\documentclass[aos]{imsart}

\RequirePackage{amsthm,amsmath,amsfonts,amssymb}
\RequirePackage[authoryear]{natbib}
\RequirePackage[colorlinks,citecolor=blue,urlcolor=blue]{hyperref}
\RequirePackage{graphicx}
\RequirePackage[T1]{fontenc}

\makeatletter 
\def\journal@name{}
\makeatother

\startlocaldefs

\usepackage[utf8]{inputenc}
\usepackage{mathtools}
\usepackage{multirow}
\usepackage{enumerate}
\usepackage{comment}
\usepackage{bbm}
\usepackage{arydshln}
\usepackage{float}
\usepackage{adjustbox}
\usepackage{dsfont}
\usepackage{bm}
\usepackage{paralist}
\usepackage[dvipsnames]{xcolor}
\usepackage{booktabs}
\usepackage{makecell}
\usepackage{etoc}
\usepackage{bibunits} 

\numberwithin{equation}{section}
\theoremstyle{plain}
\newtheorem{theorem}{Theorem}[section]

\newtheorem{lemma}[theorem]{Lemma}
\newtheorem{corollary}[theorem]{Corollary}

\theoremstyle{definition}
\newtheorem{definition}[theorem]{Definition}

\newtheorem{remark}[theorem]{Remark}

\allowdisplaybreaks[4]

\newcommand{\E}{\mathbb{E}}
\newcommand{\R}{\mathbb{R}}

\newcommand{\Z}{\mathbb{Z}}
\newcommand{\N}{\mathbb{N}}
\newcommand{\eps}{\varepsilon}

\newcommand{\Hb}{\mathbb{H}}

\newcommand{\Hc}{\mathcal{H}}

\newcommand{\Prob}{\mathbb{P}}    
\newcommand{\Exp}{\operatorname{E}}
\newcommand{\Var}{\operatorname{Var}}
\newcommand{\Cov}{\operatorname{Cov}}

\newcommand{\FWER}{\operatorname{FWER}}
\newcommand{\argmin}{\operatornamewithlimits{\arg\min}}

\newcommand{\diff}{\mathrm{d}}

\DeclareMathOperator{\lin}{span}
\newcommand{\FAC}{\operatorname{FAC}}
\newcommand{\SNR}{\operatorname{SNR}}
\newcommand{\op}{{\operatorname{op}}}
\newcommand{\HS}{{\operatorname{HS}}}
\newcommand{\Tr}{{\operatorname{Tr}}}
\DeclareMathOperator{\id}{id}
\newcommand{\chgset}{\kappa}
\newcommand{\kexp}{\varpi}

\endlocaldefs

\begin{document}
\etocdepthtag.toc{main}
\begin{bibunit}[imsart-nameyear]
\begin{frontmatter}
\title{Simultaneous Change-Point Inference for High-Dimensional Functional Time Series}
\runtitle{Simultaneous Change-Point Inference}

\begin{aug}
\author[A]{\fnms{Axel}~\snm{Bücher}\ead[label=e1]{axel.buecher@rub.de}\orcid{0000-0002-1947-1617}}
\author[A]{\fnms{Colin}~\snm{Decker}\ead[label=e2]{david.decker@rub.de}}
\address[A]{Ruhr-Universität Bochum, Fakultät für Mathematik\printead[presep={,\ }]{e1,e2}}
\end{aug}

\begin{abstract}
We develop a framework for simultaneous change-point inference of high-dimensional functional time series. The observations are modeled as temporally dependent vectors whose coordinates take values in possibly different separable Hilbert spaces, thereby covering a broad class of functional data. Heterogeneous mean changes may occur at coordinate-specific locations, and the contemporaneous dependence across coordinates is left unrestricted.
Our procedure is based on coordinatewise cumulative-sum statistics and a residual block multiplier bootstrap that provides a common critical value for the global test and all coordinatewise decisions. We establish a nonasymptotic Gaussian approximation, quantitative nonasymptotic bounds for strong family-wise error control under arbitrary mixtures of changed and unchanged coordinates, covariance-adaptive detection guarantees, and simultaneous high-probability bounds for change-point localization. The bounds accommodate high-dimensional regimes in which the number of functional coordinates grows exponentially in a power of the sample size. We investigate finite-sample performance in simulations and illustrate the method using river discharge curves and high-frequency financial log returns.
\end{abstract}
\begin{keyword}[class=MSC]
\kwdgroup[type=primary]{\kwd{62G10}
\kwd{62R10}}
\kwdgroup[type=secondary]{\kwd{62M10}}
\end{keyword}

\begin{keyword}
\kwd{Change-point detection}
\kwd{CUSUM statistics}
\kwd{Family-wise error control}
\kwd{Functional time series}
\kwd{High-dimensional inference}
\end{keyword}

\end{frontmatter}


\section{Introduction}
\label{sec:intro}

Modern applications increasingly produce large collections of functional observations recorded sequentially over time. Examples include environmental monitoring, where measurements from many sensors or stations are recorded continuously; financial markets, where intraday price trajectories are observed for large collections of assets; biomedical applications involving repeated functional imaging signals; and industrial monitoring systems recording multiple streams of functional measurements. Such data combine temporal dependence, functional observations, and high dimensionality, where the latter refers to the number of functional coordinates observed at each time rather than to the dimension of the domain of a single function. 
A fundamental question in analyzing such systems is whether their underlying structure remains stable over time, with one important form of instability being the occurrence of abrupt structural changes. Beyond detecting whether such a change has occurred somewhere in the system, scientific interest often lies in identifying which functional coordinates have changed, when the respective changes occurred, and how confidently these conclusions can be drawn simultaneously.

Answering these questions is challenging. Temporal dependence must be accounted for when calibrating the inferential procedure, while high dimensionality requires simultaneous control across a potentially large number of functional coordinates. At the same time, the coordinates may exhibit strong and otherwise unrestricted contemporaneous dependence, so that they cannot generally be treated as independent. Different coordinates may also undergo changes of different magnitudes, in different functional directions, and at different time points. Testing each coordinate separately does not account for multiplicity, whereas a global test alone cannot identify the coordinates responsible for a rejection. A simultaneous inferential procedure should therefore identify changed coordinates while controlling the overall probability of false detections and, at the same time, provide guarantees for the estimated change-point locations.

In general, change-point analysis for functional time series has received considerable attention. However, the literature has largely focused on inferential objectives different from the simultaneous coordinate-specific problem described above. Fully functional methods for detecting changes in the mean of a single functional time series have been developed by \citet{Horvath2014,Aue2018,Bastian2025,kumar2024estimation}, among others,  and, more generally in a Hilbert-space framework, by \citet{sharipov2016sequential}. Methods based on dimension reduction, including functional principal component analysis, were developed by \citet{berkes2009detecting,MR3147328,DetteKutta2021,MR2922864,MR4176154}. Multiple change-point detection for functional time series has also received considerable attention; see, for example, \citet{3600270.3602960,bastian2024multiple,kutta2025multiscale,Bai2026}. These contributions primarily concern detecting or localizing changes over time within one or more functional processes.

Two contributions are particularly close to the simultaneous inference problem described above, albeit from different perspectives. From the perspective of the inferential objective, \citet{jirak2015uniform} develops simultaneous coordinatewise change-point methodology for high-dimensional real-valued time series. From the perspective of the data structure, \citet{li2024detection} studies structural breaks in high-dimensional functional time series in \(L^2[0,1]\), including global detection and estimation of heterogeneous break-point locations under a latent group structure. Further related work on simultaneous inference includes \citet{ZD23}, who construct simultaneous confidence bands for the marginal mean functions of high-dimensional functional time series rather than addressing structural changes.

To the best of our knowledge, however, simultaneous coordinate-specific change-point inference for high-dimensional functional time series, with rigorous guarantees for both detection and localization, has not been addressed in the literature. In this paper, we develop a general framework for this problem.
We consider temporally dependent observations consisting of \(K\) (functional) coordinates, where different coordinates may take values in different separable Hilbert spaces, undergo mean changes of different functional forms and magnitudes at different time points, and exhibit unrestricted contemporaneous dependence. For each coordinate, we construct a cumulative-sum (CUSUM) statistic in the corresponding Hilbert-space norm and use the maximum of these marginal statistics for simultaneous calibration. To account for temporal and cross-coordinate dependence, we employ a residual block multiplier bootstrap in which estimated coordinate-specific mean changes are removed before resampling and common block multipliers are used across coordinates. The bootstrap yields a single data-dependent critical value that is used both for the global test and for all coordinatewise decisions, while change-point locations are estimated from the marginal CUSUM processes. A step-down refinement further improves the coordinatewise decisions by iteratively recalibrating the bootstrap critical value over the set of hypotheses not yet rejected, thereby allowing for coordinate-specific effective critical values.

The proposed methodology is accompanied by three principal nonasymptotic inferential guarantees, all of which accommodate high-dimensional settings in which the number of functional coordinates may be substantially larger than the sample size, including asymptotic regimes in which $K$ grows exponentially in a power of $n$.

\begin{enumerate}
\item \emph{Strong family-wise error control}
(Theorem~\ref{thm:fwer_control}).
We derive a quantitative nonasymptotic bound on the probability of making at least one false coordinatewise change detection, uniformly over arbitrary mixtures of changed and unchanged coordinates. In particular, the simultaneous error guarantee is not restricted to the global null hypothesis. 

\item \emph{Covariance-adaptive simultaneous detection}
(Theorem~\ref{thm:profiled_simultaneous_power}).
We establish nonasymptotic lower bounds on the probability of simultaneously detecting collections of changed coordinates in high-dimensional settings. The sufficient detection threshold is expressed through quantiles of a Gaussian reference process and therefore adapts to both temporal dependence and contemporaneous dependence across coordinates, rather than replacing the effect of simultaneous calibration by a generic worst-case function of $K$.

\item \emph{Simultaneous localization guarantees}
(Theorem~\ref{thm:linear_process_localization}).
We derive nonasymptotic high-probability bounds for the errors of the estimated change-point locations, simultaneously over collections of changed coordinates. The resulting bounds are signal-adaptive and incur only a logarithmic dependence on the number of coordinates being localized.

\end{enumerate}

A key ingredient underlying the first two results is a nonasymptotic Gaussian approximation for the global maximum CUSUM statistic under the global null hypothesis (Theorem~\ref{thm:gaussian_approximation}). By retaining the full cross-coordinate covariance structure, the Gaussian reference distribution provides the theoretical basis for the covariance-adaptive bootstrap calibration and power analysis.

We investigate the finite-sample behavior of the proposed procedure in a simulation study designed to demonstrate empirically its strong control of the family-wise error rate and covariance-adaptive coordinatewise detection. As a secondary objective, we benchmark its performance as a global max-type procedure against the aggregation-based PE-CUSUM method of \citet{li2024detection}.
We further illustrate the methodology in two data applications involving annual river-discharge curves from European measurement stations and cumulative intraday log-return curves for constituents of the Dow Jones Industrial Average.

The remainder of the paper is organized as follows. Section~\ref{sec:methodology} introduces the proposed methodology, and Section~\ref{sec:theoretical-guarantees} establishes its theoretical guarantees. Section~\ref{subsec:tuning-parameter-choice} provides practical guidance on the choice of tuning parameters. Section~\ref{sec:simulation-study} presents the simulation study, and Section~\ref{sec:case-study} contains the two empirical applications. All proofs are deferred to the supplementary material.

\section{Statistical Methodology}
\label{sec:methodology}

\subsection{Observation Scheme and Objectives}

For $k\in[K]=\{1,\dots,K\}$, let $(\Hb_k,\langle\cdot,\cdot\rangle_{\Hb_k})$ be a separable Hilbert space, for example $\R^d$ or $L^2[0,1]$. 
We consider a sequence of observations $f^{(1)},\ldots,f^{(n)}$ taking values in the product Hilbert space
\begin{equation}
\label{eq:mathcalH-decomposition}
    \mathcal H
    =
    \bigoplus_{k\in[K]}\Hb_k.
\end{equation}
Specifically, for each $i\in[n]$,
\[
    f^{(i)}
    =
    \big(f_1^{(i)},\ldots,f_K^{(i)}\big),
    \qquad
    f_k^{(i)}\in\Hb_k.
\]
Our framework accommodates high-dimensional settings in which the number $K$ of Hilbert-space-valued coordinates may be substantially larger than the sample size $n$. 
We assume that the observations follow a change-point model driven by a centered, strictly stationary noise process
\begin{equation}
\label{def:global_noise}
    \varepsilon
    =
    \big(\varepsilon^{(i)}\big)_{i\in\mathbb Z},
    \qquad
    \varepsilon^{(i)}
    =
    \big(\varepsilon_1^{(i)},\ldots,\varepsilon_K^{(i)}\big)
    \in\mathcal H,
\end{equation}
such that, for $i\in[n]$ and $k\in[K]$,
\begin{equation}
\label{eq:change_point_model}
    f_k^{(i)}
    =
    \mu_k
    +
    \delta_k\mathbf 1\{i>\omega_k\}
    +
    \varepsilon_k^{(i)}.
\end{equation}
Here, $\mu_k,\delta_k\in\Hb_k$ denote the baseline mean and mean change
function, respectively, while $\omega_k \in [n-1]$ is the change-point location
for the $k$th coordinate. If $\delta_k=0$, then $\omega_k$ is not
identifiable, and the $k$th coordinate is said to have no mean change.  The coordinates of each $\varepsilon^{(i)}$ may be contemporaneously dependent, and no structural model is imposed on this cross-coordinate dependence by our assuptions below.

The model allows both the change functions and the change-point locations
to vary across coordinates. It also permits the marginal spaces $\Hb_k$
to differ. For example, some coordinates may take values in
infinite-dimensional function spaces such as $L^2[0,1]$, while others
may belong to finite-dimensional Euclidean spaces or to Hilbert spaces
arising from transformed compositional data
\citep{aitchison1982statistical}. This flexibility is useful in
applications where different components of a multivariate observation
possess different structure. For brevity, we refer to all coordinates as functional coordinates, including the finite-dimensional special cases.

Define
\begin{equation}
\label{eq:change-point-coordinates}
\mathcal S
=
\{k\in[K]:\delta_k\neq0\},
\qquad
\mathcal S^c
=
\{k\in[K]:\delta_k=0\},
\end{equation}
to be the sets of coordinates with and without a mean change, respectively. The set $\mathcal S$ is unknown, and our primary objective is to recover it through simultaneous inference on the marginal hypotheses
\begin{equation*}
H_0^k:\delta_k=0
\qquad\text{versus}\qquad
H_A^k:\delta_k\neq0,
\qquad
k\in[K],
\end{equation*}
while strongly controlling the family-wise error rate at a prescribed level $\gamma\in(0,1)$. At the same time, the procedure should retain high power for detecting coordinates with genuine mean changes. For coordinates that undergo a change, we further seek to estimate the corresponding change-point locations $\omega_k$ and to establish simultaneous guarantees for the localization errors of the resulting estimators.
The methodology developed below is designed to address these objectives simultaneously. As a by-product, it provides a test for the global null hypothesis \citep{li2024detection}
\begin{align}
\label{eq:global-hypothesis}
H_0^{\infty}
:
\delta_k=0
\quad\text{for all }k\in[K],
\end{align}
or, equivalently, $\mathcal S=\varnothing$.

\subsection{Simultaneous testing and change-point estimation}
\label{subsec:testing-localization}
For each coordinate $k\in[K]$, evidence against $H_0^k$ is quantified by a marginal CUSUM process. For $s\in[n]$, define
\begin{equation}
\label{def:cusum}
C_{n,k}(s)
=
\frac{1}{\sqrt n}
\sum_{i=1}^{s}
\bigl(
f_k^{(i)}-\bar f_k
\bigr)
\in\Hb_k,
\qquad
\bar f_k
=
\frac1n
\sum_{i=1}^n
f_k^{(i)}.
\end{equation}
The corresponding marginal test statistic is
\begin{equation}
\label{def:marginal_test}
T_{n,k}^2
=
\max_{s\in[n]}
\|C_{n,k}(s)\|_{\Hb_k}^2,
\end{equation}
and the global max statistic is
\begin{equation}
\label{def:global_test}
T_n^2
=
\max_{k\in[K]}
T_{n,k}^2.
\end{equation}
Large values of $T_{n,k}^2$ provide evidence of a mean change in coordinate $k$, while $T_n^2$ aggregates this evidence through the maximum over all coordinates.

Suppose, for the moment, that the $(1-\gamma)$-quantile of the null distribution of $T_n^2$ was known, and denote it by $Q_{1-\gamma}$. 
A natural simultaneous procedure would then be
\[
\text{Reject } H_0^k
\qquad\text{whenever}\qquad
T_{n,k}^2>Q_{1-\gamma}.
\]
Conceptually, the multiplicity adjustment is thus achieved by a global threshold determined by the null distribution of the maximum of the marginal statistics, rather than by calibrating the coordinates separately.

In practice, the required quantile $Q_{1-\gamma}$ is unknown. We estimate it using the residual block multiplier bootstrap introduced in Section~\ref{subsec:bootstrap} below. Let $\smash{\widehat Q_{1-\gamma}}$ denote the resulting conditional $(1-\gamma)$ bootstrap quantile. The rejection set is then
\begin{equation}
\label{eq:def-hat-S_gamma}
\widehat{\mathcal S}_\gamma
:=
\left\{
k\in[K]:
T_{n,k}^2>\widehat Q_{1-\gamma}
\right\},
\end{equation}
which can also be regarded as an estimator of the unknown set $\mathcal S$ of coordinates containing a mean change.

Finally, for each coordinate $k\in[K]$, define the CUSUM change-point estimator as the first maximizer of the marginal CUSUM process,
\begin{equation}
\label{eq:definition-hat-omega-k}
\widehat\omega_k
=
\min\Bigl\{
s\in[n]:
\|C_{n,k}(s)\|_{\Hb_k}^2
=
\max_{t\in[n]}
\|C_{n,k}(t)\|_{\Hb_k}^2
\Bigr\}.
\end{equation}
When $k\in\widehat{\mathcal S}_\gamma$, we report $\widehat\omega_k$ as the estimated location of the corresponding mean change.

Section~\ref{sec:theoretical-guarantees} establishes the theoretical properties of the proposed methodology, including family-wise error control, covariance-adaptive detection guarantees, and simultaneous localization guarantees for the estimated change-point locations.

\subsection{Bootstrap calibration}
\label{subsec:bootstrap}
The bootstrap procedure combines a big-block--small-block construction with residualization designed to remain valid under both the global null hypothesis and change-point alternatives. The block lengths are
denoted by $q,r\in\mathbb N$, where $1\le r\le q\le n/4$, with $q$ and $r$ corresponding to the lengths of the big and small blocks,
respectively. Let
\begin{equation}
\label{def:m}
m:=\left\lfloor\frac{n}{q+r}\right\rfloor.
\end{equation}
For $\ell\in[m]$, define
\begin{align}
\label{def:big_little_blocks}
I_\ell&=\{(\ell-1)(q+r)+1,\ldots,(\ell-1)(q+r)+q\},\\
J_\ell&=\{(\ell-1)(q+r)+q+1,\ldots,\ell(q+r)\},\nonumber
\end{align}
to be the corresponding big and small blocks. The small blocks separate successive big blocks and are omitted from the bootstrap sums introduced below. This construction reduces the effect of temporal dependence between the retained blocks, while allowing the big blocks to preserve the serial dependence within each functional coordinate.

To prevent an estimated mean change from `contaminating' the bootstrap, we estimate the mean separately before and after the estimated
change-point, excluding a neighborhood of its estimated location. This approach to removing contamination was also used by \cite{jirak2015uniform}. Let
$g\ge0$ denote the trimming radius and let $\tau\in[0,1/2)$ be a
boundary clipping parameter. Define the clipped change-point estimator
\begin{equation}
\label{def:trimmed_omega_k}
\widetilde\omega_k:=\min\{(1-\tau)n,\max\{\tau n,\widehat\omega_k\}\},
\qquad k\in[K].
\end{equation}
Next define the block-aligned trimming indices
\begin{align}
\label{def:clipped-LR-general}
N_{L,k} &:=(q+r)\left\lfloor\frac{\widetilde\omega_k-g}{q+r}\right\rfloor,
\qquad
N_{R,k}:=(q+r) \left\lceil\frac{\widetilde\omega_k+g}{q+r}\right\rceil,
\end{align}
and left and right trimmed means 
\begin{align}
\widehat\mu_{L,k}&:=\frac{1}{N_{L,k}}\sum_{i=1}^{N_{L,k}}f_k^{(i)},
\qquad
\widehat\mu_{R,k}:=\frac{1}{n-N_{R,k}}\sum_{i=N_{R,k}+1}^{n}f_k^{(i)}.
\end{align}
If $\tau>0$ and $g+q+r\le\tau n/2$, then
\[
N_{L,k}\ge\frac{\tau n}{2},\qquad
n-N_{R,k}\ge\frac{\tau n}{2},\qquad k\in[K],
\]
so that both trimmed means are well defined and are based on sample sizes of the same order as $n$. Using these trimmed means, define the residuals
\begin{equation}
\label{def:trimmed-residuals-clipped}
\widehat\varepsilon_{i,k}=
\begin{cases}
f_k^{(i)}-\widehat\mu_{L,k},&1\le i\le N_{L,k},\\[0.4em]
0,&N_{L,k}<i\le N_{R,k},\\[0.4em]
f_k^{(i)}-\widehat\mu_{R,k},&N_{R,k}<i\le n.
\end{cases}
\end{equation}
Residuals in the trimmed region are thus set to zero so that observations near the estimated change-point do not contribute to the bootstrap sample. The purpose of this residualization is to reconstruct, as closely as possible, the fluctuations that would be observed after removal of the coordinate-specific mean changes. 
The trimming step limits the effect of change-point estimation error on the residuals, while clipping ensures that the left and right mean estimators are computed from sufficiently many observations.

Let $e_1,\ldots,e_m$ be independent standard normal random variables,
independent of the data. For $s\in[n]$ and $k\in[K]$, define the
bootstrap partial sum process
\begin{equation}
\label{def:bootstrap-partial-sum}
S_k^*(s):=\sum_{\ell=1}^{m}e_\ell
\sum_{i\in I_\ell}\widehat\varepsilon_{i,k}\mathbf 1\{i\le s\},
\end{equation}
and the corresponding bootstrap CUSUM process
\begin{equation}
\label{def:bootstrap-cusum}
C_k^*(s):=\frac{1}{\sqrt{mq}}
\left\{S_k^*(s)-\frac{s}{n}S_k^*(n)\right\}.
\end{equation}
Importantly, the same multiplier $e_\ell$ is applied to the $\ell$th
block for every coordinate $k$. Consequently, the bootstrap does not
resample the coordinates independently; it preserves the empirical
cross-coordinate dependence carried by the block sums. This feature
allows the resulting critical value to adapt to the contemporaneous
dependence structure without requiring an explicit model for the
cross-coordinate covariance.

The bootstrap analogue of the global statistic is
\begin{equation}
\label{eq:definition-Tstar}
\left(T_n^*\right)^2
=\left(T_n^*(g,q,r,\tau)\right)^2
:=\max_{k\in[K]}(T_{n,k}^*)^2,
\qquad
(T_{n,k}^*)^2:=\max_{s\in[n]}\|C_k^*(s)\|_{\Hb_k}^2.
\end{equation}
Thus, the bootstrap reproduces both sources of dependence relevant for
simultaneous inference: temporal dependence is retained through the
block sums, while cross-coordinate dependence is retained by using
common multipliers across coordinates. At the same time, the
coordinate-specific residualization removes the estimated mean changes,
so that the bootstrap distribution can be used for calibration under
arbitrary mixtures of changed and unchanged coordinates.

Finally, let $\widehat Q_{1-\gamma}$ denote the conditional $(1-\gamma)$-quantile of $(T_n^*)^2$ given the observed data. 
In practice, this quantile is approximated from $B$ independent bootstrap replicates.

\subsection{Step-down refinement}
\label{subsec:stepdown}
The single-step procedure of Section~\ref{subsec:testing-localization} can be refined using the resampling-based step-down principle of \citet{RomanoWolf2005}. For every nonempty set of coordinates $\mathcal J\subseteq[K]$, define, in analogy with~\eqref{eq:definition-Tstar},
\[
(T_{n,\mathcal J}^*)^2
:=
\max_{k\in\mathcal J}(T_{n,k}^*)^2,
\]
and let $\widehat Q_{1-\gamma}(\mathcal J)$ denote its conditional $(1-\gamma)$-quantile given the observed data.

Starting with $\mathcal R_1=[K]$, let $\mathcal R_j$ denote the set of coordinates not rejected before step $j$. At step $j$, reject all marginal null hypotheses $H_0^k$, $k\in\mathcal R_j$, for which
\[
T_{n,k}^2>
\widehat Q_{1-\gamma}(\mathcal R_j),
\]
and set
\[
\mathcal R_{j+1}
=
\mathcal R_j
\setminus
\left\{
k\in\mathcal R_j:
T_{n,k}^2>
\widehat Q_{1-\gamma}(\mathcal R_j)
\right\}.
\]
The procedure is repeated until no further rejection occurs, and the
resulting rejection set is denoted by
$\smash{\widehat{\mathcal S}^{\mathrm{SD}}_\gamma}$. The same bootstrap replications
are used throughout, with only the set of coordinates entering the bootstrap
maximum changing across steps. Corollary~\ref{cor:stepdown} shows that the
step-down refinement retains the strong family-wise error control of the
single-step procedure while providing detection guarantees that are at least
as strong.

\section{Theoretical Guarantees}
\label{sec:theoretical-guarantees}

This section establishes the theoretical guarantees underlying the
proposed methodology. 

\subsection{Assumptions on the data-generating process}
\label{subsec:assumptions}

Our theoretical results are established under two sets of conditions on the noise process. The first consists of Assumptions~\ref{cond:stationary} and~\ref{cond:scores}, which impose polynomial $\beta$-mixing and polynomial eigenvalue decay of the marginal long-run covariance operators. We control the familywise error of the proposed procedure under these assumptions and also show that it detects local alternative changes under them.  The second is Assumption~\ref{cond:linear}, which assumes a linear-process representation and is used only for the localization theory. The reason for using two different assumptions is discussed in greater detail below.

Throughout, $\|X\|_{\psi_1}$ denotes the Orlicz--1 norm of a real-valued random variable; see Section~\ref{sec:orlicz} for its definition and basic properties.

\begin{enumerate}
\renewcommand{\theenumi}{[A1]}
\renewcommand{\labelenumi}{\theenumi}

\item
\label{cond:stationary}
The $\Hc$-valued noise variables $\eps^{(1)}, \dots, \eps^{(n)}$ are an excerpt from a centered, strictly stationary series $(\eps^{(i)})_{i \in \Z}$ with $\beta$-mixing coefficients $(\beta_h)_{h\in\mathbb N}$ (Definition~\ref{def:beta_mixing_for_sequence}). There exist constants $D_\psi>0$, $C_\beta>0$ and $c_\beta>2$ such that
\[
\max_{k\in[K]}
\bigl\|
\|\varepsilon_k^{(1)}\|_{\Hb_k}
\bigr\|_{\psi_1}
\le
D_\psi,
\qquad
\beta_h
\le
C_\beta h^{-c_\beta},
\qquad
h\in\mathbb N.
\]
\end{enumerate}

Under Assumption~\ref{cond:stationary}, each marginal process
$(\varepsilon_k^{(i)})_{i\in\mathbb Z}$ possesses a self-adjoint,
positive semidefinite, trace-class long-run covariance operator
$\mathcal K_k$ satisfying
\[
\lim_{N\to\infty} \Bigl\|
\Cov
\Bigl( N^{-1/2} \sum_{i=1}^N \varepsilon_k^{(i)} \Bigr)
-
\mathcal K_k \Bigr\|_{\Tr(\Hb_k)} = 0,
\]
where $\|\cdot\|_{\Tr}$ denotes the trace norm (Lemma~\ref{lem:convergence_to_long_term_trace}). Let $\lambda_{k1}\ge\lambda_{k2}\ge\cdots$ denote the nonzero eigenvalues of $\mathcal K_k$, and let $Z_{k1},Z_{k2},\ldots$ be corresponding orthonormal eigenvectors. Their number is $r_k=\operatorname{rank}(\mathcal K_k)$, which may be finite or infinite. We assume throughout that $r_k\ge1$.

\begin{enumerate}
\renewcommand{\theenumi}{[A2]}
\renewcommand{\labelenumi}{\theenumi}

\item
\label{cond:scores}
There exist constants
$0<C_L\le C_U$
and
$1/2<\gamma_U\le\gamma_L$
such that
\[
C_L^2j^{-2\gamma_L}
\le
\lambda_{kj}
\le
C_U^2j^{-2\gamma_U},
\qquad
k\in[K],
\quad
j\in[r_k]\cap\mathbb N.
\]
\end{enumerate}

To prove our principal localization result, we instead work under the following linear-process assumption, which then replaces \ref{cond:stationary} and \ref{cond:scores}.

\begin{enumerate}
\renewcommand{\theenumi}{[B]}
\renewcommand{\labelenumi}{\theenumi}

\item
\label{cond:linear}
There exist constants $C_A,D_R>0$ such that, for every $k\in[K]$,
\[
\varepsilon_k^{(i)}
=
\sum_{j=0}^{\infty}
A_{k,j}\eta_k^{(i-j)},
\qquad
i\in\mathbb Z,
\]
where each $A_{k,j}:\Hb_k\to\Hb_k$ is a bounded linear operator satisfying
\[
\max_{k\in[K]}
\sum_{j=0}^{\infty}
(j+1)
\|A_{k,j}\|_{\op}
\le
C_A,
\]
and $(\eta^{(i)})_{i \in \Z}$ is a centered i.i.d. sequence with $\eta^{(i)} = (\eta_1^{(i)},\ldots,\eta_K^{(i)}) \in\mathcal H$ satisfying
\[
\max_{k\in[K]}
\Bigl\|
\|\eta_k^{(0)}\|_{\Hb_k}
\Bigr\|_{\psi_1}
\le
D_R.
\]
\end{enumerate}

Assumptions~\ref{cond:stationary} and~\ref{cond:scores} are standard in high-dimensional statistics and functional time series analysis. Assumption~\ref{cond:stationary} requires polynomially decaying temporal
dependence together with exponential tails of the marginal
Hilbert-space norms, while
Assumption~\ref{cond:scores} imposes polynomial eigenvalue decay of the
marginal long-run covariance operators.
Similar conditions appear in
\citet{jirak2016optimal,li2024detection,cai2006prediction,HK07,lu2022almost,lopes2025improved}.
Together they permit nonasymptotic Gaussian approximation and bootstrap results in high-dimensional settings, including regimes in which $K$ may grow  exponentially in a power of the sample size.

Assumption~\ref{cond:linear} replaces Assumptions~\ref{cond:stationary} and~\ref{cond:scores} when proving our principal localization result in Theorem~\ref{thm:linear_process_localization}. Its short-range linear dependence structure allows us to derive substantially sharper nonasymptotic localization bounds than under generic polynomial $\beta$-mixing. While the latter framework is sufficiently general for establishing consistency of the residual bootstrap procedure, the resulting localization rates are too conservative to provide a meaningful benchmark result.
Assumptions of a similar type are employed by \citet{li2024detection,jirak2018rate}.

\subsection{Gaussian approximation}
\label{subsec:gaussian-approximation}

We start by presenting a Gaussian approximation result for the  global test statistic $T^2_n$ from \eqref{def:global_test}. This approximation identifies the covariance-adaptive reference distribution underlying the simultaneous testing procedure and provides the theoretical benchmark for the bootstrap calibration developed below.
Let
\begin{align}
\label{eq:model-class-P}
\mathcal P
=
\mathcal P(D_\psi,C_U,C_L,\gamma_U,\gamma_L,C_\beta,c_\beta)
\end{align}
denote the class of data-generating distributions satisfying the change-point model \eqref{eq:change_point_model}, where the noise process fulfills Assumptions~\ref{cond:stationary} and~\ref{cond:scores} with the stated constants. Let
\begin{equation}\label{eq:model-class-P_0}
\mathcal P_0
:=
\left\{
P\in\mathcal P:
H_0^\infty\text{ holds under }P
\right\}.
\end{equation}
denote the subclass of distributions for which the global null is met, where $H_0^\infty:\delta_k=0\text{ for all }k\in[K]$.

Throughout the sequel, we occasionally attach the index $P$ to
quantities that depend on the data-generating distribution. For example,
we write
\[
\mathcal S_P=\{k\in[K]:\delta_k(P)\neq0\},
\qquad
\mathcal S_P^c=\{k\in[K]:\delta_k(P)=0\},
\]
for the sets of coordinates with and without a mean change, respectively, introduced in \eqref{eq:change-point-coordinates}.
For each $P\in\mathcal P$, let $\mathcal K=\mathcal K_P$ denote the
global long-run covariance operator of the global noise process
$\varepsilon=(\varepsilon^{(i)})_{i\in\mathbb Z}$. By
Assumption~\ref{cond:stationary}, $\mathcal K$ is a self-adjoint,
positive semidefinite trace-class operator on $\mathcal H$. Let
\begin{equation}
\label{def:Hilbert_Brownian_Bridge}
\mathbb B_{\mathcal K}
=
\{\mathbb B_{\mathcal K}(t):t\in[0,1]\}
=
\{
(\mathbb B_{\mathcal K,1}(t),\ldots,\mathbb B_{\mathcal K,K}(t))
:t\in[0,1]
\}
\end{equation}
denote a centered $\mathcal H$-valued Brownian bridge with covariance
operator $\mathcal K$, that is,
\[
\mathbb E\,
\mathbb B_{\mathcal K}(t)
=
0,
\qquad
\operatorname{Cov}
\bigl(
\mathbb B_{\mathcal K}(s),
\mathbb B_{\mathcal K}(t)
\bigr)
=
\{\min(s,t)-st\}\mathcal K,
\qquad
s,t\in[0,1].
\]
Recall the global test statistic $T_n$ from \eqref{def:global_test}.

\begin{theorem}[Uniform Gaussian approximation under the global null]
\label{thm:gaussian_approximation}
Fix \(\rho>0\) and recall $\mathcal P_0$ from \eqref{eq:model-class-P_0}; note that this means that \ref{cond:stationary} and \ref{cond:scores} holds. There exist constants $c,d,C>0$ such that, for every \(n\ge4\), \(K\ge1\) and $P \in \mathcal P_0$, we have 
\begin{align}
&\sup_{t \geq \rho}
\Big|
\Prob_P\big(T_n^2 \leq t\big)
-
\Prob\Big(
\max_{k \in [K]}
\sup_{u \in [0,1]}
\big\|
\mathbb{B}_{\mathcal{K},k}(u)
\big\|_{\Hb_k}^{2}
\leq t
\Big)
\Big| \phantom{|}
\le
C n^{-c}\log^d(n\vee K).
\label{eq:main_gaussian_approx}
\end{align}
Here, the constants $c$ and $d$ only depend on $\gamma_L$ and $\gamma_U$, while $C$ depends on $\rho$ as well as on all constants from \ref{cond:stationary} and \ref{cond:scores} as specified in \eqref{eq:model-class-P}.
\end{theorem}

Importantly, in an asymptotic framework where $K=K_n$ may depend on $n$, the upper bound in \eqref{eq:main_gaussian_approx} converges to zero if $\log K_n=o(n^{c/d})$. In particular, the Gaussian approximation remains valid even when $K_n$ grows nearly exponentially fast in a  power of $n$. 

Parseval's equality implies the almost sure representation
\[\max_{k\in[K]}\sup_{u\in[0,1]}
\left\|
\mathbb B_{\mathcal K,k}(u)
\right\|_{\Hb_k}^2=\max_{k\in[K]}
\sup_{u\in[0,1]}
\sum_{j=1}^{r_k} \lambda_{kj}\eta_{kj}^2(u),
\]
where $r_k$ is the rank of the marginal long-run covariance operator $\mathcal K_k$ as defined below Assumption~\ref{cond:stationary} and where \(\{\eta_{kj}(u):u\in[0,1]\}\) are  standard Brownian bridges that are generally dependent across the $K$ coordinates.
Therefore, Theorem~\ref{thm:gaussian_approximation} is equivalent to (when $r_k=\infty$ for all $k \in [K]$)
\begin{equation*}
\sup_{t\ge \rho}
\bigg|
\mathbb P(T^2_{n}\le t)
-
\mathbb P\Big(
\max_{k\in[K]}
\sup_{u\in[0,1]}
\sum_{j=1}^\infty \lambda_{kj}\eta_{kj}^2(u)
\le t
\Big)
\bigg|
\le
C_\rho \log^d(n\vee K)n^{-c}. 
\end{equation*}
For $K=1$, asymptotic versions of such `Hilbertian Brownian Bridge' approximations are
well known and form a standard tool in functional change-point
analysis. Corresponding results under various notions of temporal dependence have been established, for example, by \citet{sharipov2016sequential,Aue2018,li2024detection}. In this case, the Gaussian reference distribution is determined solely by the eigenvalues of the marginal long-run covariance operator, so its quantiles can be approximated by Monte Carlo simulation after estimating sufficiently many eigenvalues; see Section~2.3 of \citet{Aue2018}.

The high-dimensional setting of Theorem~\ref{thm:gaussian_approximation} is fundamentally different. For $K>1$, the Gaussian reference distribution depends not only on the eigenvalues ${\lambda_{kj}}$, but also on the cross-coordinate covariance structure of the Brownian-bridge array. Consequently, its quantiles adapt to the contemporaneous dependence across coordinates, which is the sense in which the approximation is covariance-adaptive.

Theorem~\ref{thm:gaussian_approximation} therefore does not by itself yield a practical calibration method in high dimensions. Since we impose no structural assumptions on the contemporaneous dependence, direct simulation from the Gaussian reference distribution would require estimating a potentially very high-dimensional cross-covariance structure, which is particularly challenging when $K\gg n$. Instead, the theorem provides a theoretical reference distribution for the global statistic and serves as an intermediate approximation in our bootstrap and power analyses. We next show that the bootstrap procedure of Section~\ref{subsec:bootstrap} consistently approximates the relevant quantiles without explicitly estimating this cross-covariance structure.

\subsection{Bootstrap validity and family-wise error control}
\label{subsec:FWER-control}

Fix $\gamma \in (0,1)$ and $P \in \mathcal P$, and recall the true and estimated change point locations $\mathcal S = \mathcal S_P$ and $\widehat{\mathcal S}_\gamma$ from \eqref{eq:change-point-coordinates} and \eqref{eq:def-hat-S_gamma}, respectively. Define the familywise error rate as
\begin{align*}
\FWER(P,\gamma)
=
\mathbb P_P\left(
\widehat{\mathcal S}_\gamma\cap \mathcal S_P^c\neq\varnothing
\right).
\end{align*}

\begin{theorem}[Strong family-wise error control]
\label{thm:fwer_control}
Fix \(\gamma\in(0,1)\) and recall $\mathcal P$ and $\mathcal P_0$ from \eqref{eq:model-class-P} and \eqref{eq:model-class-P_0}, respectively. Suppose that the
residual bootstrap is implemented with clipping parameter \(\tau\), gap
length \(g=q\), and parameters \(q,r\in\mathbb N\) satisfying
\begin{equation}
\tau\in(0,1/2),
\qquad 1\le r\le q,
\qquad
2q+r\le\frac{\tau n}{2}.
\label{eq:fwer-profiled-tuning}
\end{equation}
Then there exist constants $d,C>0$ 
such that, for every \(n\ge4\), \(K\ge1\), and
\(q,r, \tau\) satisfying~\eqref{eq:fwer-profiled-tuning}, we have 
\begin{alignat}{3}
&\forall P\in\mathcal P: \qquad 
&\operatorname{FWER}(P,\gamma) - \gamma \phantom{|}
&\le
\frac {C} {\sqrt {\tau}}\log^d(nK)\mathcal R_{n,q,r},
\label{eq:fwer-profiled-strong-control}
\\
&\forall P\in\mathcal P_0: 
&\big|\operatorname{FWER}(P,\gamma)-\gamma \big|
&\le
\frac {C} {\sqrt {\tau}}\log^d(nK)\mathcal R_{n,q,r},
\label{eq:fwer-profiled-null-calibration}
\end{alignat}
where 
\begin{align}
\mathcal R_{n,q,r}
:={}&
\Big(\frac{q^2}{n}\Big)^{
\frac{2\gamma_U-1}{12\gamma_U-2}}
+q^{-\frac{2\gamma_U-1}{2\gamma_L}}
+\Big(\frac rq\Big)^{1/2}
+\frac nq r^{-c_\beta}.
\label{eq:fwer-profiled-rate}
\end{align}
Here, the constant $d$ only depends on $\gamma_L$ and $\gamma_U$, while $C$ depends on all constants from \ref{cond:stationary} and \ref{cond:scores} as specified in \eqref{eq:model-class-P} as well as on $\gamma$.
\end{theorem}

Note that \eqref{eq:fwer-profiled-strong-control} yields strong FWER control uniformly over arbitrary mixtures of changed and unchanged coordinates, while \eqref{eq:fwer-profiled-null-calibration} shows that, under the complete null hypothesis, the procedure is not merely valid, but calibrated around the nominal level, with an explicit finite-sample bound on the calibration error.

To the best of our knowledge, Theorem~\ref{thm:fwer_control} appears to be the first result providing quantitative strong family-wise error control for simultaneous change-point inference in high-dimensional functional time series while allowing for temporal dependence and unrestricted contemporaneous (cross-sectional) dependence.
In an asymptotic framework in which $K, q, r$ and $\tau$ are allowed to depend on $n$, the right-hand side of \eqref{eq:fwer-profiled-null-calibration} and \eqref{eq:fwer-profiled-rate} converges to zero for a broad range of choices. Importantly, it permits sequences for which $\tau\to0$ while $K$ grows exponentially in a power of $n$:  For example, one admissible choice is
\[
q=\lfloor n^{2/5}\rfloor,\qquad
r=\lfloor n^{1/3}\rfloor,\qquad
\tau=n^{-\rho} / 4 ,\qquad
K=\lfloor \exp(n^{\rho/(2d)})\rfloor,
\]
for any 
\[
0<\rho<
\min\left\{
\frac{2\gamma_U-1}{60\gamma_U-10},
\frac{2\gamma_U-1}{5\gamma_L},
\frac1{30},
\frac{c_\beta}{3}-\frac35
\right\}.
\]

In Section ~\ref{subsec:tuning-parameter-choice} we recommend setting $q=\lfloor Cn^{1/3} \rfloor$ for a constant $C$ that is determined by a stability criterion. This choice is always admissible. Our recommended choice of $r=\lfloor q^{1/2} \rfloor$, also discussed in Section~\ref{subsec:tuning-parameter-choice}, is simultaneously admissible if $c_\beta>4$; otherwise, it must be chosen larger in order than recommended for the final additive term in \eqref{eq:fwer-profiled-rate} to become small along the sequence.  Finally, the conclusion of Theorem~\ref{thm:fwer_control} also holds, up to constants, for the recommended choice $g=(q+r)/2$. 

\subsection{Detection guarantees}
\label{subsec:detectionguarantees}

To state our main power result, we introduce a subclass of \(\mathcal P\)
consisting of data-generating distributions whose changed coordinates
have interior change points and whose nonzero change sizes are bounded
below. 
For $k \in \mathcal S_P$, let
\begin{align}
\label{eq:def-nu_k}
    \theta_k:=\frac{\omega_k}{n},
    \qquad
    \gamma_k:=\theta_k\wedge(1-\theta_k),
    \qquad
    \Delta_k:=\|\delta_k\|_{\Hb_k},
    \qquad
    \nu_k := \gamma_k\Delta_k.
\end{align}
Occasionally, we also write $\nu_k(P), \Delta_k(P), \ldots$ to highlight the dependence on the data generating model $P$.
For $\widetilde c >0$, \(\widetilde g\in\mathbb N\) and
\(\widetilde\tau\in(0,1/2)\), define
\begin{align}
    \mathcal P_{\mathrm{loc}}(\widetilde c, \widetilde g,\widetilde\tau)
    &:=
    \bigg\{
    P\in\mathcal P \,\bigg|\, \mathcal S_P \ne \varnothing \text{ and } \forall k\in\mathcal S_P:\,
    \omega_k
    \in
    [\widetilde\tau n,(1-\widetilde\tau)n],
    \nonumber\\
    & \hspace{3cm}
    \theta_k(1-\theta_k)\Delta_k
    \ge
    \widetilde c \log(nK)
    \bigg[
        \frac1{\sqrt{\widetilde g}}
        +
        \frac{
            n^{3/[2(c_\beta+2)]}\log n
        }{\widetilde g}
    \bigg]
    \bigg\};
    \label{eq:model-class-P-loc}
\end{align}
here, we suppress the dependence on $D_\psi, C_U,C_L,\gamma_U,\gamma_L,C_\beta,c_\beta$ from the notation. Note that every $P \in \mathcal P_{\mathrm{loc}}(\widetilde c, \widetilde g,\widetilde\tau)$ satisfies \ref{cond:stationary} and~\ref{cond:scores}.
For every \(P\in\mathcal P\) and every nonempty  set
\(\mathcal J\subseteq[K]\), define the Brownian-bridge maximum
\begin{equation}
    B_{\mathcal J,P}
    :=
    \max_{k\in\mathcal J}
    \sup_{u\in[0,1]}
    \left\|
        \mathbb B_{\mathcal K,k}(u)
    \right\|_{\Hb_k}.
    \label{eq:def-B-J-P}
\end{equation}
For \(a\in(0,1)\), let
\begin{equation}
    c_{\mathcal J,P}^{\mathrm{BB}}(a)
    :=
    \inf\left\{
        z\in\mathbb R:
        \mathbb P(B_{\mathcal J,P}\le z)\ge a
    \right\}
    \label{eq:def-BB-quantile-J-P}
\end{equation}
denote its \(a\)-quantile.

\begin{theorem}[Simultaneous Power]
\label{thm:profiled_simultaneous_power}
Fix \(\gamma\in(0,1)\). Suppose that the residual bootstrap is
implemented with clipping parameter \(\tau\), gap length \(g=q\), and
block lengths \(q,r\in\mathbb N\) satisfying
\begin{equation}
    \tau\in(0,1/2),
    \qquad
    1\le r\le q,
    \qquad
    2q+r\le\frac{\tau n}{2}.
    \label{eq:power-profiled-tuning}
\end{equation} 
Then there exist constants  $c_0, d, c, C>0$
such that for every \(n\ge4\), \(K\ge1\), and \(q,r,\tau\) satisfying
\eqref{eq:power-profiled-tuning}, every
\(\eta\in(0,\gamma]\), every
\(P\in\mathcal P_{\mathrm{loc}}(c_0,q,\tau)\), and
every nonempty deterministic coordinate subset
\(\chgset\subseteq\mathcal S_P\), if
\begin{equation}
\sqrt n\,\theta_k(1-\theta_k)\Delta_k
>
c_{[K],P}^{\mathrm{BB}}
\left(
    1-\gamma+
    \{\varepsilon_{n,q,r}\wedge c_\gamma\}
\right)
+
c_{\chgset,P}^{\mathrm{BB}}(1-\eta),
\qquad k\in\chgset,
\label{eq:profiled-power-signal}
\end{equation}
then
\begin{equation}
\mathbb P_P\left(
    \chgset\subseteq\widehat{\mathcal S}_\gamma
\right)
\ge
1-\eta-2\varepsilon_{n,q,r}.
\label{eq:profiled-power-conclusion}
\end{equation}
Here
\begin{equation*}
c_\gamma
:=
\frac14\{\gamma\wedge(1-\gamma)\},
\qquad
\varepsilon_{n,q,r}
:=
\frac{C}{\sqrt\tau}
\log^{d}(nK) \big[ \mathcal R_{n,q,r}
+
n^{-c} \big],
\end{equation*}
and $\mathcal R_{n,q,r}$ is defined in~\eqref{eq:fwer-profiled-rate}. The constant $c_0$ may depend on $D_\psi, C_\beta, c_\beta$. The constants $d$ and $c$ depend only on $\gamma_L$ and $\gamma_U$, 
while $C$ depends on $\gamma$ as well as on all constants from \ref{cond:stationary} and \ref{cond:scores} as specified in \eqref{eq:model-class-P}. More specifically, $c$ may be set equal to the constant of the same name from Theorem~\ref{thm:gaussian_approximation}.
\end{theorem}
Theorem~\ref{thm:profiled_simultaneous_power} involves two distinct
signal-strength requirements, which play fundamentally different roles.
We first discuss the condition
$P\in\mathcal P_{\mathrm{loc}}(c_0,q,\tau)$ and then turn to the
detection condition~\eqref{eq:profiled-power-signal}. Membership in $\mathcal P_{\mathrm{loc}}(c_0,q,\tau)$ requires, first,
that every change point lies in the region searched by the trimmed
estimator:
\begin{align}
\label{eq:interiority-condition}
\omega_k\in[\tau n,(1-\tau)n],
\qquad k\in\mathcal S_P.
\end{align}
Second, every non-null signal must satisfy
\begin{align}
\label{eq:signal-strength-condition}
\theta_k(1-\theta_k)\Delta_k
\ge
c_0\log(nK)
\left[
\frac1{\sqrt q}
+
\frac{
n^{3/[2(c_\beta+2)]}\log n
}{q}
\right],
\qquad k\in\mathcal S_P.
\end{align}
The proof shows that this condition guarantees  with high probability 
that the trimmed estimator $\widetilde\omega_k$ from
\eqref{def:trimmed_omega_k} lies within distance $g=q$ of the true
change point. Removing the corresponding gap around
$\widetilde\omega_k$ therefore prevents the coordinate-specific mean
change from contaminating the residuals retained by the bootstrap.
The residual bootstrap can consequently approximate its oracle
counterpart based on the centered noise process, and
$\widehat Q_{1-\gamma}$ can approximate the corresponding oracle null
quantile. Thus, \eqref{eq:signal-strength-condition} is not a detection
boundary; its role is to guarantee sufficiently accurate preliminary
localization for bootstrap validity under change-point alternatives.

The sufficient condition for simultaneous detection is instead \eqref{eq:profiled-power-signal}. The first term on the right-hand side of~\eqref{eq:profiled-power-signal} reflects calibration against the
global maximum over all $K$ coordinates, whereas the second determines
the probability with which all coordinates in the target set
$\chgset$ are detected simultaneously. Both are quantiles of maxima of
the Gaussian reference process and hence depend on the cross-covariance
structure of the full high-dimensional Hilbert-space Brownian Bridge . The detection threshold is
therefore covariance-adaptive: rather than replacing the effect of
simultaneous calibration by a generic worst-case function of $K$, the
theorem retains the quantiles induced by the actual cross-coordinate
dependence structure. Although these quantiles admit generic
$\sqrt{\log K}$-type upper bounds under our model assumptions, they can
be smaller when dependence reduces the effective size of the
coordinatewise maximum. This difference can be particularly meaningful when $K$ is large compared to $n$.

The distinction between preliminary localization for the purpose of establishing bootstrap validity and signal detection is also present in \cite{jirak2015uniform}. Jirak likewise
imposes the interiority condition~\eqref{eq:interiority-condition}; see
his Equation~(3.5). His Assumption~4.3\textnormal{(B2)} contains the
asymptotic localization requirement
\[
\frac{\log n}{q\Delta_{\min,P}^2}\to0,
\qquad
\Delta_{\min,P}
:=
\min_{k\in\mathcal S_P}\Delta_k.
\]
Under the interiority condition, this is equivalent up to constants to
\[
\nu_{\min,P}
\gg
\sqrt{\frac{\log n}{q}},
\qquad
\nu_{\min,P}
:=
\min_{k\in\mathcal S_P}
\theta_k(1-\theta_k)\Delta_k.
\]
This in turn is closely related to our localization requirement
\eqref{eq:signal-strength-condition}, although our condition is stronger:
its first term requires
$\nu_{\min,P}\gtrsim\log(nK)/\sqrt q$ and it contains an additional
mixing-dependent term. These differences arise from controlling the
preliminary localization error uniformly while retaining only
logarithmic dependence on $K$ under polynomially decaying
$\beta$-mixing. In contrast, \citet{jirak2015uniform} imposes the
additional cardinality restriction
\[
|\mathcal S_P|
(q\log n)^{-p/2+2}
\lesssim
n^{-C},
\]
which permits the use of sharper $L^p$ bounds for serially dependent observations. The detection condition in Equation~(3.6) of \citet{jirak2015uniform} is
likewise comparable to~\eqref{eq:profiled-power-signal}. In particular, we
replace the generic $\sqrt{\log K}$ term in his bound, viewed on the square-root scale, with the data-dependent threshold $c_{\mathcal J,P}^{\mathrm{BB}}(a)$ defined in
\eqref{eq:def-BB-quantile-J-P}. Since
$c_{\mathcal J,P}^{\mathrm{BB}}(a)$ is always of order at most $\sqrt{\log K}$, our threshold is consistent with the scaling in \citet{jirak2015uniform}. However,
Theorem~\ref{thm:profiled_simultaneous_power} retains the relevant Gaussian quantiles themselves, allowing the sufficient detection threshold to adapt to the cross-coordinate covariance structure. 

\begin{remark}
The signal-strength condition defining \(\mathcal P_{\mathrm{loc}}(c_0,q,\tau)\) can be relaxed by choosing the bootstrap gap length \(g\) larger than \(q\). We state Theorem~\ref{thm:profiled_simultaneous_power} with \(g=q\) to obtain a simple result in terms of the block length \(q\). A corresponding result for general \(g\) follows from the \(g\)-dependent bound in Theorem~\ref{thm:conditional_gaussian_subset_master}, which is a key ingredient in the proof of Theorem~\ref{thm:profiled_simultaneous_power}. Such a result would also allow for our recommended choice of $g=(q+r)/2$ given in Section~\ref{subsec:tuning-parameter-choice}.
\end{remark}

\subsection{Simultaneous localization intervals}
\label{subsec:localization-guarantees}

We show that our method can simultaneously localize any collection of
changed coordinates. 
Let
\begin{align*}
\mathcal P_{\mathrm{lin}}
=
\mathcal P_{\mathrm{lin}}(C_A, D_R)
\end{align*}
denote the class of data-generating distributions satisfying the change-point model \eqref{eq:change_point_model}, where the noise process fulfills Assumption~\ref{cond:linear} with the stated constants.
Recall the change point estimator $\widehat \omega_k$ from \eqref{eq:definition-hat-omega-k} and $\omega_k$ and $\nu_k$ from \eqref{eq:def-nu_k}.

\begin{theorem}
\label{thm:linear_process_localization}
There exists a constant $C$ that may depend only on $D_R$ and $C_A$ of Assumption~\ref{cond:linear} such that, 
for every $P \in \mathcal P_{\mathrm{lin}}$, every $\varnothing \ne \chgset \subset \mathcal S_P$ and every $\eta \in (0, 1)$,
\[
   \mathbb P_P\Big(
        \forall k\in \mathcal \chgset:
        |\widehat\omega_k-\omega_k|
        <
        r_{k}(\eta)
    \Big)
    \ge
    1-\eta,
\]
where, recalling $\nu_k$ from \eqref{eq:def-nu_k},
\begin{align} \label{eq:definition-rketa}
    r_{k}(\eta)
    :=
    \left\lceil
        C\log\left(\frac{e|\chgset|}{\eta}\right)
        \left(
            \nu_k^{-2}
            +
            \nu_k^{-1}
        \right)
    \right\rceil.
\end{align}
\end{theorem}
Theorem~2 of \cite{li2024detection} gives an asymptotic uniform
localization result that should be compared to Theorem~\ref{thm:linear_process_localization}.  Under a strengthened version of Assumption~\ref{cond:linear} with  $D_R$ providing a bound on the sub-Gaussian rather than the sub-Exponential norm of the Hilbert space norm of the latent linear process innovations, and also the uniform signal condition
\(\min_{k\in\mathcal S_P}\nu_k\ge c>0\), their coordinatewise CUSUM
estimators $\widehat\omega_k$ satisfy
\[
\max_{k\in\mathcal S_P}
|\widehat\omega_k-\omega_k|
=
o_{\mathbb P}\!\left(
[\log(n\vee K)]^{2}
\right);
\]
see Equation~(4.5) therein. In
contrast, Theorem~\ref{thm:linear_process_localization} provides a fully nonasymptotic and signal-adaptive guarantee,
which holds at each sample size and confidence level,
rather than only along an asymptotic sequence. When the signals are
uniformly bounded below, its localization radius is proportional to
\(\log(e|\chgset|/\eta)\), improving the logarithmic exponent in the
asymptotic bound of \cite{li2024detection}. 

Note that under an interiority condition like~\eqref{eq:interiority-condition} it holds for a $\tau$-dependent constant $C_\tau$ that 
\[
\nu_k^{-2}+\nu_k^{-1} \leq C_\tau\{ \Delta_k^{-2}+\Delta_k^{-1}\}
\]
The additional \(\Delta_k^{-1}\)term, which is less familiar than the
commonly provided \(\Delta_k^{-2}\) localization rate (see \citealp{verzelen2023optimal}, who work in the Gaussian scalar case), reflects the linear-tail term in a Bernstein-type maximal inequality of~\cite{pinelis1994optimum} that we use under the tail assumption of Assumption~\ref{cond:linear}. Under a sub-Gaussian
assumption like the one used by \cite{li2024detection}, a similar argument  gives, under the interiority condition~\eqref{eq:interiority-condition}, the sharper radius
\[
C\Delta_k^{-2}
\log\left(\frac{e|\chgset|}{\eta}\right),
\]
where now $C$ may depend on the interiority parameter $\tau$. This rate is known to be optimal in the scalar case with $K=1$, as discussed in \cite{verzelen2023optimal}.

Theorem~\ref{thm:linear_process_localization} is established under
Assumption~\ref{cond:linear}, rather than under
Assumptions~\ref{cond:stationary} and~\ref{cond:scores}. Under the
polynomial \(\beta\)-mixing conditions in
Assumption~\ref{cond:stationary}, we also derive a nonasymptotic
simultaneous localization bound in Corollary~\ref{cor:coordinatewise_beta_localization_new}. That bound, however, depends on a
coupling length chosen to control the mixing remainder and therefore
on the constants \(C_\beta\) and \(c_\beta\). Under polynomial mixing,
the resulting radius generally grows faster than logarithmically in
\(n\vee K\). Although this bound is sufficient for controlling
contamination in the residual-bootstrap construction, it is too
conservative to serve as our principal localization result.
We therefore impose the alternative linear-process structure in
Assumption~\ref{cond:linear}, which permits us to directly use the maximal inequality of \cite{pinelis1994optimum}
for partial sums and removes the coupling-length penalty. This yields
the sharper radius 
in Theorem~\ref{thm:linear_process_localization}. Assumption
\ref{cond:scores}, which governs the spectral structure needed for the
Gaussian and bootstrap approximations, plays no role in this
localization argument. Thus, the general mixing assumptions suffice
for bootstrap validity, whereas the linear-process assumption delivers
a sharper and more readily interpretable localization guarantee.

\subsection{Step-down refinement of simultaneous testing}
\label{subsec:stepdown-theory}

We conclude this section by recording the theoretical properties of the step-down refinement introduced in Section~\ref{subsec:stepdown}. The refinement concerns the simultaneous testing procedure and can only increase the set of detected coordinates without weakening its family-wise error guarantee. Its theoretical properties are essentially a corollary of Theorem~\ref{thm:fwer_control}.

\begin{corollary}[Step-down refinement]
\label{cor:stepdown}
For fixed $\gamma \in (0,1)$, let $\widehat{\mathcal S}^{\mathrm{SD}}_\gamma$denote the rejection set of the step-down procedure. Then 
\[
\widehat{\mathcal S}_\gamma
\subseteq
\widehat{\mathcal S}^{\mathrm{SD}}_\gamma.
\]
Moreover,  under the conditions of Theorem~\ref{thm:fwer_control}, 
\[
\forall P\in\mathcal P:\qquad
\mathbb P_P\!\left(
\widehat{\mathcal S}^{\mathrm{SD}}_\gamma
\cap\mathcal S_P^c\neq\varnothing
\right)-\gamma
\le
\frac{C}{\sqrt{\tau}}
\log^d(nK)\mathcal R_{n,q,r}.
\]
\end{corollary}

The first conclusion shows that the step-down refinement has simultaneous detection probability at least as large as that of the single-step procedure. We do not attempt to quantify the potential improvement, which depends on how much the bootstrap critical values decrease as coordinates are removed during the step-down procedure. In particular, whenever the conditions of Theorem~\ref{thm:profiled_simultaneous_power} and the signal strength condition \eqref{eq:profiled-power-signal} hold for a nonempty deterministic set $\chgset\subseteq\mathcal S_P$, the step-down procedure inherits the same simultaneous detection guarantee. Thus, the step-down refinement retains the strong family-wise error guarantee of the single-step procedure while having no smaller simultaneous detection probability.

\section{Choice of tuning parameters}
\label{subsec:tuning-parameter-choice}

This section discusses the practical selection of the tuning parameters
$q$, $r$, $g$, and $\tau$ needed for the bootstrap statistic  $\left(T_n^*\right)^2 = \left(T_n^*(g,q,r,\tau)\right)^2 $  introduced in~\eqref{eq:definition-Tstar}.

We begin with the clipping parameter
$\tau\in(0,1/2)$. The theoretical results in
Section~\ref{sec:theoretical-guarantees} require $\tau>0$, but this
condition primarily serves as a technical device to ensure that the left
and right residual samples remain sufficiently large in the proofs.
Empirically, we found no benefit from clipping, and we therefore
recommend setting $\tau=0$ in practice.

We select the big-block length $q$ using a rolling stability criterion
based on the estimated bootstrap quantile, similar in spirit to the
procedure of \citet{ZD23}. Let $\mathcal Q=\mathcal Q_n$ denote a finite
set of candidate block lengths. We recommend
\[
\mathcal Q_n
=
\left\{
q\in\mathbb N:
\left\lfloor\frac12 n^{1/3}\right\rfloor
\le q\le
\left\lceil\frac32 n^{1/3}\right\rceil,
\quad
2\le q\le \frac n4
\right\}.
\]
The candidate values in $\mathcal Q_n$ typically satisfy the conditions
required by our theoretical results. For each $q\in\mathcal Q_n$, define
\[
r_q
:=
\lfloor\sqrt q\rfloor,
\qquad
g_q
:=
\frac{q+r_q}{2},
\]
and generate repeated realizations of
$
(T_n^*(g_q,q,r_q,0))^2
$
using independent multiplier draws. Let
$\widehat Q_{1-\gamma}(q)$ denote the resulting empirical
$(1-\gamma)$-quantile. Writing the candidate values as
$
q_1<q_2<\cdots<q_G,
$
we quantify the local stability of the estimated quantiles by
\[
\widehat\sigma_j
:=
\operatorname{sd}
\left\{
\widehat Q_{1-\gamma}(q_j),
\widehat Q_{1-\gamma}(q_{j+1}),
\widehat Q_{1-\gamma}(q_{j+2})
\right\},
\qquad
j=1,\ldots,G-2.
\]
We then select the center of the most stable triple,
\[
\widehat j
\in
\argmin_{1\le j\le G-2}
\widehat\sigma_j,
\qquad
\widehat q
:=
q_{\widehat j+1}.
\]

Finally, we set $\widehat r=\lfloor\sqrt{\widehat q}\rfloor$ and $\widehat g=(\widehat q+\widehat r)/2$, and use the corresponding empirical critical value $\widehat Q_{1-\gamma}(\widehat q)$ obtained from repeated realizations of $(T_n^*(\widehat g,\widehat q,\widehat r,0))^2$.

\section{Simulation Study}
\label{sec:simulation-study}
We evaluate the finite-sample performance of our method with a simulation study. Our main focus is on coordinatewise inference, and as such we study finite-sample family-wise error control and coordinatewise power, including the gain from the stepdown refinement from Section~\ref{subsec:stepdown}. As a secondary benchmark we also compare the global test induced by our procedure with the power-enhanced CUSUM (PE-CUSUM) test of \cite{li2024detection}. 
The study has two parts: Experiment~1 examines strong FWER control and coordinatwise power.
Experiment~2 compares global rejection probabilities with PE-CUSUM.
Implementations are available at \url{https://github.com/colindecker/hd_change_point}.

\subsection{Simulation design}
\label{subsec:simulation-design}
We introduce a simulation framework that depends on several design parameters, which are introduced now and fully defined below. Each simulation is characterized by the choice of 
\[
\boldsymbol\psi
=
(K, \lambda, \FAC, \SNR, \mathcal C,\mathcal O).
\]
The first parameter, $K \in \N$, is the number of coordinates. The second parameter, $\lambda \in [0,1]$, is a factor-mixing weight used to control the strength of global contemporaneous
dependence in the noise process across coordinates. Finally, each alternative is
characterized by the fraction of affected coordinates 
$\FAC \in [0,1]$, the signal strength $\SNR>0$, a change-point configuration design
$\mathcal C\in\{\mathrm{common},\mathrm{heterogeneous}\}$, and a
functional-direction design
$\mathcal O\in\{\mathrm{fixed},\mathrm{random}\}$. 
Throughout $n=200$, and $\Hb_k = L^2([0,1])$ for all $k \in [K]$.

\subsubsection{Noise process}
Let $\phi_1, \dots, \phi_{21}$ be the first 21 Fourier basis functions of $L^2([0,1])$, defined for $u \in [0,1]$ by
\[
\phi_1(u)=1,\quad
\phi_{2j}(u)=\sqrt2\sin(2\pi ju),\quad
\phi_{2j+1}(u)=\sqrt2\cos(2\pi ju),\quad j=1,\ldots,10.
\]
Following \cite{li2024detection}, the baseline noise model is
\[
\varepsilon_{k,0}^{(i)}(u)
=
\sum_{j=1}^{21}
\{Z_{kj}^{(i)}+\eta_{kj}^{(i)}\}\phi_j(u),
\qquad
Z_j^{(i)}=AZ_j^{(i-1)}+\xi_j^{(i)}.
\]
Here, $\smash{Z_j^{(i)}=(Z_{1j}^{(i)},\ldots,Z_{Kj}^{(i)})^\top}$ induces serial and cross-sectional dependence through the vector autoregressive recursion, where $\smash{\xi_j^{(i)}\in\R^K}$ are standard Gaussian vectors that are independent across $i$ and $j$. 
Moreover, the variables $\smash{\eta_{kj}^{(i)}}$ are independent Gaussian random variables with standard deviation $1/j$, and the matrix $A$ is banded, with $a_{k\ell}=0$ whenever $|k-\ell|>3$, while the remaining entries $a_{k\ell}$ are drawn independently from $\operatorname{Unif}[-0.3,0.3]$ in each Monte Carlo replication.

To introduce global cross-sectional dependence, independently generate
\[
G_\ell^{(i)}(u)
=
\sum_{j=1}^{21}
\{X_{\ell j}^{(i)}+\zeta_{\ell j}^{(i)}\}\phi_j(u),
\qquad
X_{\ell j}^{(i)}
=
\rho_\ell X_{\ell j}^{(i-1)}+\nu_{\ell j}^{(i)},
\]
for $\ell=1,\ldots,10$,  where
$\rho_\ell$ are iid $\operatorname{Unif}[-0.3,0.3]$,
$\smash{\nu_{\ell j}^{(i)}}$ are iid standard Gaussian, and
$\smash{\zeta_{\ell j}^{(i)}}$ are independent Gaussians with standard deviation $1/j$. 
Let $b_k = (b_{k1}, \dots, b_{k10})^\top\in\mathbb R^{10}$ have independent standard Gaussian entries
and define
\[
F_k^{(i)}(u)=\sum_{\ell=1}^{10} \frac{b_{k\ell}}{\lVert b_k\rVert_2} G_\ell^{(i)}(u).
\]
The general noise process is a mixture of the baseline noise process and the factors
\begin{equation}
\label{eq:mixed-noise}
\varepsilon_{k,\lambda}^{(i)}(u)
=
\sqrt{1-\lambda}\,\varepsilon_{k,0}^{(i)}(u)
+
\sqrt{\lambda}\,F_k^{(i)}(u).
\end{equation}
The setting $\lambda=0$ recovers the baseline model, while larger $\lambda$ gives
more weight to the common factors.

\subsubsection{Alternatives}
For a specified $\FAC \in [0,1]$, the affected set
$\mathcal S\subseteq[K]$ is drawn uniformly among the subsets of $[K]$ satisfying
$|\mathcal S|= \lfloor \FAC  \cdot K \rfloor$. We set $\mu_k\equiv0$ for every $k \in [K]$,
$\delta_k\equiv0$ for $k\notin\mathcal S$, and
\[
\delta_k(u)=\sqrt{c_{k,\lambda}}\,d_k(u),
\qquad k\in\mathcal S.
\]
In the heterogeneous change-point design, $\mathcal C = \mathrm{heterogeneous}$, the $\omega_k$ are drawn
independently and uniformly from
$\{\lfloor0.25n\rfloor,\ldots,\lfloor0.75n\rfloor\}$. In the common design, $\mathcal C = \mathrm{common}$,
one location is drawn from the same set and assigned to all affected
coordinates.

Under fixed orientation, $\mathcal O = \mathrm{fixed}$, all affected coordinates share $d_k = d_{\mathrm{fix}}$ with
\[
d_{\mathrm{fix}}(u)
=
\frac{1}{\sqrt{21}}\sum_{j=1}^{21}\phi_j(u).
\]
Under random orientation, $\mathcal O = \mathrm{random}$, 
$\smash{d_k(u)=\sum_{j=1}^{21}\gamma_{kj}\phi_j(u)}$, where each
$\gamma_k=(\gamma_{k1},\allowbreak \ldots,\gamma_{k,21})^\top$ is independently uniform
on the unit sphere in $\mathbb R^{21}$. Note that both constructions have unit
coefficient norm.

All experiments use the same signal calibration based on the long-run variance. Specifically, from the signal-free noise realization at the relevant $\lambda$, 
we estimate the marginal long-run covariance matrix $\widehat\Omega_{k,\lambda}$ (using an equidistant grid of 101 values in $[0,1]$) and set, for specified $\SNR > 0$,
\begin{equation}
\label{eq:snr-calibration}
c_{k,\lambda}
=
\SNR \cdot \frac{
\operatorname{tr}(\widehat\Omega_{k,\lambda})10^{-4}  
}{
\theta_k(1-\theta_k)
}.
\end{equation}

\subsubsection{Computation}
All random design quantities are independently regenerated in each Monte Carlo replication.
Final inference uses $B=500$ bootstrap replications, and the tuning parameter selection method described in Section~\ref{subsec:tuning-parameter-choice} uses $B_{\mathrm{select}}=200$ separate replications. The experiment-specific design choices are summarized below. All presented results are based on $N_{MC}=1000$ Monte-Carlo replications.

\subsection{Experiment 1: Strong FWER control and coordinatewise power}
\label{subsec:experiment-fwer}
Throughout, we fix $n=200$, $K\in\{200,400\}$ and $\lambda \in \{0, .3, .6\}$, and consider both the complete null, $\FAC=0$, as well as partial alternatives of the form
\[
\FAC\in\{0.02,0.10,0.70\},
\qquad
\SNR\in \{0.75, 1.50, 2.50, 3.50\},
\]
always using heterogeneous change point locations and random orientations. 
We report Monte-Carlo estimates of both 
\[
\mathrm{FWER} = \Prob(\widehat{\mathcal {S}}_\gamma \cap \mathcal S^c \ne \varnothing), \qquad 
\mathrm{TPR} = \mathbb E\bigg[
\frac{|\widehat{\mathcal {S}}_\gamma\cap\mathcal S|}{|\mathcal S|}
\bigg].
\]
Table~\ref{tab:complete-null-fwer} examines finite-sample calibration under the complete null. In this setting, the FWER coincides with the rejection probability of the induced global test and is identical for the single-step procedure and the stepwise refinement from Section~\ref{subsec:stepdown}. At the nominal \(5\%\) level, the empirical FWER is close to nominal when \(\lambda=0\), at \(5.0\%\) for \(K=200\) and \(5.9\%\) for \(K=400\), consistent with the complete-null calibration guarantee of Theorem~\ref{thm:fwer_control}. As \(\lambda\) increases, the procedure becomes more conservative, with rejection rates falling to approximately \(1.5\%\) at \(\lambda=0.6\). 
This conservatism is in line with findings from other simulation studies of hypothesis tests based on the high-dimensional Gaussian multiplier bootstrap \citep{Chakraborty2023}.
At the same time, the mean bootstrap critical value decreases substantially with $\lambda$, indicating that the calibration adapts to the stronger cross-coordinate dependence and the resulting reduction in effective multiplicity. Critical values remain higher for $K=400$ than for $K=200$, as expected for a maximum over a larger number of coordinates. Overall, the results illustrate the covariance-adaptive nature of the bootstrap calibration, while also revealing some finite-sample conservatism under strong factor dependence.

\begin{table}[t]
\centering
\caption{Empirical family-wise error rates (in percent) under the complete null ($\FAC=0$) at the 5\% nominal level. The reported global critical values are bootstrap critical values for the maximum statistic.}
\label{tab:complete-null-fwer}
\small
\setlength{\tabcolsep}{8pt}
\renewcommand{\arraystretch}{1.12}
\begin{tabular}{c cc cc}
\toprule
& \multicolumn{2}{c}{$K=200$}
& \multicolumn{2}{c}{$K=400$} \\
\cmidrule(lr){2-3}\cmidrule(lr){4-5}
$\lambda$
& FWER
& Mean critical value
& FWER
& Mean critical value \\
\midrule
0.0 & 5.0 & 47.48 & 5.9 & 51.55 \\
0.3 & 3.4 & 39.08 & 4.4 & 42.22 \\
0.6 & 1.5 & 31.58 & 1.4 & 33.71 \\
\bottomrule
\end{tabular}
\end{table}

Table~\ref{tab:power-partial-fwer-small} evaluates the article's strong FWER objective under partial alternatives, where genuine changes and unchanged coordinates occur simultaneously. Across all combinations of \(K\), \(\FAC\), $\SNR$, and \(\lambda\), the empirical FWER remains at or below \(5.9\%\), with the few values above the nominal \(5\%\) level falling within the Monte Carlo variability associated with \(1{,}000\) replications. For the single-step procedure, the error rate is largely insensitive to signal strength, indicating that the residualization step prevents genuine changes from substantially contaminating the bootstrap calibration. The procedure becomes more conservative as \(\lambda\) or the affected fraction increases, the latter because fewer unchanged coordinates remain while the single-step threshold is still calibrated over all \(K\) coordinates. The stepdown procedure partially reduces this conservatism, especially under dense and strong alternatives, by recalibrating over progressively smaller sets of hypotheses. Nevertheless, its FWER remains controlled, providing finite-sample support for Theorem~\ref{thm:fwer_control} and Corollary~\ref{cor:stepdown}.

\begin{table}[t]
\centering
\caption{Empirical family-wise error rates (in percent) under partial alternatives ($\FAC>0$) at the $5\%$ nominal level. Results are reported for $\SNR=2.5$ only; those for $\SNR\in\{0.75,1.5,3.5\}$ are very similar.}
\label{tab:power-partial-fwer-small}
\small
\setlength{\tabcolsep}{6pt}
\renewcommand{\arraystretch}{1.12}
\begin{tabular}{cc cc cc cc}
\toprule
& &
\multicolumn{2}{c}{$\lambda=0$}
& \multicolumn{2}{c}{$\lambda=0.3$}
& \multicolumn{2}{c}{$\lambda=0.6$} \\
\cmidrule(lr){3-4}
\cmidrule(lr){5-6}
\cmidrule(lr){7-8}
$K$ & FAC
& Single-step & Stepdown
& Single-step & Stepdown
& Single-step & Stepdown \\
\midrule
\multirow{3}{*}{200}
& 0.02 & 4.9 & 4.9 & 3.5 & 3.6 & 1.4 & 1.4 \\
& 0.10 & 4.8 & 5.5 & 2.7 & 3.2 & 1.2 & 1.3 \\
& 0.70 & 1.3 & 2.7 & 0.7 & 2.0 & 0.2 & 0.6 \\
\addlinespace[4pt]
\multirow{3}{*}{400}
& 0.02 & 5.7 & 5.9 & 4.2 & 4.2 & 1.5 & 1.5 \\
& 0.10 & 5.2 & 5.7 & 4.4 & 4.5 & 1.5 & 1.5 \\
& 0.70 & 1.9 & 4.2 & 1.4 & 3.3 & 0.5 & 1.1 \\
\bottomrule
\end{tabular}
\end{table}

Finally, Figure~\ref{fig:tpr} reports coordinatewise true-positive rates and shows that power increases with both SNR and the factor weight \(\lambda\), while generally decreasing when the dimension rises from \(K=200\) to \(K=400\). The improvement with \(\lambda\), together with the declining critical values in Table~\ref{tab:complete-null-fwer}, illustrates the covariance-adaptive nature of the procedure: stronger cross-coordinate dependence reduces effective multiplicity and permits a lower simultaneous threshold. The single-step true-positive rate changes little as the affected fraction increases, which is consistent with comparing each marginal statistic with a common threshold calibrated over all \(K\) coordinates. By contrast, the stepdown refinement yields a substantial power gain when \(\FAC=0.70\), particularly at the intermediate and high SNR levels, because detected coordinates are removed before the remaining hypotheses are recalibrated. Its advantage is negligible under sparse alternatives, where relatively few coordinates can be removed. Figure~\ref{fig:tpr} therefore empirically supports Corollary~\ref{cor:stepdown}: stepdown can improve detection under dense alternatives without sacrificing the strong FWER control documented in Table~\ref{tab:power-partial-fwer-small}.

\begin{figure}[thbp]
\centering
\includegraphics[
  width=0.9\textwidth,
  keepaspectratio
]{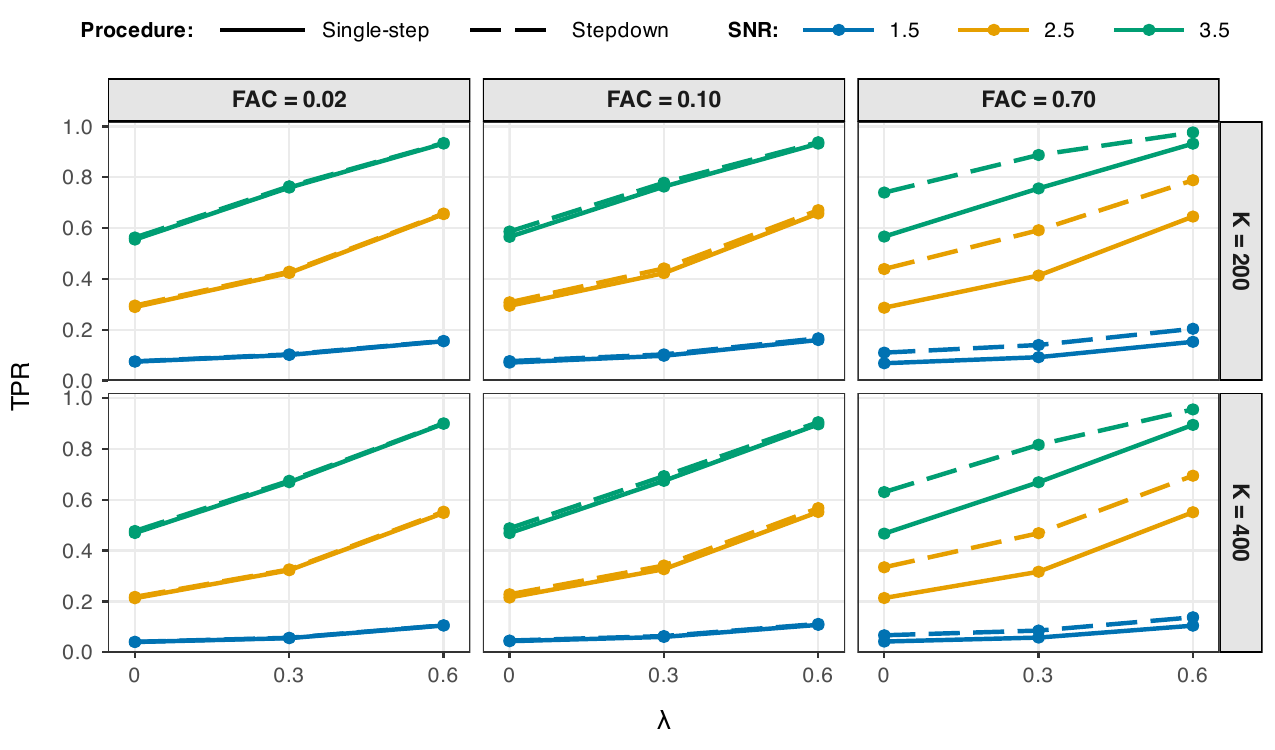}
\vspace{-.4cm}
\caption{Coordinatewise true-positive rate of the single-step and stepdown procedures across dependence strengths $\lambda$, fraction of affected coordinates $\FAC$, and dimensions $K$. Curves correspond to SNR levels $1.5$, $2.5$, and $3.5$.}
\label{fig:tpr}
\end{figure}

\subsection{Experiment 2: Global comparison with PE-CUSUM}
\label{subsec:experiment-global-comparison}

We recall the power-enhanced (PE) CUSUM procedure of \citet{li2024detection} using our notation. The procedure requires $\Hb_1=\dots=\Hb_K=L^2(\mathbb C)$ for some bounded set $\mathbb C$. Under this setting, define the aggregated functional observations and the corresponding CUSUM process by
\[
\widetilde f^{(i)}
=
\frac{1}{\sqrt K}\sum_{k=1}^K f_k^{(i)}, \qquad \widetilde C_n(s)
=
\frac{1}{\sqrt n}
\sum_{i=1}^s
\Big(
\widetilde f^{(i)}-\overline{\widetilde f}
\, \Big),
\quad s\in[n].
\]
The aggregated CUSUM statistic is
\begin{equation}\label{eq:aggregated_CUSUM_PE}
Z_n
=
\max_{s\in[n]}
\big\|
\widetilde C_n(s)
\big\|_\Hb^2.
\end{equation}
Using our marginal statistics
$
T_{n,k}^2
=
\max_{s\in[n]}
\|C_{n,k}(s)\|_\Hb^2
$
from \eqref{def:marginal_test},
the power-enhancement component can be written as
\[
Z_n^\circ
=
\sqrt{K\vee n}
\sum_{k\in[K]}
\mathbf 1
\left\{
T_{n,k}^2>\zeta_{n,K}
\right\},
\qquad
\zeta_{n,K}
=
c_\zeta
\log(K\vee n)
\log\log(K\vee n),
\]
where \(c_\zeta>0\) is a user defined tuning constant. The PE-CUSUM test statistic is then
\[
Z_n^{\mathrm{PE}}
=
Z_n+Z_n^\circ,
\]
which is compared with the Brownian-bridge-based null distribution of the aggregated CUSUM statistic $Z_n$. The constant \(c_\zeta\) controls the screening threshold used in the
power-enhancement component. Smaller values of \(c_\zeta\) make it easier
for marginal statistics to exceed the threshold and can therefore increase
power against sparse alternatives or alternatives that cancel under aggregation, at the cost of potentially inflating the rate of false positives. Following the simulation study of \citet{li2024detection}, we set 
$c_\zeta=\smash{\hat \lambda_1^{1/2}}$,
where $\hat \lambda_1$ is an estimate of the largest eigenvalue of long-run covariance operator of the process $\smash{\tilde f^{(1)}, \dots \tilde{f}^{(n)}}$ defined above \eqref{eq:aggregated_CUSUM_PE}.

We fix $\lambda=0$, use $K\in\{200,400\}$, and include all four combinations of common or heterogeneous change points and fixed or random orientations. We consider both the global null, $\FAC=0$, as well as (partial) alternatives, for which the FAC grid is
$
\{0.020,0.250\}.$
The first value corresponds to a sparse regime and we use
\[
\SNR\in \{1.0,1.5,2.0,2.5,3.0,3.5\};
\]
the second value corresponds to a moderately dense regime and we use 
\[
\SNR\in \{0.25,0.5,0.75,1.0,1.25,1.5\}.
\]
Since PE-CUSUM is global, the comparison is based on global rejection probabilities only.

We first discuss the empirical levels under the global null. For the proposed method, these are reported in Table~\ref{tab:complete-null-fwer} and equal $5.0\%$ for $K=200$ and $5.9\%$ for $K=400$. The corresponding values for the PE-CUSUM test are $4.9\%$ and $6.1\%$, respectively, consistent with the findings of \citet{li2024detection}. Overall, both tests provide a reasonable approximation to the nominal level.

Figure~\ref{fig:power-combined} compares the empirical power of the two methods over the partial-alternative grid. Viewed as a test of the global null-hypothesis, the proposed maximum-based procedure performs particularly well under sparse alternatives and heterogeneous change-point locations, where its coordinatewise CUSUM statistics can detect a small number of pronounced changes. PE-CUSUM is more powerful when many coordinates share a common functional direction, because their signals reinforce one another under aggregation.
Thus, the methods exploit complementary signal structures: PE-CUSUM favors dense, aligned global alternatives, whereas the proposed procedure applied as a test of the global null favors sparse or heterogeneous alternatives while additionally providing coordinate-specific decisions with strong FWER control. More fundamentally, the methods address different inferential objectives: PE-CUSUM tests only the global null, while the proposed procedure provides coordinate-specific decisions with strong FWER control.

\begin{figure}[t]
\centering
\includegraphics[
  width=0.99\textwidth,
  keepaspectratio
]{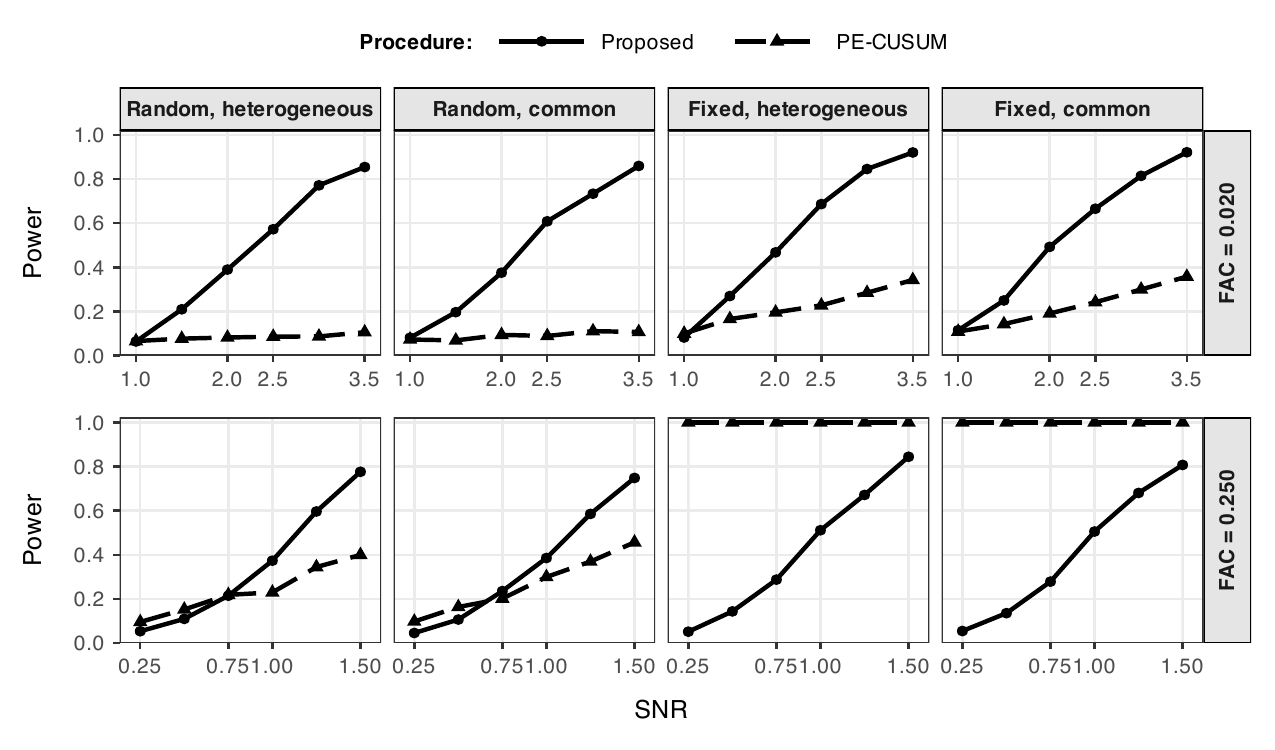}
\vspace{-.4cm}
\caption{Power of the proposed test and the PE-CUSUM test for the global null hypothesis $H_0^\infty$ from \eqref{eq:global-hypothesis}.}
\label{fig:power-combined}
\end{figure}

\section{Case studies}
\label{sec:case-study}

\subsection{European Rivers}
\label{subsec:rivers}

\subsubsection{Data and preprocessing}
Following \citet{sharipov2016sequential}, we analyze daily river-discharge measurements from the Global Runoff Data Centre (GRDC) over the period January 1, 1921 to December 31, 2020, and selected the $K=27$ stations that had complete observations over this \(n=100\) year period. February 29 was removed so that each year is represented on a common 365-day grid. 

Let \(D_k^{(i)}(t)\) denote the discharge at station \(k\), in year \(i\), and on calendar day \(t\). Each year is treated as one functional observation, and the data entering the analysis are the log-transformed
discharge curves
\begin{equation*}
f_k^{(i)}(t)
=
\log D_k^{(i)}(t).
\end{equation*}
The log transformation reduces the influence of very large discharge values and allows changes to be interpreted on a relative rather than absolute
scale. To place the stations on a comparable overall scale, we standardize each
station using a single station-specific scale factor. Specifically, for each
\(k\in[K]\), we replace 
$\smash{f_k^{(i)}(t)}$ by $\smash{f_k^{(i)}(t) / \widehat \sigma_k}$, where
\(\widehat\sigma_k\) is the pooled empirical standard deviation computed across all years and calendar days at station $k$.
For comparison, we also apply the univariate functional change-point procedure of \cite{Aue2018} separately to the standardized log-discharge curves of each station using the implementation from the R package \cite{fChange2025}. We note that \citet{sharipov2016sequential} implements a similar test based on the same test-statistic but with a different null-quantile calibration procedure. We then applied Holm–Bonferroni correction \citep{holm1979simple} to these $p$-values to obtain adjusted $p$-values that control for familywise error.

\subsubsection{Results}
Table~\ref{tab:river_station_p_values} reports the resulting $p$-values from the two procedures. For the proposed procedure, we report the simultaneous single-step $p$-values derived from the method outlined in Section~\ref{sec:methodology}.  Specifically, for each coordinate \(k\in[K]\), define
\[
\widehat p_k^{\,\mathrm{SS}}
=
\mathbb P\!\left(
(T_n^*)^2\ge T_{n,k}^2
\,\middle|\,
f^{(1)},\ldots,f^{(n)}
\right),
\]
where \((T_n^*)^2=\max_{j\in[K]}(T_{n,j}^*)^2\) is the bootstrap
maximum over all \(K\) coordinates, and the probability is conditional on the observed data and so only with respect to the multiplier randomness.  Using $B$ bootstrap
replications, we estimate this probability by
\begin{align} \label{eq:p-values}
\widehat p_{k,B}^{\,\mathrm{SS}}
=
\frac{
1+\sum_{b=1}^{B}
\mathbf 1\!\left\{
(T_n^{*(b)})^2\ge T_{n,k}^2
\right\}
}{B+1}.
\end{align}
Table~\ref{tab:river_station_p_values}
reports these simultaneous single-step \(p\)-values in the column
``Prop. Method.''

At the $5\%$ significance level, the proposed simultaneous procedure identifies a mean change at three of the 27 measurement stations: Aare at Brienzwiler and the RHINE at Diepoldsau--Rietbruecke and Domat/Ems. The corresponding estimated change-point years are 1950, 1963, and 1962, respectively. In contrast, the univariate procedure of \cite{Aue2018}, applied separately to each coordinate and followed by a post-hoc Bonferroni--Holm correction to ensure strong family-wise error control, does not detect a mean change at any station. 

The proposed method controls the family-wise error rate across all 27 tests while calibrating the rejection threshold to the dependence structure among stations. This is particularly relevant in this example, as stations on the same river may have noise fluctuations that are positively dependent. The bootstrap calibration can exploit this dependence, whereas the Bonferroni--Holm adjustment does not.

\begin{table}[htbp]
\centering
\caption{
Change-point results for the 27 river measurement stations.
The column ``Adjusted ARS'' reports the Bonferroni-Holm-adjusted $p$-values obtained from the univariate functional change-point procedure of \cite{Aue2018} applied separately to each station. The column ``Prop. Method'' reports the simultaneous marginal $p$-values
obtained from the proposed procedure. Values below $0.05$ are shown in bold. The final column reports estimated change-point when the corresponding $p$-value is significant. 
}
\label{tab:river_station_p_values}

\begin{tabular}{llccc}
\toprule
\textbf{River}
&
\textbf{Station}
&
\textbf{Adjusted ARS}
&
\textbf{Prop. Method}
&
\textbf{Change point}
\\
\midrule

{\small AARE}
& BERN-SCHOENAU
& 1.000
& 0.985
& --
\\

{\small AARE}
& BRIENZWILER
& 0.621
& \textbf{0.025}
& 1950
\\

{\small AARE}
& BRUEGG-AEGERTEN
& 1.000
& 0.992
& --
\\

{\small AARE}
& BRUGG
& 1.000
& 0.985
& --
\\

{\small AARE}
& MURGENTHAL
& 1.000
& 0.984
& --
\\

{\small AARE}
& UNTERSIGGENTHAL
& 1.000
& 0.985
& --
\\

{\small BIRSE}
& MONTIER (LA CHARRUE)
& 1.000
& 0.660
& --
\\

{\small EMME}
& EMMENMATT
& 1.000
& 0.278
& --
\\

{\small LANDQUART}
& FELSENBACH
& 1.000
& 0.957
& --
\\

{\small LUTSCHINE}
& GSTEIG
& 1.000
& 0.953
& --
\\

{\small MUOTA}
& INGENBOHL
& 0.882
& 0.320
& --
\\

{\small REUSS}
& MELLINGEN
& 0.621
& 0.901
& --
\\

{\small RHINE}
& BASEL
& 0.882
& 0.970
& --
\\

{\small RHINE}
& DIEPOLDSAU
& 0.081
& \textbf{0.033}
& 1963
\\

{\small RHINE}
& DOMAT/EMS
& 0.081
& \textbf{0.002}
& 1962
\\

{\small RHINE}
& NEUHAUSEN
& 0.225
& 0.905
& --
\\

{\small RHINE}
& REKINGEN
& 0.480
& 0.894
& --
\\

{\small SIMME}
& OBERWIL
& 0.882
& 0.836
& --
\\

{\small SITTER}
& APPENZELL
& 1.000
& 0.141
& --
\\

{\small THUR}
& ANDELFINGEN
& 1.000
& 0.411
& --
\\

{\small TOSS}
& NEFTENBACH
& 1.000
& 0.738
& --
\\

{\small MAIN}
& SCHWEINF. NEUER HAFEN
& 1.000
& 0.278
& --
\\

{\small RHINE}
& DUESSELDORF
& 1.000
& 0.936
& --
\\

{\small RHINE}
& KOELN
& 1.000
& 0.941
& --
\\

{\small RHINE}
& REES
& 1.000
& 0.965
& --
\\

{\small WUTACH}
& OBERLAUCHRINGEN
& 1.000
& 0.495
& --
\\

{\small RHINE}
& LOBITH
& 1.000
& 0.978
& --
\\

\bottomrule
\end{tabular}

\end{table}

The most notable difference between the two procedures concerns Aare--Brienzwiler. The adjusted univariate procedure yields $p=0.621$, whereas the proposed simultaneous procedure yields $\widehat p<0.025$, providing substantially stronger evidence for a change at this station. The estimated change-point year is 1950, which coincides with major hydropower development in the upper Aare region; in particular, the Räterichsboden dam was completed in 1950 (see~\citealp{thevenon2013human}, p.~416).

\subsection{Dow Jones Industrial Average}
\label{subsec:djia}

\subsubsection{Data and preprocessing}
Following \citet{Horvath2014} and \citet{li2024detection}, we analyze cumulative intraday log-return curves for constituents of the Dow Jones Industrial Average (DJIA). The data cover the period from January 2, 2018 through December 31, 2021, comprising \(n=1008\) trading days.

To match the notation of Section~\ref{sec:methodology}, the stocks are indexed
by \(k\in[K]\), with \(K=28\), and trading days by \(i\in[n]\). At intraday
time \(t\), let 
\[
P_k^{(i)}(t)
=
\frac{
\operatorname{Bid}_k^{(i)}(t)
+
\operatorname{Ask}_k^{(i)}(t)
}{2}
\]
denote the bid--ask midpoint price.
The observations are aligned to a common five-minute grid from 09:30 to 15:55,
giving 78 intraday price observations per trading day. Missing observations are completed following the
preprocessing used in \citet{li2024detection}, and the remaining missing
values are filled by interpolation.

For each stock and trading day, the midpoint prices are transformed into
cumulative intraday log returns (CIDRs). Writing \(t_0=09{:}30\) and
\(t_1,\ldots,t_{77}\) for the five-minute grid from 09:35 to 15:55, we set
\begin{equation*}
f_k^{(i)}(t_\ell)
=
100
\log\left(
\frac{P_k^{(i)}(t_\ell)}
     {P_k^{(i)}(t_0)}
\right),
\qquad
\ell=1,\ldots,77.
\end{equation*}
Thus, \(f_k^{(i)} \in \Hb_k = L^2([0,1])\) is the functional observation corresponding to stock \(k\)
on trading day \(i\), in the notation of
\eqref{eq:mathcalH-decomposition}--\eqref{eq:change_point_model}. 
In the implementation, the squared \(L^2\)-norm is approximated on this equally
spaced grid by
$
\|x\|_{\Hb_k}^2
\approx
77^{-1}
\sum_{\ell=1}^{77} x(t_\ell)^2.
$

\subsubsection{Results}
For each stock \(k\), we compute the marginal CUSUM statistic \(T_{n,k}^2\) defined in~\eqref{def:marginal_test}. 
We implement the proposed method using the tuning parameter choices of Section~\ref{subsec:tuning-parameter-choice}. Table~\ref{tab:stock_pvalues} reports the simultaneous marginal $p$-values from \eqref{eq:p-values} for the 28 stocks. None of the marginal hypotheses $H_0^k:\delta_k=0$ is rejected at level $\gamma=0.05$. The smallest $p$-value is obtained for Chevron (CVX), with $\widehat p_{\mathrm{CVX}}=0.470$. Since no coordinate is selected at the $5\%$ level, we do not report the corresponding CUSUM maximizers as detected change-point locations.

\begin{table}[htbp]
\centering
\caption{Simultaneous marginal $p$-values for the 28 DJIA stocks.}
\label{tab:stock_pvalues}
\begin{tabular}{lc|lc|lc|lc}
\toprule
\textbf{Stock} & \textbf{$p$-value} &
\textbf{Stock} & \textbf{$p$-value} &
\textbf{Stock} & \textbf{$p$-value} &
\textbf{Stock} & \textbf{$p$-value} \\
\midrule
AAPL.O & 1.000 & AXP    & 0.956 & CSCO.O & 0.980 & JNJ    & 1.000 \\
AMGN.O & 1.000 & BA     & 0.538 & CVX    & 0.470 & JPM    & 1.000 \\
CAT    & 0.988 & CRM    & 0.988 & DIS    & 0.996 & KO     & 1.000 \\
GS     & 0.689 & HD     & 0.996 & IBM    & 0.964 & INTC.O & 1.000 \\
MCD    & 0.984 & MMM    & 1.000 & MRK    & 1.000 & MSFT.O & 0.996 \\
NKE    & 0.912 & PG     & 1.000 & TRV    & 0.920 & UNH    & 0.689 \\
V      & 1.000 & VZ     & 1.000 & WBA.O  & 0.793 & WMT    & 1.000 \\
\bottomrule
\end{tabular}
\end{table}

The absence of significant coordinatewise mean changes is consistent with
earlier empirical evidence that CIDR curves can be quite stable. In particular,
\citet{Horvath2014} apply functional stationarity tests to intraday financial
data. They find systematic evidence against stationarity for sufficiently
long samples of raw intraday price curves, but not the corresponding CIDR curves.  
We note that the hypotheses are not identical to those considered here: \citet{Horvath2014} test the broader null that a functional time series is stationary and weakly dependent, whereas our marginal hypotheses \(H_0^k:\delta_k=0\) concern the absence of a change in the mean function. Nevertheless, the two findings are qualitatively consistent. 

Further empirical work supports this interpretation. Using five-minute
S\&P~500 data, \citet{RiceWirjantoZhao2020} examine cumulative intraday
log-return curves and report that functional stationarity tests suggest that
the return series are reasonably stationary, even though the same curves
exhibit pronounced conditional heteroscedasticity. Likewise, \citet{ShangJi2023} transform
five-minute S\&P/ASX All Ordinaries prices into CIDRs and report a functional
KPSS \(p\)-value of \(0.737\), providing no evidence against stationarity in
their sample.

A notable contrast is provided by \citet{li2024detection}, who analyze the
same DJIA CIDR dataset analyzed here and report evidence of a structural break,
with 22 stocks contributing to the power-enhancement component. We can think of two main reasons for this contrast. 

First, the PE-CUSUM method is designed to detect a global change in systems for which the marginal changes are similarly oriented; recall that its high power in such settings has been illustrated in Section~\ref{sec:simulation-study}. 
It may be the case that many of the CIDR trajectories experienced a small mean change in a nearly uniform direction, and that the global aggregate statistic $Z_n$ of~\eqref{eq:aggregated_CUSUM_PE} amplifies these small changes into a large global signal. 

Second, our replication of \citet{li2024detection} shows that the finite-sample conclusions can be sensitive to the choice of $c_\zeta$. Theorem~1(i) of \citet{li2024detection} establishes asymptotic Type~I error control for any fixed $c_\zeta>0$ under their stated assumptions. In finite samples, however, smaller values of $c_\zeta$ lower the screening
threshold and can substantially increase the power-enhancement term $Z_n^\circ$ by increasing the number of coordinates that contribute to it.
Since each such coordinate contributes $\sqrt{K\vee n}$ to $Z_n^\circ$, this effect can be meaningful in high dimensions, especially because the limiting Brownian-bridge reference distribution has no explicit dependence on $K$.  In their simulation study (and in our implementation of their method in our own simulation study), \citet{li2024detection} set
$c_\zeta=\smash{(\widehat\lambda_1)^{1/2}}$, where $\widehat\lambda_1$ is the estimated largest eigenvalue of the long-run covariance operator of the aggregated
process. In the code used for their empirical DJIA analysis the choice $c_\zeta=\smash{(\widehat\lambda_1)^{1/2} /12.5}$ is instead used, thereby lowering the screening threshold by a factor of $12.5$. 
While this choice is asymptotically valid, in our reimplementation of the empirical analysis using $c_\zeta=\smash{(\widehat\lambda_1)^{1/2}}$, no marginal CUSUM statistic exceeds the screening threshold and the resulting global PE-CUSUM test does not reject, with $p=0.225$.

%
%

\begin{acks}[Acknowledgments]
During the preparation of this manuscript, the authors used ChatGPT, developed by OpenAI, to improve the language and clarity of the text and to assist with the presentation of figures. The authors reviewed and edited all AI-assisted content and take full responsibility for the final manuscript.
\end{acks}
\begin{funding}
Both authors were supported by the Deutsche Forschungsgemeinschaft (DFG, German Research Foundation; Project-ID 520388526; TRR 391: Spatio-temporal Statistics for the Transition of Energy and Transport) which is gratefully acknowledged. Calculations for this publication were performed on the HPC cluster Elysium of the Ruhr University Bochum, subsidised by the DFG (INST 213/1055-1).
\end{funding}

\begin{supplement}
\stitle{Supplement to the Paper: ``Simultaneous Change-Point Inference for High-Dimensional Functional Time Series''}
\sdescription{The supplement contains all proofs for the results in this paper.}
\end{supplement}


\bibliographystyle{imsart-nameyear} 
\putbib[biblio]
\end{bibunit}
\begin{bibunit}[imsart-nameyear]

\newpage

\appendix
\thispagestyle{empty}
\numberwithin{equation}{section}

\counterwithin{figure}{section}
\counterwithin{table}{section}

\begin{center}
	
	{\bfseries SUPPLEMENT TO THE PAPER:  \\  ``\MakeUppercase{Simultaneous Change-Point Inference for High-Dimensional Functional Time Series'' }}
	\vspace{.5cm}
	
	{\textsc{By Axel Bücher and David Colin Decker}}
	
	\vspace{.28cm}
	
	{\textit{Ruhr-Universität Bochum}}
	
	\vspace{.28cm}

	\begin{center}
		\begin{minipage}{.6\textwidth}
			{\small \hspace{.5cm}
					This supplement contains all proofs for the results stated in the main paper. Section~\ref{sec:notation-and-definitions} introduces notation and definitions used throughout. The proofs of the four main results are presented in Section~\ref{sec:proofs-of-main-theorems}, with supporting arguments collected in Sections~\ref{sec:supporting-results} and~\ref{sec:miscellaneous}. Finally, Sections~\ref{sec:orlicz} and~\ref{sec:Hilbert_space_defs} summarize standard background on Orlicz norms and Hilbert spaces, respectively.
				 }
		\end{minipage}
	\end{center}

\end{center}

\vspace{.5cm}

\begingroup

\etocsettagdepth{main}{none}

\etocsettagdepth{Supplementary}{subsection}


\setcounter{tocdepth}{2}
\tableofcontents

\endgroup

\etocdepthtag.toc{Supplementary}

\appendix

\section{Notation and Definitions}
\label{sec:notation-and-definitions}

Standard notation related to separable Hilbert spaces is defined in Section~\ref{sec:Hilbert_space_defs}. In particular, $\|\cdot\|_{\Tr}$, $\|\cdot\|_{\HS}$ , $\|\cdot\|_{\op}$ denote the trace, Hilbert-Schmidt, and operator norms of a linear operator on a reference Hilbert space. If $T:\Hb_1\to \Hb_2$ is a linear map, and then $T^*:\Hb_2\to \Hb_1$ is its adjoint. If $A=\varnothing$ and $(\cdot)$ is a generic argument we define $\max_{s \in A}(\cdot):=0$ and  $\sum_{s \in A}(\cdot):=0$; in the second case, $0$ is interpreted as the $0$ element of a relevant vector space.  For example if  $a <b$ we define $\sum_{b}^{a}(\cdot):=0$. If $ z \in \N$, $[z] = \{1, \dots, z\}$, and if $z=\infty$, $[z]:=\N$. If $a \in \R$,   $\lfloor a \rfloor  = \max\{z \in \Z : z \leq a\}$, and $\lceil a \rceil = \min\{z \in \Z: z \geq a\}$, the ceiling and floor functions. It is straightforward to see that $\lceil a \rceil-\lfloor a \rfloor =1$ if $a\notin \Z$. If $X$ and $Y$ are random variables and $\rho \in \R$, we define the $\rho$-restricted Kolmogorov distance
\begin{align}
\label{eq:restricted-k-distance}
d_{\mathrm{K}}^{(\rho)}(X,Y) := \sup_{t \ge \rho} d_t\big(X,Y\big), \qquad d_t\big(X,Y\big):=\Big|\Prob\big(X\leq t\big)-\Prob\big(Y\leq t\big)\Big|.
\end{align}

If $A=(A_{ij})_{ij}$ is a matrix, $\|A\|_{\infty}=\max_{i,j}|A_{ij}|$. Let $I=[0,1]$ be the unit interval. Let $\N_{\geq 0}:=\N \cup \{0\}$ and for a natural number $x$ let $\N_{\geq x}:=\{n \in \N: n \geq x\}$. We use $\Z$ to refer to the positive and negative integers and $0$. We write $A\lesssim B$ if there is an absolute constant $C$ such that $A\leq C \cdot B$. 

For the $\beta$-mixing coefficients $\beta_h$ of Assumption~\ref{cond:stationary} define the following serial-dependence coefficient for all $j\in \N_{\geq 0}$:
\begin{equation}\label{eq:def-S_beta}
S_{\beta,j}:= \sum_{h=1}^{\infty}
    h^j\beta_h\log^2(4/\beta_h),
\end{equation}
with the convention that \(0\log^2(4/0)=0\). We will only use $S_{\beta,0}$ and $S_{\beta,1}$. Note that $S_{\beta,0}\leq S_{\beta,1}$. Since Assumption~\ref{cond:stationary} gives $\beta_h\leq C_\beta h^{-c_\beta}$ for some $c_\beta>2$, it holds that $S_{\beta,1}$ is bounded by a constant that depends only on $C_\beta$ and $c_\beta$.

\subsection{Hilbert Space Conventions and Projection Operators}
\label{subsec:direct_sum}
Recall that $\Hb_1, \dots, \Hb_K$ are separable Hilbert spaces, and that $\mathcal{H}:=\oplus_{k=1}^K\Hb_k$ is the collection of all tuples of the form $(v_1, \dots, v_K)$, where for each $k \in [K]$, $v_k\in \Hb_k$. Equipped with the direct sum vector  space structure (component wise addition and scalar multiplication distributed over the components) and the direct sum inner-product
\begin{equation*}
\big\langle (w_1, \dots, w_K),  (v_1, \dots, v_K)\big\rangle_{\mathcal{H}}=\sum_{k \in [K]}\langle w_k, v_k\rangle_{\Hb_k}
\end{equation*}
is itself a separable Hilbert space. Given an ordered subset $\chgset = \{k_{1}, \dots, k_{d}\} \subset [K]$ consisting of distinct elements (so that $|\chgset|=d$) we define $\pi_{\chgset}:\mathcal{H}\to \oplus^{d}_{j=1}\Hb_{k_j}=:\Hb_{\chgset}$ as $\pi_{\chgset}w=[w_{k_1}, \dots, w_{k_d}]$. In case $\chgset =\{k\}$ for some $k\in [K]$ we define $\pi_{k}:=\pi_{\chgset}=\pi_{\{k\}}$ and generally drop the parentheses in the subscript of a factor projection map to lighten notation.

Recall for each $k \in [K]$ the trace class linear operator $\mathcal{K}_k: \Hb_k \to \Hb_k$ defined above Assumption~\ref{cond:scores}. Let  $\mathcal{K}:\mathcal{H} \to \mathcal{H} $ be the long run covariance operator of the strictly stationary centered series strictly stationary and centered series $(\eps^{(i)})_{i \in \Z}$ of~\eqref{def:global_noise}, which exists by Lemma~\ref{lem:convergence_to_long_term_trace}.  For each $k \in [K]$ we have the identity $\mathcal{K}_k=\pi_k\circ \mathcal{K}\circ \pi^*_k$. We define 
\begin{equation}\label{def:C_Kappa}
C_{\mathcal K}:=\max_{k \in [K]}\big\|\mathcal{K}_k\|_{\Tr},
\end{equation}
which is the $L^\infty$-trace norm of $\mathcal{K}$, with respect to the decomposition $\mathcal{H}:=\oplus_{k=1}^K\Hb_k$. Under Assumption~\ref{cond:scores}, we have 
\begin{align}
\label{eq:C_t}
C_{\mathcal K}  \le \mathcal C_{\mathcal T} :=  C_U^2 \sum_{j=1}^\infty j^{-2\gamma_U} < \infty,
\end{align}
which is finite since $\gamma_U>1/2$. Further note that
\begin{align} \label{eq:trace-class-covariance-products}
\| \mathcal K \|_{\Tr} = \sum_{k \in [K]} \| \mathcal K_k\|_{\Tr} \le K \mathcal C_{\mathcal T},
\end{align}
which is valid since $\mathcal K$ is positive-semidefinite. 
For each $k\in[K]$, define
\[
    \tilde{\Hb}_k := \overline{\textrm{im}(\mathcal K_k)} \subseteq \Hb_k,
\]
where the closure is taken in the norm topology of $\Hb_k$. Since $\tilde{\Hb}_k$ is a closed linear subspace of the separable Hilbert space $\Hb_k$, it is itself a separable Hilbert space, with vector space structure and inner product inherited from $\Hb_k$. Let
\[
    \tilde{\mathcal H} := \bigoplus_{k \in [K]} \tilde{\Hb}_k .
\]
By the spectral theorem for compact, self-adjoint, positive operators, the eigenvectors of $\mathcal K_k$ corresponding to strictly positive eigenvalues form an orthonormal basis of $\tilde{\Hb}_k$. Moreover, if $X\in\Hb_k$ is a centered random element with covariance operator $\mathcal K_k$, then
\[
    \Prob(X\in \tilde{\Hb}_k)=1;
\]
see Lemma~\ref{lem:properties-covariance}(3).
Thus, although $X$ is formally an $\Hb_k$-valued random element, it almost surely lies in the principal subspace $\tilde{\Hb}_k$ generated by the eigen-directions of $\mathcal K_k$. As in the formulation of Assumption~\ref{cond:scores}, let $r_k$ denote the number of strictly positive eigenvalues of $\mathcal K_k$, counted with multiplicity; $r_k$ may be finite or infinite. We write the corresponding orthonormal eigenvectors as $\{Z_{kj}\}_{j=1}^{r_k}$. These eigenvectors form an orthonormal basis of $\tilde{\Hb}_k$, meaning that
\[
    \tilde{\Hb}_k
    =
    \overline{\operatorname{span}}\{Z_{kj}\}_{j=1}^{r_k}.
\]
Equivalently, every $x\in\tilde{\Hb}_k$ admits the expansion
\[
    x
    =
    \sum_{j=1}^{r_k}
    \langle x,Z_{kj}\rangle_{\Hb_k} Z_{kj},
\]
where the sum is finite if $r_k<\infty$ and converges in $\Hb_k$ if $r_k=\infty$. For each $M\in\N$, define
\begin{align}
\label{eq:definition-check-mk}
    \check M_k := r_k\wedge M,
\end{align}
with the convention that $\infty\wedge M=M$. We then define $ P_{kM}:{\Hb}_k\to \Hb_k$
by 
\begin{equation}\label{eq:def-PkM}
    P_{kM}x
    :=
    \sum_{j=1}^{\check M_k}
    \langle x,Z_{kj}\rangle_{\Hb_k} Z_{kj}.
\end{equation}
Thus $P_{kM}x$ is the orthogonal projection (in particular, it is self-adjoint) of $x$ onto the span of the first $\check M_k$ principal directions of $\mathcal{K}_k$.  In particular, for every $x\in{\Hb}_k$,
\begin{align}
\label{eq:orthogonal}
    \|P_{kM}x\|_{\Hb_k}^2
    +
    \|(\id_{\Hb_k}-P_{kM})x\|_{\Hb_k}^2
    =
    \|x\|_{\Hb_k}^2.
\end{align}
We then define the operator $P_M=\oplus_{k=1}^K P_{kM}: \mathcal H \to \mathcal H$, which takes values in $\tilde{\mathcal H}$.
Finally, define
\begin{align}
\label{eq:def-OkM}
    O_{kM}:
    \operatorname{span}\{Z_{kj}:1\le j\le \check M_k\}
    \to
    \R^{\check M_k}
\end{align}
to be the linear isometry that sends each basis vector $Z_{kj}$ to the $j$th standard basis vector $e_j\in\R^{\check M_k}$. Equivalently
\[
    x=\sum_{j=1}^{\check M_k} v_j Z_{kj} \mapsto O_{kM}x = (v_1,\ldots,v_{\check M_k})^\top.
\]
We have $O_{kM}^*=O_{kM}^{-1}$.

\subsection{Blocking and Coupling}
\label{subsubsec:blocking_coupling}
We first introduce the sigma-field generated by the noise variables over the observed excerpt:
\[
\label{def:mathcal_F_n}
\mathcal F_n
:=
\sigma\left(
    \varepsilon_k^{(i)}:
    1\le i\le n,\ k\in[K]
\right)
=
\sigma\left(
    \varepsilon^{(i)}:
    1\le i\le n
\right).
\]
Recalling the sets \(I_\ell\) and \(J_\ell\) from~\eqref{def:big_little_blocks}, define the big- and small-block noise sums, for \(\ell\in[m]\), by
\begin{equation}
\label{def:block_sums}
S_\ell
:=
\sum_{i\in I_\ell}\varepsilon^{(i)},
\qquad
S'_\ell
:=
\sum_{i\in J_\ell}\varepsilon^{(i)}.
\end{equation}
By Berbee's coupling lemma, Lemma~\ref{lem:berbee}, possibly after enlarging the probability space, we may construct random elements
\begin{equation*}
\label{def:tilde eps}
\tilde \eps^{(1)}, \dots, \tilde \eps^{(n)}
\end{equation*}
such that for all $i \in [n]$ it holds that $\tilde \eps^{(i)}=_d \eps^{(i)}$, and that further, it holds that \(\widetilde S_1,\ldots,\widetilde S_m\) are independent where for each $\ell \in [m]$  we define
\begin{equation}\label{def:big_block_sums}
 \widetilde S_\ell = \sum_{i\in I_\ell}\tilde \eps^{(i)}.
\end{equation}
Further, the Berbee-coupled block sums satisfy 
\[
    \widetilde S_\ell\stackrel{d}{=}S_\ell,
    \qquad
    \ell\in[m].
\]
Moreover, defining
\begin{equation}
\label{def:Omega_big}
\Omega_{\mathrm{big}}
:=
\left\{
S_\ell=\widetilde S_\ell
\text{ for all }\ell\in[m]
\right\},
\end{equation}
we have
\[
\mathbb P(\Omega_{\mathrm{big}}^c)
\le
\sum_{\ell=2}^m
\beta\!\left(
\sigma(S_\ell),
\sigma(S_1,\ldots,S_{\ell-1})
\right)
\le
(m-1)\beta_r.
\]
Here the final inequality uses the fact that consecutive big blocks are separated by a small block of length \(r\). Similarly, again possibly after enlarging the probability space, we may construct random elements
\[
\widetilde S'_1,\ldots,\widetilde S'_m
\]
such that \(\widetilde S'_1,\ldots,\widetilde S'_m\) are independent and
\[
    \widetilde S'_\ell\stackrel{d}{=}S'_\ell,
    \qquad
    \ell\in[m].
\]
Defining
\begin{equation*}
\Omega_{\mathrm{small}}
:=
\left\{
S'_\ell=\widetilde S'_\ell
\text{ for all }\ell\in[m]
\right\},
\end{equation*}
we have
\[
\mathbb P(\Omega_{\mathrm{small}}^c)
\le
\sum_{\ell=2}^m
\beta\!\left(
\sigma(S'_\ell),
\sigma(S'_1,\ldots,S'_{\ell-1})
\right)
\le
(m-1)\beta_q,
\]
because consecutive small blocks are separated by a big block of length \(q\). Consequently, on the event
\begin{equation}
\label{eq:Omega-block}
\Omega_{\mathrm{block}}
:=
\Omega_{\mathrm{big}}\cap\Omega_{\mathrm{small}},
\end{equation}
both the big-block sums and the small-block sums agree with their respective independent coupled versions. Moreover,
\begin{equation}
\label{eq:Omega-block-probability}
\mathbb P(\Omega_{\mathrm{block}}^c)
\le
(m-1)\beta_r+(m-1)\beta_q.
\end{equation}
The two coupled families are used separately below; the construction does not require the coupled big-block family and the coupled small-block family to be mutually independent. Finally, we define the sigma-field generated by the coupled big-block sums:
\begin{equation*}
\label{def:tilde_mathcal_F_n}
\widetilde{\mathcal F}_n
:=
\sigma\left(
\widetilde S_\ell:
1\le \ell\le m
\right).
\end{equation*}
We define the join of $\widetilde{\mathcal F}_n$ and $\mathcal F_n$, the smallest $\sigma$-field containing both $\widetilde{\mathcal F}_n$ and $\mathcal F_n$, as 
\begin{equation}\label{def:join_data_coupled}
\widetilde{\mathcal F}_n \vee \mathcal F_n =\sigma(\widetilde{\mathcal F}_n, \mathcal F_n).
\end{equation}

\subsection{Tensor product notation and projections}
\label{subsec:tensor-products}
For a symmetric matrix $A=(a_{st})_{s,t=1}^n \in \R^{n \times n}$ and an operator $T:\mathcal H \to \mathcal H$, we define the Kronecker product operator $(A \otimes T):\mathcal H^n \to \mathcal H^n$ by
\begin{align} \label{eq:Kronecker-product-operator}
\begin{pmatrix}
    x_1 \\ \vdots \\ x_n
\end{pmatrix}  \mapsto 
(A \otimes T) 
\begin{pmatrix}
    x_1 \\ \vdots \\ x_n
\end{pmatrix} 
:=
\begin{pmatrix}
    a_{11} T & \dots & a_{1n} T \\
    \vdots & & \vdots \\
    a_{n1} T & \dots & a_{nn} T
\end{pmatrix} 
\begin{pmatrix}
    x_1 \\ \vdots \\ x_n
\end{pmatrix}
=
\begin{pmatrix}
    \sum_{s=1}^n a_{1s} T x_s \\
    \vdots  \\
    \sum_{s=1}^n a_{ns} T x_s
\end{pmatrix} 
\end{align}
Clearly, $A \otimes T$ is bi-linear. A standard argument based on the spectral theorem shows that
\begin{align} \label{eq:operator-norm-Kronecker}
\| A \otimes T \|_{\op} = \| A \|_{\op} \| T \|_{\op}. 
\end{align}
For $\sigma \subset [n]$ and $\chgset \subset [K|$, write 
\begin{align}
\label{eq:product-spaces-with-indices}
\mathcal H_\chgset = \bigoplus_{k \in \chgset} \Hb_k, \qquad \mathcal H_\chgset^\sigma = \bigoplus_{s \in \sigma}  \mathcal H_\chgset.
\end{align}
Note that $\mathcal H = \mathcal H_{[K]}$ and $\mathcal H^n = \mathcal H^{[n]}$. In later sections, we will be interested in the projection
\begin{align}
\label{eq:projection-pisigmak}
\pi_{\sigma,\chgset}: \mathcal H^n \to \mathcal H_\chgset^\sigma, 
\qquad 
(x_1, \dots, x_n) \mapsto \big( (x_{s_1k})_{k \in \chgset}, \dots, (x_{s_{|\sigma|}k})_{k \in \chgset}\big),
\end{align}
and in particular, in the operator
\[
\pi_{\sigma ,\chgset }(A \otimes T) \pi_{\sigma ,\chgset }^*:  \mathcal H_\chgset^\sigma \to  \mathcal H_\chgset^\sigma.
\]
This operator can be rewritten in terms of further projections as follows: write
\begin{align*}
    p_\sigma: &\mathcal H^n \to \mathcal H^\sigma,
    \qquad
    (x_1, \dots, x_n) \mapsto (x_s)_{s \in \sigma}, 
    \\
    \pi^\sigma_\chgset :  & \mathcal H^\sigma \to \mathcal H_\chgset^\sigma,
    \qquad
    (x_s)_{s \in \sigma } \mapsto ( (x_{sk})_{k \in \chgset})_{s \in \sigma}, 
\end{align*}
and note that
\[
\pi_{\sigma,\chgset} 
=
\pi_\chgset^\sigma \circ p_{\sigma}. 
\]
A straightforward calculation shows that
\[
p_\sigma (A \otimes T) p_\sigma^* = A_\sigma \otimes T : \mathcal H^\sigma \to \mathcal H^\sigma,
\]
where $A_\sigma =(a_{st})_{s,t \in \sigma}$ is the submatrix of $A$ with rows and columns in $\sigma$. Next,
\[
\pi_\chgset^\sigma (A_\sigma \otimes T) (\pi_\chgset^\sigma)^*
=
A_\sigma \otimes (\pi_\chgset T \pi_\chgset^* )
\]
where we write $\pi_ \chgset = \pi^{[n]}_\chgset$. As a consequence, 
\begin{align}
\label{eq:representation-projection-kronecker}
\pi_{\sigma ,\chgset }(A \otimes T) \pi_{\sigma, \chgset }^* = A_\sigma \otimes  (\pi_\chgset T \pi_\chgset^* ).
\end{align}

\subsection{Discretization of the Brownian Bridge}
\label{subsec:discrete_brownian-bridge}
For each $u, v  \in I=[0,1]$ let 
\begin{equation}\label{def:brownian_covariance_kernel_variance_2}
K(u,v) := u\wedge v -uv, \hspace  {0.5cm}V(u):= K(u,u)=u(1-u).
\end{equation}
Recalling $\mathbb B_{\mathcal K}$ from \eqref{def:Hilbert_Brownian_Bridge} define the discretized Brownian Bridge $\mathcal{B}^{(n)} \in \Hc^n$ by
\begin{align}
\label{eq:def-bn}
\mathcal{B}^{(n)}
:= 
\bigl(
\mathbb B_{\mathcal K}(1/n),
\ldots,
\mathbb B_{\mathcal K}(1)
\bigr)^\top.
\end{align}
Note that the associated covariance operator $\Sigma^{(n)}:=\Cov(\mathcal B^{(n)}):\mathcal H^n \to \mathcal H^n$ can be written as 
\[
\begin{pmatrix}
    x_1 \\ \vdots \\ x_n
\end{pmatrix}  \mapsto 
\Sigma^{(n)} \begin{pmatrix}
    x_1 \\ \vdots \\ x_n
\end{pmatrix} 
:=
\begin{pmatrix}
    \Sigma^{(n)}{(1,1)} & \dots & \Sigma^{(n)}{(1,n)} \\
    \vdots & & \vdots \\
    \Sigma^{(n)}{(n,1)} & \dots & \Sigma^{(n)}{(n,n)}
\end{pmatrix} 
\begin{pmatrix}
    x_1 \\ \vdots \\ x_n
\end{pmatrix}
\]
with $\Sigma^{(n)}(s,t): \mathcal H \to \mathcal H$ given by
\[
\Sigma^{(n)}(s,t) = K\Big(\frac sn, \frac tn\Big) \cdot \mathcal K.
\]
Using the Kronecker product operator from \eqref{eq:Kronecker-product-operator}, we can write this as
\begin{align} 
\label{eq:Kronecker-Kn}
\Sigma^{(n)} = K_n \otimes \mathcal K, \qquad K_n = [K(s/n,t/n)]_{s,t=1}^n \in \R^{n \times n}.
\end{align}
Finally, for $ (k,s)\in[K]\times[n]$, we also write
\[
\mathcal{B}_{sk}^{(n)}
= \pi_{sk} \mathcal B^{(n)}
=
\mathbb B_{\mathcal K,k}(s/n) 
\in \Hb_k,
\]
with $\pi_{sk}:=\pi_{\{s\},\{k\}}$ from \eqref{eq:projection-pisigmak}.

\subsection{Key Intermediate Gaussian Variables}
\label{subsec:key-intermediate-variables}
Recall from Section~\ref{subsec:bootstrap} that $
    m:=\left\lfloor n/(q+r)\right\rfloor$. For each \(s\in[n]\) define
\begin{equation}\label{eq:def-Bs}
    B_s:=\left\lceil \frac{s}{q+r}\right\rceil ,
\end{equation}
so that \(1\le B_s\le m+1\). For \(s\in[n]\) and \(\ell\in[m]\), define
\begin{equation}
\label{eq:def-a_sl}
    a_{s\ell}
    :=
    \mathbf 1\{\ell\le B_s\}-\frac{s}{n},
    \qquad
    a_s:=(a_{s1},\ldots,a_{sm})^\top\in\mathbb R^m .
\end{equation}
Let $A_n \in \mathbb R^{n \times n}$ be the matrix with entries
\begin{equation*}
    [A_n]_{st}
    :=
    \frac{1}{m}
    \langle a_s,a_t\rangle \in [-1,1],
    \qquad s,t\in[n].
\end{equation*}
Since \(A_n\) is a rescaled Gram matrix, it is positive semidefinite. 
Let
\[
G^{(n,q)} = (G^{(q)}_1, \dots, G_n^{(q)}) \in \mathcal H^n
\]
be centered Gaussian with covariance operator $\Sigma^{(n,q)}: \mathcal H^n \to \mathcal H^n$ defined by 
\[
\begin{pmatrix}
    x_1 \\ \vdots \\ x_n
\end{pmatrix}  \mapsto 
\Sigma^{(n,q)} \begin{pmatrix}
    x_1 \\ \vdots \\ x_n
\end{pmatrix} 
:=
\begin{pmatrix}
    \Sigma^{(n,q)}{(1,1)} & \dots & \Sigma^{(n,q)}{(1,n)} \\
    \vdots & & \vdots \\
    \Sigma^{(n,q)}{(n,1)} & \dots & \Sigma^{(n,q)}{(n,n)}
\end{pmatrix} 
\begin{pmatrix}
    x_1 \\ \vdots \\ x_n
\end{pmatrix}
\]
where, for $s,t \in [n]$, $\Sigma^{(n,q)}(s,t): \mathcal H \to \mathcal H$ is defined by 
\[
\Sigma^{(n,q)}{(s,t)}:= \frac1m \langle a_s, a_t\rangle \cdot \Sigma_q
\]
and where
\[
    \Sigma_q
    :=
    \Cov\bigl(q^{-1/2}S_1\bigr),
\]
with \(S_1 \in \mathcal H\) the first big-block noise sum defined in~\eqref{def:block_sums}. Using the Kronecker product operator from \eqref{eq:Kronecker-product-operator}, we can write
\[
\Sigma^{(n,q)} = A_n \otimes \Sigma_q
\]
Note that 
\[
\Sigma^{(n,q)}{(s,t)}
=
\Cov(G^{(q)}_s,G_t^{(q)} ), \qquad  \Sigma^{(n,q)}{(s,t)} = (\Sigma^{(n,q)}{(t,s)})^*.
\]

Fix $s \in [n]$, $k \in [K]$ and $M \in \N$. Recall $P_{kM}$ from \eqref{eq:def-PkM}, $O_{kM}$ from \eqref{eq:def-OkM}, and $\pi_{sk} : \mathcal H^n \to \Hb_k, (x_1, \dots, x_n) \to x_{sk}$, where $x_s=(x_{s1}, \dots, x_{sK}) \in \mathcal H$. Define
\begin{align}
\label{eq:definition-GMsk}
G_{sk}^{M} = G_{sk}^{(M,n,q)} :=  (O_{kM} \circ P_{kM} \circ \pi_{sk}) G^{(n,q)},
\end{align}
which is a random variable in $\R^{\check M_k}$ with $\check M_k = r_k \wedge M$. A straightforward calculation shows that the cross-covariance between $G_{sk}^M$ and $G_{t\ell}^M$ is
\begin{align}
\label{eq:covariance-GMsk}
\Cov( G_{sk}^M, G_{t\ell}^M ) 
&=  \nonumber
O_{kM} \circ P_{kM} \circ \pi_{sk} \circ \Sigma^{(n,q)} \circ \pi_{t\ell}^* \circ P_{\ell M}^* \circ O_{\ell M}^*
\\&=
\frac1m \langle a_s , a_t \rangle \cdot \Big( \langle \Sigma_q \pi_\ell^* Z_{\ell j}, \pi_k^* Z_{ki} \rangle_{\mathcal H} \Big)_{i,j=1}^M \in \R^{M \times M}
\end{align}
with row index $i$ and column index $j$. Here, we made use of the fact that $P_{\ell M}^* = P_{\ell M}$. 
By Lemma~\ref{lem:properties-covariance} and since $\pi_\ell^*(y_\ell) = (0, \dots, 0, y_\ell, 0 \dots 0)$, we have the alternative representation
\begin{align}
\label{eq:covariance-GMsk-alternative-expresssion}
\langle \Sigma_q \pi_\ell^* Z_{\ell j}, \pi_k^* Z_{ki} \rangle_{\mathcal H}
=
\frac1q \E\Big[ \langle S_1, Z_{\ell j} \rangle_{\Hb_k}\langle S_1, Z_{k i} \rangle_{\Hb_k}  \Big].
\end{align}

\section{Proofs of Main Theorems}
\label{sec:proofs-of-main-theorems}

\begin{proof}[Proof of Theorem~\ref{thm:gaussian_approximation}]
Fix \(P\in\mathcal P_0\), and define
\[
B_P
:=
\max_{k\in[K]}\sup_{u\in [0,1]}
\|\mathbb B_{\mathcal K_P,k}(u)\|_{\Hb_k}^2,
\]
where \(\mathbb B_{\mathcal K_P}\) is from \eqref{def:Hilbert_Brownian_Bridge}. Under \(H_0^\infty\), the baseline mean \(\mu_k\) from \eqref{eq:change_point_model} cancels from every marginal CUSUM statistic in $T_n$ from \eqref{def:global_test}, such that $T_n = T_n^O$ with \(T_n^O\) from \eqref{eq:def_oracle_test_statistic}. It is therefore sufficient to show that 
\begin{align}
\label{eq:gap-proof1}
d_{\mathrm K}^{(\rho)}\big( (T_n^O)^2, B_P \big)
\le
C n^{-c}\log^d(n\vee K).
\end{align}

For \(x\in(0,3/16]\) and \(M\in\mathbb N_{\ge1}\), define
\[
Z
:=
\max_{k\in[K]}\max_{s\in S_x}
\|G_{sk}^M\|_{\check M_k}^2,
\]
where \(S_x\) and \(G_{sk}^M\) are from
\eqref{eq:definition-S_x} and \eqref{eq:definition-GMsk}, respectively.
The distribution of \(Z\) depends on
\(P,n,q,r,x\), and \(M\);
this dependence is suppressed.
Let \(C_1,c_1>0\) and \(C_2,c_2>0\) be the constants in
Theorems~\ref{Thm_1} and~\ref{Thm_2}, respectively, evaluated at the
fixed threshold \(\rho\). Whenever \(x,q,r,M\) satisfy the hypotheses
of both theorems, their conclusions and the triangle inequality for
\(d_{\mathrm K}^{(\rho)}\) give
\begin{align}
d_{\mathrm K}^{(\rho)}\bigl((T_n^O)^2,B_P\bigr)
&\le
C_0x^{-3/2}\log^{A_0}(nK)
\left\{
\left(\frac{M^4q^2}{n}\right)^{1/6}
+\left(\frac{M^2}{q\wedge m}\right)^{1/3}
\right.
\nonumber\\[-0.2em]
&\hspace{4.5cm}\left.
+M^{1-2\gamma_U}
+\left(\frac rq\right)^{1/2}
\right\}
+2C_\beta\frac nq r^{-c_\beta},
\label{eq:short-combined-bound}
\end{align}
where we have used \(\beta_r\le C_\beta r^{-c_\beta}\) and where one may take
\begin{equation*}
C_0:=C_1+C_2,
\qquad
A_0:=c_1\vee c_2.
\end{equation*}
Note that the constants in \eqref{eq:short-combined-bound} are uniform over
\(P\in\mathcal P_0\). In particular,
\[
C_0
=
C_0\bigl(
\rho,D_\psi,C_\beta,c_\beta,
C_L,C_U,\gamma_L,\gamma_U
\bigr),
\qquad
A_0=A_0(\gamma_L,\gamma_U).
\]

Let \(C_X^{(1)},C_X^{(2)}\) and \(C_M^{(1)},C_M^{(2)}\) be the constants constructed in Theorems~\ref{Thm_1} and~\ref{Thm_2}, respectively, 
and define
\begin{align*}
C_X
&:=
C_X^{(1)}\vee C_X^{(2)}
\vee
\frac{16}{3(\log4)^8},
\qquad 
C_M
:=
C_M^{(1)}\vee C_M^{(2)}.
\end{align*}
If we then take
\[
x:=\frac1{C_XL_{nK}^4},
\qquad
q:=\lfloor n^{9/20}\rfloor,
\qquad
r:=\lfloor n^{3/10}\rfloor
\]
we have \(x\in(0,3/16]\) and condition \text{\normalfont {(X)}} from~\eqref{eq:side-conditions} holds.
Define
\begin{equation}
\eta_0
:=
1\wedge C_M^{-1/(2\gamma_L)},
\qquad
a_n
:=
\eta_0
x^{1/(2\gamma_L)}n^{1/(100\gamma_L)},
\qquad
M:=\lfloor a_n\rfloor.
\label{eq:proof-choice-M}
\end{equation}
First suppose that
\begin{equation}
L_{nK} := \log(nK)\log n \le n^{1/400}.
\label{eq:short-moderate-regime}
\end{equation}
There exists an absolute integer \(n_{0,1}\) such that, for every
\(n\ge n_{0,1}\),
\begin{equation}
1\le r<q
\le
\min\left\{\frac n4,\frac{\sqrt n}{\log n}\right\},
\qquad
q\wedge m=q,
\qquad
q\ge n^{1/50}.
\label{eq:proof-block-admissibility}
\end{equation}
Indeed, these assertions follow directly from the definitions of
\(q,r\) and \(m=\lfloor n/(q+r)\rfloor\). Under
\eqref{eq:short-moderate-regime},
\begin{align*}
a_n
=
\eta_0
C_X^{-1/(2\gamma_L)}
L_{nK}^{-2/\gamma_L}n^{1/(100\gamma_L)}
\ge
\eta_0
C_X^{-1/(2\gamma_L)}
n^{1/(200\gamma_L)}.
\end{align*}
Consequently, if
\begin{equation*}
n
\ge
n_{0,2}
:=
\left\lceil
\left(
2
\eta_0^{-1}
C_X^{1/(2\gamma_L)}
\right)^{200\gamma_L}
\right\rceil,
\end{equation*}
then \(a_n\ge2\), and hence
\begin{equation}
\frac{a_n}{2}\le M\le a_n.
\label{eq:proof-control-floor-M}
\end{equation}
We used \(\lfloor u\rfloor\ge u/2\) whenever \(u\ge2\). In
particular, \(M\ge1\). Moreover, by
\eqref{eq:proof-choice-M} and
\eqref{eq:proof-block-admissibility},
\begin{align*}
M^{2\gamma_L}
&\le
a_n^{2\gamma_L}
=
\eta_0^{2\gamma_L}
xn^{1/50}
\le
\frac{xq}{C_M}
=
\frac{x(q\wedge m)}{C_M}.
\end{align*}
Thus condition \text{\normalfont {(M)}} from~\eqref{eq:side-conditions} also holds. 

We next evaluate the terms in~\eqref{eq:short-combined-bound}. After
increasing \(n_{0,1}\) if necessary, the floor bounds for \(q,r\),
the inequality \(x\le1\), and \(\gamma_L>1/2\) give
for absolute constants
\(C_{\mathrm{rate},1},C_{\mathrm{rate},2},
C_{\mathrm{rate},4}>0\)
\begin{align*}
\left(\frac{M^4q^2}{n}\right)^{1/6}
&\le
C_{\mathrm{rate},1}
n^{-1/60+1/(150\gamma_L)}
\le
C_{\mathrm{rate},1}
n^{-1/300},
\\
\left(\frac{M^2}{q\wedge m}\right)^{1/3}
&\le C_{\mathrm{rate},2}
n^{-3/20+1/(150\gamma_L)}
\le C_{\mathrm{rate},2}
n^{-41/300},
\\
\left(\frac rq\right)^{1/2}
&\le
C_{\mathrm{rate},4}
n^{-3/40}.
\end{align*}
Because \(1-2\gamma_U<0\), the lower bound for \(M\) in
\eqref{eq:proof-control-floor-M} is needed to control the 
remainder \(M^{1-2\gamma_U}\) in \eqref{eq:short-combined-bound}.
It allows to choose \(C_{\mathrm{rate},3}>0\) sufficiently large and depending only on
\(\eta_0,C_X,\gamma_L\), and \(\gamma_U\), such that
\begin{align*}
M^{1-2\gamma_U}
&\le
(a_n/2)^{1-2\gamma_U}
\le
C_{\mathrm{rate},3}
L_{nK}^{\,2(2\gamma_U-1)/\gamma_L}
n^{-(2\gamma_U-1)/(100\gamma_L)},
\end{align*}
After increasing \(n_{0,1}\) if necessary, the floor bounds give
$q \ge n^{9/20}/2$ and $r\ge n^{3/10}/2$.
Thus, with $C_{\mathrm{rate},5}
:=
2^{c_\beta+1}$, we have
\begin{align*}
\frac nq r^{-c_\beta}
&\le
C_{\mathrm{rate},5}
n^{11/20-3c_\beta/10}
\le
C_{\mathrm{rate},5}n^{-1/20},
\end{align*}
where the final inequality uses \(c_\beta>2\). 
Note that all constants $C_{\mathrm{rate},j}$ are uniform over \(P\in\mathcal P_0\).

Since $x^{-3/2}=C_X^{3/2}L_{nK}^{6}$ and $\log(nK)\le L_{nK}$, 
the derived upper bounds on the terms in~\eqref{eq:short-combined-bound} from the previous five displayed lines imply that, in the case $L_{nk} \le n^{1/400}$ from \eqref{eq:short-moderate-regime} and if 
\(n\ge n_{0,1}\vee n_{0,2}\), 
\begin{equation}
d_{\mathrm K}^{(\rho)}\bigl((T_n^O)^2,B_P\bigr)
\le
C_{\mathrm{mod}}n^{-z}L_{nK}^{A},
\label{eq:short-moderate-bound}
\end{equation}
where 
\begin{equation}
z
:=
\min\left\{
\frac1{300},
\frac{2\gamma_U-1}{100\gamma_L}
\right\}>0,
\qquad
A
:=
A_0+6+\frac{2(2\gamma_U-1)}{\gamma_L}
\label{eq:proof-A}
\end{equation}
and where
\[
C_{\mathrm{mod}}
:=
C_0C_X^{3/2}
\sum_{j=1}^{4}C_{\mathrm{rate},j}
+2C_\beta C_{\mathrm{rate},5}.
\]
Finally, because $
L_{nK}\le2\log^2(n\vee K)$, 
the  bound in \eqref{eq:short-moderate-bound} yields
\begin{equation}
d_{\mathrm K}^{(\rho)}\bigl((T_n^O)^2,B_P\bigr)\le 
2^A C_{\mathrm{mod}}
n^{-z}\log^{2A}(n\vee K).
\label{eq:proof-moderate-final-form}
\end{equation}

We now treat the complementary case of \eqref{eq:short-moderate-regime}, so that \(L_{nK}>n^{1/400}\). Then
\[
\log(n\vee K)
\ge
\frac12\log(nK)
>
\frac{n^{1/400}}{2\log n}
\ge 
n^{1/800},
\]
where the final inequality holds true for all sufficiently large $n$, say, all
\(n \ge n_{0,3}\). Hence, using the elementary bound \(d_{\mathrm K}^{(\rho)}\le1\), we obtain, for all \(n\ge n_{0,3}\),
\begin{align*}
d_{\mathrm K}^{(\rho)}\bigl((T_n^O)^2,B_P\bigr)
\le 1
\le 
n^{-1/800} \log(n\vee K).
\end{align*}

Combining with the first-case bound from \eqref{eq:proof-moderate-final-form}, and defining 
\begin{equation*}
c
:=
z\wedge\frac1{800},
\qquad
d:=2A\vee1
\end{equation*}
with $z$ and $A$ from \eqref{eq:proof-A},
we have shown that, for all $n \ge n_0:=n_{0,1}\vee n_{0,2}\vee n_{0,3}$, 
\[
d_{\mathrm K}^{(\rho)}\bigl((T_n^O)^2,B_P\bigr)
\le 
2^A C_{\mathrm{mod}}
n^{-c}\log^{d}(n\vee K).
\]

To cover the remaining
values \(4\le n<n_0\), define
\[
C_{\mathrm{fin}}
:=
\max_{4\le j<n_0}
\frac{j^c}{\log^d j},
\]
with the maximum interpreted as zero if the index set is empty. Since
\(\log(n\vee K)\ge\log n\),
\[
C_{\mathrm{fin}}n^{-c}\log^d(n\vee K)\ge1,
\qquad 4\le n<n_0.
\]
Thus the elementary bound \(d_{\mathrm K}^{(\rho)}\le1\) covers all
remaining \(n\). The claimed bound in \eqref{eq:gap-proof1} hence holds for all $n \ge 4$ if we choose $C
\ge
1\vee 2^AC_{\mathrm{mod}}\vee C_{\mathrm{fin}}$. The proof is finished.
\end{proof}

\begin{proof}[Proof of Theorem~\ref{thm:fwer_control}]
Fix \(P\in\mathcal P\). If \(\mathcal S_P^c=\varnothing\), then \(\operatorname{FWER}(P,\gamma)=0\), so \eqref{eq:fwer-profiled-strong-control} is immediate, and there is nothing to be shown for \eqref{eq:fwer-profiled-null-calibration}. We may therefore assume henceforth that \(\mathcal S_P^c\ne\varnothing\).

For every \(k\in\mathcal S_P^c\), we have $f_k^{(i)}=\mu_k + \eps_k^{(i)}$, such that $C_{n,k}(s)=C_{n,k}^O(s)$  for all $s \in [n]$,
with $C_{n,k}$ and $C_{n,k}^O$ from \eqref{def:cusum} and \eqref{eq:definition-CnkO}, respectively. 
Consequently, with $\widehat{\mathcal S}_\gamma$ from \eqref{eq:def-hat-S_gamma},
\begin{align} \label{eq:fwer-starting-point}
\operatorname{FWER}(P,\gamma)
&=
\Prob\big(\widehat{\mathcal S}_\gamma\cap\mathcal S_P^c\ne\varnothing\big)
=
\Prob\big( \exists k \in S_P^c: T_{nk}^2 > \hat Q_{1-\gamma} \big) 
=
\mathbb P_P\big(X_0>\widehat Q_{1-\gamma}\big)
\end{align}
where
\begin{align*}
X_0
&:= 
\max_{k\in\mathcal S_P^c}\max_{s\in[n]}
\|C_{n,k}^O(s)\|_{\Hb_k}^2.
\end{align*}
Define intermediate variables
\begin{align*}
(T_{n,0}^*)^2
&:= 
\max_{k\in\mathcal S_P^c}\max_{s\in[n]}
\|C_k^*(s)\|_{\Hb_k}^2,
\qquad 
Y_0
:= 
\max_{k\in\mathcal S_P^c}\max_{s\in S_x}
\|G_{sk}^M\|_{2}^2.
\end{align*}
We will show below that there exist constants $D$ and $d$ with dependencies as required in the theorem and an event $\Omega_0 \in \mathcal F_n$ such that
\begin{align}
\mathbb P_P(\Omega_0^c)
&\le
\varepsilon_{n,q,r},
\label{eq:fwer-proof-three-profiled-event}
\\
\sup_{t\ge\rho_\gamma}
\left|
\mathbb P_P(X_0\le t)
-
\mathbb P_P(Y_0\le t)
\right|
&\le
\varepsilon_{n,q,r},
\label{eq:fwer-proof-three-profiled-oracle}
\\
\text{On }\Omega_0:\qquad
\sup_{t\ge\rho_\gamma}
\left|
\mathbb P_P\left(
(T_{n,0}^*)^2\le t
\,\middle|\,
\mathcal F_n
\right)
-
\mathbb P_P(Y_0\le t)
\right|
&\le
\varepsilon_{n,q,r},
\label{eq:fwer-proof-three-profiled-bootstrap}
\end{align}
where 
\begin{align}
\label{eq:rho-gamma}
\rho_\gamma:= \frac3{50}C_L^2
\left[
\Phi^{-1}\left(\frac{3-\gamma}{4}\right)
\right]^2,
\qquad
\varepsilon_{n,q,r} := \frac {D} {\sqrt {\tau}}\log^d(nK)\mathcal R_{n,q,r}   
\end{align}
and where \(\Phi\) is the cdf of a standard normal random variable.
If $\varepsilon_{n,q,r} > \{\gamma\wedge(1-\gamma)\} /2$,
then both conclusions in \eqref{eq:fwer-profiled-strong-control} and \eqref{eq:fwer-profiled-null-calibration} follow from
\(0\le\operatorname{FWER}(P,\gamma)\le1\), by choosing $C$ such that
$
C\ge 2D/ \{ \gamma\wedge(1-\gamma) \}.
$
Hence assume that $\varepsilon_{n,q,r}  \le  \{\gamma\wedge(1-\gamma)\} /2$.
Define, for $t \in \R$,
\[
F_{*,0}(t)
:=
\mathbb P_P\left(
(T_{n,0}^*)^2\le t
\,\middle|\,
\mathcal F_n
\right),
\qquad
G_0(t)
:=
\mathbb P_P(Y_0\le t), 
\]
and, for $u \in (0,1)$,
\[
\widehat Q^0_{1-\gamma}
:=
\inf\left\{
t\in\mathbb R:
F_{*,0}(t)\ge1-\gamma
\right\},
\qquad 
Q_{Y_0}(u)
:=
\inf\left\{
t\in\mathbb R:
\mathbb P_P(Y_0\le t)\ge u
\right\}.
\]
We will also show  below that
\begin{equation}
\label{eq:fwer-proof-three-quantile-sandwich}
\text{On } \Omega_0: \qquad \rho_\gamma \le q_-:= Q_{Y_0}(1-\gamma-\varepsilon_{n,q,r}) 
\le
\widehat Q^0_{1-\gamma}
\le
Q_{Y_0}(1-\gamma+\varepsilon_{n,q,r}) =: q_+.
\end{equation}
With \eqref{eq:fwer-proof-three-profiled-event}, \eqref{eq:fwer-proof-three-profiled-oracle}, \eqref{eq:fwer-proof-three-profiled-bootstrap} and \eqref{eq:fwer-proof-three-quantile-sandwich} at hand, we will now prove the two claims from the theorem. Starting from \eqref{eq:fwer-starting-point}, we have
\begin{align*}
\operatorname{FWER}(P,\gamma)
&=
\mathbb P_P\left(
X_0>\widehat Q_{1-\gamma},\,
\Omega_0
\right)
+
\mathbb P_P\left(
X_0>\widehat Q_{1-\gamma},\,
\Omega_0^c
\right)
\\
&\le
\mathbb P_P\left(
X_0>\widehat Q_{1-\gamma},\,
\Omega_0
\right)
+
\eps_{n,q,r},
\end{align*}
where we have used \eqref{eq:fwer-proof-three-profiled-event}.
Since 
\(T_n^*\ge T_{n,0}^*\) by definition, we have $\widehat Q_{1-\gamma}
\ge
\widehat Q^0_{1-\gamma}$. Hence, on \(\Omega_0\), the lower bound
in~\eqref{eq:fwer-proof-three-quantile-sandwich} gives
$\widehat Q_{1-\gamma} \ge \widehat Q^0_{1-\gamma} \ge q_-$,
such that
\begin{equation*}
\operatorname{FWER}(P,\gamma) \le \mathbb P_P(X_0>q_-) + \eps_{n,q,r}.
\end{equation*}
Because \(q_-\ge\rho_\gamma\) by \eqref{eq:fwer-proof-three-quantile-sandwich}, the oracle approximation \eqref{eq:fwer-proof-three-oracle-bound} is applicable at \(t=q_-\). It gives
\[
\left|
\mathbb P_P(X_0\le q_-)-G_0(q_-)
\right|
\le
\varepsilon_{n,q,r}.
\]
Consequently,
\begin{align*}
\mathbb P_P(X_0>q_-)
&=
1-\mathbb P_P(X_0\le q_-)
\le
1-G_0(q_-)+\varepsilon_{n,q,r}
=
\gamma + 2 \varepsilon_{n,q,r},
\end{align*}
where the final equality follows from continuity of $G_0$ (note that $Y_0$ is a maximum of squared norms of Gaussian variables) and the definition of $q_-$.
Consequently, 
\begin{align} \label{eq:fwer-lower}
\operatorname{FWER}(P,\gamma)
\le
\gamma+3\varepsilon_{n,q,r}
\end{align}
and~\eqref{eq:fwer-profiled-strong-control} follows provided the final
constant \(C\) is chosen so that $C\ge3D$.
For the proof of \eqref{eq:fwer-profiled-null-calibration}, suppose \(P\in\mathcal P_0\). Then
\[
\mathcal S_P^c=[K],
\qquad
X_0=T_n^2,
\qquad
T_{n,0}^*=T_n^*,
\qquad
\widehat Q^0_{1-\gamma}
=
\widehat Q_{1-\gamma}
\]
such that,  on \(\Omega_0\), we have 
$q_+ \ge \widehat Q^0_{1-\gamma} = \widehat Q_{1-\gamma}$ by \eqref{eq:fwer-proof-three-quantile-sandwich}. As a consequence, 
\begin{align*}
\operatorname{FWER}(P,\gamma)
=
\mathbb P_P(X_0>\widehat Q_{1-\gamma})
&\ge
\mathbb P_P(X_0>\widehat Q_{1-\gamma}, \Omega_0)
\\&\ge
\mathbb P_P(X_0>  q_+,\Omega_0)
\\&\ge
\mathbb P_P(X_0>q_+) -\Prob(\Omega_0^c)
\\&\ge
1-G_0(q_+)
-
2\varepsilon_{n,q,r}
\\&=
\gamma-3\varepsilon_{n,q,r}.
\end{align*}
Together with the lower bound from \eqref{eq:fwer-lower}, this yields \eqref{eq:fwer-profiled-null-calibration}.

It remains to show \eqref{eq:fwer-proof-three-profiled-event}, \eqref{eq:fwer-proof-three-profiled-oracle}, \eqref{eq:fwer-proof-three-profiled-bootstrap} and \eqref{eq:fwer-proof-three-quantile-sandwich}. We start by introducing some helpful notation. Let
\[
\Lambda:=\log(nK),
\qquad
b:=2\gamma_U-1,
\qquad
\alpha:=\frac{b}{12\gamma_U-2} \in (0,1/6),
\qquad
\beta:=\frac{b}{2\gamma_L},
\]
such that 
\[
\mathcal R_{n,q,r}
=
\Big(\frac{q^2}{n}\Big)^{\alpha}
+q^{-\beta}
+\Big(\frac rq\Big)^{1/2}
+\frac nq r^{-c_\beta}.
\]
Let \(c_1=c_1(\gamma_L,\gamma_U)\) and
\(c_3=c_3(\gamma_L,\gamma_U)\) be the logarithmic exponents in
Theorem~\ref{Thm_1} and
Theorem~\ref{thm:conditional_gaussian_subset_master}, respectively, and
set
\begin{equation}
d:=14+(c_1\vee c_3)+8\beta.
\label{eq:fwer-proof-three-d}
\end{equation}
Let \(C_X^{(1)}\) and \(C_X^{(3)}\) be sufficiently large constants required by Theorems~\ref{Thm_1} and \ref{thm:conditional_gaussian_subset_master}, respectively, evaluated at the threshold \(\rho_\gamma\) defined in \eqref{eq:rho-gamma}. Choose \(C_X\) to satisfy
\begin{equation}
\label{eq:choice-CX-fwer}
C_X
\ge
C_X^{(1)}
\vee C_X^{(3)}
\vee
\frac{16}{3(\log4)^8}.
\end{equation}
Further set 
\begin{align}
\label{eq:hnk-x}
H_{nK}
:=
\{\log(nK)\log n\}^4,
\qquad
x
:=
\frac1{C_XH_{nK}}.
\end{align}
This choice ensures that condition \text{\normalfont (X)} from \eqref{eq:side-conditions}  holds for
both Theorem~\ref{Thm_1} and
Theorem~\ref{thm:conditional_gaussian_subset_master}. Moreover, because
\(n\ge4\) and \(K\ge1\), $
H_{nK}\ge(\log4)^8$, 
and hence
\[
0<x
\le
\frac1{C_X(\log4)^8}
\le
\frac3{16}.
\]
Choose any $k_\star\in\mathcal S_P^c$ and choose $s_\star
\in
\operatorname*{arg\,max}_{s\in[n]}V(s/n)$. As observed below~\eqref{eq:definition-S_x}, $V(s_\star/n) = \max_{s\in[n]}V(s/n) \ge 6/25 > 3/16$.
It follows that \(s_\star\in S_x\) for the particular choice of $x$ made above. Furthermore, Assumption~\ref{cond:scores} gives $\lambda_{k_\star1}\ge C_L^2$. 
Let
\[
A_{\mathcal P}
:=
D_\psi^2S_{\beta,1}\vee C_{\mathcal T}
\]
with $C_{\mathcal T}$ from \eqref{eq:C_t}
and note that $C_{\mathcal K} \le C_{\mathcal T}$ by Assumption~\ref{cond:scores}. 
Let \(C_M^{(1)}\) and \(C_M^{(3)}\) be sufficiently large constants
required by Theorems~\ref{Thm_1} and
\ref{thm:conditional_gaussian_subset_master} for $\rho=\rho_\gamma$, respectively, and choose
\begin{equation}
\label{eq:choice-CM-fwer}
C_M
\ge
C_M^{(1)}
\vee C_M^{(3)}
\vee
\frac{2C_4A_{\mathcal P}}{C_L^2},
\end{equation}
where \(C_4\) is the universal constant from
Lemma~\ref{lem:bound-on-operator-norm-covariances}. Put
\begin{align} \label{eq:2-ms}
M_{\max}
:=
\left(
\frac{q}{4C_XC_MH_{nK}}
\right)^{1/(2\gamma_L)},
\qquad
M_\star
:=
\left(\frac n{q^2}\right)^{1/(12\gamma_U-2)}
=
\left(\frac n{q^2}\right)^{\alpha/b}.
\end{align}
Choose the constant \(C\) in the theorem so that
\begin{equation*}
C
\ge
4^\beta
\vee 2^b
\vee
\left\{
2^b(4C_XC_M)^\beta
\right\}.
\end{equation*}
This choice demonstrates that we may assume that 
\begin{equation}
\label{eq:fwer-proof-three-nontrivial}
4\le q\le\frac{\sqrt n}{\log n},
\qquad
M_{\max}\ge2,
\qquad
M_\star\ge2,
\end{equation}
because if some of the inequalities does not hold, then  $C\Lambda^d\mathcal R_{n,q,r} \ge 1$ and the claimed inequalities in \eqref{eq:fwer-profiled-strong-control} and \eqref{eq:fwer-profiled-null-calibration} are trivial. Indeed, 
if \(q<4\), then
\(\mathcal R_{n,q,r}\ge q^{-\beta}\ge4^{-\beta}\), and since
\(\Lambda=\log(nK)\ge\log4>1\), we have $C\Lambda^d\mathcal R_{n,q,r} > 4^{\beta-\beta}=1$.
If \(q>\sqrt n/\log n\), then $(q^2/n)^\alpha > (\log n)^{-2\alpha}$, and since \(\log n\le\Lambda\) and \(d\ge2\alpha\), the definition of $R_{n,q,r}$ from \eqref{eq:fwer-profiled-rate} yields $C\Lambda^d\mathcal R_{n,q,r} \ge \Lambda^d(q^2/n)^\alpha > (\log n)^{d-2\alpha} \ge 1$.
If \(M_{\max}<2\), then $M_{\max}^{-b} = (4C_XC_M)^\beta
H_{nK}^{\beta}q^{-\beta}$ yields $q^{-\beta} > 2^{-b}(4C_XC_M)^{-\beta}H_{nK}^{-\beta}$. Hence, since $H_{nK} \le \Lambda^8$ and $d\ge8\beta$, we obtain
\begin{align*}
C\Lambda^d\mathcal R_{n,q,r}
&\ge
2^b(4C_XC_M)^\beta
\Lambda^d q^{-\beta}>
\Lambda^dH_{nK}^{-\beta}
\ge
\Lambda^{d-8\beta}
\ge1.
\end{align*}
Finally, if \(M_\star<2\), then $(q^2/{n})^\alpha = M_\star^{-b} > 2^{-b}$, so $C\Lambda^d\mathcal R_{n,q,r} \ge 2^b (q^2/{n})^\alpha > 1$.

It therefore suffices to continue under the assumptions of
\eqref{eq:fwer-proof-three-nontrivial}.
Recall from~\eqref{def:m} that $m=\lfloor n/(q+r)\rfloor$. Since \(r\le q\) and \(q\le n/4\), this yields $n/(q+r) \ge 2$ and hence $m \ge n/\{2(q+r)\} \ge n/(4q)$.
Thus, under~\eqref{eq:fwer-proof-three-nontrivial}, $m \ge \sqrt n \log n/4 \ge q/4$.
Define
\begin{equation}
\label{eq:fwer-proof-three-M-ell}
M
:=
\left\lfloor M_\star\wedge M_{\max}\right\rfloor,
\qquad
\ell
:=
\left\lceil
\sqrt q\vee n^{3/[2(c_\beta+2)]}
\right\rceil.
\end{equation}
Then \(M\ge2\) and \(\ell\in[n]\), with the inequality $\ell \leq n$ following from $c_\beta >2$ in Assumption ~\ref{cond:stationary}. Furthermore,
\begin{align}
\label{eq:mbound1}
M^{2\gamma_L}
\le
M_{\max}^{2\gamma_L}
=
\frac{q}{4C_XC_MH_{nK}}
=
\frac{xq}{4C_M}
\le
\frac{x(q\wedge m)}{C_M}.
\end{align}
Thus conditions \text{\normalfont (X)} and
\text{\normalfont (M)}  from \eqref{eq:side-conditions}  hold and Theorem~\ref{Thm_1} and
Theorem~\ref{thm:conditional_gaussian_subset_master} can be freely applied for the chosen values of $x$ and $M$. 

Theorem~\ref{Thm_1} remains valid after replacing the index set $[K]$ by \(\mathcal S_P^c\). Since \(|\mathcal S_P^c|\le K\), we may apply the bound of Theorem~\ref{Thm_1} as it is stated, and obtain
\begin{align}
&\sup_{t\ge\rho_\gamma}
\big|
\mathbb P_P(X_0\le t)
-
\mathbb P_P(Y_0\le t)
\big|
\nonumber\\
&\hspace{2cm}\le
D_1x^{-3/2}
\left[
\left(\frac{M^4q^2}{n}\right)^{1/6}
+
M^{-b}
+
\left(\frac rq\right)^{1/2}
\right]\Lambda^{c_1}
+
2C_\beta\frac nq r^{-c_\beta}.
\label{eq:fwer-proof-three-oracle-bound}
\end{align}
Here and below, \(D_1,D_3,D_{\mathrm{boot}}>0\) depend only on the
parameters permitted in the statement of the theorem. Next apply
Theorem~\ref{thm:conditional_gaussian_subset_master} with $\mathcal J=\mathcal S_P^c$ and $g=q$.
Because $
\mathcal S_P^c\cap\mathcal S_P=\varnothing$ the interior condition of ~\eqref{eq:intertior-condition} is vacuous and $
\Delta_{\mathcal S_P^c,\mathrm{wk}}(q,\ell)=0$ by definition in~\eqref{eq:def-delta-J-wk}.
Consequently, by Theorem~\ref{thm:conditional_gaussian_subset_master}, there exists an event
\(\Omega_0\in\mathcal F_n\) satisfying
\begin{align}
\mathbb P_P(\Omega_0^c)
\le
D_{\mathrm{boot}}
\bigg\{
\frac1n
+
C_\beta
\left\lceil\frac n\ell\right\rceil
\ell^{-c_\beta}
+
C_\beta\frac nq r^{-c_\beta}
\bigg\},
\label{eq:fwer-proof-three-event}
\end{align}
such that, on \(\Omega_0\),
\begin{align}
&\sup_{t\ge\rho_\gamma}
\left|
\mathbb P_P\left(
(T_{n,0}^*)^2\le t
\,\middle|\,
\mathcal F_n
\right)
-
\mathbb P_P(Y_0\le t)
\right|\le
D_3x^{-3/2}
\Bigg[
\frac1{\sqrt \tau}\left(\frac{2q}{n}\right)^{1/2}
\left(1\vee\frac{\ell}{\sqrt q}\right)
\nonumber\\
&\hspace{4.5cm}
+
M^{-b}
+
M^{2/3}
\left\{
\left(\frac qn\right)^{1/6}
\vee
\left(\frac{q^2}{n}\right)^{1/3}
\right\}
\Bigg]
\log^{c_3}(nKM).
\label{eq:fwer-proof-three-bootstrap-bound}
\end{align}
The claims in \eqref{eq:fwer-proof-three-profiled-event}, \eqref{eq:fwer-proof-three-profiled-oracle}, \eqref{eq:fwer-proof-three-profiled-bootstrap} will now follow by showing that we can choose $D$ sufficiently large such that each upper bound in \eqref{eq:fwer-proof-three-oracle-bound}, \eqref{eq:fwer-proof-three-event} and \eqref{eq:fwer-proof-three-bootstrap-bound} can be bounded by $\eps_{n,q,r}$ from \eqref{eq:rho-gamma}. For that purpose, we start by showing that there exists a constant $D_M$ having only the permitted parameter dependence such that
\begin{align}
\label{eq:m4q}
\left(\frac{M^4q^2}{n}\right)^{1/6} + M^{-b}
\le 
D_M\Lambda^{8\beta}
\left[
\left(\frac{q^2}{n}\right)^\alpha
+
q^{-\beta}
\right]
\end{align}
Indeed, recall $M=\left\lfloor M_\star\wedge M_{\max}\right\rfloor \ge 2$ from~\eqref{eq:fwer-proof-three-M-ell}, with $M_\star$ and $M_{\max}$ from \eqref{eq:2-ms}. On the one hand, 
\[
\left(\frac{M_\star^4q^2}{n}\right)^{1/6} = M_\star^{-b} = \left(\frac{q^2}{n}\right)^{\alpha}. 
\]
On the other hand,
\[
M_{\max}^{-b}
=
(4C_XC_M)^\beta H_{nK}^{\beta}q^{-\beta}
\le
(4C_XC_M)^\beta 
\Lambda^{8\beta}q^{-\beta}.
\]
Together with the bounds $(M_\star\wedge M_{\max})/2 \le M  \le M_\star\wedge M_{\max}$ this proves \eqref{eq:m4q}.

Similarly, the choice of \(\ell\) from \eqref{eq:fwer-proof-three-M-ell} gives
\begin{equation*}
\left(\frac qn\right)^{1/2}
\vee
\frac{\ell}{\sqrt n}
+
n\ell^{-(c_\beta+1)}
\le
D_\ell
\left[
\left(\frac qn\right)^{1/2}
+
n^{-\frac{c_\beta-1}{2(c_\beta+2)}}
\right]
\end{equation*}
where \(D_\ell>0\) is absolute. Here, since $\alpha \in (0,1/6)$, 
\[
\Big(\frac qn\Big)^{1/2}
\le
\Big(\frac{q^2}{n}\Big)^\alpha.
\]
Moreover, 
\begin{align}
\label{eq:bound-on-hn}
h_n := n^{-\frac{c_\beta-1}{2(c_\beta+2)}} \le \Big(\frac{q^2}{n}\Big)^\alpha \vee \Big(\frac rq\Big)^{1/2}
\vee
\Big(\frac nq r^{-c_\beta}\Big).
\end{align}
Indeed, since $q < \sqrt n$ by \eqref{eq:fwer-proof-three-nontrivial}, we have
\begin{align*}
\Big(\frac rq\Big)^{1/2}
\vee
\Big(\frac nq r^{-c_\beta}\Big)
&\ge
\Big(\frac rq\Big)^{\frac{c_\beta}{2c_\beta+1}}
\Big(\frac nq r^{-c_\beta}\Big)^{\frac1{2c_\beta+1}}
=
n^{\frac1{2c_\beta+1}}
q^{-\frac{c_\beta+1}{2c_\beta+1}}
\ge
n^{-\frac{c_\beta-1}{4c_\beta+2}}
\ge
h_n.
\end{align*}
Next, because  $q <\sqrt n$ also yields $q^2/n<1$ from~\eqref{eq:fwer-proof-three-nontrivial}, such that
\begin{align}
\label{eq:m4q-2}
M^{2/3}
\bigg\{
\left(\frac qn\right)^{1/6}
\vee
\left(\frac{q^2}{n}\right)^{1/3}
\bigg\}
\le
M^{2/3}\left(\frac{q^2}{n}\right)^{1/6}
\end{align}
Finally, $1/n \le (q^2/n)^{\alpha}$ and $\log(nKM) \le \log(n^2K) \le2\Lambda$ and, recalling \eqref{eq:hnk-x},
\[
x^{-3/2}
=
C_X^{3/2}H_{nK}^{3/2}
\le
C_X^{3/2}\Lambda^{12}.
\]
Assembling bounds and recalling the definition of $d = 14 + (c_1 \vee c_3) + 8 \beta$ from \eqref{eq:fwer-proof-three-d}
the proof of \eqref{eq:fwer-proof-three-profiled-event}, \eqref{eq:fwer-proof-three-profiled-oracle}, \eqref{eq:fwer-proof-three-profiled-bootstrap} is finished by choosing $D$ sufficiently large.

It remains to show \eqref{eq:fwer-proof-three-quantile-sandwich}.  We will show below that 
\begin{align}
\Var_P\!\left([G_{s_\star k_\star}^M]_1\right)
&\ge
\frac{3C_L^2}{25} =: v_0.
\label{eq:fixed-coordinate-variance-lower-bound}
\end{align}
Consequently, since \(Y_0\ge [G_{s_\star k_\star}^M]_1^2\), 
\begin{align*}
\mathbb P_P(Y_0\le2\rho_\gamma)
\le
\mathbb P_P\left(
[G_{s_\star k_\star}^M]_1^2\le2\rho_\gamma
\right)
\le
\mathbb P\left(Z^2\le\frac{2\rho_\gamma}{v_0}\right)
=
2\Phi\left(\sqrt{\frac{2\rho_\gamma}{v_0}}\right)-1
=
\frac{1-\gamma}{2}
\end{align*}
where \(Z\) is standard normal. Hence
\begin{equation*}
Q_{Y_0}\big((1-\gamma)/2\big)\ge2\rho_\gamma.
\end{equation*}
Recall our assumption $\varepsilon_{n,q,r} \le \{\gamma\wedge(1-\gamma)\}/2$ from above. It implies $1-\gamma - \varepsilon_{n,q,r} \ge (1-\gamma)/2$, and thus
\[
q_-
=
Q_{Y_0}(1-\gamma-\varepsilon_{n,q,r})
\ge 
Q_{Y_0}\big( (1-\gamma)/2\big) \ge 2\rho_\gamma,
\]
which is the first inequality in \eqref{eq:fwer-proof-three-quantile-sandwich}. Moreover, this inequality yields
\[
G_0(\rho_\gamma)+\varepsilon_{n,q,r}<1-\gamma.
\]
The remaining two inequalities in \eqref{eq:fwer-proof-three-quantile-sandwich}, valid on $\Omega_0$, therefore follow from an application of Lemma~\ref{lem:koltoquant} with
\[
F=F_{*,0},
\qquad
G=G_0,
\qquad
t_*=\rho_\gamma,
\qquad
\delta=\varepsilon_{n,q,r},
\qquad
\alpha=1-\gamma.
\]

It remains to show \eqref{eq:fixed-coordinate-variance-lower-bound}.
For that purpose, note that
\begin{align}
\label{eq-v0-bound}
\frac{C_4A_{\mathcal P}}{q\wedge m}
\le
\frac{C_4A_{\mathcal P}x}{C_MM^{2\gamma_L}}
\le
\frac{xC_L^2}{2M^{2\gamma_L}}
\le
\frac{3C_L^2}{32} \le \frac{3C_L^2}{25} = v_0,
\end{align}
where we have used 
\eqref{eq:mbound1}, \eqref{eq:choice-CM-fwer}  and \(x\le3/16\) and \(M\ge1\). 

Next, write $L_\star := O_{k_\star M}\circ P_{k_\star M}\circ\pi_{s_\star k_\star}$ and note that $G_{s_\star k_\star}^M = L_\star G^{(n,q)}$ with $\Cov(G^{(n,q)}) = \Sigma^{(n,q)}$ by definition in \eqref{eq:definition-GMsk}. 
Further recall $\Sigma^{(n)} = \Cov(\mathcal B^{(n)}) = K_n \otimes \mathcal K$ from \eqref{eq:Kronecker-Kn}.
It then follows from Lemma~\ref{lem:bound-on-operator-norm-covariances}, applied with
\(\sigma=\{s_\star\}\) and \(\chgset=\{k_\star\}\), that
\begin{align*}
\left\|
\Cov(G_{s_\star k_\star}^M)
-
L_\star\Sigma^{(n)}L_\star^*
\right\|_{\mathrm{op}}
&= \nonumber
\left\|
 L_\star\Sigma^{(n,q)}L_\star^* -  L_\star\Sigma^{(n)}L_\star^*
\right\|_{\mathrm{op}}
\\&\le \nonumber
\left\|
(P_{k_\star M}\circ\pi_{s_\star k_\star} )(\Sigma^{(n,q)} - \Sigma^{(n)}) (P_{k_\star M}\circ\pi_{s_\star k_\star} )
\right\|_{\mathrm{op}}
\\&= \nonumber
\left\|
\Cov\big( (P_{k_\star M}\circ\pi_{s_\star k_\star})G^{(n,q)} \big) 
-
\Cov\big( (P_{k_\star M}\circ\pi_{s_\star k_\star})\mathcal B^{(n)} \big)
\right\|_{\mathrm{op}}
\\&\le
\frac{C_4A_{\mathcal P}}{q\wedge m},
\end{align*}
where we have used sub-multiplicativity of the operator norm and $\|O_{kM}\|_\op= 1$. Next, from \eqref{eq:representation-projection-kronecker},
\[
\pi_{s_\star k_\star} \Sigma^{(n)} \pi_{s_\star k_\star}^*
=
V(s_\star/n)\mathcal (\pi_k \mathcal K \pi_k^*) = V(s_\star/n) \mathcal K_{k_\star},
\]
where  \(V(u)=u(1-u)\), such that 
\begin{align*}
L_\star\Sigma^{(n)}L_\star^*
&=
V(s_\star/n)
O_{k_\star M}P_{k_\star M}
\mathcal K_{k_\star}
P_{k_\star M}O_{k_\star M}^*
=
V(s_\star/n)
\operatorname{diag}
(
\lambda_{k_\star1},
\ldots,
\lambda_{k_\star\check M_{k_\star}}
).
\end{align*}
As a consequence, with \(\mathbf e_1\in\mathbb R^{\check M_{k_\star}}\) the first unit vector, we have
\begin{align*}
\Var\!\left([G_{s_\star k_\star}^M]_1\right)
&=
\big\langle
\mathbf e_1, \Cov_P(G_{s_\star k_\star}^M)\mathbf e_1
\big\rangle
\\
&=
\big\langle
\mathbf e_1, L_\star\Sigma^{(n)}L_\star^*\mathbf e_1
\big\rangle
+
\left\langle
\mathbf e_1,
\left\{
\Cov(G_{s_\star k_\star}^M)
-
L_\star\Sigma^{(n)}L_\star^*
\right\}
\mathbf e_1
\right\rangle
\\
&\ge
V(s_\star/n)\lambda_{k_\star1}
-
\left\|
\Cov_P(G_{s_\star k_\star}^M)
-
L_\star\Sigma^{(n)}L_\star^*
\right\|_{\mathrm{op}}
\\
&\ge
V(s_\star/n)\lambda_{k_\star1}
-
\frac{C_4A_{\mathcal P}}{q\wedge m}
\\& \ge  2v_0 - v_0 = v_0,
\end{align*}
where we have used \eqref{eq-v0-bound} and $V(s_\star/n) \ge 6/25$ and $\lambda_{k_\star1}\ge C_L^2$. We have thus shown  \eqref{eq:fixed-coordinate-variance-lower-bound} and the proof is complete.
\end{proof}

\begin{proof}[Proof of Theorem~\ref{thm:profiled_simultaneous_power}]
Define $c_0 = 5(C_{\mathrm{loc}} \vee 1)$ with $C_{\mathrm{loc}}=C_{\mathrm{loc}}(D_\psi, C_\beta, c_\beta)$ from Corollary \ref{cor:coordinatewise_beta_localization_new}.
Fix \(\eta\in(0,\gamma]\), \(P\in\mathcal P_{\mathrm{loc}}(c_0,q,\tau)\) and a nonempty deterministic set
\(\chgset\subseteq\mathcal S_P\). As in~\eqref{eq:rho-gamma}, let
\[
\rho_\gamma
:=
\frac{3C_L^2}{50}
\left[
\Phi^{-1}\left(\frac{3-\gamma}{4}\right)
\right]^2.
\]
For the value of \(x\) and \(M\) chosen below in~\eqref{eq:power-proof-M-ell} set $
Y_{[K]}
:=
\max_{k\in[K]}\max_{s\in S_x}\|G_{sk}^M\|_2^2$. 
We will show below that there exist constants
\[
D_{\mathrm F}
=
D_{\mathrm F}\!\left(
\gamma,D_\psi,C_\beta,c_\beta,C_L,C_U,\gamma_L,\gamma_U
\right)>0,
\qquad
d_{\mathrm F}
=
d_{\mathrm F}(\gamma_L,\gamma_U)>0,
\]
such that, with
\begin{equation}
e_{\mathrm F}
:=
\frac{D_{\mathrm F}}{\sqrt\tau}
\log^{d_{\mathrm F}}(nK)\mathcal R_{n,q,r},
\label{eq:power-proof-eF}
\end{equation}
there exists an event \(\Omega_{\mathrm F}\in\mathcal F_n\)
satisfying
\begin{align}
\mathbb P_P(\Omega_{\mathrm F}^c)
&\le e_{\mathrm F},
\label{eq:power-proof-profiled-event}
\\
\sup_{t\ge\rho_\gamma}
\left|
\mathbb P_P\left((T_n^O)^2\le t\right)
-\mathbb P_P(Y_{[K]}\le t)
\right|
&\le e_{\mathrm F},
\label{eq:power-proof-profiled-oracle-finite}
\\
\text{on }\Omega_{\mathrm F}: \qquad
\sup_{t\ge\rho_\gamma}
\left|
\mathbb P_P\left(
(T_n^*)^2\le t
\,\middle|\,
\mathcal F_n
\right)
-\mathbb P_P(Y_{[K]}\le t)
\right|
&\le e_{\mathrm F}.
\label{eq:power-proof-profiled-bootstrap-finite}
\end{align}
Recall that $T^O_n$ is defined in~\eqref{eq:def_oracle_test_statistic} as the noise only global CUSUM statistic; therefore its distribution is independent of the mean parameter structure $(\mu_k, \delta_k)_{k\in [K]}$ induced by $P$. 

We first complete the proof assuming
\eqref{eq:power-proof-profiled-event}--\eqref{eq:power-proof-profiled-bootstrap-finite}, and to do so will apply Lemma~\ref{lem:simultaneous_power_generic}. To apply the lemma,  first note that although Theorem~\ref{thm:gaussian_approximation} is stated under the
global null, it applies to \(T_n^O\) under any \(P\in\mathcal P\):
indeed, apply that theorem to the global-null model obtained by
suppressing the deterministic mean changes under \(P\) while retaining
the centered noise process induced by $P$. Hence
Theorem~\ref{thm:gaussian_approximation}, evaluated at
\(\rho_\gamma\), gives constants
\[
D_{\mathrm G}
=
D_{\mathrm G}\!\left(
\gamma,D_\psi,C_\beta,c_\beta,C_L,C_U,\gamma_L,\gamma_U
\right)>0,
\]
\[
c_{\mathrm{G}}
=
c_{\mathrm{G}}(\gamma_L,\gamma_U)>0,
\qquad
d_{\mathrm G}
=
d_{\mathrm G}(\gamma_L,\gamma_U)>0,
\]
such that, with
\[
e_{\mathrm G}
:=
D_{\mathrm G}n^{-c_{\mathrm{G}}}
\log^{d_{\mathrm G}}(n\vee K),
\]
we have
\begin{equation}
\sup_{t\ge\rho_\gamma}
\left|
\mathbb P_P\left((T_n^O)^2\le t\right)
-\mathbb P\left(B_{[K],P}^2\le t\right)
\right|
\le e_{\mathrm G}.
\label{eq:power-proof-global-oracle-BB}
\end{equation}
The same Theorem~\ref{thm:gaussian_approximation}, applied to the random variable $T_{n,\chgset}^O$ defined in~\eqref{noise:kappa} gives 
\begin{equation}
\sup_{t\ge\rho_\gamma}
\left|
\mathbb P_P\left((T_{n,\chgset}^O)^2\le t\right)
-\mathbb P\big(B_{\chgset,P}^2\le t\big)
\right|
\le e_{\mathrm G},
\label{eq:power-proof-subset-oracle-BB}
\end{equation}
because \(|\chgset|\le K\). Combining \eqref{eq:power-proof-profiled-oracle-finite},
\eqref{eq:power-proof-profiled-bootstrap-finite}, and
\eqref{eq:power-proof-global-oracle-BB} by the triangle inequality
shows that, on \(\Omega_{\mathrm F}\),
\begin{align}
&\sup_{t\ge\rho_\gamma}
\left|
\mathbb P_P\left(
(T_n^*)^2\le t
\,\middle|\,
\mathcal F_n
\right)
-\mathbb P\big(B_{[K],P}^2\le t\big)
\right|
\le2e_{\mathrm F}+e_{\mathrm G}.
\label{eq:power-proof-bootstrap-BB}
\end{align}
Choose $d
\ge
d_{\mathrm F}\vee d_{\mathrm G}$ and $c=c_{\mathrm G}$ in the formulation of the theorem we are proving. Since $
\log(n\vee K)\le\log(nK)$ and also $\tau^{-1/2}\ge1$ 
we have
\[
\max\left\{
e_{\mathrm F},
e_{\mathrm G},
2e_{\mathrm F}+e_{\mathrm G}
\right\}
\le
\frac{D_0}{\sqrt\tau}
\log^{d}(nK)
\left\{
\mathcal R_{n,q,r}+n^{-c}
\right\},
\]
where $D_0:=2D_{\mathrm F}+D_{\mathrm G}$. Choose \(C\) in the statement of the theorem we are proving sufficiently
large that
\begin{equation}
C
\ge
D_0
\vee
\frac{D_0}{2c_\gamma}
= D_0 \max\{1, 1/(2c_\gamma)\}.
\label{eq:power-proof-choice-Cpow}
\end{equation}
Note that if \(\varepsilon_{n,q,r}\ge1/2\), then \eqref{eq:profiled-power-conclusion} is immediate since
\[
1-\eta-2\varepsilon_{n,q,r}
\le
-\eta
<0
\le
\mathbb P_P\big(
\chgset\subseteq\widehat{\mathcal S}_\gamma
\big)
\]
holds trivially. Hence from now on assume that
\(\varepsilon_{n,q,r}<1/2\), and set
\[
\delta_0
:=
\frac{D_0}{C}\varepsilon_{n,q,r}.
\]
We then have $
\delta_0
<
{D_0}/( 2C)
\le c_\gamma
$ and also $\delta_0\le\varepsilon_{n,q,r}$ by~\eqref{eq:power-proof-choice-Cpow}.
Consequently,
\begin{equation}
0<\delta_0
\le
\varepsilon_{n,q,r}\wedge c_\gamma.
\label{eq:power-proof-delta0}
\end{equation}
We apply Lemma~\ref{lem:simultaneous_power_generic} with
\[
\rho=\rho_\gamma,
\qquad
\Omega=\Omega_{\mathrm F},
\qquad
\varpi=\delta_0,
\qquad
\delta=\delta_0.
\]
Since
$
\delta_0\le c_\gamma
=
\{\gamma\wedge(1-\gamma)\}/4
$
the requirement \(0<\delta<\gamma \wedge (1-\gamma)\) is satisfied. Moreover,
\eqref{eq:power-proof-profiled-event},
\eqref{eq:power-proof-subset-oracle-BB}, and
\eqref{eq:power-proof-bootstrap-BB}
show that
\eqref{eq:generic-power-calibration-event},
\eqref{eq:generic-power-oracle-comparison}, and 
\eqref{eq:generic-power-bootstrap-comparison}, respectively, hold with
these choices. Next, by
\eqref{eq:power-proof-delta0} and the fact that quantile functions are nondecreasing it holds that 
\[
c_{[K],P}^{\mathrm{BB}}(1-\gamma+\delta_0)
\le
c_{[K],P}^{\mathrm{BB}}
\left(
1-\gamma+
\{\varepsilon_{n,q,r}\wedge c_\gamma\}
\right),
\]
such that \eqref{eq:profiled-power-signal} implies
\[
\sqrt n\,\theta_k(1-\theta_k)\Delta_k
>
c_{[K],P}^{\mathrm{BB}}(1-\gamma+\delta_0)
+
c_{\chgset,P}^{\mathrm{BB}}(1-\eta),
\qquad k\in\chgset.
\]
This is
\eqref{eq:generic-power-two-quantile-condition} in
Lemma~\ref{lem:simultaneous_power_generic}, with
\(\delta=\delta_0\).
It remains to verify the two lower-threshold conditions
\eqref{eq:generic-power-bootstrap-threshold} and
\eqref{eq:generic-power-oracle-threshold} of
Lemma~\ref{lem:simultaneous_power_generic}.
For $\mathcal J \subset [K]$, fix \(\tilde{k}\in\mathcal J\), and
let \(Z_{\tilde{k}1}\) be a unit eigenfunction of
\(\mathcal K_{P,\tilde{k}}\) corresponding to its largest
eigenvalue \(\lambda_{\tilde{k}1}\). Note that
\[
B_{\mathcal J,P}
\ge
\Big|
\left\langle
\pi_{\tilde{k}1}
\mathbb B_{\mathcal K_P}(1/2),
Z_{\tilde{k}1}
\right\rangle_{\Hb_{\tilde{k}}}
\Big|
\]
and
\[
\left\langle
\pi_{\tilde{k}}
\mathbb B_{\mathcal K_P}(1/2),
Z_{\tilde{k}1}
\right\rangle_{\Hb_{\tilde{k}}}
\sim
N\big(0,\lambda_{\tilde{k}1}/4\big),
\]
the second statement following from the Brownian-bridge covariance at \(1/2\),
\[
\operatorname{Cov}\!\left(
\pi_{\widetilde k}\mathbb B_{\mathcal K_P}(1/2)
\right)
= \mathcal K_{P,\widetilde k} / 4,
\]
and the eigenfunction identity
\(\mathcal K_{P,\widetilde k}Z_{\widetilde k1}
=\lambda_{\widetilde k1}Z_{\widetilde k1}\).
Then Assumption~\ref{cond:scores}, which gives
\(\lambda_{k_{\mathcal J}1}\ge C_L^2\), implies
\begin{align*}
\mathbb P\left(
B_{\mathcal J,P}^2\le\rho_\gamma
\right)
&\le
\mathbb P\left(
Z^2\le\frac{4\rho_\gamma}
{\lambda_{k_{\mathcal J}1}}
\right)\le
2\Phi\left(
\frac{2\sqrt{\rho_\gamma}}{C_L}
\right)-1,
\end{align*}
where \(Z\sim N(0,1)\) and $\Phi$ is the cumulative distribution function of $Z$. Put
\[
z_\gamma
:=
\Phi^{-1}\left(\frac{3-\gamma}{4}\right)>0.
\]
By the definition of \(\rho_\gamma\), we have $2\sqrt{\rho_\gamma}/{C_L}
=
\sqrt{6/25}\,z_\gamma
<
z_\gamma$.
Consequently,
\begin{equation}
\mathbb P\left(
B_{\mathcal J,P}^2\le\rho_\gamma
\right)
<
2\Phi(z_\gamma)-1
=
\frac{1-\gamma}{2}.
\label{eq:power-proof-BB-lower-tail}
\end{equation}
Applying \eqref{eq:power-proof-BB-lower-tail} with
\(\mathcal J=[K]\), and using
\(\delta_0\le c_\gamma\le(1-\gamma)/4\), recall ~\eqref{eq:power-proof-delta0}, gives
\begin{align*}
\mathbb P\left(
B_{[K],P}^2\le\rho_\gamma
\right)
+\delta_0
&<
\frac{1-\gamma}{2}
+\frac{1-\gamma}{4}<
1-\gamma.
\end{align*}
This verifies
\eqref{eq:generic-power-bootstrap-threshold}. Applying \eqref{eq:power-proof-BB-lower-tail} 
with \(\mathcal J=\chgset\), 
and recalling that \(\eta\le\gamma\), gives
\[
\mathbb P\left(
B_{\chgset,P}\le\sqrt{\rho_\gamma}
\right)
<
\frac{1-\gamma}{2}
<
1-\eta.
\]
By the definition of the quantile function, it follows that
\[
c_{\chgset,P}^{\mathrm{BB}}(1-\eta)
\ge
\sqrt{\rho_\gamma},
\]
which verifies
\eqref{eq:generic-power-oracle-threshold}. The conclusion of Lemma~\ref{lem:simultaneous_power_generic} therefore gives
\begin{align*}
\Prob\big(
\chgset\subseteq\widehat{\mathcal S}_\gamma
\big)
\ge
1-\eta-\varpi-\delta
=
1-\eta-2\delta_0
\ge
1-\eta-2\varepsilon_{n,q,r},
\end{align*}
where the final inequality follows from
\(\delta_0\le\varepsilon_{n,q,r}\). This proves
\eqref{eq:profiled-power-conclusion}, subject only to the verification
of
\eqref{eq:power-proof-profiled-event}--
\eqref{eq:power-proof-profiled-bootstrap-finite}.

For that purpose, put
\[
\Lambda:=\log(nK),
\qquad
b:=2\gamma_U-1,
\qquad
\alpha:=\frac{b}{12\gamma_U-2},
\qquad
\beta_M:=\frac{b}{2\gamma_L}.
\]
Choose \(C_X,C_M,H_{nK},x,M_{\max}\), and \(M_\star\) exactly as in
\eqref{eq:choice-CX-fwer}, \eqref{eq:hnk-x},
\eqref{eq:choice-CM-fwer}, and~\eqref{eq:2-ms}, with
\(\rho=\rho_\gamma\), and set
\begin{equation}
M
:=
\left\lfloor M_\star\wedge M_{\max}\right\rfloor,
\qquad
\ell
:=
\left\lceil n^{3/[2(c_\beta+2)]}\right\rceil;
\label{eq:power-proof-M-ell}
\end{equation}
note that only the definition of $\ell$ is different from the one used in  the proof of  Theorem~\ref{thm:fwer_control}; see \eqref{eq:fwer-proof-three-M-ell}.
The elementary-case argument leading to
\eqref{eq:fwer-proof-three-nontrivial} applies without change, so that it suffices to assume
\begin{equation*}
4\le q\le\frac{\sqrt n}{\log n},
\qquad
M_{\max}\ge2,
\qquad
M_\star\ge2.
\end{equation*}
The calculation preceding \eqref{eq:mbound1} then gives $m \ge q/4$
and, since the definition of \(M\) is unchanged,
\eqref{eq:mbound1} gives
\[
M^{2\gamma_L}
\le
\frac{x(q\wedge m)}{C_M}.
\]
Thus conditions \textnormal{(X)} and \textnormal{(M)} from
\eqref{eq:side-conditions} hold.

Recall 
\(\nu_k=\gamma_k\Delta_k\) and $\gamma_k = \theta_k \wedge (1-\theta_k) \ge \theta_k (1-\theta_k)$ from \eqref{eq:def-nu_k}. Membership in
\(\mathcal P_{\mathrm{loc}}(c_0,q,\tau)\) gives
\begin{equation*}
\nu_k
\ge
c_0\Lambda
\left\{
\frac1{\sqrt q}
+
\frac{n^{3/[2(c_\beta+2)]}\log n}{q}
\right\},
\qquad k\in\mathcal S_P.
\end{equation*}
Note that 
\[
\ell\le2n^{3/[2(c_\beta+2)]},
\qquad
\log\left(\frac{en}{\ell}\right)\le2\log n,
\]
which follow from $n^{3/[2(c_\beta+2)}\ge 1$ and $\lceil x \rceil \leq x+1$ and $\ell \ge 1$ and $\log(en)=1+\log(n) \leq 2\log(n)$, since $n\ge 4$. Recall the definition of \(r_{\beta,k}^{(n)}(\ell)\) in~\eqref{eq:def-coordinatewise-rbeta-new}. By our choice of $c_0 =  5(C_{\mathrm{loc}} \vee 1)$ from the beginning of this proof, 
we have 
\[
C_{\mathrm{loc}}
\left(
\frac1{c_0^2}+\frac4{c_0}
\right)
\le1.
\]
Then, by~\eqref{eq:def-coordinatewise-rbeta-new} and the preceding
bound it holds for all $k\in\mathcal S_P$ that 
\begin{align*}
r_{\beta,k}^{(n)}(\ell)
&=
\left\lceil
C_{\mathrm{loc}}
\frac{\Lambda}{\nu_k}
\left\{
\frac{\Lambda}{\nu_k}
+
\ell\log\left(\frac{en}{\ell}\right)
\right\}
\right\rceil
\\
&\le
\left\lceil
C_{\mathrm{loc}}
\left(
\frac1{c_0^2}+\frac4{c_0}
\right)q
\right\rceil
\\
&\le
\lceil q\rceil
=
q.
\end{align*}
Recalling the definitions of~\eqref{def:W_j(g)} and~\eqref{eq:def-delta-J-wk} it follows that 
\begin{equation*}
\mathcal W_{[K]}(q,\ell)=\varnothing,
\qquad
\Delta_{[K],\mathrm{wk}}(q,\ell)=0,
\end{equation*}
with the second equality follow from the first, recalling the convention used throughout this manuscript that an a maximum over an empty set is defined to equal $0$.

Theorem~\ref{Thm_1} now gives the bound in 
\eqref{eq:fwer-proof-three-oracle-bound}, but with \(X_0\) replaced by
\((T_n^O)^2\) and \(Y_0\) replaced by \(Y_{[K]}\). Likewise,
Theorem~\ref{thm:conditional_gaussian_subset_master}, applied with
\(\mathcal J=[K]\) and \(g=q\), gives the bounds
\eqref{eq:fwer-proof-three-event} and
\eqref{eq:fwer-proof-three-bootstrap-bound}, but with $(T_{n,0}^*)^2$ replaced by $(T_n^*)^2$ and \(Y_0\) replaced by \(Y_{[K]}\). Here, 
the interior condition from \eqref{eq:intertior-condition}
follows from \(P\in\mathcal P_{\mathrm{loc}}(c_0,q,\tau)\).
 The claims in \eqref{eq:power-proof-profiled-event},
\eqref{eq:power-proof-profiled-oracle-finite}, and 
\eqref{eq:power-proof-profiled-bootstrap-finite} hence follow once we show that each of three the upper bounds can be upper bounded by $e_{\mathrm F}$ from \eqref{eq:power-proof-eF}. This in turn follows analogously to the proof of Theorem~\ref{thm:fwer_control}.

All \(M\)-dependent terms in these bounds are controlled by
\eqref{eq:m4q} and \eqref{eq:m4q-2}. 
Only the calculation involving \(\ell\) differs from the FWER proof, since $\ell$ has a different definition in this one. Recall
\[
h_n
:=
n^{-(c_\beta-1)/[2(c_\beta+2)]}
\le 
\Big(\frac{q^2}{n}\Big)^\alpha \vee \Big(\frac rq\Big)^{1/2}
\vee
\Big(\frac nq r^{-c_\beta}\Big) \le \mathcal R_{n,q,r}
\]
from \eqref{eq:bound-on-hn}.
The choice of $\ell$ in \eqref{eq:power-proof-M-ell} gives
\begin{align*}
&\left(\frac qn\right)^{1/2}
\left(1\vee\frac\ell{\sqrt q}\right)
\le
\left(\frac qn\right)^{1/2}+\frac\ell{\sqrt n}
\le
\left(\frac qn\right)^{1/2}+2h_n,
\\
&\left\lceil\frac n\ell\right\rceil\ell^{-c_\beta}
\le
2n\ell^{-(c_\beta+1)}
\le
2h_n.
\end{align*}
Moreover,
\[
\left(\frac qn\right)^{1/2}
\le
\left(\frac{q^2}{n}\right)^\alpha,
\qquad
\frac1n
\le
\left(\frac{q^2}{n}\right)^\alpha,
\]
Finally, the bounds following \eqref{eq:m4q-2} give
\[
x^{-3/2}\le C_X^{3/2}\Lambda^{12},
\qquad
\log(nKM)\le2\Lambda.
\]
Combining these bounds 
establishes
\eqref{eq:power-proof-profiled-event}--\eqref{eq:power-proof-profiled-bootstrap-finite}
after choosing \(D_{\mathrm F}\) sufficiently large and
\[
d_{\mathrm F}
\ge
14+(c_1\vee c_3)+8\beta_M,
\]
where \(c_1\) and \(c_3\) are the logarithmic exponents in
Theorem~\ref{Thm_1} and
Theorem~\ref{thm:conditional_gaussian_subset_master}, respectively.
This completes the proof.

\end{proof}

\begin{proof}[Proof of Theorem~\ref{thm:linear_process_localization}]
Fix \(\eta\in(0,1)\), and set
\begin{align}
\label{eq:definition-x_eta}
    x_\eta:=C_x\log\left(\frac{e|\chgset|}{\eta}\right),
\end{align}
where \(C_x\) is a sufficiently large absolute constant chosen later.  At the end of the
proof, \(C_x\) is absorbed into the constant \(C\) appearing in the definition
of \(r_k(\eta)\) from \eqref{eq:definition-rketa}.

Without loss of generality, assume that \(r_k(\eta) < n\) for all $k \in \chgset$; otherwise, the desired localization inequality is trivial for all coordinates with $r_k(\eta) \ge  n$ and we may remove them from $\chgset$.

Recall $\mathcal H_n(r)=\mathcal H_n(r, \chgset)$ from \eqref{eq:definition-Hnr}, and let 
\[
\mathcal H_{n,k}(r)
:= \mathcal H_{n}(r, \{k\})
=
\sup_{\substack{
1\le s\le n\\
|s-\omega_k|\ge r
}}
\frac{
\sqrt n\,
\|C_{n,k}^{O}(s)-C_{n,k}^{O}(\omega_k)\|_{\Hb_k}
}{
|s-\omega_k|\,\gamma_k\Delta_k
}, \quad r <n,
\]
with $C_{n,k}^O$ from \eqref{eq:definition-CnkO} and $\Delta_k$ from \eqref{eq:def-nu_k}. By Lemma~\ref{lem:general_localization_lemma}, applied to each coordinate $k \in \chgset$, it is enough to show that
\[
    \mathbb P\Big(
        \forall k\in\chgset:
        \mathcal H_{n,k}(r_k(\eta))<\frac15
    \Big)
    \ge 1-\eta .
\]
We will first work with each coordinate separately. Let
\[
    L_{k,+}:=n-1-\omega_k,
    \qquad
    L_{k,-}:=\omega_k-1.
\]
For \(j\ge0\), set
\begin{align}
\label{eq:def-xj}
    m_j:=2^jr,
    \qquad
    x_j:=x_\eta+j .
\end{align}
These sequences will be used below to obtain a sharper localization rate via a standard dyadic segment argument. 

By the Beveridge--Nelson decomposition associated with
Assumption~\ref{cond:linear} (see \cite{li2024detection}, Equation (B.2) on page 34 of \url{https://arxiv.org/pdf/2304.07003}), we can write
\begin{align}
\label{eq:beveridge-nelson}
    \varepsilon_k^{(i)}
    =
    X_{i,k}+R_{i-1,k}-R_{i,k},
    \qquad i\in\mathbb Z,
\end{align}
with
\[
    X_{i,k}:=A_k\eta_{i,k},
    \qquad
    R_{i,k}:=\sum_{\ell=0}^{\infty}\widetilde A_{k,\ell}\eta_{i-\ell,k}
\]
and 
\[
    A_k:=\sum_{\ell=0}^{\infty}A_{k,\ell},
    \qquad
    \widetilde A_{k,\ell}:=\sum_{j=\ell+1}^{\infty}A_{k,j}.
\]
Note that the $X_{i,k}$ are centered and independent over $i$.
The summability and tail conditions in Assumption~\ref{cond:linear} imply 
\begin{align}
\label{eq:bound-on-rik}
    \sup_{i\in\mathbb Z}\max_{k\in\chgset}
    \big\|
        \|R_{i,k}\|_{\Hb_k}
    \big\|_{\psi_1}
    \le
    D_R C_A. 
\end{align}
Indeed, for any fixed $(i,k)$, it holds for  any fixed $N \in \N$ that by~\ref{orlicz:triangle} and the triangle inequality on $\Hb_k$ that for
\[
R_{i,k}^{(N)}
:=
\sum_{j=0}^{N}
\widetilde A_{k,j}\eta_{i-j,k}.
\]
we have 
\begin{align*}
\left\|
    \|R_{i,k}^{(N)}\|_{\Hb_k}
\right\|_{\psi_1}
&\le
\sum_{\ell=0}^{N}
\left\|
    \|\widetilde A_{k,\ell}\eta_{i-\ell,k}\|_{\Hb_k}
\right\|_{\psi_1}
\\
&\le
\sum_{\ell=0}^{N}
\|\widetilde A_{k,\ell}\|_{\mathrm{op}}
\big\|
    \|\eta_{i-\ell,k}\|_{\Hb_k}
\big\|_{\psi_1}
\le
D_R
\sum_{\ell=0}^{N}
\|\widetilde A_{k,\ell}\|_{\mathrm{op}}
\le
D_R C_A .
\end{align*}
It is straightforward to verify that the sequence $\|R_{i,k}^{(N)}\|_{\Hb_k}$ is Cauchy in the $\psi_1$ norm. Then, following \cite[Equation H.2]{Dec24}, let $L^{\psi_1}=\{X: \|X\|_{\psi_1}<\infty\}$, where each $X$ in the collection is a random variable defined on a fixed probability space.  By Lemma H.0.1 of \cite{Dec24},  $L^{\psi_1}$ is a Banach space for the $\psi_1$ norm, and therefore complete in the $\psi_1$ norm. Hence the Cauchy sequence $\|R_{i,k}^{(N)}\|_{\Hb_k}$converges to a limit in the $\psi_1$ norm. It then follows by~\ref{orlicz:limit} that
\[
\Big\| \lim_{N\to \infty}\Big\|\sum_{j=0}^{N}
\widetilde A_{k,j}\eta_{i-j,k}\Big\|_{\Hb_k}\Big\|_{\psi_1} 
\leq 
D_R C_A,
\]
which justifies \eqref{eq:bound-on-rik}.
Similarly, Assumption~\ref{cond:linear} gives
\[
\|A_k\|_{\mathrm{op}}
\le
\sum_{\ell=0}^{\infty}
\|A_{k,\ell}\|_{\mathrm{op}}
\le
\sum_{\ell=0}^{\infty}
(\ell+1)\|A_{k,\ell}\|_{\mathrm{op}}
\le
C_A,
\]
which in turn implies
\[
\big\|
    \|X_{i,k}\|_{\Hb_k}
\big\|_{\psi_1}
=
\big\|
    \|A_k\eta_{i,k}\|_{\Hb_k}
\big\|_{\psi_1}
\le
\|A_k\|_{\mathrm{op}}
\big\|
    \|\eta_{i,k}\|_{\Hb_k}
\big\|_{\psi_1}
\le
D_R C_A .
\]
Therefore, by \ref{orlicz:expectation},
\begin{align}
\label{eq:Bernstein-bound-assumption}
\mathbb E\|X_{i,k}\|_{\Hb_k}^{p} 
\le (2p!) \cdot 
\big\|
    \|X_{i,k}\|_{\Hb_k}
\big\|_{\psi_1}^p
\le
(2p!) \cdot (D_R C_A)^p
\le 
\frac{p!}2 \cdot (2D_R C_A)^p
\end{align}
as required for later use of the Hilbert-space Bernstein inequality of~\cite{pinelis1994optimum}, recorded in Lemma~\ref{lem:pinelis-maximal-bernstein}, with $B=2D_RC_A$.

For arbitrary integers \(u<v\), the Beveridge--Nelson decomposition from \eqref{eq:beveridge-nelson} gives
\[
    \sum_{i=u+1}^{v}\varepsilon_k^{(i)}
    =
    \sum_{i=u+1}^{v}X_{i,k}
    +
    R_{u,k}-R_{v,k}.
\]
Moreover, from~\eqref{eq:COnk-difference-decomposition1}, writing \(s=\omega_k+d\) with \(1\le d\le L_{k,+}\),
\[
\sqrt n
\{C_{n,k}^{O}(\omega_k+d)-C_{n,k}^{O}(\omega_k)\}
=
\sum_{i=\omega_k+1}^{\omega_k+d}\varepsilon_k^{(i)}
-
\frac dn\sum_{i=1}^{n}\varepsilon_k^{(i)},
\]
whereas for \(s=\omega_k-d\) with \(1\le d\le L_-\), from \eqref{eq:COnk-difference-decomposition2},
\[
\sqrt n
\{C_{n,k}^{O}(\omega_k-d)-C_{n,k}^{O}(\omega_k)\}
=
-
\sum_{i=\omega_k-d+1}^{\omega_k}\varepsilon_k^{(i)}
+
\frac dn\sum_{i=1}^{n}\varepsilon_k^{(i)}.
\]
Therefore, by the triangle inequality, the right-side increment satisfies
\begin{align*}
&\sqrt n
\left\|C_{n,k}^{O}(\omega_k+d)-C_{n,k}^{O}(\omega_k)
\right\|_{\Hb_k}
\\&\hspace{4cm} \le
\Big\|
\sum_{i=\omega_k+1}^{\omega_k+d}X_{i,k}
\Big\|_{\Hb_k}
+
\|R_{\omega_k,k}\|_{\Hb_k}
      +\|R_{\omega_k+d,k}\|_{\Hb_k}
+
dT_{n,k},
\end{align*}
where
\[
    T_{n,k}
    :=
    \frac1n
    \Big\|
        \sum_{i=1}^{n}X_{i,k}
    \Big\|_{\Hb_k}
    +
    \frac{\|R_{0,k}\|_{\Hb_k}
          +\|R_{n,k}\|_{\Hb_k}}{n}.
\]
Similarly, the left-side increment satisfies
\begin{align*}
&\sqrt n
\left\|
\{C_{n,k}^{O}(\omega_k-d)-C_{n,k}^{O}(\omega_k)
\right\|_{\Hb_k}
\\ &\hspace{4cm}\le
\Big\|
\sum_{i=\omega_k-d+1}^{\omega_k}X_{i,k}
\Big\|_{\Hb_k}
+
\|R_{\omega_k-d,k}\|_{\Hb_k}
      +\|R_{\omega_k,k}\|_{\Hb_k}
+
dT_{n,k}.
\end{align*}
To remove an extra additive factor of size $\log(n)$ that results from using a naive union bound, we use the standard ``peeling'' argument from empirical process theory
(see Section 5.3 of \citealp{vandeGeer2000}), applied to the range 
\(d\ge r\). Let 
\[
    \mathcal D_j(r)
    :=
    \{d\in\mathbb N: 2^jr\le d<2^{j+1}r\}
    =
    \{d\in\mathbb N: m_j\le d<2m_j\},
\]
where $m_j = 2^j r$.
For \(j\ge0\), define the contribution from the $j$th  segment $\mathcal D_j(r)$  as
\begin{align}
    \label{eq:def-ejkr}
E_{j,k}(r)
:= \nonumber
&\sup_{\substack{
d\in\mathcal D_j(r)\\
d\le L_{k,+}
}}
\frac1d
\Big\|
    \sum_{i=\omega_k+1}^{\omega_k+d}X_{i,k}
\Big\|_{\Hb_k}
+
\sup_{\substack{
d\in\mathcal D_j(r)\\
d\le L_{k,-}
}}
\frac1d
\Big\|
    \sum_{i=\omega_k-d+1}^{\omega_k}X_{i,k}
\Big\|_{\Hb_k}
\\
&+
\sup_{\substack{
d\in\mathcal D_j(r)\\
d\le L_{k,+}
}}
\frac{
    \|R_{\omega_k,k}\|_{\Hb_k}
    +
    \|R_{\omega_k+d,k}\|_{\Hb_k}
}{d}
+
\sup_{\substack{
d\in\mathcal D_j(r)\\
d\le L_{k,-}
}}
\frac{
    \|R_{\omega_k-d,k}\|_{\Hb_k}
    +
    \|R_{\omega_k,k}\|_{\Hb_k}
}{d}.
\end{align}
Since the union (over all $j\geq 0$) of the segments $\mathcal D_j(r)$  covers the region $d\geq r$, the previous bounds imply that
\begin{align}
\label{eq:bound-on-Hnkr}
    \mathcal H_{n,k}(r)
    \le
    \frac1{\gamma_k\Delta_k}
    \left[
        \sup_{j\ge0}E_{j,k}(r)
        +
        T_{n,k}
    \right].
\end{align}
We now bound each $E_{j,k}(r)$. Since \(d\ge m_j\) it holds for each $j\geq 0$ that 
\[
\sup_{\substack{
d\in\mathcal D_j(r)\\
d\le L_{k,+}
}}
\frac1d
\Big\|
    \sum_{i=\omega_k+1}^{\omega_k+d}X_{i,k}
\Big\|_{\Hb_k}\leq \frac1{m_j}
\max_{1\le d\le (2m_j-1)\wedge L_+}
\Big\|
    \sum_{i=\omega_k+1}^{\omega_k+d}X_{i,k}
\Big\|_{\Hb_k}.
\]
Recalling \eqref{eq:Bernstein-bound-assumption}, we may apply the Hilbert-space Bernstein inequality of~\cite{pinelis1994optimum}, recorded in Lemma~\ref{lem:pinelis-maximal-bernstein}, to obtain that, for every \(j\ge0\) and $x_j \ge 1$,
\[
\begin{aligned}
&\mathbb P\Bigg(
\max_{1\le d\le (2m_j-1)\wedge L_+}
\Big\|
    \sum_{i=\omega_k+1}^{\omega_k+d}X_{i,k}
\Big\|_{\Hb_k}
>
D_R C_A\{\sqrt{m_jx_j}+x_j\}
\Bigg)
\le
2e^{-x_j/4}.
\end{aligned}
\]
Hence, there exists an event $\Omega_{S,+,j}$ of probability at most  \(2e^{-x_j/4}\), such that on $\Omega^c_{S,+,j}$ it holds that 
\[
   \sup_{\substack{
d\in\mathcal D_j(r)\\
d\le L_{k,+}
}}
\frac1d
\Big\|
    \sum_{i=\omega_k+1}^{\omega_k+d}X_{i,k}
\Big\|_{\Hb_k} \leq  D_R C_A
    \left\{
        \sqrt{\frac{x_j}{m_j}}
        +
        \frac{x_j}{m_j}.
    \right\}
\]
Using the same reasoning, there is an event $\Omega_{S,-,j}$ of probability at most  \(2e^{-x_j/4}\), such that on $\Omega^c_{S,-,j}$ it holds that 
\[
   \sup_{\substack{
d\in\mathcal D_j(r)\\
d\le L_{k,-}
}}
\frac1d
\Big\|
    \sum_{i=\omega_k-d+1}^{\omega_k}X_{i,k}
\Big\|_{\Hb_k} \leq D_R C_A
    \left\{
        \sqrt{\frac{x_j}{m_j}}
        +
        \frac{x_j}{m_j}
    \right\}.
\]
In view of \eqref{eq:bound-on-rik}, 
a union bound over at most \(2m_j\) remainder variables together with \ref{orlicz:tail_bound} gives
\[
\begin{aligned}
&\mathbb P\Bigg(
\max_{0\le d\le (2m_j-1)\wedge L_{k,+}}
\|R_{\omega_k+d,k}\|_{\Hb_k}
>
D_R C_A\{x_j+\log(em_j)\}
\Bigg)
\le
2e^{-x_j} \le 2 e^{-x_j/4}.
\end{aligned}
\]

The same estimate holds for the second remainder term, corresponding to region to the left of the change point. Therefore, considering the left and right partial sum terms, and left and right remainder terms, for each fixed $j\geq 0$, there is an event $\Omega_j$ of probability at most \(8e^{-x_j/4}\)
such that, on $\Omega^c_j$, it holds
\[
 E_{j,k}(r)\leq  D_R C_A
    \left\{
        \sqrt{\frac{x_j}{2^jr}}
        +
        \frac{x_j}{2^jr}
        +
        \frac{\log(e2^jr)}{2^jr}
    \right\},
\]
where $E_{j,k}(r)$ is from \eqref{eq:def-ejkr}, and we substituted $m_j:=2^jr$. Since \(x_j=x_\eta+j\) by definition in \eqref{eq:def-xj}, we have
\[
   \sum_{j=0}^{\infty} 8 e^{-x_j/4}
    =
    8 e^{-x_\eta/4}\sum_{j=0}^{\infty}e^{-j/4}
    =
    \frac{8}{1-e^{-1/4}} e^{-x_\eta/4}
    =
    C_1 e^{-x_\eta/4},
\]
where $C_1=8/(1-e^{-1/4})\approx 36.1664933$.  Therefore, outside an event of probability at most \(C_1 e^{-x_\eta/4}\),
\[
\sup_{j\ge0}E_{j,k}(r)
\le
D_R C_A
\sup_{j\ge0}
    \left\{
        \sqrt{\frac{x_\eta+j}{2^jr}}
        +
        \frac{x_\eta+j}{2^jr}
        +
        \frac{\log(e2^jr)}{2^jr}
    \right\}.
\]
Since \(x_\eta\ge1\),
\[
    \sup_{j\ge0}
    \sqrt{\frac{x_\eta+j}{2^jr}}\leq  \sup_{j\ge0}
    \sqrt{\frac{x_\eta}{2^jr}}+ 
    \sqrt{\frac{1}{2r}}
    \le
    3\sqrt{\frac{x_\eta}{r}},
\]
and
\[
    \sup_{j\ge0}
    \frac{x_\eta+j}{2^jr}
    \le
    \frac{x_\eta}{r},
\]
and, for an absolute constant $C_2$,
\[
    \sup_{j\ge0}
    \frac{\log(e2^jr)}{2^jr}
    = \frac{\log(er)}r
    \le
    \frac{C_2}{\sqrt r}
    \le
    C_2\sqrt{\frac{x_\eta}{r}}.
\]
Thus, with probability at least \(1-C_1e^{-x_\eta/4}\), it holds for the absolute constant $C_3 = 4 + C_2$ that 
\[
   \sup_{j\ge0}E_{j,k}(r)\leq  C_3 D_R C_A
    \left\{
        \sqrt{\frac{x_\eta}{r}}
        +
        \frac{x_\eta}{r}
    \right\}.
\]
It remains to control  \(T_{n,k}\). Applying the
same Hilbert-space Bernstein inequality to \(\sum_{i=1}^{n}X_{i,k}\), and
using the sub-exponential tail bound from \eqref{eq:bound-on-rik} for \(R_{0,k}\) and \(R_{n,k}\), gives for an absolute constant $C_4$
\[
    T_{n,k}
    \le
    C_4D_R C_A
    \left\{
        \sqrt{\frac{x_\eta}{n}}
        +
        \frac{x_\eta}{n}
    \right\}
\]
outside an event of probability at most \(6e^{-x_\eta/4}\). Combining the previous two displays with \eqref{eq:bound-on-Hnkr}, recalling that $\nu_k = \gamma_k\Delta_k$ by definition in \eqref{eq:def-nu_k}, and using that $x_\eta/n \le x_\eta/r$ we obtain that there is an absolute constant $C_5$ such that 
\[
\mathbb P\left(
\mathcal H_{n,k}(r)
>
\frac{C_5D_R C_A}{\nu_k}
\left\{
    \sqrt{\frac{x_\eta}{r}}
    +
    \frac{x_\eta}{r}
\right\}
\right)
\le
C_5e^{-x_\eta/4},
\]
with $x_\eta=C_x\log(e|\chgset|/\eta)$ from \eqref{eq:definition-x_eta}. 
Now take \(r=r_k(\eta)\) from \eqref{eq:definition-rketa}. 
Since
\[
    r_k(\eta) = \Big\lceil \frac{C}{C_x}x_\eta (\nu_k^{-2}+ \nu_k^{-1}) \Big\rceil
    \ge \frac{C}{C_x} x_\eta\nu_k^{-2}
    \quad \text{ and } \quad
    r_k(\eta)\ge \frac{C}{C_x} x_\eta\nu_k^{-1}.
\]
we have
\[
    \frac{D_R C_A}{\nu_k}
    \sqrt{\frac{x_\eta}{r_k(\eta)}}
    \le
    \frac{D_R C_A}{\sqrt{C/C_x}}
    \quad \text{ and } \quad
    \frac{D_R C_A}{\nu_k}
    \frac{x_\eta}{r_k(\eta)}
    \le
    \frac{D_R C_A}{C/C_x}.
\]
Therefore, choosing $C$ as a sufficiently large multiple of $C_x$, we obtain that $\mathcal H_{n,k}(r_k(\eta))<1/5$ 
outside an event of probability at most \(C_5e^{-x_\eta/4}\). A union bound over \(k\in\chgset\) gives
\[
\mathbb P\left(
    \exists k\in\chgset:
    \mathcal H_{n,k}(r_k(\eta))\ge\frac15
\right)
\le
C_5|\chgset|e^{-x_\eta/4}.
\]
Since \(x_\eta=C_x\log(e|\chgset|/\eta)\), choosing \(C_x\) sufficiently large
makes the right-hand side at most \(\eta\),
which completes the proof.
\end{proof}

\begin{proof}[Proof of Corollary~\ref{cor:stepdown}]
For the first assertion, the first step of the step-down procedure uses \(\mathcal R_1=[K]\), and hence
\[
\widehat Q_{1-\gamma}(\mathcal R_1)
=
\widehat Q_{1-\gamma}.
\]
Thus the coordinates rejected at the first step are exactly
\[
\left\{
k\in[K]:
T_{n,k}^2>\widehat Q_{1-\gamma}
\right\}
=
\widehat{\mathcal S}_\gamma.
\]
Since subsequent steps can only add rejections,
$
\widehat{\mathcal S}_\gamma
\subseteq
\widehat{\mathcal S}^{\mathrm{SD}}_\gamma
$ as asserted.

For the second assertion, fix $P\in\mathcal P$ and write
$
\mathcal N_P:=\mathcal S_P^c
$
for the set of true null coordinates. If $\mathcal N_P=\varnothing$, then there are no true null hypotheses, so
$
\widehat{\mathcal S}^{\mathrm{SD}}_\gamma
\cap \mathcal S_P^c
=
\varnothing.
$
Hence
\[
\mathbb P_P\!\left(
\widehat{\mathcal S}^{\mathrm{SD}}_\gamma
\cap \mathcal S_P^c\neq\varnothing
\right)
=0,
\]
and the first claim is immediate. From now on we suppose that \(\mathcal N_P\neq\varnothing\). For every nonempty \(\mathcal J\subseteq[K]\), recall from Section~\ref{subsec:stepdown} that 
\[
(T_{n,\mathcal J}^*)^2
=
\max_{k\in\mathcal J}(T_{n,k}^*)^2.
\]
If \(\mathcal J\subseteq\mathcal J'\), then, for every bootstrap realization,
\[
(T_{n,\mathcal J}^*)^2
\le
(T_{n,\mathcal J'}^*)^2,
\]
since the second variable is a maximum over more coordinates. Therefore
\begin{equation}
\label{eq:stepdown-quantile-monotonicity}
\widehat Q_{1-\gamma}(\mathcal J)
\le
\widehat Q_{1-\gamma}(\mathcal J').
\end{equation}
Consider the event that the step-down procedure rejects at least one true null, and let \(j_\star\) be the first step at which this occurs. Since no true null has been rejected before step \(j_\star\),
\[
\mathcal N_P\subseteq\mathcal R_{j_\star}.
\]
At step \(j_\star\), some \(k\in\mathcal N_P\) satisfies
\[
T_{n,k}^2>
\widehat Q_{1-\gamma}(\mathcal R_{j_\star}).
\]
By~\eqref{eq:stepdown-quantile-monotonicity},
\[
\widehat Q_{1-\gamma}(\mathcal N_P)
\le
\widehat Q_{1-\gamma}(\mathcal R_{j_\star}),
\]
and hence
\begin{equation*}
\left\{
\widehat{\mathcal S}^{\mathrm{SD}}_\gamma
\cap\mathcal S_P^c\neq\varnothing
\right\}
\subseteq
\left\{
\max_{k\in\mathcal S_P^c}T_{n,k}^2
>
\widehat Q_{1-\gamma}(\mathcal S_P^c)
\right\}.
\end{equation*}

It remains to control the probability on the right-hand side. This is precisely the true-null subset bootstrap problem considered in the proof of Theorem~\ref{thm:fwer_control}. In the notation of that proof,
\[
X_0=\max_{k\in\mathcal S_P^c}T_{n,k}^2,
\qquad
\widehat Q_{1-\gamma}(\mathcal S_P^c)
=
\widehat Q^0_{1-\gamma}.
\]
By~\eqref{eq:fwer-proof-three-profiled-event} and the lower bound in
\eqref{eq:fwer-proof-three-quantile-sandwich},
\[
\mathbb P_P\!\left(X_0>\widehat Q^0_{1-\gamma}\right)
\le
\mathbb P_P(X_0>q_-)+\varepsilon_{n,q,r}.
\]
Since $q_-\ge\rho_\gamma$; see~\eqref{eq:rho-gamma} and~\eqref{eq:fwer-proof-three-quantile-sandwich} for definitions;  
\eqref{eq:fwer-proof-three-profiled-oracle} gives
\[
\mathbb P_P(X_0>q_-)
\le
\gamma+2\varepsilon_{n,q,r}.
\]
Consequently,
\[
\mathbb P_P\!\left(
\max_{k\in\mathcal S_P^c}T_{n,k}^2
>
\widehat Q_{1-\gamma}(\mathcal S_P^c)
\right)
\le
\gamma+3\varepsilon_{n,q,r}.
\]
The same choice of constants as in the argument leading to
\eqref{eq:fwer-profiled-strong-control} therefore yields
\[
\mathbb P_P\!\left(
\widehat{\mathcal S}^{\mathrm{SD}}_\gamma
\cap\mathcal S_P^c\neq\varnothing
\right)-\gamma
\le
\frac{C}{\sqrt{\tau}}
\log^d(nK)\mathcal R_{n,q,r},
\]
completing the proof.
\end{proof}

\section{Supporting Results}
\label{sec:supporting-results}

\subsection{Main Supporting Results}

Throughout this section, we fix $\rho>0$, which serves as a non-uniform Kolmogorov threshold. We assume that Assumptions~\ref{cond:stationary} and~\ref{cond:scores} hold with fixed values of the parameters appearing in those assumptions.

Unless stated otherwise, a “constant” means a positive quantity that may depend on these fixed parameters and on $\rho$, but is otherwise independent of $n,q,r,K,M,x$. When useful, we will specify that a constant depends only on a particular subset of these quantities. Unless a result imposes further restrictions, all results in this section hold for all
\begin{equation*}
(n,q,r,K,M,x)
\in
\mathbb{N}_{\geq 4}
\times \mathbb{N}
\times \mathbb{N}
\times \mathbb{N}
\times \mathbb{N}
\times (0,3/16],
\end{equation*}
subject to the additional restrictions $1 \leq r\leq q\leq n/4$. We use the following side conditions:
\begin{equation}
\label{eq:side-conditions}
\begin{aligned}
\text{(X)}\quad
&x
\leq
\frac{1}
{C_X{[\log(nK)\log(n)]}^{4}}, \qquad
\text{(M)}\quad
M^{2\gamma_L}
\leq
\frac{x(q\wedge m)}{C_M},
\end{aligned}
\end{equation}
where  $m:=\left\lfloor n/(q+r)\right\rfloor$ as in~\eqref{def:m}. Here, $C_X>0$ is a constant that is permitted to depend on $D_\psi,\ C_\beta,\ c_\beta,\ C_L,\ C_U,\ \gamma_U,\ \gamma_L, \rho$, and \(C_M>0\) is a constant that is independent of \(\rho\) but may depend on $D_\psi,\ C_\beta,\ c_\beta,\ C_L,\ C_U,\ \gamma_U,\ \gamma_L$. These side conditions are imposed only when explicitly invoked in the statement of a result.
Recall the definition for each $u \in I=[0,1]$ of $V(u)=u(1-u)$ from~\eqref{def:brownian_covariance_kernel_variance_2} and for each $x >0$ let
\begin{equation}\label{eq:definition-S_x}
S_x
:=
\left\{
s\in[n]:
V(s/n)\ge x
\right\}.
\end{equation}
Since $\max_{s\in[n]}V(s/n) \ge 6/25 > 3/16$ (the first bound is sharp for $n=5$),
both \(S_x\) and \([n]\setminus S_x\) are nonempty. Hence, by construction,
\begin{equation}\label{eq:infsup-V-S_x}
\min_{s\in S_x}V(s/n)\ge x,
\qquad
\max_{s\in[n]\setminus S_x}V(s/n)\le x.
\end{equation}
In each of the following theorems, the sufficiently large constants \(C_X\) and \(C_M\) appearing in the conditions from \eqref{eq:side-conditions} may be chosen differently.

To state Theorem~\ref{Thm_1} we introduce the oracle analogue of each marginal CUSUM process $C_{n,k}$ from~\eqref{def:cusum} by first defining the oracle partial sum process
\begin{equation*}
S_{n,k}^{O}(s)
=
\sum_{i=1}^{n}\varepsilon^{(i)}_{k}\mathbf 1(i\le s),
\qquad 0\le s\le n,
\end{equation*}
and then setting
\begin{equation}
\label{eq:definition-CnkO}
C_{n,k}^{O}(s)
=
\frac{1}{\sqrt n}
\left[
S_{n,k}^{O}(s)
-
\frac{s}{n}S_{n,k}^{O}(n)
\right],
\qquad 0\le s\le n.
\end{equation}
The oracle analogue of \(T_n\) from~\eqref{def:global_test} is then
\begin{equation}
\label{eq:def_oracle_test_statistic}
T^O_n
:=
\max_{k\in[K]}
\max_{s\in[n]}
\left\|C_{n,k}^O(s)\right\|_{\Hb_k}.
\end{equation}
Note that $C_{n,k}^{O}(0)=C_{n,k}^{O}(n)=0$. Further recall the definition of $G_{sk}^M$ from \eqref{eq:definition-GMsk}, $\check M_k$ from \eqref{eq:definition-check-mk}, and $\mathbb B_{\mathcal K}$ from \eqref{def:Hilbert_Brownian_Bridge}.

\begin{theorem}\label{Thm_1}
    Fix $\rho>0$, and suppose that Assumptions~\ref{cond:stationary} and~\ref{cond:scores} hold. There exist sufficiently large constants \(C_X,C_M>0\) and constants
\(C_1,c_1>0\) such that, whenever conditions \text{\normalfont {(X)}} and \text{\normalfont {(M)}} from~\eqref{eq:side-conditions} hold,
\begin{align*}
d_{\mathrm{K}}^{(\rho)}\left( (T_n^O)^2,
\max_{k\in[K]}
\max_{s\in S_x}
\|G_{sk}^M\|_{2}^2
\right)
\leq
C_1\Delta_1\log^{c_1}(nK)
+
2\frac{n}{q}\beta_r,
\end{align*}
where $d_{\mathrm{K}}^{(\rho)}$ is the restricted K-distance from \eqref{eq:restricted-k-distance} and where
\begin{align}
\label{eq:def-delta_1}
\Delta_1
:=
x^{-3/2}
\left[
\left(\frac{M^4q^2}{n}\right)^{1/6}
+
M^{1-2\gamma_U}
+
\left(\frac rq\right)^{1/2}
\right].
\end{align}
The constant \(c_1\) depends only on \(\gamma_U\) and \(\gamma_L\).
\end{theorem}

\begin{theorem}\label{Thm_2}
Fix $\rho>0$, and suppose Assumptions~\ref{cond:stationary} and~\ref{cond:scores} hold.
There exist sufficiently large constants \(C_X,C_M>0\) and constants
\(C_2,c_2>0\) such that, whenever conditions \text{\normalfont {(X)}} and \text{\normalfont {(M)}} from~\eqref{eq:side-conditions} hold,
\begin{align*}
d_{\mathrm{K}}^{(\rho)}\left(
\max_{k\in[K]}
\max_{s\in S_x}
\|G_{sk}^M\|_{2}^2,
\max_{k\in[K]}
\sup_{u\in I}
\|\pi_k\mathbb B_{\mathcal K}(u)\|_{\Hb_k}^2
\right)
\leq
C_2\Delta_2\log^{c_2}(nK),
\end{align*}
where $d_{\mathrm{K}}^{(\rho)}$ is the restricted K-distance from \eqref{eq:restricted-k-distance} and where
\[
\Delta_2
:=
x^{-3/2}
\left[
\left(\frac{M^2}{q\wedge m}\right)^{1/3}
+
M^{1-2\gamma_U}
\right].
\]
The constant $c_2$ depends only on $\gamma_U$ and $\gamma_L$.
\end{theorem}

Recall the set of changed coordinates
\(\mathcal S\subseteq[K]\) from~\eqref{eq:change-point-coordinates}.
For each \(k\in\mathcal S\) and \(\ell\in[n]\), let
\begin{equation*}
\label{eq:def-coordinatewise-rbeta-new_2}
    r_{\beta,k}^{(n)}(\ell)
    :=
    \left\lceil
        C_{\mathrm{loc}}
        \frac{\log(nK)}{\nu_k}
        \left\{
            \frac{\log(nK)}{\nu_k}
            +
            \ell\log\left(\frac{en}{\ell}\right)
        \right\}
    \right\rceil 
\end{equation*}
denote the coordinatewise
localization radius defined in~\eqref{eq:def-coordinatewise-rbeta-new}.
The subscript \(\beta\) indicates that its probability guarantee in Corollary~\ref{cor:coordinatewise_beta_localization_new} is
obtained under the \(\beta\)-mixing conditions of
Assumption~\ref{cond:stationary}. For \(g\in\mathbb N\), \(\ell\in[n]\), and
\(\mathcal J\subseteq[K]\), define
\begin{equation}
\label{def:W_j(g)}
\mathcal W_{\mathcal J}(g,\ell)
:=
\left\{
k\in\mathcal J\cap\mathcal S:
r_{\beta,k}^{(n)}(\ell)>g
\right\}.
\end{equation}
Thus \(\mathcal W_{\mathcal J}(g,\ell)\) contains the changed
coordinates in \(\mathcal J\) for which
Corollary~\ref{cor:coordinatewise_beta_localization_new} does not
guarantee localization within distance \(g\). Finally, recalling $\Delta_k=\| \delta_k\|_{\Hb_k}$ from \eqref{eq:def-nu_k}, define
\begin{equation}
\label{eq:def-delta-J-wk}
\Delta_{\mathcal J,\mathrm{wk}}(g,\ell)
:=
\max_{k\in\mathcal W_{\mathcal J}(g,\ell)}\Delta_k,
\end{equation}
where the maximum over the empty set is defined to equal zero following the convention throughout this manuscript.

\begin{theorem}[Conditional Gaussian approximation on a coordinate subset]
\label{thm:conditional_gaussian_subset_master}
Fix \(\rho>0\) and suppose
Assumptions~\ref{cond:stationary} and~\ref{cond:scores} hold.
Fix \(\varnothing\ne\mathcal J\subseteq[K]\) and suppose that the tuning parameters $q,r,g,\tau$ for the residual bootstrap from Section~\ref{subsec:bootstrap} satisfy
\begin{equation}
\label{eq:adaptive-subset-block-conditions}
   \tau \in (0,1/2) \qquad 1\le r\le q,
    \qquad
    g+q+r\le\frac{\tau n}{2}
\end{equation}
and also that
\begin{equation}
\label{eq:intertior-condition}
    \omega_k\in[\tau n,(1-\tau)n]
    \quad \forall\,
    k\in\mathcal J\cap\mathcal S,
\end{equation}
where this condition is understood to be vacuous if
\(\mathcal J\cap\mathcal S=\varnothing\). 
Then there exist sufficiently large constants \(C_X,C_M>0\), constants \(C_3,c_3>0\), and an absolute constant \(C_{\mathrm{boot}}>0\) such that, for every $\ell \in [n]$ and every $x \in (0,3/16]$ and every $M \in \N_{\ge 1}$ such that \text{\normalfont (X)} and \text{\normalfont (M)} from~\eqref{eq:side-conditions} hold, there exists an event \(\Omega_{\mathrm{boot},\mathcal J}\in\mathcal F_n\) satisfying
\begin{equation*}
\mathbb P\left(
\Omega_{\mathrm{boot},\mathcal J}^c
\right)
\le
6
\left\{
\frac1n
+
\left\lceil\frac n{\ell}\right\rceil\beta_{\ell}
+
\frac nq\beta_r
\right\},
\end{equation*}
such that, on \(\Omega_{\mathrm{boot},\mathcal J}\),
\begin{align}
&\sup_{t\ge\rho}
\Bigg|
\mathbb P\left(
\max_{k\in\mathcal J}
\max_{s\in[n]}
\|C_k^*(s)\|_{\Hb_k}^2
\le t
\,\middle|\,
\mathcal F_n
\right)
-
\mathbb P\left(
\max_{k\in\mathcal J}
\max_{s\in S_x}
\|G_{sk}^M\|_{2}^2
\le t
\right)
\Bigg|
\nonumber
\\& \hspace{9cm} \le
C_3
\Delta_{3,\mathcal J}^{\mathrm{ad}}
\log^{c_3}(nKM),
\label{eq:subset-conditional-gaussian-approximation}
\end{align}
where, recalling \(\Delta_{\mathcal J,\mathrm{wk}}(g,\ell)\) from~\eqref{eq:def-delta-J-wk}, 
\begin{align} \nonumber
\Delta_{3,\mathcal J}^{\mathrm{ad}}
:={}&
x^{-3/2}
\Bigg[
\frac1{\sqrt\tau}
\Bigg\{
\left(\frac{q+g}{n}\right)^{1/2}
\left(
1\vee\frac{\ell}{\sqrt q}
\right)
+
\sqrt q\,
\Delta_{\mathcal J,\mathrm{wk}}(g,\ell)
\Bigg\}
\\ & \hspace{4cm} +
M^{1-2\gamma_U}
+
M^{2/3}
\left\{
\left(\frac qn\right)^{1/6}
\vee
\left(\frac{q^2}{n}\right)^{1/3}
\right\}
\Bigg].
\label{eq:def-delta-3-adaptive-subset}
\end{align}
Here, the exponent \(c_3\) depends only on
\(\gamma_L\) and \(\gamma_U\), the constants \(C_3\) and $C_X$ depend only
on \(\rho\) and the fixed constants in
Assumptions~\ref{cond:stationary} and~\ref{cond:scores}, while \(C_M\)  depend only on those fixed constants and may be chosen
independently of \(\rho\). All constants are otherwise
independent of
$
n,K,\mathcal J,q,r,g,\ell,x,\tau$ and $M$.
\end{theorem}

\begin{proof}[Proof of Theorem~\ref{Thm_1}]
Throughout the proof, constants may change from line to line but depend only on the fixed constants appearing in Assumptions~\ref{cond:stationary} and~\ref{cond:scores}, and on \(\rho\). We write \(C,c\) for such constants.

Without loss of generality, we may assume that
\begin{align*}
r\le q,
\qquad
q\le \frac n4,
\qquad
q\le \frac{\sqrt n}{\log n}
\end{align*}
(which in turn implies $m = \lfloor n/(q+r) \rfloor \ge \lfloor n/(2q) \rfloor \ge 2$).
Indeed, if \(r>q\), then, since \(x\le1\), we have
$
\Delta_1
\ge
x^{-3/2}(r/q)^{1/2}
>1$, and hence the claimed bound holds trivially (note that $d_{\mathrm K}^{(\rho)}\le1$). Next, if \(q>n/4\), then, since \(M\ge1\) and \(x\le1\), we have
$
\Delta_1
\ge
\left(M^4q^2/{n}\right)^{1/6}
\ge
\left(q^2/n\right)^{1/6}
>
16^{-1/6},
$
and hence the claimed bound holds choosing $C_1$ at least $16^{1/6}$. Finally, if \(q>\sqrt n/\log n\), then $\Delta_1 \log^{1/3}(nK) \ge \left({q^2}/{n}\right)^{1/6} \log^{1/3}(n) \ge 1$, so the claimed bound holds by choosing $c_1$ at least $1/3$.

The proof proceeds by successively relating the quantities $U_i$ and $U_{i+1}$, where
\begin{align*}
U_0
&:=
\max_{k\in[K]}
\max_{s\in[n]}
\big\|C_{n,k}^O(s)\big\|_{\Hb_k} \\
U_1
&:=
\max_{k\in[K]}\max_{s \in [n]}\widetilde T_{sk},\\
U_2
&:=
\max_{k\in[K]}\max_{s\in S_x}\widetilde T_{sk},\\
U_3
&:=
\max_{k\in[K]}\max_{s\in S_x}\widetilde T^M_{sk}, \\
U_4
&:=
\max_{k\in[K]}\max_{s\in S_x}
\|G^M_{sk}\|_{\check M_k},
\end{align*}
with $\widetilde T_{sk}$ and $\widetilde T_{sk}^M$ defined in \eqref{eq:tilde-tsk} and \eqref{eq:tilde-tskm}, respectively. 
Since $U_0=T_n^O$, the quantity appearing in the assertion of the theorem can be written as \(d_{\mathrm K}^{(\rho)}(U_0^2,U_4^2)\). 

Recall the definition of 
$\Omega_{\mathrm{block}}
:=
\Omega_{\mathrm{big}}\cap\Omega_{\mathrm{small}}
$ from \eqref{eq:Omega-block}. By \eqref{eq:Omega-block-probability}, we have
\[
\mathbb P(\Omega_{\mathrm{block}}^c)
\le
(m-1)\beta_r+(m-1)\beta_q
\le
2\frac nq\beta_r,
\]
where we have used that $\beta_q\le \beta_r$ by monotonicity and $m\le {n}/({q+r}) \le n/q$.

We now construct a `good' event on which a comparison chain,  involving the $U_i$ variables defined above, holds.  By Lemma~\ref{lem:berbee_hp_reduction}, there exists an event
\(\Omega_1\) with
$
\mathbb P(\Omega_1)\ge 1-1/n
$
such that, on \(\Omega_1\cap\Omega_{\mathrm{block}}\),
\[
|U_0-U_1|
\le
\delta_1
:=
C D_\psi(S_{\beta,0}\vee1)^{1/2}
\log(nK)
\left(
\frac rq\vee \frac{q^2}{n}
\right)^{1/2}.
\]

As shown in~\eqref{bound_on_tilde_T}, there is a constant $A_1$ that depends only on $D_\psi,C_\beta,c_\beta$, and an event $\Omega^{(1)}$ of probability greater than $1-1/(3n)$ such that on $\Omega^{(1)}$ it holds that
\begin{equation*}
U_1=\max_{k\in[K]}\max_{s \in [n]}\widetilde T_{sk}\leq A_1\log(nK).
\end{equation*}
Let
$
\Omega_{01}
:=
\Omega^{(1)}\cap\Omega_1\cap\Omega_{\mathrm{block}}.
$
Then
\[
\mathbb P(\Omega_{01})
\ge
1-\frac3n-2\frac nq\beta_r.
\]
On \(\Omega_{01}\), we have
\[
U_0
\le
U_1+\delta_1
\le
A_1\log(nK)+\delta_1,
\]
so
\[
U_0+U_1
\le
2A_1\log(nK)+\delta_1.
\]
Hence, 
\begin{align}
\label{eq:u0-u1-bound}
|U_0^2-U_1^2|
=
|U_0+U_1|\,|U_0-U_1|
\le
\delta_1\{2A_1\log(nK)+\delta_1\}
=: \eta_{01}.
\end{align}
Since \(r\le q\) and \(q\le \sqrt n/\log n\),
\[
\left(
\frac rq\vee \frac{q^2}{n}
\right)^{1/2}
\le1,
\]
and hence \(\delta_1\le D_1\log(nK)\) for a constant $D_1$ that depends only on $D_\psi,C_\beta,c_\beta$.
It follows that
\[
\eta_{01}
\le
C\log^2(nK)
\left(
\frac rq\vee \frac{q^2}{n}
\right)^{1/2}.
\]
Furthermore,
\[
\left(
\frac rq\vee \frac{q^2}{n}
\right)^{1/2}
\le
\left(\frac rq\right)^{1/2}
+
\frac q{\sqrt n},
\]
and since \(q^2/n\le1\) and \(M\ge1\),
\[
\frac q{\sqrt n}
=
\left(\frac{q^2}{n}\right)^{1/2}
\le
\left(\frac{q^2}{n}\right)^{1/6}
\le
\left(\frac{M^4q^2}{n}\right)^{1/6}.
\]
Thus
\begin{equation}\label{eq:eta01-bound}
\eta_{01}
\le
C\log^2(nK)
\left[
\left(\frac rq\right)^{1/2}
+
\left(\frac{M^4q^2}{n}\right)^{1/6}
\right].
\end{equation}

Define

\[
U^2_{\partial,x}
:=
\max_{k\in[K]}
\max_{s\in[n]\setminus S_x}
(\widetilde T_{sk})^2.
\]
Let
\[
B_{\partial}
:=
C_2D_\psi^2(S_{\beta,0}\vee1),
\]
where \(C_2\) is the universal constant from
Lemma~\ref{lem:coupled_big_block_boundary}. We may assume without loss
of generality that
\begin{equation}\label{eq:q-boundary-reduction-proof-optimized}
\frac{q^2\log^2 n}{n}
\le
\frac{\rho}
{4B_{\partial}\log^2(nK)}.
\end{equation}
Indeed, if \eqref{eq:q-boundary-reduction-proof-optimized} fails, then
\[
\frac{q^2}{n}
>
\frac{\rho}
{4B_{\partial}\log^2(nK)\log^2 n}
\ge
\frac{\rho}
{4B_{\partial}\log^4(nK)}.
\]
Using \(M\ge1\) and \(x\le1\), this yields
\[
\Delta_1
\ge
\left(\frac{M^4q^2}{n}\right)^{1/6}
>
\left(\frac{\rho}{4B_{\partial}}\right)^{1/6}
\log^{-2/3}(nK),
\]
so the desired theorem is trivial after enlarging the constants
$c_1$ and $C_1$.

Choose \(C_X\) sufficiently large so that $C_X\ge 4B_{\partial}/{\rho}$.
Then the first inequality in the condition from \eqref{eq:side-conditions} implies
\[
x
\le
\frac{1}{C_X[\log(nK)\log n]^4}
\le
\frac{1}{C_X\log^2(nK)}
\le
\frac{\rho}{4B_{\partial}\log^2(nK)}.
\]
Together with \eqref{eq:q-boundary-reduction-proof-optimized}, this gives
\[
x\vee \frac{q^2\log^2 n}{n}
\le
\frac{\rho}{4B_{\partial}\log^2(nK)}.
\]
Therefore Lemma~\ref{lem:coupled_big_block_boundary}, applied with
\(t=\rho/4\), yields 
$
\mathbb P(\Omega_{\partial})\ge 1-2/n,
$
where
\[
\Omega_{\partial}
:=
\left\{
U_{\partial,x}^2\le\frac{\rho}{4}
\right\}.
\]
Now note that 
$
U_1^2= U_2^2\vee U^2_{\partial,x}.
$
Consequently, on \(\Omega_{\partial}\), for every \(u\ge \rho/4\),
\begin{equation}\label{eq:u1-u2-boundary-equivalence}
\{U_1^2\le u\}
=
\{U_2^2\le u\}.
\end{equation}

We now compare \(U_2^2\) and \(U_3^2\). Choose \(C_M\) sufficiently large so
that the second inequality in the condition from \eqref{eq:side-conditions}  implies the condition in \eqref{eq:bound-on-M} of
Corollary~\ref{lem:good_eigen_values_2}. It is enough to take
\[
C_M
\ge
\frac{2C_4\bigl(D_\psi^2S_{\beta,1}\vee C_{\mathcal K}\bigr)}
{C_L^2}.
\]
Indeed, then
\[
M^{2\gamma_L}
\le
\frac{x(q\wedge m)}{C_M}
\le
\frac{
C_L^2x(q\wedge m)
}{
2C_4\bigl(D_\psi^2S_{\beta,1}\vee C_{\mathcal K}\bigr)
}.
\]
By Lemma~\ref{lem:coupled_big_block_projection_error}, there exists an event
\(\Omega_{\mathrm{proj}}\) satisfying $\mathbb P(\Omega_{\mathrm{proj}})\ge 1-2/n$ and a constant $C>0$
such that, on \(\Omega_{\mathrm{proj}}\),
\begin{align}
\label{eq:u2-u3-bound}
|U_2^2-U_3^2|
&\le
\eta_{23}
:=
C\log^2(nK)
\left[
M^{1-2\gamma_U}
+\frac1q
+\frac{q^2\log^2 n}{n}
\right].
\end{align}
Moreover, since \(r\ge1\), \(M\ge1\), and \(q^2/n\le1\),
\[
\frac1q\le\left(\frac rq\right)^{1/2},
\qquad
\frac{q^2\log^2n}{n}
\le
\left(\frac{M^4q^2}{n}\right)^{1/6}\log^2(nK).
\]
Consequently,
\begin{align*}
\eta_{23}
\le
C\log^4(nK)
\left[
M^{1-2\gamma_U}
+\left(\frac rq\right)^{1/2}
+\left(\frac{M^4q^2}{n}\right)^{1/6}
\right].
\end{align*}
Combining the preceding display with \eqref{eq:eta01-bound}, we obtain
\begin{align}
\label{eq:eta-total-bound}
\eta:=\eta_{01}+\eta_{23}
&\le
C\log^4(nK)
\left[
\left(\frac rq\right)^{1/2}
+\left(\frac{M^4q^2}{n}\right)^{1/6}
+M^{1-2\gamma_U}
\right]
\nonumber\\
&=
C\log^4(nK)x^{3/2}\Delta_1,
\end{align}
where the final equality follows from the definition of $\Delta_1$ from \eqref{eq:def-delta_1}.
In particular,
\begin{equation}\label{eq:x-eta-total-bound}
x^{-3/2}\eta
\le
C\Delta_1\log^4(nK).
\end{equation}
In the following, we may assume without loss of generality that $\eta\le {\rho}/{2}$.
Indeed, if \(\eta>\rho/2\), then by \eqref{eq:eta-total-bound} and
\(x^{-3/2}\ge1\),
\[
\Delta_1\log^4(nK) \ge
\frac1C x^{-3/2} \eta > \frac1{2C} \rho,
\]
so the desired bound is trivial after enlarging
$C_1$ and $c_1$.
Now define the good event
$
\Omega_\star
:=
\Omega_{01}\cap\Omega_{\partial}\cap\Omega_{\mathrm{proj}},
$
which satisfies
\[
\mathbb P(\Omega_\star^c)
\le
\frac7n+2\frac nq\beta_r.
\]
On \(\Omega_\star\), for every \(t\ge\rho\),
\begin{equation}\label{eq:main-sandwich}
\{U_3^2\le t-\eta\}
\subseteq
\{U_0^2\le t\}
\subseteq
\{U_3^2\le t+\eta\}.
\end{equation}
Indeed, for the upper inclusion, if \(U_0^2\le t\), then
\[
U_1^2\le t+\eta_{01}
\]
by \eqref{eq:u0-u1-bound}.
Since \(t+\eta_{01}\ge\rho>\rho/4\), \eqref{eq:u1-u2-boundary-equivalence}
implies
\[
U_2^2\le t+\eta_{01}.
\]
Therefore, by \eqref{eq:u2-u3-bound},
\[
U_3^2\le U_2^2+\eta_{23}\le t+\eta.
\]
For the lower inclusion, if \(U_3^2\le t-\eta\), then, by \eqref{eq:u2-u3-bound},
\[
U_2^2\le U_3^2+\eta_{23}\le t-\eta_{01}.
\]
Since \(t-\eta_{01}\ge t-\eta\ge \rho/2>\rho/4\), again by
\eqref{eq:u1-u2-boundary-equivalence},
\[
U_1^2\le t-\eta_{01}.
\]
Thus, by \eqref{eq:u0-u1-bound},
\[
U_0^2\le U_1^2+\eta_{01}\le t.
\]
It remains to compare \(U_3^2\) with \(U_4^2\) and then remove the shift
\(\eta\). Let
$
\kexp:=\gamma_L / (2\gamma_U-1).
$
Applying Lemma~\ref{lem:gaussian_approx_projected_coupled} with
threshold \(\rho/2\), and using \(\log n\le\log(nK)\), gives
\begin{align}
\label{eq:u3-u4-bound}
\sup_{u\ge\rho/2}
\left|
\mathbb P(U_3^2\le u)-\mathbb P(U_4^2\le u)
\right|
\le
\varepsilon_G
:=
C x^{-3/2}
\left(\frac{M^4q^2}{n}\right)^{1/6}
\log^{6\kexp+7/2}(nK).
\end{align}
for a sufficiently large constant $C$ that may depend on the parameters listed in Lemma~\ref{lem:gaussian_approx_projected_coupled}.
Since \(\eta\le\rho/2\), for every \(t\ge\rho\) we have
\(t-\eta\ge\rho/2\). From \eqref{eq:main-sandwich},
\[
\mathbb P(U_0^2\le t)
\le
\mathbb P(U_3^2\le t+\eta)+\mathbb P(\Omega_\star^c)
\]
and
\[
\mathbb P(U_0^2\le t)
\ge
\mathbb P(U_3^2\le t-\eta)-\mathbb P(\Omega_\star^c).
\]
Thus, by \eqref{eq:u3-u4-bound},
\[
\mathbb P(U_0^2\le t)
\le
\mathbb P(U_4^2\le t+\eta)+\varepsilon_G+\mathbb P(\Omega_\star^c),
\]
and
\[
\mathbb P(U_0^2\le t)
\ge
\mathbb P(U_4^2\le t-\eta)-\varepsilon_G-\mathbb P(\Omega_\star^c).
\]
Therefore
\[
\sup_{t\ge\rho} \left|
\mathbb P(U_0^2\le t)-\mathbb P(U_4^2\le t)
\right|
\le
\varepsilon_G+\mathbb P(\Omega_\star^c)+
\mathcal A_4(\eta),
\]
where
\[
\mathcal A_4(\eta)
:=
\sup_{u\ge\rho/2}
\mathbb P\left(u<U_4^2\le u+\eta\right).
\]
By Corollary~\ref{cor:projected_mixed_norm_ac}, applied with restricted
threshold \(\rho/2\),
\[
\mathcal A_4(\eta)
\le
C x^{-3/2}
\left(\eta+\frac1n\right)
\log^{1/2+3\kexp}(nK)\log^{3\kexp}n.
\]
From \eqref{eq:x-eta-total-bound}, we have the bound
$x^{-3/2}\eta\le C\Delta_1\log^4(nK)$.
Moreover, since \(M,q\ge1\), we have
\[
x^{-3/2} \frac1n
\le 
x^{-3/2}n^{-1/6} 
\le
x^{-3/2}\left(\frac{M^4q^2}{n}\right)^{1/6}
\le 
\Delta_1
\le \Delta_1 \log^2 (nK).
\]
Thus, since $\log(n)\le \log(nK)$, we obtain that 
\[
\mathcal A_4(\eta)
\le
C\Delta_1\log^{9/2+6\kexp}(nK).
\]
Finally, \(7n^{-1}\le C\Delta_1\), so
\[
\mathbb P(\Omega_\star^c)
\le
C\Delta_1+2\frac nq\beta_r.
\]
Combining the preceding bounds and taking the supremum over \(t\ge\rho\),
we obtain
\[
d_{\mathrm K}^{(\rho)}(U_0^2,U_4^2)
\le
C\Delta_1\log^{6\kexp+9/2}(nK)
+
2\frac nq\beta_r,
\]
and the desired conclusion follows with
$c_1=6\kexp+9/2$.
\end{proof}

\begin{proof}[Proof of Theorem~\ref{Thm_2}]
Throughout the proof, constants may change from line to line but depend only on the fixed constants appearing in Assumptions~\ref{cond:stationary} and~\ref{cond:scores}, and on \(\rho\). Let
\begin{align*}
Z_{M,q,x}
&:=
\max_{k\in[K]}\max_{s\in S_x}
\|G_{sk}^M\|^2_{\check M_k} \\
U_0
&:=
\max_{k\in[K]}
\sup_{u\in[0,1]}
\|\pi_k\mathbb B_{\mathcal K}(u)\|^2_{\Hb_k},\\
U_1
&:=
\max_{k\in[K]}
\max_{s\in[n]}
\|\pi_k\mathbb B_{\mathcal K}(s/n)\|^2_{\Hb_k},\\
U_2
&:=
\max_{k\in[K]}
\max_{s\in S_x}
\|\pi_k\mathbb B_{\mathcal K}(s/n)\|^2_{\Hb_k},\\
U_3
&:=
\max_{k\in[K]}
\max_{s\in S_x}
\|(P_{kM}\circ\pi_k)\mathbb B_{\mathcal K}(s/n)\|^2_{\Hb_k},
\end{align*}
and note that we need to bound the restricted Kolmogorov distance between $U_0$ and $Z_{M,q,x}$.
We have the almost sure inequalities
\[
U_3\le U_2\le U_1\le U_0,
\]
where the first one holds since $P_{kM}$ is an orthogonal projection.  

We start by showing that there is a constant $A_1$ depending only on $C_U$ and $\gamma_U$ and an event $\Omega_1$ satisfying $\Prob(\Omega_1)\geq 1-4/n$ such that on $\Omega_1$ it holds that
\begin{align} \label{eq:inequality-u0-u1}
U_0\leq U_1 + \eta_1, \qquad \eta_1:=A_1\frac{\log(nK)}{\sqrt{n}}.
\end{align}
Indeed, Lemma~\ref{lem:cts_to_discrete_hilbert_bridge_high_prob} gives on an event $\Omega_{11}$ satisfying $\Prob(\Omega_{11})\geq 1-2/n$ the inequality
\[
U_0 \leq U_1+ 2\delta_1\sqrt{U_1}+\delta^2_1,
\]
for 
\[
\delta_1
:=
D_1
\sqrt{
\frac{
C_{\mathcal K}\log(nK)
}{n}
},
\]
where $D_1$ is an absolute constant. Then by Lemma~\ref{lem:psi_2_controlled_by_L_2} and the union bound it is seen that on an event $\Omega_{12}$ satisfying $\Prob(\Omega_{12})\geq 1-2/n$ it holds that for an absolute constant $D_2$ that 
\[
\sqrt{U_1}\leq D_2 \sqrt{C_{\mathcal K}\log(nK)}
\]
Therefore on $\Omega_1=\Omega_{11}\cap \Omega_{12}$ it holds for an absolute constant $D_3$   that 
\begin{align*}
U_0& \leq U_1+ 2D_1
\sqrt{
\frac{
C_{\mathcal K}\log(nK)
}{n}
}D_2\sqrt{\log(nK)C_{\mathcal K}}+\Big(D_1
\sqrt{
\frac{
C_{\mathcal K}\log(nK)
}{n}
}\Big)^2\\
& \leq  U_1 + D_3(C_{\mathcal K}\vee 1) \Big[ 
\frac{\log(nK)}{\sqrt{n}}+ \frac{\log(nK)}{n}
\Big]
\leq 
U_1 + 2D_3(C_{\mathcal K}\vee 1)\frac{\log(nK)}{\sqrt{n}}. 
\end{align*}
Then set $A_1=2D_3(C_{\mathcal K}\vee 1)$ and note that$(C_{\mathcal K}\vee 1)$ may be bounded in terms of $C_U$ and $\gamma_U$ only, as the proof of Lemma~\ref{lem:interior_supremum_to_projected_interior_supremum_high_prob}. This yields \eqref{eq:inequality-u0-u1}.

We now consider the case that $\eta_1\geq \rho/4$. A straightforward calculation, using $q \wedge m \leq n$, $x\leq 1$ and $M\geq 1$, shows that $\eta_1 \leq A_1\Delta_2\log(nK)$. Therefore,
\begin{align*}
d_{\mathrm{K}}^{(\rho)}\left(Z_{M,q,x},
U_0
\right)
\leq
1\leq (4/\rho)A_1\Delta_2\log(nK).
\end{align*}
Thus in this case the bound we seek to prove holds by setting $C_2=(4/\rho)A_1$ and $c_2=1$. In the remainder of the proof we consider the case $\eta_1< \rho/4$.

Define the boundary maximum
\[
U_{\partial,x}
:=
\max_{k\in[K]}
\max_{s\in[n]\setminus S_x}
\|\pi_k\mathbb B_{\mathcal K}(s/n)\|^2_{\Hb_k}.
\]
Choose $C_X$ in the formulation of the Theorem to satisfy 
$C_X\geq (4/\rho)C_\partial C_{\mathcal{K}}$,
where $C_\partial$ is the absolute constant of Lemma~\ref{lem:supremum_to_interior_supremum_high_prob}. Then by \eqref{eq:side-conditions}, it holds that 
\[
x
\leq
\frac{1}
{C_X{[\log(nK)\log(n)]}^{4}} <
\frac{\rho/4}
{C_\partial C_{\mathcal K}\log(nK)}.
\]
Therefore, by Lemma~\ref{lem:supremum_to_interior_supremum_high_prob}, 
there exists an event $\Omega_2$ satisfying $\Prob(\Omega_2)\geq 1-2/n$ such that
\[
U_{\partial,x} \leq \rho/4.
\]
on $\Omega_2$.
By Lemma~\ref{lem:interior_supremum_to_projected_interior_supremum_high_prob} there is a constant $A_3$ and an event $\Omega_3$ satisfying $\Prob(\Omega_3)\geq 1-2/n$ such that on $\Omega_3$
\[
U_2\leq U_3 + \eta_3, \qquad \eta_3:=A_3\log(nK)M^{1-2\gamma_U}.
\]
Combining the previous bounds, we find that
\[
\Omega
:=
\{U_0\le U_1+\eta_1\}
\cap
\{U_{\partial,x}\le\rho/4\}
\cap
\{U_2\le U_3+\eta_3\}
\]
satisfies
\[
\Prob(\Omega^c)\le \frac8n.
\]
Fix \(t\ge\rho\). Since \(U_3\le U_0\) almost surely,
\[
0
\le
\Prob(U_3\le t)-\Prob(U_0\le t)
=
\Prob(U_3\le t<U_0).
\]
On the event
\[
\{U_3\le t<U_0\}\cap\Omega,
\]
we have \(U_0>t\), and therefore
\[
U_1
\ge
U_0-\eta_1
>
t-\eta_1
\ge
\rho-\eta_1>3\rho/4.
\]
Since \(U_{\partial,x}\le \rho/4\) on \(\Omega\), the maximum defining \(U_1\) must be attained over \(S_x\). Hence \(U_1=U_2\) on this event and so it holds on $\{U_3\le t<U_0\}\cap\Omega$ that
\[
U_0-\eta_1 -\eta_3 \leq U_3 \leq U_0.
\]
Consequently,
\[
\{U_3\le t<U_0\}\cap\Omega
\subseteq
\left\{
t-\eta_1-\eta_3
<
U_3
\le t
\right\},
\]
and so
\[
\Prob(U_3\le t)-\Prob(U_0\le t) \leq \frac 8n + \Prob(t-\eta_1-\eta_3
<
U_3
\le t).
\]
As we did with $\eta_1$ above, it suffices to consider the case of $\eta_3 <\rho/4$; in the complementary case the constants $C_2$ and $c_2$ may be chosen large enough so that the bound we seek to prove holds.  

Then $\eta_1+\eta_3 <\rho/2$ and so $t-\eta_1-\eta_3>\rho/2$. Write
\[
\Delta_{GA}(s)=d_K^{(s)}\big( U_3, Z_{M,q,x} \big) = \sup_{y\geq s}\left|
\Prob(U_3\le y)
-
\Prob(Z_{M,q,x}\le y)
\right|,
\]
and let 
\[
\Delta_{AC}(s)=\sup_{y\geq s}\Prob\left(y-\eta_1-\eta_3 <Z_{M,q,x} \leq y
\right).
\]
It then holds that 
\[
\Prob(U_3\le t)-\Prob(U_0\le t) \leq \frac 8n + 2\Delta_{GA}(\rho/2)+\Delta_{AC}(\rho),
\]
which in turn yields
\[
\big|\Prob(Z_{M,q,x}\le t)-\Prob(U_0\le t)\big| \leq \frac 8n + 3\Delta_{GA}(\rho/2)+\Delta_{AC}(\rho).
\]
By Corollary~\ref{cor:projected_mixed_norm_kolmogorov}, applied with lower threshold \(\rho/2\), there exists a constant $c_{GA} > 7/6$ and a constant $C$ such that 
\[
3\Delta_{GA}(\rho/2) \leq 
C
x^{-1}
M^{2/3}
(m\wedge q)^{-1/3}
\log^{c_{GA}}(nK).
\]
Corollary~\ref{cor:projected_mixed_norm_ac} gives for a constant $c_{\mathrm{AC}}>1/2$
\[
\Delta_{AC}(s)
\le
C
x^{-3/2}
\left(\eta_1+\eta_3+\frac1n\right)
\log^{c_{\mathrm{AC}}}(nK).
\]
Recall that
\[
    \eta_1=A_1\frac{\log(nK)}{\sqrt n},
    \qquad
    \eta_3=A_3\log(nK)M^{1-2\gamma_U}.
\]
Hence, for every \(t\ge \rho\),
\[
\begin{aligned}
\left|
\Prob(Z_{M,q,x}\le t)-\Prob(U_0\le t)
\right|
\le
\frac8n
&+
C x^{-1}
M^{2/3}(m\wedge q)^{-1/3}
\log^{c_{GA}}(nK)
\\
&+
C x^{-3/2}
\left(
    \frac{\log(nK)}{\sqrt n}
    +
    \log(nK)M^{1-2\gamma_U}
    +
    \frac1n
\right)
\log^{c_{\mathrm{AC}}}(nK),
\end{aligned}
\]
where we have absorbed \(A_1\) and \(A_3\) into the generic constant \(C\).
Choose \(c_\star\ge \max\{c_{GA},c_{\mathrm{AC}}+1\}\). Increasing \(C\) if
necessary, and using \(\log(nK)\ge 1\), we obtain
\[
\begin{aligned}
\left|
\Prob(Z_{M,q,x}\le t)-\Prob(U_0\le t)
\right|
\le
\frac8n
&+
C x^{-1}
M^{2/3}(m\wedge q)^{-1/3}
\log^{c_\star}(nK)
\\
&+
C x^{-3/2}
\left(
    \frac1{\sqrt n}
    +
    M^{1-2\gamma_U}
    +
    \frac1n
\right)
\log^{c_\star}(nK).
\end{aligned}
\]
Continuing, since \(m\wedge q\le n\),
\(M\ge1\), and \(x\le1\),
\[
\frac1n
\le
\frac1{\sqrt n}
\le
(m\wedge q)^{-1/3}
\le
M^{2/3}(m\wedge q)^{-1/3}
\le
x^{-3/2}M^{2/3}(m\wedge q)^{-1/3}.
\]
Moreover, \(x\le1\) implies \(x^{-1}\le x^{-3/2}\). Therefore, after
increasing \(C\) if necessary,
the upper bound in the penultimate display
is bounded by
\[
Cx^{-3/2}
\left[
\left(\frac{M^2}{q\wedge m}\right)^{1/3}
+
M^{1-2\gamma_U}
\right] \log^{c_\star}(nK)
= 
C\Delta_2\log^{c_\star}(nK)
\]
The assertion follows since $t \ge \rho$ was arbitrary.
\end{proof}

\begin{proof}[Proof of Theorem~\ref{thm:conditional_gaussian_subset_master}]
Set $c_3$ as
$
c_3:=6+6\kexp,
$
where $\kexp = \gamma_L / (2\gamma_U-1)$.
We may assume throughout that
\begin{equation}
\Delta_{3,\mathcal J}^{\mathrm{ad}}
\log^{c_3}(nKM)
\le1.
\label{thm_3_constraint_1}
\end{equation}
Indeed, otherwise \eqref{eq:subset-conditional-gaussian-approximation} is immediate, after taking \(C_3\ge1\).
The proof is based on comparing the following six random variables:
\begin{align*}
U_0
&:=
\max_{k\in\mathcal J}\max_{s\in[n]}
\|C_k^*(s)\|_{\Hb_k},
\\
U_1
&:=
\max_{k\in\mathcal J}\max_{s\in[n]}
\|C_k^{O*}(s)\|_{\Hb_k},
\\
U_2
&:=
\max_{k\in\mathcal J}\max_{s\in[n]}
\|\widetilde C_k^{O*}(s)\|_{\Hb_k},
\\
U_3
&:=
\max_{k\in\mathcal J}\max_{s\in S_x}
\|\widetilde C_k^{O*}(s)\|_{\Hb_k},
\\
U_4
&:=
\max_{k\in\mathcal J}\max_{s\in S_x}
\|P_{kM}\widetilde C_k^{O*}(s)\|_{\Hb_k},
\\
U_5
&:=
\max_{k\in\mathcal J}\max_{s\in S_x}
\|G_{sk}^M\|_{2}.
\end{align*}
with $C_k^*$, $C_k^{O*}$ and $\widetilde C_k^{O*}$ defined in~\eqref{eq:Ckstar}, \eqref{def:oracle-ck} and \eqref{def:berbee-coupled-block-oracle}, respectively. 
Note that the two distribution functions in
\eqref{eq:subset-conditional-gaussian-approximation} are those of
\(U_0^2\), conditionally on \(\mathcal F_n\), and \(U_5^2\),
respectively. 
For $u \in \R$, let
\[
F_5(u) = \Prob(U_5^2 \le u), \qquad F_j(u) = \Prob(U_j^2 \le u \mid \mathcal G_n), \qquad j \in \{0,1,2,3,4\},
\]
where $\mathcal G_n := \mathcal F_n\vee\widetilde{\mathcal F}_n$. 
We will show below that there exists an event $\Omega_{04}$
with 
\begin{align}
\mathbb P(\Omega_{04}^c)
\le
\frac5n
+2\left\lceil\frac n\ell\right\rceil\beta_\ell
+\frac nq\beta_r.
\label{eq:subset-proof-bootstrap-event-new}
\end{align}
and a constant $C_{04}>0$ such that with $\delta := C_{04} x^{3/2}\Delta_{3,\mathcal J}^{\mathrm{ad}}\log^5(nMK)$, we have the following: if $\delta \le \rho/2$, then,
on $\Omega_{04}$
\begin{align}
\forall t \ge \rho: \qquad 
F_4(t-\delta)
-\frac{10}{n}
\le
F_0(t)
\le
F_4(t+\delta)
+\frac{10}n
\label{eq:subset-proof-cdf-sandwich-new}.
\end{align}
Furthermore, we will show that there is an event $\Omega_{45}$ satisfying $\Omega_{45}\subset \Omega_{04}$ and 
\begin{align}
\label{eq:omega45}
\mathbb P(\Omega_{45}^c)\leq \frac6n
+2\left\lceil\frac n\ell\right\rceil\beta_\ell
+\frac nq\beta_r
\end{align}
and a constant $C_{45}>0$ such that on $\Omega_{45}$ we have the following:
\begin{align}\label{eq:Thm_3_gaussian_approx}
&\sup_{u \ge \rho/2} \big| F_4(u) - F_5(u)| \le C_{45}\Delta_{3,\mathcal J}^{\mathrm{ad}}\log^{c_3}(KMn).
\end{align}
Finally, we will show that there is a constant $C_5>0$ such that, with $\varepsilon
:=
C_5
\Delta_{3,\mathcal J}^{\mathrm{ad}}
\log^{c_3}(nKM)$,
\begin{equation}\label{Thm_3:main_ac}
\forall t \ge \rho: \qquad
F_5(t)-\varepsilon
\le
F_5(t-\delta)
\le
F_5(t)
\le
F_5(t+\delta)
\le
F_5(t)+\varepsilon.
\end{equation}

With this three assertions at hand, we can now proof the theorem. 
We distinguish two cases. First, if $\delta >\rho/2$, then, taking $C_3\geq 2C_{04}/\rho$, the claimed upper bound in \eqref{eq:subset-conditional-gaussian-approximation} is trivial. Indeed, 
we then have 
\begin{align*}
C_3
\Delta_{3,\mathcal J}^{\mathrm{ad}}
\log^{c_3}(nKM)=(C_3/C_{04})x^{-3/2}\delta\log^{c_3 - 5}(KMn)\geq (2/\rho)\times (\rho/2)=1
\end{align*}
where we used that $x\leq 1$, that $c_3-5\geq 1$ and that $\log(KMn)\geq \log(8)\geq 1$, since $n\geq 8$ as verified as a first step in the paragraph justifying~\eqref{eq:subset-proof-cdf-sandwich-new}, below~\eqref{eq:subset-proof-main-cdf-bound}. 
Otherwise, if \(\delta\le\rho/2\), fix \(t\ge\rho\) and note that
\[
t-\delta\ge\frac\rho2,
\qquad
t+\delta\ge\frac\rho2.
\]
Consequently, 
\eqref{eq:Thm_3_gaussian_approx} evaluated at \(u=t\pm\delta\), gives, on \(\Omega_{45}\),
\begin{align}
F_4(t-\delta)
&\ge
F_5(t-\delta)
-
C_2\Delta_{3,\mathcal J}^{\mathrm{ad}}
\log^{c_3}(nKM),
\label{eq:subset-proof-lower-Gaussian-at-shift}
\\
F_4(t+\delta)
&\le
F_5(t+\delta)
+
C_2\Delta_{3,\mathcal J}^{\mathrm{ad}}
\log^{c_3}(nKM).
\label{eq:subset-proof-upper-Gaussian-at-shift}
\end{align}
Consequently, successively applying
\eqref{eq:subset-proof-cdf-sandwich-new},
\eqref{eq:subset-proof-upper-Gaussian-at-shift}, and
\eqref{Thm_3:main_ac} gives
\begin{align*}
F_0(t)
&\le
F_4(t+\delta)+\frac{10}{n}
\\
&\le
F_5(t+\delta)
+
C_{45}\Delta_{3,\mathcal J}^{\mathrm{ad}}
\log^{c_3}(nKM)
+\frac{10}{n}
\\
&\le
F_5(t)
+
\varepsilon
+
C_{45}\Delta_{3,\mathcal J}^{\mathrm{ad}}
\log^{c_3}(nKM)
+\frac{10}{n}
\\
&=
F_5(t)
+
\left(C_5+C_{45}\right)
\Delta_{3,\mathcal J}^{\mathrm{ad}}
\log^{c_3}(nKM)
+\frac{10}{n}.
\end{align*}
Likewise, 
for the lower bound,
\eqref{eq:subset-proof-cdf-sandwich-new},
\eqref{eq:subset-proof-lower-Gaussian-at-shift}, and
\eqref{Thm_3:main_ac} give
\begin{align*}
F_0(t)
&\ge
F_4(t-\delta)-\frac{10}{n}
\\
&\ge
F_5(t-\delta)
-
C_{45}\Delta_{3,\mathcal J}^{\mathrm{ad}}
\log^{c_3}(nKM)
-\frac{10}{n}
\\
&\ge
F_5(t)
-
\varepsilon
-
C_{45}\Delta_{3,\mathcal J}^{\mathrm{ad}}
\log^{c_3}(nKM)
-\frac{10}{n}
\\
&=
F_5(t)
-
\left(C_5+C_{45}\right)
\Delta_{3,\mathcal J}^{\mathrm{ad}}
\log^{c_3}(nKM)
-\frac{10}{n}.
\end{align*}
It is straightforward to see that $1/n\le \Delta_{3,\mathcal J}^{\mathrm{ad}} \log^{c_3}(nKM)$.
Therefore, on \(\Omega_{45}\),
\begin{equation}
\sup_{t\ge\rho}
\left|
F_0(t)-F_5(t)
\right|
\le
C_{3}
\Delta_{3,\mathcal J}^{\mathrm{ad}}
\log^{c_3}(nKM),
\label{eq:subset-proof-main-cdf-bound}
\end{equation}
for any
\[
C_{3}
\ge 
C_5+C_{45}+10.
\]
This implies the assertion of the theorem after noting that, in the definition of $F_0$, we may replace the conditioning event $\mathcal G_n$ by $\mathcal F_n$ without changing the function (almost surely).
The proof is therefore complete once the three assertions above are justified.

\smallskip
\noindent
\emph{Justification of~\eqref{eq:subset-proof-cdf-sandwich-new}.}
We justify~\eqref{eq:subset-proof-cdf-sandwich-new} by
comparing \(U_0\) with \(U_4\) using a chain of supporting lemmas.
We choose \(C_X\) and \(C_M\) explicitly below. We begin by comparing
\(U_0^2\) and \(U_2^2\) on an event of high probability.
The tuning condition implies \(n>8\). Indeed, since
\(q,r\ge1\), \(g\ge0\), and \(\tau<1/2\),
$
2
\le q+r
\le g+q+r
\le\tau n/{2}
<  n/4,
$
and therefore \(n>8\).
The tuning condition implies \(n>8\). Hence
\[
\log\left(\frac{en}{\ell}\right)
\le 1+\log n
\le \frac32\log n
\le \frac32\log(nKM).
\]
Consequently, the rate in
Lemma~\ref{lem:residual_oracle_adaptive_subset} satisfies
\begin{align*}
R_{n,\mathcal J}^{\mathrm{ad}}(q,g,\ell,K)
&\le
\sqrt{\log(nKM)}
\bigg[
\log(nKM)
\left(\frac{q+g}{n}\right)^{1/2}
\left\{
1+\frac{3\ell\log(nKM)}{2\sqrt q}
\right\}
+B_{\mathcal J}
\bigg]
\\
&\le
2\mathcal A\log^{5/2}(nKM)
+B_{\mathcal J}\log^{1/2}(nKM)
\\
&\le
2R_{\mathcal J}\log^{5/2}(nKM),
\end{align*}
where
\begin{align}
\label{eq:rj}
\mathcal A
:=
\left(\frac{q+g}{n}\right)^{1/2}
\left(1\vee\frac{\ell}{\sqrt q}\right),
\qquad
B_{\mathcal J}
:=
\sqrt q\,\Delta_{\mathcal J,\mathrm{wk}}(g,\ell),
\qquad
R_{\mathcal J}:=\mathcal A+B_{\mathcal J},
\end{align}
and where we used
\(1+3\ell\log(nKM)/(2\sqrt q)\le2\log(nKM)\{ 1\vee (\ell/\sqrt q) \}\),
which follows from \(\log(nKM)\ge\log8>2\). Applying Lemma~\ref{lem:residual_oracle_adaptive_subset} and noting that
\[
|U_0 - U_1| \le \max_{k \in \mathcal J} \max_{s \in [n]}
\|C_k^*(s) - C_k^{O*}(s)\|_{\Hb_k} 
\]
gives an
event \(\Omega_{\mathrm{res}}\in\mathcal F_n\) satisfying
\begin{equation*}
\mathbb P(\Omega_{\mathrm{res}}^c)
\le
\frac1n
+2\left\lceil\frac n\ell\right\rceil\beta_\ell
\end{equation*}
such that, on \(\Omega_{\mathrm{res}}\),
\begin{equation}
\mathbb P\left(\mathcal A_{\mathrm{res}}
\,\middle|\,
\mathcal F_n
\right)
\ge 1-\frac2n,
\qquad
\mathcal A_{\mathrm{res}}
:=
\left\{
|U_0-U_1|
\le
2C_{\mathrm{res}}
\frac{R_{\mathcal J}}{\sqrt\tau}
\log^{5/2}(nKM)
\right\},
\label{eq:subset-proof-residual-comparison-new}
\end{equation}
Here \(C_{\mathrm{res}}\) is the constant in
Lemma~\ref{lem:residual_oracle_adaptive_subset}; it depends only on
\(D_\psi,C_\beta,c_\beta\). Because the multiplier sequence is independent of
\(\widetilde{\mathcal F}_n\) conditionally on \(\mathcal F_n\), and
the coefficients defining \(U_0\) and \(U_1\) are
\(\mathcal F_n\)-measurable, the same conditional bound holds with
\(\mathcal F_n\) replaced by
\(\mathcal G_n:=\mathcal F_n\vee\widetilde{\mathcal F}_n\). Lemma~\ref{lem:oracle_block_alignment}, applied to the coordinate
subset \(\mathcal J\), gives an event
\(\Omega_{\mathrm{align}}\in\mathcal F_n\) satisfying
\[
\mathbb P(\Omega_{\mathrm{align}}^c)\le\frac1n
\]
such that, on
\(\Omega_{\mathrm{align}}\cap\Omega_{\mathrm{big}}\) with $\Omega_{\mathrm{big}}$ from \eqref{def:Omega_big},
\begin{equation}
\mathbb P\left(\mathcal A_{\mathrm{align}}
\,\middle|\,
\mathcal G_n
\right)
\ge 1-\frac2n,
\qquad
\mathcal A_{\mathrm{align}}
:=
\left\{
|U_1-U_2|
\le
C_{\mathrm{align}}
\frac q{\sqrt n}
\log^{3/2}(nK)
\right\},
\label{eq:subset-proof-alignment-comparison}
\end{equation}
The constant \(C_{\mathrm{align}}\) depends only on \(D_\psi\). 
Finally, let \(C_{\mathrm{size}}\) and
\(\Omega_{\mathrm{size}}^{O*}\in\widetilde{\mathcal F}_n\) be the
constant and event supplied by
Lemma~\ref{lem:coupled_block_oracle_size}, with $\mathbb P((\Omega_{\mathrm{size}}^{O*})^c)
\le1/n$. Then, on \(\Omega_{\mathrm{size}}^{O*}\), replacing the conditioning event $\widetilde{\mathcal F}_n$ by $\mathcal G_n$ again,
\begin{align}
\mathbb P(
\mathcal A_{\mathrm{size}} \mid 
\mathcal G_n)
\ge1-\frac2n,
\qquad
\mathcal A_{\mathrm{size}}
:=
\left\{
U_2
\le
C_{\mathrm{size}}
\left(1+\frac{q\log n}{\sqrt n}\right)
\log^{3/2}(nK)
\right\}.
\label{eq:subset-proof-size-control-raw}
\end{align}

Now define
\[
\Omega_{02}
:=
\Omega_{\mathrm{res}}
\cap\Omega_{\mathrm{align}}
\cap\Omega_{\mathrm{big}}
\cap\Omega_{\mathrm{size}}^{O*} \in \mathcal G_n.
\]
Using
\(\mathbb P(\Omega_{\mathrm{big}}^c)\le(n/q)\beta_r\), the union
bound gives
\begin{equation}
\mathbb P(\Omega_{02}^c)
\le
\frac3n
+2\left\lceil\frac n\ell\right\rceil\beta_\ell
+\frac nq\beta_r.
\label{eq:subset-proof-Omega12-probability}
\end{equation}
We will next show that there exists a constant $C_{\mathrm{sq}}$ such that, on \(\Omega_{02}\),
\begin{equation}
\mathbb P\left(
\mathcal A_{02}
\,\middle|\,
\mathcal G_n
\right)
\ge 1-\frac6n,
\qquad
\mathcal A_{02}
:=
\left\{
|U_0^2-U_2^2|
\le
C_{\mathrm{sq}}x^{3/2}
\Delta_{3,\mathcal J}^{\mathrm{ad}}\log^5(nKM)
\right\}.
\label{eq:subset-proof-squared-comparison-current}
\end{equation}
In view of \eqref{eq:subset-proof-residual-comparison-new}, \eqref{eq:subset-proof-alignment-comparison} and \eqref{eq:subset-proof-size-control-raw}, it is sufficient to show that 
\[
\widetilde{\mathcal A}_{02} := \mathcal A_{\mathrm{res}} \cap \mathcal A_{\mathrm{align}} \cap \mathcal A_{\mathrm{size}} 
\subset 
\mathcal A_{02}.
\]
For that purpose note that the constraint in \eqref{thm_3_constraint_1} from the beginning of this proof yields
\(\Delta_{3,\mathcal J}^{\mathrm{ad}}\le1\). In particular, by the definition
of \(\Delta_{3,\mathcal J}^{\mathrm{ad}}\) in \eqref{eq:def-delta-3-adaptive-subset} and using
\(M\ge1\) and \(x\le1\), 
\begin{equation}
\left(\frac{q^2}{n}\right)^{1/3}
\le \label{eq:q2-1}
x^{3/2}\Delta_{3,\mathcal J}^{\mathrm{ad}}
\le1.
\end{equation}
Consequently,
\begin{align}
\frac q{\sqrt n}
=
\left(\frac{q^2}{n}\right)^{1/2}
\le
\left(\frac{q^2}{n}\right)^{1/3}
\le
x^{3/2}\Delta_{3,\mathcal J}^{\mathrm{ad}}.
\label{eq:subset-proof-q2n-current}
\end{align}
Further, we have 
\[
\frac{R_{\mathcal J}}{\sqrt\tau}
\le
x^{3/2}\Delta_{3,\mathcal J}^{\mathrm{ad}},
\]
by the definition of $\mathcal R_{\mathcal J}$ and 
\(\Delta_{3,\mathcal J}^{\mathrm{ad}}\) in \eqref{eq:rj} and \eqref{eq:def-delta-3-adaptive-subset}, respectively. Finally, 
by \eqref{eq:subset-proof-q2n-current}, \(c_3\ge1\), and
\(\log n\le\log(nKM)\),
\[
\frac{q\log n}{\sqrt n}
\le
x^{3/2}\Delta_{3,\mathcal J}^{\mathrm{ad}}\log(nKM)
\le
\Delta_{3,\mathcal J}^{\mathrm{ad}}
\log^{c_3}(nKM)
\le1,
\]
with the last equality following from \eqref{thm_3_constraint_1}.
Overall, on $\widetilde{\mathcal A}_{02}$, 
\begin{align*}
|U_0-U_1| 
&\le 
2 C_{\mathrm{res}} x^{3/2} \Delta_{3,\mathcal J}^{\mathrm{ad}}\log^{5/2}(nKM),
\\
|U_1-U_2| 
&\le  
C_{\mathrm{align}} x^{3/2}\Delta_{3,\mathcal J}^{\mathrm{ad}}\log^{3/2}(nKM),
\\
|U_2 | 
&\le 2 C_{\mathrm{size}} \log^{3/2}(nKM).
\end{align*}
Combining the first two inequalities with the triangular inequality yields yields
\[
|U_0 - U_2| \le (2 C_{\mathrm{res}} + C_{\mathrm{align}}) x^{3/2} \Delta_{3,\mathcal J}^{\mathrm{ad}}\log^{5/2}(nKM)
= 
\tilde C x^{3/2} \Delta_{3,\mathcal J}^{\mathrm{ad}}\log^{5/2}(nKM),
\]
where $\tilde C =2 C_{\mathrm{res}} + C_{\mathrm{align}}$.
This in turn can be combined with the last inequality to obtain, on $\widetilde{\mathcal A}_{02}$, 
\begin{align*}
|U_0^2-U_2^2|
&= \nonumber
|U_0-U_2|(U_0+U_2)
\\
&\le
|U_0-U_2| \big( |U_0-U_2|+2U_2 \big)
\\&=|U_0-U_2|^2 + 2U_2|U_0-U_2| 
\\
&\le 
\tilde C^2  x^3 (\Delta_{3,\mathcal J}^{\mathrm{ad}})^2 \log^{5}(nKM) + 4 C_{\mathrm{size}}\tilde C x^{3/2} \Delta_{3,\mathcal J}^{\mathrm{ad}}\log^{4}(nKM)
\\
&\le 
C_{\mathrm{sq}}
x^{3/2}\Delta_{3,\mathcal J}^{\mathrm{ad}}
\log^5(nKM),
\label{eq:subset-proof-squared-difference-bound}
\end{align*}
where
\[
C_{\mathrm{sq}}
:=
\tilde C^2 + 
4\tilde CC_{\mathrm{size}}.
\]
and where have used that $x \le 1$ and $\Delta_{3,\mathcal J}^{\mathrm{ad}} \le 1$. This proves \eqref{eq:subset-proof-squared-comparison-current}.

We proceed by comparing $U_2^2$ and $U_3^2$, for which we use the decomposition
\begin{equation}
U_2^2=U_3^2\vee U_{\partial}^2,
\qquad U_{\partial}^2
:=
\max_{k\in\mathcal J}
\max_{s\in[n]\setminus S_x}
\|\widetilde C_k^{O*}(s)\|_{\Hb_k}^2.
\label{eq:subset-proof-boundary-decomposition-current}
\end{equation}
Let \(C_{\partial}\) be the absolute constant from
Lemma~\ref{lem:oracle_coupled_boundary_restriction_squared}, and define $B_{\partial}
:=
C_{\partial}D_\psi^2(S_{\beta,0}\vee1)$. The goal is to apply that lemma with $t=\rho/4$, for which we need to check the lemma's condition on $t$.
Choose $C_X$ from the formulation of the theorem we are proving as
\begin{equation*}
C_X
:=
1\vee\frac{4B_{\partial}}{\rho}.
\end{equation*}
Condition \textnormal{(X)} from \eqref{eq:side-conditions} then gives
\begin{align*}
B_{\partial}x\log^3(nK)
&\le
\frac{B_{\partial}\log^3(nK)}
{C_X\{\log(nK)\log n\}^4}
=
\frac{B_{\partial}}
{C_X\log(nK)\log^4 n}
\le
\frac{B_{\partial}}{C_X}
\le\frac\rho4,
\end{align*}
where we used \(n>8\). If $B_{\partial}
\{
\Delta_{3,\mathcal J}^{\mathrm{ad}}
\log^{c_3}(nKM)
\}^{3}
>
\rho/4,$
then the asserted inequality in \eqref{eq:subset-conditional-gaussian-approximation} is elementary after increasing the
final constant so that
\(C_3\ge(4B_{\partial}/\rho)^{1/3}\). We may therefore assume the
reverse inequality. Since \(c_3\ge2\) and by the definition of \eqref{eq:def-delta-3-adaptive-subset}, we get
\begin{align*}
B_{\partial}
\frac{q^2\log^2 n}{n}\log^3(nK)
&\le
B_{\partial}\frac{q^2}{n}\log^5(nKM)
\le
B_{\partial}
\left\{
\Delta_{3,\mathcal J}^{\mathrm{ad}}
\log^{c_3}(nKM)
\right\}^{3}
\le\frac\rho4.
\end{align*}
Thus, combining the previous two displays
\[
\frac\rho4 \ge B_{\partial}\log^3(nK)  
\left(
x\vee\frac{q^2\log^2 n}{n}
\right)
=
C_{\partial}D_\psi^2(S_{\beta,0}\vee1)\log^3(nK)  
\left(
x\vee\frac{q^2\log^2 n}{n}
\right).
\]
Hence, the condition on $t=\rho/4$ from Lemma~\ref{lem:oracle_coupled_boundary_restriction_squared} is satisfied, and there exists an
even \(\Omega_{\partial}^{O*}\in\widetilde{\mathcal F}_n\) with $\mathbb P((\Omega_{\partial}^{O*})^c) \le 1/n$
such that, on \(\Omega_{\partial}^{O*}\),
\begin{equation}
\mathbb P\left(
\mathcal A_{\partial}
\,\middle|\,
\mathcal G_n
\right)
\ge1-\frac2n,
\qquad
\mathcal A_{\partial}
:=
\left\{U_{\partial}^2\le\frac\rho4\right\}.
\label{eq:subset-proof-boundary-control-current}
\end{equation}
Here we used \(\mathcal J\subseteq[K]\) and the independence of the
Gaussian multipliers to replace conditioning on
\(\widetilde{\mathcal F}_n\) by conditioning on \(\mathcal G_n\). 

We next compare $U_3^2$ and $U_4^2$, for which we use Lemma~\ref{lem:oracle_coupled_projection_error_2}. We obtain that 
there exists an event
$
\Omega_{\mathrm{proj}}^{O*}\in\widetilde{\mathcal F}_n
$, potentially depending on $M$ and $S$,
satisfying
$
\mathbb P((\Omega_{\mathrm{proj}}^{O*})^c)\le 1/n 
$
such that, on \(\Omega_{\mathrm{proj}}^{O*}\),
\begin{align}
\mathbb P\Big(&
U_3^2
\le
U_4^2 
+\Delta_{\mathrm{proj}}^O
\,\Bigm|\,
\widetilde{\mathcal F}_n
\Big)
\ge1-\frac2n,
\label{eq:subset-proof-projection-control-raw}
\end{align}
where
\[
\Delta_{\mathrm{proj}}^O
:=
C_{\mathrm{proj}}\log^3(nK)
\left[
M^{1-2\gamma_U}
+\frac1q
+\frac{q^2\log^2n}{n}
\right].
\]
We proceed by upper bounding $\Delta_{\mathrm{proj}}^O$, for which we will make use of \textnormal{(M)} from \eqref{eq:side-conditions}, with a choice of $C_M$ that will be helpful for proving \eqref{eq:Thm_3_gaussian_approx} below.  Specifically, we choose
\begin{equation}
C_M
:=
1\vee\frac{2C_{\mathrm{cov}}(D_\psi^2S_{\beta,1}\vee C_{\mathcal T})}{C_L^2}.
\label{eq:subset-proof-CM-choice-current}
\end{equation}
where $\mathcal C_{\mathcal T}$ is from \eqref{eq:C_t} and where \(C_{\mathrm{cov}}>0\) is the absolute constant from
Lemma~\ref{lem:bound-on-operator-norm-covariances}. 

Condition \textnormal{(M)} from \eqref{eq:side-conditions} also gives 
\[
\frac1q
\le
\frac1{q\wedge m}
\le
\frac{x}{C_MM^{2\gamma_L}}
\le
\frac1{C_M}M^{1-2\gamma_U},
\]
where we used \(x\le1\), \(M\ge1\), and
\(2\gamma_L\ge2\gamma_U-1\). Moreover,
\eqref{eq:q2-1}, \(M\ge1\), and
\(\log n\le\log(nKM)\) imply
\[
\frac{q^2\log^2 n}{n}
\le
M^{2/3}
\left(\frac{q^2}{n}\right)^{1/3}
\log^2(nKM).
\]
Consequently, by the definition of $\Delta_{3,\mathcal J}^{\mathrm{ad}}$ in \eqref{eq:def-delta-3-adaptive-subset}, 
\[
\Delta_{\mathrm{proj}}^O
\le
C_{\mathrm P}x^{3/2}
\Delta_{3,\mathcal J}^{\mathrm{ad}}\log^5(nKM),
\qquad
C_{\mathrm P}
:=
C_{\mathrm{proj}}\left(2+\frac1{C_M}\right).
\]
Therefore, from \eqref{eq:subset-proof-projection-control-raw},  on \(\Omega_{\mathrm{proj},\mathcal J}^{O*}\),
\begin{equation}
\mathbb P\left(
\mathcal A_{34}
\,\middle|\,
\mathcal G_n
\right)
\ge1-\frac2n,
\qquad
\mathcal A_{34}:=
\left\{
U_3^2
\le
U_4^2+C_{\mathrm P}x^{3/2}
\Delta_{3,\mathcal J}^{\mathrm{ad}}\log^5(nKM)
\right\}
\label{eq:subset-proof-projection-control-current}
\end{equation}

We finally justify~\eqref{eq:subset-proof-cdf-sandwich-new} using \eqref{eq:subset-proof-squared-comparison-current}, \eqref{eq:subset-proof-boundary-control-current} and \eqref{eq:subset-proof-projection-control-current}. Define
\[
C_{04}:=C_{\mathrm{sq}}+C_{\mathrm P},
\qquad 
\Omega_{04}
:=
\Omega_{02}
\cap\Omega_{\partial}^{O*}
\cap\Omega_{\mathrm{proj},\mathcal J}^{O*}
\]
and recall $\delta = C_{04} x^{3/2}\Delta_{3,\mathcal J}^{\mathrm{ad}} \log^5(nKM)$.
By \eqref{eq:subset-proof-Omega12-probability} and the two preceding
event bounds,
\begin{equation}
\mathbb P(\Omega_{04}^c)
\le
\frac5n
+2\left\lceil\frac n\ell\right\rceil\beta_\ell
+\frac nq\beta_r
\label{eq:subset-proof-Omega14-probability}
\end{equation}
as required in \eqref{eq:subset-proof-bootstrap-event-new}.
Fix \(t\ge\rho\). 
On
\(\mathcal A_{02}\cap\mathcal A_{\partial}\cap\mathcal A_{34}\),
the inequality \(U_4^2\le t-\delta\) implies
\[
U_3^2
\le
t-
C_{\mathrm{sq}}x^{3/2}
\Delta_{3,\mathcal J}^{\mathrm{ad}}\log^5(nKM).
\]
Also, since \(t\ge\rho\) and \(\delta\le\rho/2\) and $C_{\mathrm{sq}} \le C_{04}$,
\[
U_{\partial}^2
\le\frac\rho4
\le t - \frac\rho2
\le
t-
C_{\mathrm{sq}}x^{3/2}
\Delta_{3,\mathcal J}^{\mathrm{ad}}\log^5(nKM).
\]
The decomposition
\eqref{eq:subset-proof-boundary-decomposition-current} therefore gives
\[
U_2^2
\le
t-
C_{\mathrm{sq}}x^{3/2}
\Delta_{3,\mathcal J}^{\mathrm{ad}}\log^5(nKM),
\]
and the definition of \(\mathcal A_{02}\) in \eqref{eq:subset-proof-squared-comparison-current} then yields \(U_0^2\le t\). Hence
\[
\{U_4^2\le t-\delta\}
\cap\mathcal A_{02}
\cap\mathcal A_{\partial}
\cap\mathcal A_{34}
\subseteq
\{U_0^2\le t\}.
\]
The conditional union bound gives
\begin{equation*}
F_0(t)
\ge
F_4(t-\delta)-\frac{10}{n},
\end{equation*}
which is the first inequality in \eqref{eq:subset-proof-cdf-sandwich-new}.
For the reverse comparison, \(U_4\le U_3\le U_2\) identically.
Thus, on \(\mathcal A_{02}\), the event \(U_0^2\le t\) implies
\[
U_4^2
\le U_2^2
\le
t+C_{\mathrm{sq}}x^{3/2}
\Delta_{3,\mathcal J}^{\mathrm{ad}}\log^5(nKM)
\le t+\delta.
\]
Using only the \(6/n\) failure probability of \(\mathcal A_{02}\),
we obtain
\begin{equation*}
F_0(t)
\le
F_4(t+\delta)+\frac6n \le F_4(t+\delta)+\frac{10}n,
\end{equation*}
which is the second inequality in \eqref{eq:subset-proof-cdf-sandwich-new}.

\smallskip
\noindent
\emph{Justification of~\eqref{eq:Thm_3_gaussian_approx}.}
Define
\[
\Omega_{45}
:=
\Omega_{04}
\cap
\Omega_{\mathrm{GB},\mathcal J}^{O*},
\]
where \(\Omega_{\mathrm{GB},\mathcal J}^{O*}\) is the event supplied by
Lemma~\ref{lem:projected_coupled_oracle_gaussian_comparison} for the
coordinate subset \(\mathcal J\). That lemma gives that 
$
\mathbb P((\Omega_{\mathrm{GB},\mathcal J}^{O*})^c)\leq 1/n
$
which combined with the bound on $\Omega_{04}$ from ~\eqref{eq:subset-proof-Omega14-probability}

\[
\mathbb P(\Omega_{45}^c)\leq\frac6n
+2\left\lceil\frac n\ell\right\rceil\beta_\ell
+\frac nq\beta_r
\]
as required in \eqref{eq:omega45}.

We now prove~\eqref{eq:Thm_3_gaussian_approx}. Condition
\textnormal{(M)} from \eqref{eq:side-conditions} and the choice of \(C_M\) from \eqref{eq:subset-proof-CM-choice-current} ensure that the hypothesis in \eqref{eq:M-condition-conditional-Gaussian-comparison} 
of Lemma~\ref{lem:projected_coupled_oracle_gaussian_comparison} is
satisfied. Moreover, by the definition of \(O_{kM}\),
\[
\|P_{kM}\widetilde C_k^{O*}(s)\|_{\Hb_k}
=
\|(O_{kM}\circ P_{kM})
\widetilde C_k^{O*}(s)\|_2.
\]
Thus the random variables compared in that lemma are precisely
\(U_4^2\) and \(U_5^2\). 
Since \(m\ge n/(4q)\),
\[
\frac1{\sqrt m}
\le
2\left(\frac qn\right)^{1/2},
\qquad
\frac{q\log^2n}{m}
\le
4\frac{q^2}{n}\log^2n.
\]
Using \((a+b)^{1/3}\le a^{1/3}+b^{1/3}\), we obtain
\begin{align*}
\left[
\frac1{\sqrt m}
+
\frac{q\log^2n}{m}
\right]^{1/3}
&\le
2^{1/3}\left(\frac qn\right)^{1/6}
+
4^{1/3}
\left(\frac{q^2}{n}\right)^{1/3}
\log^{2/3}n
\\
&\le
\left(2^{1/3}+4^{1/3}\right)
\left\{
\left(\frac qn\right)^{1/6}
\vee
\left(\frac{q^2}{n}\right)^{1/3}
\right\}
\log^{2/3}(nKM).
\end{align*}
Further, because \(x\le1\) we have \(x^{-1}\le x^{-3/2}\) and hence, by the
definition of \(\Delta_{3,\mathcal J}^{\mathrm{ad}}\) from in \eqref{eq:def-delta-3-adaptive-subset},
\[
x^{-1}M^{2/3}
\bigg\{
\left(\frac qn\right)^{1/6}
\vee
\left(\frac{q^2}{n}\right)^{1/3}
\bigg\}
\le
\Delta_{3,\mathcal J}^{\mathrm{ad}}.
\]
Substituting these estimates into the error bound in
Lemma~\ref{lem:projected_coupled_oracle_gaussian_comparison} gives,
on \(\Omega_{\mathrm{GB},\mathcal J}^{O*} \subset\Omega_{45} \),
\begin{align*}
\sup_{u\ge\rho/2}
\big|
F_4(u)
-
F_5(u)
\big|
&\le
C_{45}
\Delta_{3,\mathcal J}^{\mathrm{ad}}
\log^{5/2+6\kexp}(nKM)
\le
C_{45}
\Delta_{3,\mathcal J}^{\mathrm{ad}}
\log^{c_3}(nKM)
\end{align*}
where
\[
C_{45}
:=
\left(2^{1/3}+4^{1/3}\right)C_{\mathrm{GB}}.
\]
and where the final inequality follows from $c_3 = 6+6\kexp \ge 5/2 + \kexp$.
This proves \eqref{eq:Thm_3_gaussian_approx}.

\smallskip
\noindent
\emph{Justification of~\eqref{Thm_3:main_ac}.}
Corollary~\ref{cor:projected_mixed_norm_ac}, applied to the coordinate
subset indexed by \(\mathcal J\), gives
\begin{equation*}
\sup_{u\ge\rho/2}
\mathbb P\left(
\left|U_5^2-u\right|\le\delta
\right)
\le
C_{\mathrm{AC}}x^{-3/2}
\left(
\delta+\frac1n
\right)
\log^{1/2+6\kexp}(nKM).
\end{equation*}
Recall that $\delta
=
C_{04} x^{3/2}
\Delta_{3,\mathcal J}^{\mathrm{ad}}
\log^5(nKM)$ and that
\begin{align*}
\frac{x^{-3/2}}n
\le 
x^{-3/2}n^{-1/6}
\le
\left(\frac qn\right)^{1/6}
\le 
x^{-3/2}M^{2/3}
\le 
\Delta_{3,\mathcal J}^{\mathrm{ad}}
\end{align*}
by the definition of \(\Delta_{3,\mathcal J}^{\mathrm{ad}}\) in \eqref{eq:subset-conditional-gaussian-approximation}.
Consequently, 
\begin{align}
\sup_{u\ge\rho/2}
\mathbb P\left(
\left|U_5^2-u\right|\le\delta
\right)
&\le \nonumber
C_{\mathrm{AC}}
\left\{
C_{04}\Delta_{3,\mathcal J}^{\mathrm{ad}}
\log^5(nKM)
+
\Delta_{3,\mathcal J}^{\mathrm{ad}}
\right\}
\log^{1/2+6\kexp}(nKM)
\\ \nonumber
&\le
C_{\mathrm{AC}}(C_{04}+1)
\Delta_{3,\mathcal J}^{\mathrm{ad}}
\log^{11/2+6\kexp}(nKM)
\\
&\le \label{eq:subset-proof-U5-anticoncentration-rate}
C_5 \Delta_{3,\mathcal J}^{\mathrm{ad}}
\log^{c_3}(nKM),
\end{align}
where $C_5 := C_{\mathrm{AC}}(C_{04}+1)$ and where we used that $c_3=6+6\kexp \ge 11/2 + \kexp$.
Recall
\[
\varepsilon
:=
C_5
\Delta_{3,\mathcal J}^{\mathrm{ad}}
\log^{c_3}(nKM)
\]
and fix \(t\ge\rho\). Since \(t\ge\rho/2\),
\eqref{eq:subset-proof-U5-anticoncentration-rate}, evaluated at
\(u=t\), gives
$
\mathbb P\left(
\left|U_5^2-t\right|\le\delta
\right)
\le\varepsilon,
$
which implies \eqref{Thm_3:main_ac}. Indeed, the first inequality  follows from
\begin{align*}
F_5(t)-F_5(t-\delta)
&=
\mathbb P\left(
t-\delta<U_5^2\le t
\right)\le
\mathbb P\left(
\left|U_5^2-t\right|\le\delta
\right)
\le\varepsilon,
\end{align*}
and the last one follows similarly. 
\end{proof}

\subsection{Gaussian Comparison Bounds}
\label{sec:Gaussian_Comparison}

Given a covariance operator $\mathcal{K}:\mathcal H \to \mathcal H$ recall from~\eqref{def:Hilbert_Brownian_Bridge} that $\mathbb{B}_{\mathcal K}$ is a $\mathcal H$-valued Brownian Bridge with covariance $\mathcal{K}$.  Further, recall that for each $k\in [K]$ it holds that 
\[
\mathcal{K}_k=\pi_k\circ \mathcal{K}\circ \pi^{*}_k;
\]
this identity is discussed in the paragraph above~\eqref{def:C_Kappa}.
The notation $\mathcal{K}_k$ is introduced in the statement of Assumption~\ref{cond:scores} and each marginal covariance operator $\mathcal{K}_k$ is realized from $\mathcal {K}$ by conjugation with $\pi_k$. Recall from~\eqref{def:C_Kappa} and Definition~\ref{def:trace_norm} that
\begin{equation*}
C_{\mathcal K}:=\max_{k \in [K]}\big\|\mathcal{K}_k\|_{\Tr}=\max_{k \in [K]}\sum_{j=1}^{r_k}\lambda_{kj},
\end{equation*}
where as in the statement of Assumption~\ref{cond:scores} the collection of positive real numbers $\{\lambda_{kj}\}_{j=1}^{r_k}$ are the eigenvalues of $\mathcal{K}_k$ where $r_k:=\textrm{rank}(\mathcal{K}_k)\geq 1$. We assume throughout this section that $C_{\mathcal K}>0$ which excludes the case case that $\mathbb B_{\mathcal K}$ is almost surely equal to $0 \in \mathcal{H}$.  Eventually the results in this section will applied the global long run covariance operator $\mathcal{K}$ whose spectral structure is controlled in Assumption~\ref{cond:scores}, which also excludes this total degeneracy. Recall also the set $S_x=
\left\{
s\in[n]:
V(s/n)\ge x
\right\}$ defined in~\eqref{eq:definition-S_x}.
We define $S^c_x:=[n]\setminus S_x$. Recall also the constant $D_\psi$ of Assumption~\ref{cond:stationary}, and the serial-dependence coefficients $S_{\beta,j}$ from~\eqref{eq:def-S_beta}.

\begin{lemma}[High-probability discretization of the Hilbert-space Brownian bridge]
\label{lem:cts_to_discrete_hilbert_bridge_high_prob}
There exists an absolute constant \(C_1>0\) such that, for every
\(K\in\mathbb N_{\ge1}\) and \(n\in\mathbb N_{\ge3}\), defining
\[
\delta_1
:=
C_1
\sqrt{
\frac{
C_{\mathcal K}\log(nK)
}{n}
},
\]
we have
\[
\mathbb P\bigg(
\max_{k\in[K]}
\sup_{u\in[0,1]}
\|\pi_k\mathbb B_{\mathcal K}(u)\|_{\Hb_k}
>
\max_{k\in[K]}
\max_{s\in[n]}
\|\pi_k\mathbb B_{\mathcal K}(s/n)\|_{\Hb_k}
+
\delta_1
\bigg)
\le
\frac{2}{n}.
\]
Equivalently, with probability at least \(1-2/n\),
\[
0
\le
\max_{k\in[K]}
\sup_{u\in[0,1]}
\|\pi_k\mathbb B_{\mathcal K}(u)\|_{\Hb_k}
-
\max_{k\in[K]}
\max_{s\in[n]}
\|\pi_k\mathbb B_{\mathcal K}(s/n)\|_{\Hb_k}
\le
\delta_1.
\]
\end{lemma}

\begin{lemma}[High-probability control of the discrete boundary supremum]
\label{lem:supremum_to_interior_supremum_high_prob}
There exists an absolute constant \(C_2>0\) such that, for every
\(K\in\mathbb N_{\ge1}\), \(n\in\mathbb N_{\ge3}\), \(t>0\), and
\(x\in(0,3/16]\) satisfying
\[
x
\le
\frac{t}
{C_2C_{\mathcal K}\log(nK)},
\]
we have
\[
\mathbb P\left(
\max_{k\in[K]}
\max_{s\in[n]\setminus S_x}
\left\|
\pi_k\mathbb B_{\mathcal K}(s/n)
\right\|_{\Hb_k}^2
>
t
\right)
\le
\frac{2}{n}.
\]
\end{lemma}

\begin{lemma}[High-probability control of the projection error]
\label{lem:interior_supremum_to_projected_interior_supremum_high_prob}
Suppose Assumption~\ref{cond:scores} holds. There exists a constant
$
C_3=C_3(\gamma_U,C_U)>0
$
depending only on \(\gamma_U\) and \(C_U\) such that, for every
\(K\in\mathbb N_{\ge1}\), \(n\in\mathbb N_{\ge3}\),
\(M\in\mathbb N_{\ge1}\), and all  $S \subset [n]$,
defining
\[
\delta_2
:=
C_3\log(nK)M^{1-2\gamma_U},
\]
we have
\[
\mathbb P\left(
\max_{k\in[K]}
\max_{s\in S}
\left\|
(\operatorname{id}_{\Hb_k}-P_{kM})
\pi_k\mathbb B_{\mathcal K}(s/n)
\right\|_{\Hb_k}^2
>
\delta_2
\right)
\le
\frac{2}{n}.
\]
Moreover,
\[
\mathbb P\left(
\max_{k\in[K]}
\max_{s\in S}
\left\|
\pi_k\mathbb B_{\mathcal K}(s/n)
\right\|_{\Hb_k}^2
>
\max_{k\in[K]}
\max_{s\in S}
\left\|
(P_{kM}\circ\pi_k)
\mathbb B_{\mathcal K}(s/n)
\right\|_{\Hb_k}^2
+
\delta_2
\right)
\le
\frac{2}{n}.
\]
\end{lemma}

Recall from \eqref{eq:projection-pisigmak} that, for $\sigma \subset [n]$ and $\chgset \subset [K]$, we write
\[
\pi_{\sigma,\chgset}: \mathcal H^n \to \mathcal H_\chgset^\sigma, 
\qquad 
(x_1, \dots, x_n) \mapsto \big( (x_{s_1k})_{k \in \chgset}, \dots, (x_{s_{|\sigma|}k})_{k \in \chgset}\big), 
\]
where $\mathcal H_\chgset = \bigoplus_{k \in \chgset} \Hb_k$ and $ \mathcal H_\chgset^\sigma = \bigoplus_{s \in \sigma}  \mathcal H_\chgset$ are from \eqref{eq:product-spaces-with-indices}. For $M \in \N$, define 
\[
P_{M\sigma \chgset} := 
\bigoplus_{s \in \sigma} \Big(\bigoplus_{k \in \chgset} P_{kM}\Big)
:
\mathcal H_\chgset^\sigma
\to
\mathcal H_\chgset^\sigma
\]
with $P_{kM}:\Hb_k \to \Hb_k$ from \eqref{eq:def-PkM}.

\begin{lemma}\label{lem:bound-on-operator-norm-covariances}
There exists a universal constant \(C_4>0\) such that, for any $\sigma \subset [n]$, $\chgset \subset [K]$, $M\in \N$ and $q \in \N$ with $q \le n$, we have
\begin{align*}
&\Big\|
\Cov\Big(
(P_{M\sigma \chgset} \circ  \pi_{\sigma,\chgset }) G^{(n,q)}
\Big)
-
\Cov\Big(
( P_{M\sigma \chgset} \circ \pi_{\sigma,\chgset })\mathcal B^{(n)}
\Big)
\Big\|_{\op}
\\
&\hspace{7cm}
\le
C_4
\big(D_{\psi}^2S_{\beta,1}\vee C_{\mathcal K}\big)
|\chgset|\,|\sigma|\,
(m\wedge q)^{-1}.
\end{align*}
\end{lemma}

\begin{corollary}\label{lem:good_eigen_values_2}
Recall the variables
$
G^M_{sk}
:=
G^{(M,n,q)}_{sk}
=
(O_{kM}\circ P_{kM}\circ\pi_{sk}) G^{(n,q)} \in \R^{\check M_k}$
defined in~\eqref{eq:definition-GMsk}. Suppose that Assumption~\ref{cond:stationary} and Assumption~\ref{cond:scores}
hold, and let \(x\in(0,3/16]\). Suppose further that
\begin{align} 
\label{eq:bound-on-M}
M^{2\gamma_L}
\leq
\frac{
C_L^2 x (m\wedge q)
}{
2C_4\bigl(D_\psi^2S_{\beta,1}\vee C_{\mathcal K}\bigr)
},
\end{align}
where \(C_4\) is the universal constant from Lemma~\ref{lem:bound-on-operator-norm-covariances}.
Then, for each \((s,k)\in S_x\times[K]\) with $S_x$ from \eqref{eq:definition-S_x}, in the finite-dimensional sense of
Definition~\ref{def:covariance-decay}, it holds that 
\begin{align} \label{eq:cov-gmsk-decay}
\operatorname{Cov}(G^M_{sk})
\in
\operatorname{Decay}
\left(
\frac{xC_L^2}{2},
2C_U^2,
2\gamma_L,
2\gamma_U
\right).
\end{align}
\end{corollary}

The following follows directly from Corollary~\ref{lem:good_eigen_values_2} and Lemma~\ref{lem:acnew}.

\begin{corollary}[Anti-concentration for projected Gaussian mixed norms]
\label{cor:projected_mixed_norm_ac}
Suppose that Assumptions~\ref{cond:stationary} and~\ref{cond:scores} hold. Let \(x\in(0,3/16]\), and suppose that \eqref{eq:bound-on-M} is satisfied for this \(x\).  Let $\kexp:=\gamma_L /(2\gamma_U-1)$.
Then, for every \(\rho>0\), there exists a constant
\[
C=C(C_U,C_L,\gamma_U,\gamma_L,\rho)>0
\]
such that, for every \(\delta>0\),
\[
\sup_{t\ge \rho}
\Prob\bigg(
\Big|\max_{k\in[K]}\max_{s\in S_x}
\|G^M_{sk}\|_{2}-t\Big|\le \delta
\bigg)
\le
C
x^{-3/2}
\left(\delta+\frac1n\right)
\log^{1/2+3\kexp}(nK)
\log^{3\kexp}(n).
\]
The same upper bound holds with $
\|G^M_{sk}\|_{2}$ replaced by $
\|G^M_{sk}\|_{2}^2$ on the right-hand side.
\end{corollary}

\begin{corollary}[Non-uniform Kolmogorov comparison for projected mixed norms]
\label{cor:projected_mixed_norm_kolmogorov}
Suppose that Assumptions~\ref{cond:stationary} and~\ref{cond:scores} hold. Let \(x\in(0,3/16]\), and suppose that \eqref{eq:bound-on-M} is satisfied for this \(x\). 
Then, for every \(\rho>0\), there exists a constant
\[
C=C(D_\psi,C_U,C_L,\gamma_U,\gamma_L,\rho)>0
\]
such that, with $\kexp:=\gamma_L /(2\gamma_U-1)$,
\begin{align*}
&d_{\mathrm{K}}^{(\rho)}\left(
\max_{k\in[K]}\max_{s\in S_x}
\|G^M_{sk}\|_{2}^2
,
\max_{k\in[K]}\max_{s\in S_x}
\left\|
(P_{kM}\circ\pi_k)\mathbb B_{\mathcal K}(s/n)
\right\|_{\Hb_k}^2
\right)
\nonumber\\
&\hspace{6cm} \le
C
x^{-1}
M^{2/3}
\left(
\frac{
1
}{
m\wedge q
}
\right)^{1/3}
\log^{7/6+6 \kexp}(nK).
\end{align*}
\end{corollary}

\begin{proof}[Proof of Lemma~\ref{lem:cts_to_discrete_hilbert_bridge_high_prob}]
For $(x,y) \in I^2$ let $d(x,y):=|x-y|^{1/2}$. If $J \subset I$ is an interval write $\textrm{diam}(J):=d(\inf J, \sup J)$ for its length. For each $k \in [K]$ and each $(u,v) \in I^2$ let 
\begin{align}
\label{eq:Ink}
I_n^{(k)}(u,v)
:=
\sqrt{\frac 38}C_\mathcal{K}^{-1/2}\Big|\|\pi_k\mathbb{B}_{\mathcal K}(u)\|_{\Hb_k} - \|\pi_k\mathbb{B}_{\mathcal K}(v)\|_{\Hb_k}\Big|.
\end{align}
Furthermore, for $\delta \in (0,1)$,  let
\[
B(\delta)
:=
\max_{k \in [K]}\sup_{J \subset I, \textrm{diam}(J)\leq \delta}
\Big\|\sup_{(x,y) \in J^2}I_n^{(k)}(x,y)\Big\|_{\psi_2},
\]
where the supremum is over all intervals $J \subset I$ with $\textrm{diam}(J)\leq \delta$.
By the reverse triangle inequality for the Euclidean $L^\infty$ norm, it holds that
\begin{align*}
A
:\! &=
\max_{k \in [K]}\sup_{u \in I}\|\pi_k\mathbb{B}_{\mathcal K}(u)\|_{\Hb_k} - \max_{k \in [K]}\sup_{s \in [n]}\|\pi_k\mathbb{B}_{\mathcal K}(s/n)\|_{\Hb_k}
\\&=
\Big|\max_{k \in [K]}\sup_{u \in I}\|\pi_k\mathbb{B}_{\mathcal K}(u)\|_{\Hb_k} - \max_{k \in [K]}\sup_{s \in [n]}\|\pi_k\mathbb{B}_{\mathcal K}(s/n)\|_{\Hb_k}\Big| \\
&\le 
\max_{k \in [K]}\Big|\sup_{u \in I}\|\pi_k\mathbb{B}_{\mathcal K}(u)\|_{\Hb_k} - \sup_{s \in [n]}\|\pi_k\mathbb{B}_{\mathcal K}(s/n)\|_{\Hb_k}\Big|\\
&\leq 
\max_{k \in [K]}\max_{s \in [n]}\Big|\sup_{u\in [\frac {s-1}{n},\frac sn]}\|\pi_k\mathbb{B}_{\mathcal K}(u)\|_{\Hb_k} -\|\pi_k\mathbb{B}_{\mathcal K}(s/n)\|_{\Hb_k} \Big|\\
&\leq 
\sqrt{\frac 83}C_\mathcal{K}^{1/2} \max_{k \in [K]}\sup_{s \in [n]}\sup_{u \in [\frac {s-1}{n},\frac sn]}I_n^{(k)}(u,s/n). 
\end{align*}
We show below that show that \(B(\delta)\lesssim\delta\). Taking this bound for granted for now, we complete the proof. Since for every $s \in [n]$ it holds that  $\textrm{diam}([\frac {s-1}{n},\frac sn]) \leq n^{-1/2}$, we have, for any $k \in [K]$, 
\[
\Big\|
\sup_{u\in[\frac{s-1}{n},\frac sn]}
I_n^{(k)}\left(u,s/n\right)
\Big\|_{\psi_2}
\le
B(n^{-1/2})
\lesssim n^{-1/2}.
\]
Therefore by~\ref{orlicz:tail_bound} it holds for all $y >0$ and for a universal constant $c$ that can be extracted from the implied constant in the bound $B(\delta) \lesssim \delta$,
\[
\Prob\Big(
\sup_{u\in[\frac{s-1}{n},\frac sn]}
I_n^{(k)}\left(u,s/n\right)
> y
\Big)
\le
2\exp(-cny^2).
\]
Taking a union bound over $k\in[K]$ and $s\in[n]$, we obtain for every $y>0$ that 
\[
\mathbb P\Big(
A>
\sqrt{\frac 83}C_{\mathcal K}^{1/2}y
\Big)
\le
2nK\exp(-cny^2).
\]
Choosing
\[
y=\delta_1 = C_1\sqrt{\frac{\log(nK)}{n}}
\]
with $C_1>0$ sufficiently large gives the claimed bound,
\[
\mathbb P\Big(
A>
C_1
\sqrt{\frac{C_{\mathcal K}\log(nK)}{n}}
\Big)
\le
\frac{2}{n}.
\]

It remains to show that $B(\delta)\lesssim \delta$.  For each $k \in [K]$ it holds by definition of $\mathbb{B}_{\mathcal K}$ and bi-linearity of the covariance operator that
\begin{align*}
&\phantom{{}={}} 
\Cov\big(\pi_k\mathbb{B}_{\mathcal K}(x)-\pi_k \mathbb{B}_{\mathcal K}(y)\big)
\\&= 
\Cov\big(\pi_k\mathbb{B}_{\mathcal K}(x),\pi_k\mathbb{B}_{\mathcal K}(x)\big)+\Cov\big(\pi_k\mathbb{B}_{\mathcal K}(y),\pi_k\mathbb{B}_{\mathcal K}(y)\big)
\\&\hspace{5cm}
-\Cov\big(\pi_k\mathbb{B}_{\mathcal K}(x),\pi_k\mathbb{B}_{\mathcal K}(y)\big)-\Cov\big(\pi_k\mathbb{B}_{\mathcal K}(y),\pi_k\mathbb{B}_{\mathcal K}(x)\big)
\\&=
\big[ (x-x^2)+(y-y^2) -(x\wedge y -xy)-(x\wedge y - xy)\big]\mathcal{K}_k
\\&=
\big[ x\vee y - x\wedge y -(x-y)^2\big]\mathcal{K}_k
\\&=
\big(|x-y|(1- |x-y|)\big)\mathcal{K}_k. 
\end{align*}
Therefore, by definition of $C_{\mathcal K}$ in \eqref{def:C_Kappa}, for each $k \in [K]$ it holds that
\begin{align*}
\max_{k\in [K]}\big\|\Cov\big(\pi_k\mathbb{B}_{\mathcal K}(x)-\pi_k\mathbb{B}_{\mathcal K}(y)\big)\big\|_{\Tr} \leq C_{\mathcal K} |x-y|,
\end{align*}
and so by definition of $I_{n}^{(k)}$ in \eqref{eq:Ink}, the reverse triangle inequality and Lemma~\ref{lem:psi_2_controlled_by_L_2} it holds for all $k \in [K]$ and all $(x,y)\in I^2$ that
\begin{align*}
\big\|I_n^{(k)}(x,y)\big\|_{\psi_2}
&\leq 
\sqrt{\frac 38}C_\mathcal{K}^{-1/2}\big\|\|\pi_k\mathbb{B}_{\mathcal K}(x)-\pi_k\mathbb{B}_{\mathcal K}(y)\|_{\Hb_k}\big\|_{\psi_2}
\\&\leq  
C_\mathcal{K}^{-1/2}\big\|\Cov\big(\pi_k\mathbb{B}_{\mathcal K}(x)-\pi_k\mathbb{B}_{\mathcal K}(y)\big)\big\|^{1/2}_{\Tr}
\leq 
|x-y|^{1/2}.
\end{align*}
Fix $k\in[K]$ and let $J\subset I$ be an interval with $\operatorname{diam}(J)\le\delta$. Choose any point $u_J\in J$, and define the real-valued process $Z=\{Z(u):u\in J\}$
by
\[
Z(u)
:=
\sqrt{\frac38}\,C_{\mathcal K}^{-1/2}
\Big(
\left\|\pi_k\mathbb B_{\mathcal K}(u)\right\|_{\Hb_k}
-
\left\|\pi_k\mathbb B_{\mathcal K}(u_J)\right\|_{\Hb_k}
\Big).
\]
Then, for all $u,v\in J$,
\[
\|Z(u)-Z(v)\|_{\psi_2}
=
\left\|I_n^{(k)}(u,v)\right\|_{\psi_2}
\le |u-v|^{1/2}.
\]

Let $N(\tau)$ be the minimal number of closed balls of radius $\tau$ in the metric $d(x,y)=|x-y|^{1/2}$ required to cover an interval of radius $\delta$ in the metric $d$. By Corollary 2.2.5 of \cite{VW96} if $\textrm{diam}(J)\leq \delta$, 
\begin{align*}
\Big\|\sup_{(x,y) \in J^2}I_n^{(k)}(x,y)\Big\|_{\psi_2}
&\lesssim 
\int^{\delta}_{0}\psi^{-1}_2\big(N(\eps/2,d)\big)\, \diff\eps 
\\&\leq 
\int^{\delta}_{0}\sqrt{\log\big(3\delta^2/\eps^2\big)} \, \diff\eps
=
\delta
\int_0^1
\sqrt{\log\big(3 u^{-2}\big)} \, \diff u
\lesssim \delta 
\end{align*}
We used that $\psi_2(x):=\exp(x^2)-1$, and $N(\eps/2,d) \leq 2(\delta/\eps)^2+1 \le 3 (\delta/\eps)^2$ as shown at the very end of the proof and the change of variables $u=\eps/\delta$ and finiteness of the integral in the second last bound. This justifies the claim that $B(\delta)\lesssim \delta$.  

It remains to justify the covering number statement. Let  $J\subset \mathbb R$ be an interval with $d$-diameter at most $\delta$, such that the ordinary Euclidean diameter of $J$ is at most
$\delta^2$.  A $d$-ball of radius $\varepsilon/2$ is of the form
$
    B_d(x,\varepsilon/2)
    =
    \{y: d(x,y)\le \varepsilon/2\}
    =
    \{y: |x-y|\le \varepsilon^2/4\}$.
Thus it is an ordinary Euclidean interval of radius $\varepsilon^2/4$,
and hence of length $\varepsilon^2/2$.
The claim then follows from the fact that covering an interval of ordinary length at most $\delta^2$
requires no more than
$
    \left\lceil \delta^2/(\varepsilon^2/2)\right\rceil
    =
    \left\lceil2 (\delta/\varepsilon)^2\right\rceil \le 2 (\delta/\eps)^2 + 1
$
such balls. 
\end{proof}

\begin{proof}[Proof of Lemma~\ref{lem:supremum_to_interior_supremum_high_prob}]
By definition of $\mathbb{B}_{\mathcal K}$ in \eqref{def:Hilbert_Brownian_Bridge}, we have $\Cov(\pi_k\mathbb{B}_{\mathcal K}(s/n))=V(s/n)\mathcal{K}_k$.
Since $ \| \mathcal{K}_k \|_{\Tr} \le C_{\mathcal K}$ by definition of $C_{\mathcal K}$ in \eqref{def:C_Kappa}, we obtain from Lemma~\ref{lem:psi_2_controlled_by_L_2} that
\begin{align*}
\big\|\|\pi_k\mathbb{B}_{\mathcal K}(s/n)\|_{\Hb_k}\big\|_{\psi_2}^2
\leq 
\frac 83 C_{\mathcal{K}} V(s/n). 
\end{align*}
Recall that $V(s/n)\le x$ for all $s \in S_x^c$; see \eqref{eq:infsup-V-S_x}.
Write 
\[
U=\max_{k \in [K]}\sup_{s\in S_x^c}\big\|\pi_k\mathbb{B}_{\mathcal K}(s/n)\big\|_{\Hb_k}.
\]
By the union bound and~\ref{orlicz:tail_bound}, we obtain that
\begin{align*}
\Prob( U^2 \geq t) = P(U > \sqrt t) 
\leq 
2Kn\exp\left(
-\frac{3t}{8C_{\mathcal K}x}
\right) .
\end{align*}
By assumption, we have $t \ge x C_2 C_{\mathcal K} \log(nK)$, so the right hand side of the previous display is upper bounded by $2Kn \exp(-3C_2 \log(nK)/8)$. This is bounded by $2/n$ if we choose $C_2$ sufficiently large; for instance, $C_2 = 16/3$.
\end{proof}

\begin{proof}[Proof of Lemma~\ref{lem:interior_supremum_to_projected_interior_supremum_high_prob}]
To lighten notation let 
\[
P^{\perp}_{kM}=(\id_{{\Hb}_k}-P_{kM}),
\]
and recall that for each $x\in \Hb_k$,
\[
\big\|P_{kM}x\big\|^2_{\Hb_k}
+
\big\|P^{\perp}_{kM} x\big\|^2_{\Hb_k}
=
\big\|x\big\|^2_{\Hb_k}.
\]
Using the elementary inequality
$|\max_i a_i-\max_i b_i|
\le
\max_i |a_i-b_i|$,
it holds that
\begin{align*}
\Big|
\max_{k\in[K]}\max_{s\in S}
\|\pi_k\mathbb B_{\mathcal K}(s/n)\|^2_{\Hb_k}
-&
\max_{k\in[K]}\max_{s\in S}
\|(P_{kM}\circ \pi_k)\mathbb B_{\mathcal K}(s/n)\|^2_{\Hb_k}
\Big|\\
&
\le
\max_{k\in[K]}\max_{s\in S}
\Big|
\|\pi_k\mathbb B_{\mathcal K}(s/n)\|^2_{\Hb_k}
-
\|(P_{kM}\circ \pi_k)\mathbb B_{\mathcal K}(s/n)\|^2_{\Hb_k}
\Big| \\
&=
\max_{k\in[K]}\max_{s\in S}
\|(P^{\perp}_{kM}\circ \pi_k)\mathbb B_{\mathcal K}(s/n)\|^2_{\Hb_k}.
\end{align*}
Let \(\Omega_0\) denote the event on which
\[
\pi_k\mathbb B_{\mathcal K}(s/n)\in \tilde{\Hb}_k
\qquad
\text{for all }(s,k)\in[n]\times[K].
\]
By the discussion preceding the definition of \(P_{kM}\) in~\eqref{eq:def-PkM} each of these inclusions holds almost surely, and since the index set \([n]\times[K]\) is finite it holds that \(\mathbb P(\Omega_0)=1\). On $\Omega_0$ it holds that
\[
R_{sk} := \|(P^{\perp}_{kM}\circ \pi_k)\mathbb B_{\mathcal K}(s/n)\|^2_{\Hb_k}=\|(P^{\perp}_{kM}\circ \pi_k)\mathbb B_{\mathcal K}(s/n)\|^2_{\tilde{\Hb}_k}=\sum^{r_k}_{j=\check M_k+1}\langle\pi_k\mathbb B_{\mathcal K}(s/n), Z_{kj}\rangle_{\tilde{\Hb}_k}^2,
\]
where the sum is interpreted as \(0\) if \(\check M_k\ge r_k\) following our convention throughout this manuscript. To obtain the first inequality we used that if  $x \in \tilde{\Hb}_k$ then $P^{\perp}_{kM} x \in \tilde{\Hb}_k$, and that  for all $x \in \tilde{\Hb}_k$ it holds that
$
\|x\|^2_{\tilde{\Hb}_k}:=\|x\|^2_{\Hb_k}
$
as discussed in the paragraphs below~\eqref{def:C_Kappa}. To obtain the second inequality we used  that the eigenvectors of $\mathcal{K}_k$,  $\{Z_{kj}\}_{j=1}^{r_k}$, form an orthonormal basis of $\tilde{\Hb}_k$ and also used Parseval's inequality, recalling the definition of $P_{kM}$ in~\eqref{eq:def-PkM}. Then note that 
\begin{align*}
\E[R_{sk}]
=
\sum^{r_k}_{j=\check M_k+1}\E\Big[\langle\pi_k\mathbb B_{\mathcal K}(s/n), Z_{kj}\rangle_{\tilde{\Hb}_k}^2\Big]
&=
\sum^{r_k}_{j=\check M_k+1}\left\langle\Cov(\pi_k\mathbb B_{\mathcal K}(s/n))Z_{kj}, Z_{kj}\right\rangle_{\tilde{\Hb}_k}
\\&=
V(s/n)\sum^{r_k}_{j=\check M_k+1}\left\langle \mathcal{K}_k Z_{kj}, Z_{kj}\right\rangle_{\tilde{\Hb}_k}
\\&=
V(s/n)\sum^{r_k}_{j=\check M_k+1}\lambda_{kj}.
\end{align*}
Here, the second equality follows by Lemma~\ref{lem:properties-covariance}(1), the third by~\eqref{def:Hilbert_Brownian_Bridge}, and the fourth because each $Z_{kj}$ is a normalized eigenvector with eigenvalue $\lambda_{kj}$. Using $V(s/n)\le 1/4$, Assumption~\ref{cond:scores}, and \(\gamma_U>1/2\), we obtain
\begin{equation*}
\E[R_{sk}]
\le
\frac14
\sum_{j=\check M_k+1}^{r_k}\lambda_{kj} \le
\frac14
C_U^2
\sum_{j=M+1}^{\infty} j^{-2\gamma_U} \le
\frac{C_U^2}{4(2\gamma_U-1)}
M^{1-2\gamma_U}.
\end{equation*}
Define
\[
Y_{k,s}
:=
(P^{\perp}_{kM}\circ \pi_k)\mathbb B_{\mathcal K}(s/n), \qquad k \in  [K], s \in S_x,
\]
which is a centered Gaussian element of \(\tilde{\Hb}_k\).
By  Lemma~\ref{lem:psi_2_controlled_by_L_2} we obtain that
\[
\big\|
\|Y_{k,s}\|_{\Hb_k}
\big\|_{\psi_2}^2 \leq \frac83\E
\|Y_{k,s}\|_{\Hb_k}^2
=
\frac83
\E[R_{sk}].
\]
where the final equality holds by definition. Thus, by~\ref{orlicz:tail_bound}, for every \(u>0\),
\[
\mathbb P\left(
\|Y_{k,s}\|_{\Hb_k}^2>u
\right)
\le
2\exp\left(
-\frac38
\frac{u}{\E[R_{sk}]}
\right).
\]
Taking
$
u=\delta_2
=
C_3\log(nK)M^{1-2\gamma_U}
$
and using the bound on $\E[R_{sk}]$ it holds that 
\[
\mathbb P\left(
\|Y_{k,s}\|_{\Hb_k}^2>\delta_2
\right)
\le
2\exp\left(
-
\frac{3(2\gamma_U-1)}{2C_U^2}
C_3
\log(nK)
\right).
\]
Taking a union bound over \(k\in[K]\) and \(s\in S\), and using \(|S|\le n\), gives
\begin{align*}
&\mathbb P\left(
\max_{k\in[K]}
\max_{s\in S}
\|Y_{k,s}\|^2_{\Hb_k}
>
\delta_2
\right) 
\le
2nK
\exp\left(
-
\frac{3(2\gamma_U-1)}{2C_U^2}
C_3
\log(nK)
\right).
\end{align*}
Choosing
$C_3=4C_U^2 / \{ 3(2\gamma_U-1) \}$, the exponential factor become $(nK)^{-2} \le n^{-2}K^{-1}$, and this completes the proof.
\end{proof}

\begin{proof}[Proof of Lemma~\ref{lem:bound-on-operator-norm-covariances}]
Recall the notation $\Sigma^{(n)} = \Cov(\mathcal B^{(n)})$ from Section~\ref{subsec:discrete_brownian-bridge} and $\Sigma^{(n,q)} = \Cov(G^{(n,q)})$ from \eqref{subsec:key-intermediate-variables}. We have
\begin{align*}
&\phantom{{}={}} 
\Big\|
\Cov\Big(
(P_{M\sigma \chgset} \circ  \pi_{\sigma,\chgset }) G^{(n,q)}
\Big)
-
\Cov\Big(
( P_{M\sigma \chgset} \circ \pi_{\sigma,\chgset })\mathcal B^{(n)}
\Big)
\Big\|_{\op}
\\&=
\Big\|
P_{M\sigma \chgset} \circ \pi_{\sigma,\chgset } ( \Sigma^{(n,q)} - \Sigma^{(n)} ) \pi_{\sigma,\chgset }^* \circ P_{M\sigma \chgset}^*
\Big\|_{\op}
\\&\le 
\| P_{M\sigma \chgset}  \| _{\op} 
( \Sigma^{(n,q)} - \Sigma^{(n)} ) \pi_{\sigma,\chgset }^* 
\Big\|_{\op}
\| P_{M\sigma \chgset}^*  \| _{\op} 
\\&\le
\Big\|
\pi_{\sigma,\chgset } ( \Sigma^{(n,q)} - \Sigma^{(n)} ) \pi_{\sigma,\chgset }^* 
\Big\|_{\op},
\end{align*}
where we have used sub-multiplicativity of the operator norm and the fact that $\| P_{M\sigma\chgset} \|_\op \le 1$. The latter follows from
\[
\| P_{M\sigma\chgset} x \|_{\mathcal H_\chgset^\sigma}^2
=
\sum_{k \in \chgset,s \in  \sigma } \| P_{kM} x_{sk} \|_{\Hb_k}^2
\le 
\sum_{k \in \chgset,s \in  \sigma } \| x_{sk} \|_{\Hb_k}^2
= \| x \|_{\mathcal H_\chgset^\sigma}^2.
\]
Next, note that
\begin{align*}
\pi_{\sigma ,\chgset } ( \Sigma^{(n,q)} - \Sigma^{(n)} ) \pi_{\sigma ,\chgset }^* 
&=
\pi_{\sigma ,\chgset } ( A_n \otimes \Sigma_q - K_n \otimes \mathcal K ) \pi_{\sigma ,\chgset }^* 
\\&= 
A_{n,\sigma} \otimes (\pi_\chgset\Sigma_q\pi_\chgset^*) - K_{n,\sigma} \otimes (\pi_\chgset \mathcal K \pi_\chgset^*)
\\&=
A_{n,\sigma} \otimes (\pi_\chgset[ \Sigma_q - \mathcal K]\pi_\chgset^*)  + (A_{n,\sigma}  - K_{n,\sigma}) \otimes  (\pi_\chgset \mathcal K \pi_\chgset^*),
\end{align*}
where we have used \eqref{eq:representation-projection-kronecker}. Using \eqref{eq:operator-norm-Kronecker}, we have
\begin{align*}
\| A_{n,\sigma} \otimes (\pi_\chgset[ \Sigma_q - \mathcal K]\pi_\chgset^*) \|_{\op}
&=
\| A_{n,\sigma}  \|_{\op}
\| \pi_\chgset[ \Sigma_q - \mathcal K]\pi_\chgset^* \|_{\op},
\\
\| (A_{n,\sigma} -K_{n,\sigma}) \otimes  (\pi_\chgset \mathcal K \pi_\chgset^*) \|_\op 
&=
\| (A_{n,\sigma} -K_{n,\sigma})\|_\op \|  (\pi_\chgset \mathcal K \pi_\chgset^*) \|_\op .
\end{align*}
Since $A_{n,\sigma}$ has entries in $[-1,1]$, we have $\| A_{n,\sigma} \|_{\op} \le |\sigma| \cdot \|A_{n,\sigma}\|_\infty \le  |\sigma|$. 
Next, we have 
\[
\| \pi_\chgset \mathcal K \pi_\chgset^* \|_{\op}
\le 
\| \pi_\chgset \mathcal K \pi_\chgset^* \|_{\Tr}
=
\sum_{k \in \chgset} \| \pi_k \mathcal K \pi_k^*\|_{\Tr}
=
\sum_{k \in \chgset} \| \mathcal K_k\|_{\Tr}
\le |\chgset| C_{\mathcal K}
\]
with $C_{\mathcal K}$ from \eqref{def:C_Kappa}, where the first equality follows from a standard argument on CONS in product spaces; see also \eqref{eq:trace-class-covariance-products}. Overall, we obtain that 
\[
\| \pi_{\sigma ,\chgset } ( \Sigma^{(n,q)} - \Sigma^{(n)} ) \pi_{\sigma ,\chgset }  \|_{\op}
\le 
|\sigma| \cdot \| \pi_\chgset[ \Sigma_q - \mathcal K]\pi_\chgset^* \|_{\op} + C_{\mathcal K} |\chgset| \cdot \| A_{n,\sigma} - K_{n,\sigma} \|_\op.
\]
As a consequence of Lemma~\ref{lem:prop a_s_ell}, we have 
\[
\| A_{n,\sigma} - K_{n,\sigma} \|_\op
\le 
|\sigma| \cdot \| A_{n,\sigma} - K_{n,\sigma} \|_\infty \le \frac{2|\sigma|}m,
\]
and it remains to bound $\| \pi_\chgset[ \Sigma_q - \mathcal K]\pi_\chgset^* \|_{\op}$.

For this purpose, we will apply Lemma~\ref{lem:convergence_to_long_term_trace} with $X_i = \pi_\chgset \eps^{(i)} \in \mathcal H_\chgset$, with $N=q$, and with candidate long run variance $\mathcal K_\infty = \pi_\chgset \mathcal K \pi_\chgset^*$. For that purpose, we need to bound $\|\|\pi_\chgset \eps^{(i)}\|_{\mathcal H_\chgset} \|_{\psi_1}$. 
First, by the definition of the inner product on product Hilbert spaces, we have
\[
\|\pi_\chgset \eps^{(i)}\|_{\mathcal H_\chgset}^2
=
\sum_{k \in \chgset} \| \eps_k^{(i)}\|_{\Hb_k}^2.
\]
As a consequence, by \ref{orlicz:power} and Assumption~\ref{cond:stationary},
\begin{align*}
    \big\|\|\pi_\chgset \eps^{(i)}\|_{\mathcal H_\chgset} \big\|_{\psi_1}^2
    =
    \big\|\|\pi_\chgset \eps^{(i)}\|_{\mathcal H_\chgset}^2 \big\|_{\psi_{1/2}}
    =
    \Big\|  \sum_{k \in \chgset} \| \eps_k^{(i)}\|_{\Hb_k}^2 \Big\|_{\psi_{1/2}}
    &\lesssim 
    \sum_{k \in \chgset} \big\|\|\eps_k^{(i)}\|_{\Hb_k}^2 \big\|_{\psi_{1/2}}
    \\&=
    \sum_{k \in \chgset} \big\|\|\eps_k^{(i)}\|_{\Hb_k} \big\|_{\psi_{1}}^2
    \le 
    |\chgset| D_\psi^2.
\end{align*}
We may hence apply Lemma~\ref{lem:convergence_to_long_term_trace} to deduce that 
\begin{align*}
\| \pi_\chgset[ \Sigma_q - \mathcal K]\pi_\chgset^* \|_{\op}
&\le 
\| \pi_\chgset[ \Sigma_q - \mathcal K]\pi_\chgset^* \|_{\Tr}
\\&=
\Big\| \Cov\Big(q^{-1/2} \sum_{i=1}^q \pi_\chgset \eps^{(i)} \Big) - \pi_\chgset\mathcal K\pi_\chgset^* \Big\|_\Tr
\lesssim
|\chgset| D_\psi^2 S_{\beta,1} \frac1q.
\end{align*}
Overall, 
\[
\| \pi_{\sigma ,\chgset } ( \Sigma^{(n,q)} - \Sigma^{(n)} ) \pi_{\sigma ,\chgset }  \|_{\op}
\lesssim 
(D_\psi^2 S_{\beta, 1} \vee C_{\mathcal K}) \cdot  |\sigma| \cdot  |\chgset| \frac{1}{q \wedge m},
\]
which yields the claim.
\end{proof}

\begin{proof}[Proof of Corollary~\ref{lem:good_eigen_values_2}]
Fix \((s,k)\in S_x\times[K]\), and let 
\[
\lambda^{(ks)}_1\ge \lambda^{(ks)}_2\ge \cdots
\ge \lambda^{(ks)}_{\check M_k}
\]
be the eigenvalues of \(\operatorname{Cov}(G^M_{sk})\). 
Since \(O_{kM}\) from \eqref{eq:def-OkM} is an isometry, the eigenvalues of 
\(\Cov(G^M_{sk})\) agree with those of $\Cov((P_{kM}  \circ \pi_{sk} )G^{(n,q)})$.
Let \(\lambda_j\) denote the \(j\)-th eigenvalue of $\Cov((P_{kM}  \circ \pi_{sk} )\mathcal B^{(n)})$.
By Weyl's inequality and Lemma~\ref{lem:bound-on-operator-norm-covariances}, applied with
\(\chgset=\{k\}\) and \(\sigma=\{s\}\), for every
\(j\in[\check M_k]\),
\begin{align}
\label{eq:bound-from-Weyl}
\left|
\lambda^{(ks)}_j-\lambda_j
\right|
\le
C_4
\bigl(D_\psi^2S_{\beta,1}\vee C_{\mathcal K}\bigr)
(m\wedge q)^{-1} =: E_{n,q}.
\end{align}
By Assumption~\ref{cond:scores}, the eigenvalues \(\lambda_{kj}\) of
\(\mathcal K_k\) satisfy
\[
C_L^2j^{-2\gamma_L}
\le
\lambda_{kj}
\le
C_U^2j^{-2\gamma_U}.
\]
Moreover, since
$
\Cov(\mathcal B^{(n)})=K_n \otimes \mathcal K
$
by \eqref{eq:Kronecker-Kn}
and since \(K_n(s,s)=V(s/n)\), we have, applying \eqref{eq:representation-projection-kronecker},
\[
\Cov(\pi_{sk}\mathcal B^{(n)})
=
V(s/n) \cdot \pi_k \mathcal K \pi_k^* = V(s/n) \cdot \mathcal K_k,
\]
so that the eigenvalues of
$\Cov((P_{kM}  \circ \pi_{sk} )\mathcal B^{(n)})$
are precisely
$
\lambda_j
=
V(s/n)\lambda_{kj}$, for $j \in [\check M_k]$.
Therefore,
\begin{align}
\label{eq:bound-on-lambdaj}
V(s/n)C_L^2j^{-2\gamma_L}
\le
\lambda_j
\le
V(s/n)C_U^2j^{-2\gamma_U}.
\end{align}

Since \(s\in S_x\), we have \(V(s/n)\ge x\) as remarked in~\eqref{eq:infsup-V-S_x}. Hence, in combination with \eqref{eq:bound-from-Weyl}, we get that
\[
\lambda^{(ks)}_j
\ge
xC_L^2j^{-2\gamma_L}
-
E_{n,q}.
\]
The assumed upper bound on \(M\) from \eqref{eq:bound-on-M} is equivalent to
$
E_{n,q}
\le
xC_L^2M^{-2\gamma_L} / 2.
$
Since \(j\le \check M_k\le M\), it follows that
$
M^{-2\gamma_L}
\le
j^{-2\gamma_L}
$
and thus
\[
\lambda^{(ks)}_j
\ge
\frac{xC_L^2}{2}j^{-2\gamma_L}.
\]

For the upper bound, note that \eqref{eq:bound-on-lambdaj} in combination with \eqref{eq:bound-from-Weyl} and the fact that $V(s/n) \le 1$ gives
\[
\lambda^{(ks)}_j
\le
C_U^2j^{-2\gamma_U}
+
E_{n,q}.
\]
Using the same bound on \(E_{n,q}\), together with
\(\gamma_U\le\gamma_L\), \(x\le1\), and \(C_L\le C_U\), we have
$
E_{n,q}
\le
xC_L^2j^{-2\gamma_L}/2
\le
C_U^2j^{-2\gamma_U}.
$
Therefore,
\[
\lambda^{(ks)}_j
\le
2C_U^2j^{-2\gamma_U}.
\]

Combining the lower and upper bounds, for every
\(j\in[\check M_k]\),
\[
\frac{xC_L^2}{2}j^{-2\gamma_L}
\le
\lambda^{(ks)}_j
\le
2C_U^2j^{-2\gamma_U},
\]
which is exactly \eqref{eq:cov-gmsk-decay}
in the sense of Definition~\ref{def:covariance-decay}.
\end{proof}

\begin{proof}[Proof of Lemma~\ref{cor:projected_mixed_norm_kolmogorov}]
Let $\mathcal I:=S_x\times[K]$ and $
K_x:=|\mathcal I|=|S_x|K\le nK$.
Define the centered Gaussian random vectors
\[
X
:=
\big(
G^M_{sk}
\big)_{(s,k)\in\mathcal I}
\in
(\mathbb R^M)^{K_x},
\]
where we pad the coordinates of $G^M_{sk}$ with $M-\check{M}_k$ zeroes. Likewise, let
\[
Y
:=
\left(
\big((O_{kM}\circ P_{kM}\circ\pi_k)\otimes\pi_s\big)
\mathcal B^{(n)}
\right)_{(s,k)\in\mathcal I}
\in
(\mathbb R^M)^{K_x},
\]
where we also add zeros if necessary.
Since \(O_{kM}\) from \eqref{eq:def-OkM} is an isometry, and since
$\mathcal B^{(n)}
=
(
\mathbb B_{\mathcal K}(1/n),
\ldots,
\mathbb B_{\mathcal K}(1)
)^\top
$ by definition in  \eqref{eq:def-bn},
we have
\[
\max_{(s,k)\in\mathcal I}\|Y_{sk}\|_{2}
=
\max_{k\in[K]}\max_{s\in S_x}
\left\|
(P_{kM}\circ\pi_k)\mathbb B_{\mathcal K}(s/n)
\right\|_{\Hb_k}.
\]
By Corollary~\ref{lem:good_eigen_values_2}, for every \((s,k)\in\mathcal I\), the covariance matrix of \(X_{sk}=G^M_{sk}\) satisfies, in the finite-dimensional sense of Definition~\ref{def:covariance-decay},
\begin{align}
\label{eq:cov-decay-x}
\Cov(X_{sk})
\in
\operatorname{Decay}
\left(
\frac{xC_L^2}{2},
2C_U^2,
2\gamma_L,
2\gamma_U
\right).
\end{align}
Thus Lemma~\ref{lem:gaussbootsphere-new} may be applied to \(X\) and \(Y\).
We first bound \(\|\Sigma_X-\Sigma_Y\|_{\infty}\), where $\Sigma_X:=\Cov(X)$ and $\Sigma_Y:=\Cov(Y)$. Fix two scalar coordinates, say
$(s,k,j)$ and $
(t,\ell,r)$,
where \(s,t\in S_x\), \(k,\ell\in[K]\), and \(j,r\in[M]\). Apply Lemma~\ref{lem:bound-on-operator-norm-covariances} to the ordered coordinate set formed by the distinct elements of \(\{k,\ell\}\) and the ordered time set formed by the distinct elements of \(\{s,t\}\). Since these two sets have cardinalities at most \(2\), Lemma~\ref{lem:bound-on-operator-norm-covariances} gives
\[
\|\Sigma_X-\Sigma_Y\|_{\max}
\le 
\|\Sigma_X-\Sigma_Y\|_{\op}
\le
4C_4
\frac{
D_\psi^2S_{\beta,1}\vee C_{\mathcal K}
}{
m\wedge q
}.
\]
We now apply Lemma~\ref{lem:gaussbootsphere-new} with auxiliary integer $m_{\mathrm{aux}}=m\wedge q$ and with threshold lower bound \(\sqrt\rho\). We obtain
\begin{align*}
&\sup_{u\ge \sqrt\rho}
\bigg|
\Prob\Big(
\max_{(s,k)\in\mathcal I}\|X_{sk}\|_{2}\le u
\Big)
-
\Prob\Big(
\max_{(s,k)\in\mathcal I}\|Y_{sk}\|_{2}\le u
\Big)
\bigg|
\\
& \quad \le C_1
x^{-1}
\bigg[
M^{2/3}
\Big(
\frac{
D_\psi^2S_{\beta,1}\vee C_{\mathcal K}
}{
m\wedge q
}
\Big)^{1/3}
+
(m\wedge q)^{-2/3}
\bigg]
\log^{7/6 + 2 \kexp}
\big(
e(m\wedge q)(K_x+1)(U\vee1)
\big).
\end{align*}
where $C_1=C_1(C_U,C_L,\gamma_U,\gamma_L,\rho)$ and where 
\[
U:=
\max_{(s,k)\in\mathcal I}
\left\{
\sqrt{\Tr(\Cov(X_{sk}))}
\vee
\sqrt{\Tr(\Cov(Y_{sk}))}
\right\}.
\]
Since the eigenvalues of $\Cov(X_{sk})$ and $\Cov(Y_{sk})$ are bounded from above by $2C_U^2j^{-2\gamma_U}$, which follows from \eqref{eq:cov-decay-x} and \eqref{eq:bound-on-lambdaj}, respectively, $U$ can be bounded by a constant $C_2=C_2(C_U,c_U)$ that only depends on $C_U$ and $c_U$.

We next simplify the logarithmic factor. Since \(K_x\le nK\), \(m\wedge q\le n\) and $U \le C_2$, there exists a constant $C_3=C_3(C_U,c_U)$ such that
\[
\log^{7/6 + 2 \kexp}
\left(
e(m\wedge q)(K_x+1)(U\vee1)
\right)
\le
C
\log^{7/6 + 2 \kexp}(nK).
\]

Since \(M\ge1\), \(m\wedge q\ge1\), and Assumption~\ref{cond:scores} implies \(C_{\mathcal K}\ge C_L^2\), we have
\[
(m\wedge q)^{-2/3}
\le
C_L^{-2/3}
M^{2/3}
\left(
\frac{
D_\psi^2S_{\beta,1}\vee C_{\mathcal K}
}{
m\wedge q
}
\right)^{1/3}.
\]
Hence, for some constant $C_4=C_4(C_U,C_L,\gamma_U,\gamma_L,\rho)$,
\begin{align*}
&\sup_{u\ge \sqrt\rho}
\bigg|
\Prob\Big(
\max_{(s,k)\in\mathcal I}\|X_{sk}\|_{2}\le u
\Big)
-
\Prob\Big(
\max_{(s,k)\in\mathcal I}\|Y_{sk}\|_{2}\le u
\Big)
\bigg|
\\
&\hspace{4cm}\le
C_4
x^{-1}
M^{2/3}
\left(
\frac{
D_\psi^2S_{\beta,1}\vee C_{\mathcal K}
}{
m\wedge q
}
\right)^{1/3}
\log^{7/6 + 2 \kexp}(nK).
\end{align*}
After the change of variable $u=\sqrt t$, this yields the claim, with $C= C_4 (D_\psi^2 S_{\beta,1} \vee C_{\mathcal K})^{2/3}$.
\end{proof}

\subsection{Gaussian Approximation}

Recall the Berbee coupled big-block sums $
\widetilde S_1,\ldots,\widetilde S_m$
defined in~\eqref{def:big_block_sums}, and write
\(\widetilde S_{\ell k}:=\pi_k \widetilde S_{\ell} \in \Hb_k\) for the \(k\)-th coordinate of
$\widetilde S_\ell$. For each \((s,k)\in[n]\times[K]\), define
\begin{align}
\label{eq:tilde-tsk}
\widetilde T_{sk}
:=
(mq)^{-1/2}
\Big\|
\sum_{\ell\in[m]}a_{s\ell}\widetilde S_{\ell k}
\Big\|_{\Hb_k},
\end{align}
where the constants $\{a_{s\ell}\}_{s \in [n],\ell\in [m]}$ are defined in~\eqref{eq:def-a_sl}. Then define each projected analogue as
\begin{align}
\label{eq:tilde-tskm}
\widetilde T^M_{sk}
:=
(mq)^{-1/2}
\Big\|
P_{kM}\sum_{\ell\in[m]}a_{s\ell}\widetilde S_{\ell k}
\Big\|_{\Hb_k}.
\end{align}
Finally, set
\[
\widetilde T
:=
\max_{k\in[K]}\max_{s \in [n]}\widetilde T_{sk},
\qquad
\widetilde T^M
:=
\max_{k\in[K]}\max_{s \in [n]}\widetilde T^M_{sk},
\]
and recall $T_n^O$ from \eqref{eq:def_oracle_test_statistic}.

\begin{lemma}[Comparison of oracle process to coupled process]
\label{lem:berbee_hp_reduction}
Suppose Assumption~\ref{cond:stationary} holds. There exists a universal
constant \(C_1>0\) such that, for every $K\in\mathbb N_{\ge1}$, $n\in\mathbb N_{\ge4}$, $q,r\in\mathbb N$
satisfying
\[
1\le r\le q \le
\min\left\{
\frac n4,\frac{\sqrt n}{\log n}
\right\},
\]
there exists an event \(\Omega_1\) with probability at least $1-1/n$
such that, on
\(\Omega_1\cap\Omega_{\mathrm{block}}\), with
\(\Omega_{\mathrm{block}}\) defined in~\eqref{eq:Omega-block},
\[
\left|T_n^O-\widetilde T\right|
\le\delta_1 
:=
C_1D_\psi(S_{\beta,0}\vee1)^{1/2}
\log(nK)
\left(
\frac rq\vee\frac{q^2}{n}
\right)^{1/2},
\]
where \(S_{\beta,0}\) is defined in~\eqref{eq:def-S_beta}.
\end{lemma}

\begin{lemma}[Boundary bound for the squared coupled big-block statistic]
\label{lem:coupled_big_block_boundary}
Suppose Assumption~\ref{cond:stationary} holds. There exists a universal
constant \(C_2>0\) such that, for every $K\in\mathbb N_{\ge1},n\in\mathbb N_{\ge4},q,r\in\mathbb N$
satisfying
\[
1\le r\le q\le\frac n4,
\]
and every \(t>0\) and \(x\in(0,3/16]\) satisfying
\[
x\vee\frac{q^2\log^2 n}{n}
\le
\frac{t}
{C_2D_\psi^2(S_{\beta,0}\vee1)\log^2(nK)},
\]
we have
\[
\mathbb P\left(
\max_{k\in[K]}
\max_{s\in[n]\setminus S_x}
(\widetilde T_{sk})^2
>
t
\right)
\le
\frac2n,
\]
where \(S_x\) is defined in~\eqref{eq:definition-S_x} and
\(S_{\beta,0}\) is defined in~\eqref{eq:def-S_beta}.
\end{lemma}

\begin{lemma}[Projection error for the coupled big-block statistic]
\label{lem:coupled_big_block_projection_error}
Suppose Assumptions~\ref{cond:stationary} and~\ref{cond:scores} hold.
Then there exists a constant
\[
C_{\mathrm{proj}}
=
C_{\mathrm{proj}}
(D_\psi,C_U,\gamma_U,C_\beta,c_\beta)>0
\]
such that, for every \(K\in\mathbb N_{\ge1}\),
\(n\in\mathbb N_{\ge4}\), \(M\in\mathbb N_{\ge1}\), and
\(q,r\in\mathbb N\) satisfying
\[
1\le r\le q\le\frac n4,
\]
and every nonempty \(S\subset[n]\), defining
\[
\delta_{\mathrm{proj}}
:=
C_{\mathrm{proj}}\log^2(nK)
\left[
M^{1-2\gamma_U}
+\frac1q
+\frac{q^2\log^2 n}{n}
\right],
\]
we have
\[
\mathbb P\left(
\max_{k\in[K]}\max_{s\in S}
(\widetilde T_{sk})^2
>
\max_{k\in[K]}\max_{s\in S}
(\widetilde T^M_{sk})^2
+
\delta_{\mathrm{proj}}
\right)
\le
\frac2n.
\]
\end{lemma}

\begin{lemma}[Gaussian approximation for the projected coupled statistic]
\label{lem:gaussian_approx_projected_coupled}
Fix \(\rho>0\), and suppose that
Assumptions~\ref{cond:stationary} and~\ref{cond:scores} hold.
Then there exists a constant
\[
C
=
C(C_U,C_L,\gamma_U,\gamma_L,\rho,
D_\psi,C_\beta,c_\beta)>0
\]
such that, for every \(K\in\mathbb N_{\ge1}\),
\(n\in\mathbb N_{\ge4}\), \(M\in\mathbb N_{\ge1}\), and
\(q,r\in\mathbb N\) satisfying
\[
1\le r\le q\le\frac n4,
\]
and every \(x\in(0,3/16]\) satisfying
\[
M^{2\gamma_L}
\le
\frac{
C_L^2x(m\wedge q)
}{
2C_4\bigl(D_\psi^2S_{\beta,1}\vee C_{\mathcal K}\bigr)
},
\]
where \(C_4\) is the universal constant from
Lemma~\ref{lem:bound-on-operator-norm-covariances},
\(S_{\beta,1}\) is defined in~\eqref{eq:def-S_beta}, and
\(C_{\mathcal K}\) is defined in~\eqref{def:C_Kappa}, we have
\begin{align*}
&d_{\mathrm K}^{(\rho)}
\Big(
\max_{k\in[K]}\max_{s\in S_x}
(\widetilde T_{sk}^{M})^2,
\max_{k\in[K]}\max_{s\in S_x}
\|G_{sk}^{M}\|_{2}^2
\Big)
\le
\frac{C}{x^{3/2}}
\left(
\frac{M^4q^2}{n}
\right)^{1/6}
\log^{2+3\kexp}(nK)
\log^{3/2+3\kexp}(n),
\end{align*}
where $\kexp := \gamma_L /(2\gamma_U-1)$.
\end{lemma}

\begin{proof}[Proof of Lemma~\ref{lem:berbee_hp_reduction}]
For \((s,i)\in [n]\times [n]\), define $
c_{si}:=\mathbf 1(i\leq s)-s/n$. Define also 
$C_n:=n^{-1/2}-(mq)^{-1/2}$ and let
\begin{align*}
\Delta_0
&:=
\max_{k\in[K]}\max_{s\in[n]}
\|\Delta^{(s,k)}_0\|_{\Hb_k},
&
\Delta^{(s,k)}_0
&:=
\sum_{i=1}^n c_{si}\eps_k^{(i)},
\\
\Delta_1
&:=
\max_{k\in[K]}\max_{s\in[n]}
\|\Delta^{(s,k)}_1\|_{\Hb_k},
&
\Delta^{(s,k)}_1
&:=
\sum_{\ell=1}^m a_{s\ell}S_{\ell k},
\\
\tilde{\Delta}_1
&:=
\max_{k\in[K]}\max_{s\in[n]}
\|\tilde{\Delta}^{(s,k)}_1\|_{\Hb_k},
&
\tilde{\Delta}^{(s,k)}_1
&:=
\sum_{\ell=1}^m a_{s\ell}\tilde{S}_{\ell k},
\\
\Delta_2
&:=
\max_{k\in[K]}\max_{s\in[n]}
\|\Delta^{(s,k)}_2\|_{\Hb_k},
&
\Delta^{(s,k)}_2
&:=
\sum_{\ell=1}^m a_{s\ell}S'_{\ell k},
\\
\tilde{\Delta}_2
&:=
\max_{k\in[K]}\max_{s\in[n]}
\|\tilde{\Delta}^{(s,k)}_2\|_{\Hb_k},
&
\tilde{\Delta}^{(s,k)}_2
&:=
\sum_{\ell=1}^m a_{s\ell}\tilde{S}'_{\ell k},
\end{align*}
where the various $S_{\ell k}$ variables are defined near~\eqref{def:big_block_sums}. Further, 
note that $\Delta_0 = \sqrt n T_n^O$ with $T_n^O$ from \eqref{eq:def_oracle_test_statistic} and that $\widetilde \Delta_1=(mq)^{1/2}\max_{k\in[K]}\max_{s \in [n]}\widetilde T_{sk} $ from \eqref{eq:tilde-tsk}.
Finally, define
\begin{align*}
\Delta_3
&:=
\max_{k\in[K]}\max_{s\in[n]}
\|\Delta^{(s,k)}_3\|_{\Hb_k},
&
\Delta^{(s,k)}_3
&:=
\sum_{i=1}^n c_{si}\eps_k^{(i)}
-
\Delta^{(s,k)}_1
-
\Delta^{(s,k)}_2.
\end{align*}
By the reverse triangle inequality for each Hilbert space norm and the \(L^\infty\) norm over \([K]\times[n]\), followed by the triangle inequality for each Hilbert space norm and the \(L^\infty\) norm, it holds that
\begin{align} 
\label{eq:long-bound:)}
&\phantom{{}={}}  \nonumber
\Big|
\max_{k\in [K]}\max_{s \in [n]}
\big\|n^{-1/2}\Delta^{(s,k)}_0\big\|_{\Hb_k}
-
\max_{k\in [K]}\max_{s \in [n]}
\big\|(mq)^{-1/2}\Delta^{(s,k)}_1\big\|_{\Hb_k}
\Big|
\\
&\leq \nonumber
\max_{k\in [K]}\max_{s \in [n]}
\big\|
n^{-1/2}\Delta^{(s,k)}_0
-
(mq)^{-1/2}\Delta^{(s,k)}_1
\big\|_{\Hb_k}
\\
&= \nonumber
\max_{k\in [K]}\max_{s \in [n]}
\big\|
n^{-1/2}
(\Delta^{(s,k)}_0-\Delta^{(s,k)}_1-\Delta^{(s,k)}_2)
+
C_n\Delta^{(s,k)}_1
+
n^{-1/2}\Delta^{(s,k)}_2
\big\|_{\Hb_k}
\\ 
&\leq
\max_{k\in [K]}\max_{s \in [n]}
\big\|
n^{-1/2}\Delta^{(s,k)}_3
\big\|_{\Hb_k}
+
|C_n|
\max_{k\in [K]}\max_{s \in [n]}
\big\|
\Delta^{(s,k)}_1
\big\|_{\Hb_k}
+
\max_{k\in [K]}\max_{s \in [n]}
\big\|
n^{-1/2}\Delta^{(s,k)}_2
\big\|_{\Hb_k}.
\end{align}

We note for later that
\begin{align} \label{eq:bound-on-cn}
|C_n|\sqrt{mq}
\le
3\left(\frac rq\vee \frac{q^2}{n}\right)^{1/2}.
\end{align}
To see this, let \(L:=q+r\), and write \(n=mL+d\) with \(0\le d<L\). Since
\(mq\le n\), we have
\[
|C_n|\sqrt{mq}
=
\left((mq)^{-1/2}-n^{-1/2}\right)\sqrt{mq}
=
1-\sqrt{\frac{mq}{n}}
\le
1-\frac{mq}{n}
=
\frac{n-mq}{n}
=
\frac{mr+d}{n}.
\]
Moreover, since \(d<L=q+r\le 2q\) and \(mq\le n\),
\[
\frac{mr+d}{n}
\le
\frac{mr}{n}+\frac{2q}{n}
\le
\frac rq+\frac{2q}{n}
\le
\sqrt{\frac rq}+\frac{2q}{\sqrt n}
\le
3\left(\sqrt{\frac rq}\vee \frac q{\sqrt n}\right).
\]

We now bound each of the pieces on the right-hand side of the inequality above. To do so, work on the event
$
\Omega_{\mathrm{block}}
:=
\Omega_{\mathrm{big}}\cap\Omega_{\mathrm{small}}
$ from \eqref{eq:Omega-block}, on which we have, for all
$(s,k)\in[n]\times[K]$,
\[
\Delta_1^{(s,k)}=\tilde\Delta_1^{(s,k)},
\qquad
\Delta_2^{(s,k)}=\tilde\Delta_2^{(s,k)}
\]
Therefore, on \(\Omega_{\mathrm{block}}\) from \eqref{eq:long-bound:)},
\begin{align}
\Big|
T_n^O
-
\max_{k\in[K]}\max_{s \in [n]}\widetilde T_{sk}
\Big|
&\leq
\max_{k\in [K]}\max_{s \in [n]}
\big\|n^{-1/2}\Delta^{(s,k)}_3\big\|_{\Hb_k}
\nonumber\\
&\quad +
|C_n|
\max_{k\in [K]}\max_{s \in [n]}
\big\|\tilde \Delta^{(s,k)}_1\big\|_{\Hb_k}
+
\max_{k\in [K]}\max_{s \in [n]}
\big\|n^{-1/2}\tilde \Delta^{(s,k)}_2\big\|_{\Hb_k}.
\label{eq:berbee_reduction_three_terms}
\end{align}

We first bound the small-block term involving $\tilde \Delta_2$. By Lemma~\ref{lem:uified_bounds}(3) and~\eqref{bound_norm_a_s}, recalling that
$
m=\lfloor {n}/(q+r)\rfloor\ge 2$
under the present assumptions, it holds for all \((s,k)\in[n]\times[K]\) that
\begin{align*}
\Big\|
\big\|\tilde{\Delta}^{(s,k)}_2\big\|_{\Hb_k}
\Big\|_{\psi_1}
&=
\Big\|
\Big\|
\sum_{\ell=1}^{m} a_{s\ell}\tilde{S}'_{\ell k}
\Big\|_{\Hb_k}
\Big\|_{\psi_1}
\\
&\leq
C_3D_{\psi}({C}_{\beta, r}\vee 1)^{1/2}
\left[
\|a_s\|_2\sqrt{r}
+
r\log(m+1)
\right]
\\
&\lesssim
D_{\psi}(S_{\beta,0} \vee 1)^{1/2}
\left[
\sqrt{mr}
+
r\log n
\right]
\\
&\lesssim
D_{\psi}(S_{\beta,0} \vee 1)^{1/2}\sqrt{mr}.
\end{align*}
In second inequality we used that the constant $C_{\beta, r}$ of Lemma~\ref{lem:uified_bounds} is bounded, up to an absolute constant, by $S_{\beta,0}\vee 1$, and use this fact without further comment below. 
In the last step, we used \(q\log n\le\sqrt n\),
\(m\ge n/(4q)\), and \(r\le q\), which give
\[
r\log n
\le
\frac{r\sqrt n}{q}
\le
2\frac{r\sqrt{mq}}q
=
2\sqrt{mr}\sqrt{\frac rq}
\le
2\sqrt{mr}.
\]
Consequently, since $mq \le n$,
\[
\left\|
n^{-1/2}\|\tilde{\Delta}^{(s,k)}_2\|_{\Hb_k}
\right\|_{\psi_1}
\lesssim
D_\psi (S_{\beta,0} \vee 1)^{1/2}
\left(\frac rq\right)^{1/2}.
\]
By the Orlicz tail bound~\ref{orlicz:tail_bound} and the union bound, there exists a sufficiently large absolute constant \(D_2>0\) such that
\begin{align} \label{eq:bound-e2}
\Prob\left(
\max_{k\in[K]}\max_{s\in[n]}
\big\|
n^{-1/2}\tilde\Delta^{(s,k)}_2
\big\|_{\Hb_k}
\ge
D_2 D_\psi (S_{\beta,0} \vee 1)^{1/2}
\left(\frac rq\right)^{1/2}
\log(nK)
\right)
\le
\frac{1}{3n}.
\end{align}
Let \(E_2\) denote the complement of the event inside the probability above.

Next, by Lemma~\ref{lem:uified_bounds}(3) and~\eqref{bound_norm_a_s}, for all \((s,k)\in[n]\times[K]\), 
\[
\Big\|
\big\|\tilde{\Delta}^{(s,k)}_1\big\|_{\Hb_k}
\Big\|_{\psi_1}
=
\Big\|
\Big\|
\sum_{\ell=1}^{m} a_{s\ell}\tilde{S}_{\ell k}
\Big\|_{\Hb_k}
\Big\|_{\psi_1}
\lesssim
D_\psi (S_{\beta,0} \vee 1)^{1/2}\sqrt{mq}.
\]
Therefore, using the bound on $|C_n|$ from \eqref{eq:bound-on-cn},
we obtain
\[
\left\|
|C_n|\,
\big\|\tilde{\Delta}^{(s,k)}_1\big\|_{\Hb_k}
\right\|_{\psi_1}
\lesssim
D_\psi (S_{\beta,0}  \vee 1)^{1/2}
\left(
\frac rq\vee \frac{q^2}{n}
\right)^{1/2}.
\]
Thus, again by~\ref{orlicz:tail_bound} and the union bound, there exists a sufficiently large absolute constant \(D_1>0\) such that
\begin{align} \label{eq:bound-e1}
\Prob\left(
|C_n|
\max_{k\in[K]}\max_{s\in[n]}
\|\tilde{\Delta}^{(s,k)}_1\|_{\Hb_k}
\ge
D_1D_\psi (S_{\beta,0} \vee 1)^{1/2}
\left(
\frac rq\vee \frac{q^2}{n}
\right)^{1/2}
\log(nK)
\right)
\le
\frac{1}{3n}.
\end{align}
Let \(E_1\) denote the complement of the event inside the probability above.

We note for a different proof that we have also just shown that there exists a sufficiently large constant $A_1$ that depends only on $C_\beta$ and $c_\beta$ (through $S_{\beta,0}$) and on $D_\psi$ and an event $\tilde E_1$ of probability larger than $1-1/(3n)$ such that, on $\tilde E_1$, it holds that 
\begin{equation}\label{bound_on_tilde_T}
\max_{k\in[K]}\max_{s \in [n]}\widetilde T_{sk}
=
\frac{1}{\sqrt{mq}} \max_{k\in[K]}\max_{s\in[n]}
\|\tilde{\Delta}^{(s,k)}_1\|_{\Hb_k}
\leq A_1\log(nK),
\end{equation}
where the first equality has been stated at the beginning of this proof.

It remains to bound the first term on the right-hand side of \eqref{eq:berbee_reduction_three_terms}. For any \((s,k)\in[n]\times[K]\),
\begin{align*}
\Big\|
n^{-1/2}
\|\Delta^{(s,k)}_3\|_{\Hb_k}
\Big\|_{\psi_1}
&=
n^{-1/2}
\Big\|
\Big\|
\sum_{i\in[n]}c_{si}\eps^{(i)}_k
-
\sum_{\ell=1}^{m}a_{s\ell}
(S_{\ell k}+S'_{\ell k})
\Big\|_{\Hb_k}
\Big\|_{\psi_1}
\\
&\quad\leq
n^{-1/2}
\Big\|
\Big\|
\sum_{i=1}^{m(q+r)}c_{si}\eps^{(i)}_k
-
\sum_{\ell=1}^{m}a_{s\ell}
(S_{\ell k}+S'_{\ell k})
\Big\|_{\Hb_k}
\Big\|_{\psi_1}
\\
&\hspace{4cm}+
n^{-1/2}
\Big\|
\Big\|
\sum_{i=m(q+r)+1}^{n}c_{si}\eps^{(i)}_k
\Big\|_{\Hb_k}
\Big\|_{\psi_1}.
\end{align*}
Since $n-m(q+r)<q+r\le2q$
and \(|c_{si}|\le1\), the second term is bounded by
$
2D_\psi qn^{-1/2}.
$
For the first term, observe that, except for indices lying in the single big block or small block intersecting the cutoff location \(s\), the coefficient \(c_{si}\) agrees with the corresponding block coefficient \(a_{s\ell_i}\), where \(\ell_i\) is the block index containing \(i\). Hence the difference contains at most \(q+r\le2q\) summands, each with coefficient bounded in absolute value by \(1\). Therefore,
\[
n^{-1/2}
\Big\|
\Big\|
\sum_{i=1}^{m(q+r)}c_{si}\eps^{(i)}_k
-
\sum_{\ell=1}^{m}a_{s\ell}
(S_{\ell k}+S'_{\ell k})
\Big\|_{\Hb_k}
\Big\|_{\psi_1}
\le
2D_\psi\frac{q}{\sqrt n}.
\]
Combining the two bounds gives
\[
\Big\|
n^{-1/2}
\|\Delta^{(s,k)}_3\|_{\Hb_k}
\Big\|_{\psi_1}
\le
4D_\psi
\Big(\frac{q^2}{n}\Big)^{1/2}.
\]
Therefore, by~\ref{orlicz:tail_bound} and the union bound, there exists a sufficiently large absolute constant \(D_3>0\) such that
\begin{align} \label{eq:bound-e3}
\Prob\bigg(
\max_{k\in[K]}\max_{s\in[n]}
\big\|
n^{-1/2}\Delta_3^{(s,k)}
\big\|_{\Hb_k}
\ge
D_3D_\psi
\Big(\frac{q^2}{n}\Big)^{1/2}
\log(nK)
\bigg)
\le
\frac{1}{3n}.
\end{align}
Let \(E_3\) denote the complement of the event inside the probability above. Now define
\[
\Omega_1:=E_1\cap E_2\cap E_3.
\]
By the three preceding probability bounds,
$
\Prob(\Omega_1)
\ge
1-1/n$.
On \(\Omega_1\cap\Omega_{\mathrm{block}}\) with $\Omega_{\mathrm{block}}$ from \eqref{eq:long-bound:)},
the reduction~\eqref{eq:berbee_reduction_three_terms} and the bounds from \eqref{eq:bound-e2}, \eqref{eq:bound-e1} and \eqref{eq:bound-e3}  imply that
\begin{align*}
\left|
T_n^O
-
\max_{k\in[K]}\max_{s \in [n]}\widetilde T_{sk}
\right|
&\le
D_2 D_\psi (S_{\beta,0} \vee 1)^{1/2}
\left(\frac rq\right)^{1/2}
\log(nK)
\\
&\quad+
D_1 D_\psi (S_{\beta,0} \vee 1)^{1/2}
\left(
\frac rq\vee \frac{q^2}{n}
\right)^{1/2}
\log(nK)
\\
&\quad+
D_3 D_\psi
\left(\frac{q^2}{n}\right)^{1/2}
\log(nK).
\end{align*}
This yields the claimed bound with $C_1 = D_1 + D_2 + D_3$.
\end{proof}

\begin{proof}[Proof of Lemma~\ref{lem:coupled_big_block_boundary}]
Define
\[
\eta_x
:=
\max_{k\in[K]}
\max_{s\in[n]\setminus S_x}
\widetilde T_{sk}.
\]
We need to prove that \(\Prob(\eta_x>\sqrt t)\le 2/n\).

Fix \((s,k)\in[n]\times[K]\). By the definition of \(\widetilde T_{sk}\) in \eqref{eq:tilde-tsk}, Lemma~\ref{lem:uified_bounds} applied with \(Q=\mathrm{id}_{\Hb_k}\), and the definition of \(S_{\beta,0}\) from \eqref{eq:def-S_beta}, and the inequalities
\(C_{\beta,q}\le S_{\beta,0}\) and
\(\|\|\varepsilon_k^{(1)}\|_{\Hb_k}\|_{\psi_1}\le D_\psi\), there exists a universal constant \(C>0\) such that
\begin{align}
\big\|\widetilde T_{sk}\big\|_{\psi_1}
&=
(mq)^{-1/2}
\Big\|
\Big\|
\sum_{\ell=1}^{m}a_{s\ell}\widetilde S_{\ell k}
\Big\|_{\Hb_k}
\Big\|_{\psi_1}
\nonumber\\
&\le
C D_\psi (S_{\beta,0}\vee 1)^{1/2}
\left[
\frac{\|a_s\|_2}{\sqrt m}
+
\sqrt{\frac qm}\log(m+1)
\right].
\label{eq:boundary_fixed_sk_orlicz}
\end{align}
Here \(\|a_s\|_2\) denotes the Euclidean norm of
\(a_s=(a_{s1},\dots,a_{sm})\in\mathbb R^m\). We next simplify the right hand side of~\eqref{eq:boundary_fixed_sk_orlicz}. By Lemma~\ref{lem:prop a_s_ell}, recalling $V(u) =u-u^2$, we have
\[
\|a_s\|_2^2
\le
mV(s/n)+2.
\]
Therefore, by subadditivity of the square root,
\[
\frac{\|a_s\|_2}{\sqrt m}
\le
\sqrt{V(s/n)}+\sqrt{\frac2m}.
\]
Under our assumptions \(1\le r\le q\) and \(q\le n/4\), we have
\(q+r\le 2q\le n/2\), and hence $m = \lfloor n/(q+r) \rfloor \ge 2$.
Thus \(\sqrt{2/m}\lesssim \sqrt{q/m}\log(m+1)\), since \(q\ge1\) and
\(\log(m+1)\ge1\). Also, since \(m\ge n/\{2(q+r)\}\ge n/(4q)\) and
\(m+1\le n+1\le 2n\), we have
\[
\sqrt{\frac qm}\log(m+1)
\lesssim
\frac{q\log n}{\sqrt n}.
\]
Consequently,
\begin{align}
\frac{\|a_s\|_2}{\sqrt m}
+
\sqrt{\frac qm}\log(m+1)
&\lesssim
\sqrt{V(s/n)}
+
\frac{q\log n}{\sqrt n}.
\label{eq:boundary_weight_bound_all_s}
\end{align}

If \(s\in[n]\setminus S_x\), then \(V(s/n)<x\) as remarked in~\eqref{eq:infsup-V-S_x}. Hence, by
\eqref{eq:boundary_fixed_sk_orlicz} and~\eqref{eq:boundary_weight_bound_all_s},
\[
\big\|\widetilde T_{sk}\big\|_{\psi_1}
\lesssim
D_\psi(S_{\beta,0}\vee1)^{1/2}
\left[
\sqrt{x}
+
\frac{q\log n}{\sqrt n}
\right]
\lesssim
D_\psi(S_{\beta,0}\vee1)^{1/2}
B_x,
\]
where
\[
B_x
:=
\sqrt{x}\vee \frac{q\log n}{\sqrt n}.
\]
By the \(\psi_1\) tail bound~\ref{orlicz:tail_bound}, there exists a universal
constant \(c>0\) such that, for every \(u>0\),
\[
\Prob(|\widetilde T_{sk}|>u)
\le
2\exp\left\{
-\frac{c u}{D_\psi(S_{\beta,0}\vee1)^{1/2}B_x}
\right\}.
\]
Taking \(u=\sqrt t \) and applying the union bound over at most \(nK\) pairs
\((s,k)\), we get
\[
\Prob(\eta_x> \sqrt t )
\le
2nK
\exp\left\{
-\frac{c \sqrt t }{D_\psi(S_{\beta,0}\vee1)^{1/2}B_x}
\right\}.
\]
By the hypothesis of the lemma,
\[
B_x
\le
\frac{\sqrt t }{C_2^{1/2}D_\psi(S_{\beta,0}\vee1)^{1/2}\log(nK)}.
\]
Therefore,
\[
\Prob(\eta_x>\sqrt t)
\le
2nK
\exp\left\{
-cC_2\log(nK)
\right\}.
\]
Choosing the universal constant \(C_2>0\) sufficiently large yields  the claim.
\end{proof}

\begin{proof}[Proof of Lemma~\ref{lem:coupled_big_block_projection_error}]
To lighten notation let
\[
P^{\perp}_{kM}=(\id_{{\Hb}_k}-P_{kM}),
\]
and recall that
$
\|P_{kM}x\|^2_{\Hb_k}
+
\|P^{\perp}_{kM} x\|^2_{\Hb_k}
=
\|x\|^2_{\Hb_k}
$
for every $x\in \Hb_k$
as noted in \eqref{eq:orthogonal}. Therefore, 
defining 
\[
R_{sk}
:=
(mq)^{-1/2}
\Big\|
P_{kM}^{\perp}
\sum_{\ell=1}^m a_{s\ell}\widetilde S_{\ell k}
\Big\|_{\Hb_k}, \qquad (s,k)\in S\times[K]
\]
and recalling the notation from \eqref{eq:tilde-tsk} and \eqref{eq:tilde-tskm},
we obtain the decomposition
$
(\widetilde T_{sk})^2
=
(\widetilde T^M_{sk})^2+R_{sk}^2.
$
Consequently,
\begin{align}
\max_{k\in[K]}\max_{s\in S}
(\widetilde T_{sk})^2
\leq
\max_{k\in[K]}\max_{s\in S}
(\widetilde T^M_{sk})^2+
\max_{k\in[K]}\max_{s\in S}
(R_{sk})^2.
\label{eq:projection-error-deterministic-reduction}
\end{align}
By Lemma~\ref{lem:uified_bounds}(4), applied with
\(Q=P_{kM}^{\perp}\) and $U=D_\psi$, it holds that, for the absolute
constant $C=C_4$ of Lemma ~\ref{lem:uified_bounds} (4),
\begin{align}
&
\bigg\|
\Big\|
P_{kM}^{\perp}
\sum_{\ell=1}^m a_{s\ell}\widetilde S_{\ell k}
\Big\|_{\Hb_k}
\bigg\|_{\psi_1}
\le
C\left(
\E\Big[\Big\|\sum_{\ell=1}^m a_{s\ell}\cdot
P_{kM}^{\perp}\widetilde S_{\ell k}\Big\|_{\Hb_k}\Big]
+\log(m+1)qD_\psi
\right).
\label{eq:projected_residual_block_orlicz}
\end{align}
Since \(1\le r\le q\le n/4\), we have \(m\ge n/(4q)\).
Together with \(m+1\le2n\) and \(\|a_s\|_\infty\le1\), this gives
\begin{equation}
\frac{\|a_s\|_\infty qD_\psi\log(m+1)}{\sqrt{mq}}
\leq
C D_\psi\frac{q\log n}{\sqrt n}.
\label{eq:projection-error-large-block-term}
\end{equation}
Moving onto the first term in~\eqref{eq:projected_residual_block_orlicz},
by the Cauchy--Schwarz inequality and independence and centeredness of the
sequence $\widetilde S_{1k},\dots,\widetilde S_{mk}$, it holds that
\begin{align}
\E\Big[\Big\|\sum_{\ell=1}^m a_{s\ell}\cdot
P_{kM}^{\perp}\widetilde S_{\ell k}\Big\|_{\Hb_k}\Big]
\leq
\left(\sum_{\ell=1}^m a_{s\ell}^2
\E\left[\big\|P_{kM}^{\perp}\widetilde S_{\ell k}\big\|^2_{\Hb_k}\right]
\right)^{1/2}.
\label{eq:projection-error-expectation}
\end{align}
Further, by Lemma~\ref{lem:properties-covariance}(2), it holds that
\begin{align*}
\E\left[\big\|P_{kM}^{\perp}\widetilde S_{\ell k}\big\|^2_{\Hb_k}\right]
&=
\left\|\Cov\left(P_{kM}^{\perp}\widetilde S_{\ell k}\right)\right\|_{\Tr}
\nonumber\\
&=
q\left\|\Cov\left(q^{-1/2}P_{kM}^{\perp}\widetilde S_{\ell k}\right)\right\|_{\Tr}
\nonumber\\
&=
q\left\|
P_{kM}^{\perp}
\left(\Cov\left(q^{-1/2}\widetilde S_{\ell k}\right)-\mathcal K_k\right)
(P_{kM}^{\perp})^*
+P_{kM}^{\perp}\mathcal K_k(P_{kM}^{\perp})^*
\right\|_{\Tr},
\end{align*}
where the final equality is obtained by adding and subtracting
\(P_{kM}^{\perp}\mathcal K_k(P_{kM}^{\perp})^*\).
By the triangle inequality for the trace norm,
\begin{align*}
&\left\|
P_{kM}^{\perp}
\left(\Cov\left(q^{-1/2}\widetilde S_{\ell k}\right)-\mathcal K_k\right)
(P_{kM}^{\perp})^*
+P_{kM}^{\perp}\mathcal K_k(P_{kM}^{\perp})^*
\right\|_{\Tr}
\nonumber\\
&\qquad\leq
\left\|
P_{kM}^{\perp}
\left(\Cov\left(q^{-1/2}\widetilde S_{\ell k}\right)-\mathcal K_k\right)
(P_{kM}^{\perp})^*
\right\|_{\Tr}
+
\left\|P_{kM}^{\perp}\mathcal K_k(P_{kM}^{\perp})^*\right\|_{\Tr}.
\end{align*}

Since $P_{kM}^{\perp}$ is the orthogonal projection onto the
orthogonal complement of the first $M\wedge r_k$ eigenvectors of
$\mathcal K_k$, the positive eigenvalues of
$P_{kM}^{\perp}\mathcal K_kP_{kM}^{\perp}$ are
$\{\lambda_{kj}:M<j\le r_k\}$. Hence, using
Assumption~\ref{cond:scores},
\begin{align*}
\left\|P_{kM}^{\perp}\mathcal K_k(P_{kM}^{\perp})^*\right\|_{\Tr}
&=
\sum_{j=M+1}^{r_k}\lambda_{kj}
\le
C_U^2\sum_{j=M+1}^{\infty}j^{-2\gamma_U}
\nonumber\\
&\le
C_U^2\int_M^\infty u^{-2\gamma_U}\,du
=
\frac{C_U^2}{2\gamma_U-1}M^{1-2\gamma_U},
\end{align*}
where the integral is finite because $\gamma_U>1/2$. Since \(\widetilde S_{\ell k}\) has the same distribution as
\(\sum_{i=1}^q\varepsilon_k^{(i)}\),
Lemma~\ref{lem:convergence_to_long_term_trace} gives
\[
\left\|
\Cov\left(q^{-1/2}\widetilde S_{\ell k}\right)-\mathcal K_k
\right\|_{\Tr}
\leq
64D_\psi^2S_{\beta,1}\frac1q.
\]
Recall the following property of the trace norm: if \(T\) is trace class
and \(A,B\) are bounded operators, then
\[
\|ATB\|_{\Tr}
\le
\|A\|_{\op}\|T\|_{\Tr}\|B\|_{\op};
\]
see, for example, \citet[Chapter~2]{simon2005trace}. 
Since
\(P_{kM}^{\perp}\) is an orthogonal projection, its operator norm is at
most one. Consequently,
\begin{align*}
&\left\|
P_{kM}^{\perp}
\left\{
\Cov\left(q^{-1/2}\widetilde S_{\ell k}\right)-\mathcal K_k
\right\}
(P_{kM}^{\perp})^*
\right\|_{\Tr}
\le
\left\|
\Cov\left(q^{-1/2}\widetilde S_{\ell k}\right)-\mathcal K_k
\right\|_{\Tr}
\le
64D_\psi^2S_{\beta,1}\frac1q.
\end{align*}
Therefore, there is a constant $C_3$ depending only on
$D_\psi,C_U,\gamma_U,C_\beta,c_\beta$ such that, for all $\ell\in[m]$,
\begin{equation*}
\frac1q\E\left[
\big\|P_{kM}^{\perp}\widetilde S_{\ell k}\big\|^2_{\Hb_k}
\right]
\leq
C_3\Big(M^{1-2\gamma_U}+\frac1q\Big).
\end{equation*}
Hence, from \eqref{eq:projection-error-expectation},
\begin{align*}
q^{-1/2}\E\Big[\Big\|\sum_{\ell=1}^m a_{s\ell}\cdot
P_{kM}^{\perp}\widetilde S_{\ell k}\Big\|_{\Hb_k}\Big]
\leq
C_3^{1/2}\Big(M^{1-2\gamma_U}+\frac1q\Big)^{1/2}
\Big(\sum_{\ell=1}^m a_{s\ell}^2\Big)^{1/2}.
\end{align*}
Since, by Lemma~\ref{lem:prop a_s_ell},
\[
\Big(\sum_{\ell=1}^m a_{s\ell}^2\Big)^{1/2}
\leq \frac32\sqrt m,
\]
there is a constant $C_4$ such that
\begin{equation}
(mq)^{-1/2}\E\Big[\Big\|\sum_{\ell=1}^m a_{s\ell}\cdot
P_{kM}^{\perp}\widetilde S_{\ell k}\Big\|_{\Hb_k}\Big]
\leq
C_4\Big(M^{1-2\gamma_U}+\frac1q\Big)^{1/2}.
\label{eq:projection-error-expectation-final}
\end{equation}
Combining \eqref{eq:projected_residual_block_orlicz},
\eqref{eq:projection-error-large-block-term}, and
\eqref{eq:projection-error-expectation-final}, and using
\(\sqrt{x}+\sqrt{y}\le\sqrt{2(x+y)}\), gives a constant
\(C_5=C_5(D_\psi,C_U,\gamma_U,C_\beta,c_\beta)>0\) such that
\begin{equation*}
\left\|R_{sk}\right\|_{\psi_1}
\leq
C_5\Big(
M^{1-2\gamma_U}+\frac1q+\frac{q^2\log^2 n}{n}
\Big)^{1/2}.
\end{equation*}
There are at most \(nK\) pairs in \(S\times[K]\). Hence, by the
\(\psi_1\) tail bound~\ref{orlicz:tail_bound} and the union bound,
there exists a sufficiently large constant
\(C_6=C_6(D_\psi,C_U,\gamma_U,C_\beta,c_\beta)>0\) such that
\begin{align*}
\mathbb P\Bigg(&
\max_{k\in[K]}\max_{s\in S}R_{sk}
>
C_6\log(nK)
\Big(M^{1-2\gamma_U}+\frac1q+\frac{q^2\log^2 n}{n}\Big)^{1/2}
\Bigg)
\leq\frac2n.
\end{align*}
Consequently, outside an event of probability at most \(2/n\),
\[
\max_{k\in[K]}\max_{s\in S}R_{sk}^2
\leq
C_6^2\log^2(nK)
\Big(M^{1-2\gamma_U}+\frac1q+\frac{q^2\log^2 n}{n}\Big).
\]
The conclusion follows from
\eqref{eq:projection-error-deterministic-reduction} upon taking
\(C_{\mathrm{proj}}:=C_6^2\).
\end{proof}

\begin{proof}[Proof of Lemma~\ref{lem:gaussian_approx_projected_coupled}]
To lighten notation, we write $S=S_x$ throughout the proof.
Define
\[
G^M
:=
\bigoplus_{s\in S}
\bigoplus_{k \in [K]}
G^M_{sk}
\in
\bigoplus_{s\in S}
\bigoplus_{k \in [K]}
\mathbb R^{\check M_k}.
\]
For each \(\ell\in[m]\), define
\[
X_\ell
:=
\bigoplus_{s\in S}
\bigoplus_{k \in [K]}
X_{\ell sk}
\in
\bigoplus_{s\in S}
\bigoplus_{k \in [K]}
\mathbb R^{\check M_k},
\]
where, recalling $O_{kM}$ from \eqref{eq:def-OkM},
\[
X_{\ell sk}
:=
q^{-1/2}a_{s\ell}
(O_{kM}\circ P_{kM})\widetilde S_{\ell k}.
\]
Further, let
\[
S_m^X
:=
m^{-1/2}\sum_{\ell=1}^m X_\ell,
\qquad
S^X_{msk}:=m^{-1/2}\sum_{\ell=1}^m X_{\ell sk}.
\]
Since \(O_{kM}\) is an isometry and by definition of $\widetilde T^M_{sk}$ in \eqref{eq:tilde-tskm}, we have
\begin{equation}
\label{eq:projected-coupled-as-block-norm-tracked}
\max_{k\in[K]}\max_{s\in S}(\widetilde T^M_{sk})^2
=
\max_{s\in S}\max_{k\in[K]}
\|S^X_{msk}\|_{2}^2.
\end{equation}
The vectors \(X_1,\ldots,X_m\) are independent and centered because the
Berbee-coupled block sums are independent and centered.  Moreover, for
\((s_a,k_a,j_a)\), \(a=1,2\),
\begin{align*}
\Cov
\left(
[S_m^X]_{s_1k_1j_1},[S_m^X]_{s_2k_2j_2}
\right)
&=
\frac{\langle a_{s_1},a_{s_2}\rangle_{\mathbb R^m}}{mq}
\Cov
\left(
\left\langle Z_{k_1j_1},\pi_{k_1}S_1\right\rangle_{\Hb_{k_1}},
\left\langle Z_{k_2j_2},\pi_{k_2}S_1\right\rangle_{\Hb_{k_2}}
\right)
\\
&=
\Cov
\left(G^M_{s_1k_1j_1},G^M_{s_2k_2j_2}\right),
\end{align*}
where we used stationarity and
\(\widetilde S_\ell\stackrel d=S_\ell\), as well as \eqref{eq:covariance-GMsk} and \eqref{eq:covariance-GMsk-alternative-expresssion}.  Thus
\(\Cov(S_m^X)=\Cov(G^M)\). 

Let \(Y_1,\ldots,Y_m\) be independent centered Gaussian random elements,
independent of the \(X_\ell\)'s, with
\(\Cov(Y_\ell)=\Cov(X_\ell)\), and set
\(S_m^Y=m^{-1/2}\sum_{\ell=1}^mY_\ell\).  Then
\[
\Cov(S_m^Y)
=\frac1m\sum_{\ell=1}^m\Cov(Y_\ell)
=\Cov(S_m^X)
=\Cov(G^M).
\]
Since \(S_m^Y\) and \(G^M\) are centered Gaussian vectors,
\begin{equation}
\label{eq:matched-gaussian-law-tracked}
S_m^Y\stackrel d=G^M.
\end{equation}

We next verify the hypotheses of
Lemma~\ref{new:high_dimensional_CLT_mixed_norms}, with $n$ in that lemma replaced by $m$.  By
Corollary~\ref{lem:good_eigen_values_2}, for every
\((s,k)\in S\times[K]\),
\begin{align}
\label{eq:decay-proof}
\Cov(S^X_{msk})
=
\Cov(G^M_{sk})
\in
\operatorname{Decay}
\left(
\frac{xC_L^2}{2},
2C_U^2,
2\gamma_L,
2\gamma_U
\right).
\end{align}
To verify the condition involving \(J_m\) in \eqref{eq:Jn}, fix
\(v\in\mathbb R^{\check M_k}\) with \(\|v\|_{\check M_k}=1\), and put
\[
\zeta_i
:=
\left\langle
v,(O_{kM}\circ P_{kM})\varepsilon_k^{(i)}
\right\rangle.
\]
For a real-valued random variable \(W\), write
\[
Q'(W,u):=\inf\{t\in\mathbb R:\mathbb P(|W|>t)\le u\},
\qquad u\in(0,1).
\]
Because \(O_{kM}\) is an isometry and \(P_{kM}\) is a contraction,
\(|\zeta_i|\le\|\varepsilon_k^{(i)}\|_{\Hb_k}\), and hence
\(Q'(\zeta_i,u)\le D_\psi\log(2/u)\). Indeed, \ref{orlicz:tail_bound} applied in case $p=1$ gives 
\[
\mathbb P(|W|>D_\psi\log(2/u))\leq 2\exp\left(-\log(2/u)\right)\leq u,
\]
and this implies the quantile bound written above.  Set
\[
C_R
:=
\int_0^1
\left(\inf\{h\in\mathbb N:\beta_h\le2u\}\right)^2
\log^4(2/u)\,du.
\]
The polynomial mixing bound from \ref{cond:stationary} gives
\[
\inf\{h\in\mathbb N:\beta_h\le2u\}
\le
1+\left(\frac{C_\beta}{2u}\right)^{1/c_\beta}
\le 
C_1 u^{-1/c_\beta},
\]
where $C_1 = 1 + (C_\beta/2)^{1/c_\beta}$. Hence, recalling that $c_\beta>2$, $C_R$ can be bounded by a constant that only depends on \(C_\beta,c_\beta\).  Since \(\alpha_h\le\beta_h/2\),
Theorem~2.5 of~\cite{rio2017asymptotic} yields
\begin{align} \label{eq:rio1}
\mathbb E\Big|\sum_{i\in I_\ell}\zeta_i\Big|^4
\le 256q^2D_\psi^4C_R.
\end{align}
Using \(|a_{s\ell}|\le1\),
\(\widetilde S_\ell\stackrel d=S_\ell\), and Lyapunov's inequality,
we therefore obtain, uniformly in \((s,k,\ell)\) and \(v\),
\[
\mathbb E|\langle v,X_{\ell sk}\rangle|^3
=
\mathbb E\Big| q^{-1/2}\sum_{i \in I_\ell} a_{s\ell} \zeta_i\Big|^3
\le
q^{-3/2}
\Big(
\mathbb E\Big|\sum_{i\in I_\ell}\zeta_i\Big|^4
\Big)^{3/4}
\le
64D_\psi^3C_R^{3/4}.
\]
Consequently, we may take
\[
J_m=C_J(D_\psi^3C_R^{3/4}\vee1),
\]
where \(C_J\) is universal and no larger than 64.

Next we identify \(H_m\) from \eqref{eq:Hn3}. For every
\((s,k,\ell)\in S\times[K]\times[m]\),
\begin{align*}
\big\|\|X_{\ell sk}\|_2\big\|_{\psi_1}
&\le
q^{-1/2}
\left\|\|\widetilde S_{\ell k}\|_{\Hb_k}\right\|_{\psi_1}
\\
&=
q^{-1/2}
\left\|\|S_{\ell k}\|_{\Hb_k}\right\|_{\psi_1}
\le
q^{-1/2}\sum_{i\in I_\ell}
\left\|\|\varepsilon_k^{(i)}\|_{\Hb_k}\right\|_{\psi_1}
\le
D_\psi\sqrt q.
\end{align*}
Hence we may choose
\[
H_m
:=
C_H\bigl(D_\psi^3C_R^{3/4}\vee D_\psi\vee1\bigr)\sqrt q,
\]
where \(C_H\) is universal and sufficiently large that
\(1\le J_m\le H_m\).

We can now apply Lemma~\ref{new:high_dimensional_CLT_mixed_norms} with
sample size \(m\), block index set \(S\times[K]\), and largest block
dimension at most \(M\).  Since \(|S|K\le nK\) and the lower eigenvalue
constant is \(C_L\sqrt{x/2}\) by \eqref{eq:decay-proof}, that lemma gives
\begin{align}
&d_{\mathrm K}^{(\rho)}
\left(
\max_{s\in S}\max_{k\in[K]}
\|S^X_{msk}\|_{2}^2,
\max_{s\in S}\max_{k\in[K]}
\|S^Y_{msk}\|_{2}^2
\right)
\nonumber\\
&\quad\le
\frac{C}{x^{3/2}}
J_m
\left(\frac{M^4H_m^2}{m}\right)^{1/6}
\log^{3\kexp+2}(nK)
\log^{3\kexp + 3/2}(m),
\label{eq:mixed-norm-clt-application-tracked}
\end{align}
where \(C\) depends only on
\(C_L,C_U,\gamma_L,\gamma_U,\rho\).
Indeed, the factor involving the lower eigenvalue constant is
\[
\bigl(C_L\sqrt{x/2}\bigr)^{-3}
=2^{3/2}C_L^{-3}x^{-3/2},
\]
which accounts for the displayed power of \(x\).

Finally, because \(r\le q\le n/4\), we have $m \ge n/(4q)$. Moreover, the quantities \(J_m\) and \(H_m/\sqrt q\) are bounded by constants
depending only on \(D_\psi,C_\beta,c_\beta\), which yields
\[
J_m
\left(\frac{M^4H_m^2}{m}\right)^{1/6}
\le
C\left(\frac{M^4q}{m}\right)^{1/6}
\le
C\left(\frac{M^4q^2}{n}\right)^{1/6}.
\]
Finally,
$
\log^{3 \kexp + 3/2}(m) \le \log^{3 \kexp + 3/2}(n)$.
Combining these estimates with
\eqref{eq:projected-coupled-as-block-norm-tracked},
\eqref{eq:matched-gaussian-law-tracked}, and
\eqref{eq:mixed-norm-clt-application-tracked} proves the claim.
\end{proof}

\subsection{Conditional Comparison}
Throughout this section, the residual, oracle, and coupled oracle multiplier processes are constructed using the same iid standard Gaussian multiplier sequence
\(e_1,\ldots,e_m\) that is independent of
$\mathcal F_n\vee\widetilde{\mathcal F}_n$
from \eqref{def:join_data_coupled}. 
Recall the residual bootstrap procedure introduced in
Section~\ref{subsec:bootstrap}, with is a gap-length tuning parameter \(g=g_n\) . In the theory below we allow general \(g_n\).

We start by defining an oracle multiplier bootstrap process.
Let $q,r\in\mathbb N$ satisfy $1\le r\leq n/4$, and let
\(I_\ell\), \(\ell\in[m]\), denote the big blocks defined
in~\eqref{def:big_little_blocks}. Recall the residual bootstrap multiplier partial sum and the bootstrap CUSUM process from \eqref{def:bootstrap-partial-sum} and \eqref{def:bootstrap-cusum}, respectively, defined for each \(k\in[K]\) and \(s\in[n]\) by
\begin{equation} \label{eq:Ckstar}
    S_k^*(s)
    :=
    \sum_{\ell=1}^{m}
    e_\ell
    \sum_{i\in I_\ell}
    \widehat\varepsilon_{i,k}(g) \mathbf 1 \{ i \le s \},
    \qquad
    C_k^*(s)
    :=
    \frac{1}{\sqrt{mq}}
    \left\{
    S_k^*(s)-\frac{s}{n}S_k^*(n)
    \right\}.
\end{equation}
We then introduce the oracle analogues:
\[
    S_k^{O*}(s)
    :=
    \sum_{\ell=1}^{m}
    e_\ell
    \sum_{i\in I_\ell}
    \varepsilon_k^{(i)}\mathbf 1\{i\le s\},
    \qquad 
    C_k^{O*}(s)
    :=
    \frac{1}{\sqrt{mq}}
    \left\{
    S_k^{O*}(s)
    -
    \frac{s}{n}S_k^{O*}(n)
    \right\}.
\]
Note that we can write
\begin{align}
\label{def:oracle-ck}
    C_k^{O*}(s)
    =
    \frac{1}{\sqrt{mq}}
    \sum_{\ell=1}^{m}
    e_\ell
    \sum_{i\in I_\ell}
    \varepsilon_k^{(i)}
    \left\{
    \mathbf 1\{i\le s\}
    -
    \frac{s}{n}
    \right\}.
\end{align}
Analogously to the definition of $T_n^{*}$ from \eqref{eq:definition-Tstar}, define the oracle multiplier statistic by
\[
    T_n^{O*}
    :=
    \max_{k\in[K]}
    \max_{s\in[n]}
    \left\|C_k^{O*}(s)\right\|_{\Hb_k}.
\]
We will also use a block-aligned and Berbee-coupled oracle multiplier
process. For each \((s,k)\in[n]\times[K]\), define
\begin{equation}
\label{def:berbee-coupled-block-oracle}
    \widetilde C_k^{O*}(s)
    :=
    \frac{1}{\sqrt{mq}}
    \sum_{\ell=1}^m
    e_\ell \sum_{i\in I_\ell}
    \tilde \varepsilon_k^{(i)}
    \left\{
    \mathbf 1\{ \ell \le B_s\}
    -
    \frac{s}{n}
    \right\}
    =
    \frac{1}{\sqrt{mq}}
    \sum_{\ell=1}^m
    a_{s\ell}e_\ell \widetilde S_{\ell k}.
\end{equation}
Here \(\widetilde S_1,\ldots,\widetilde S_m\) are the Berbee-coupled
big-block sums defined in~\eqref{def:big_block_sums}, and
\(\widetilde S_{\ell k}:=\pi_k\widetilde S_\ell\in\Hb_k\). The coefficient
\(a_{s\ell}\) is defined in~\eqref{eq:def-a_sl}; thus
\(\widetilde C_k^{O*}(s)\) is the block-aligned version of the oracle
multiplier process, with whole big-block sums replacing partial block
contributions.

Set
\[
    \widetilde T_{sk}^{O*}
    :=
    \big\|\widetilde C_k^{O*}(s)\big\|_{\Hb_k}.
\]
For a truncation level \(M\), define the projected coordinate vector
\[
    \widetilde G_{sk}^{M,O*}
    :=
    (O_{kM}\circ P_{kM})
    \widetilde C_k^{O*}(s)
    \in\mathbb R^{\check M_k}.
\]
Equivalently,
\[
    \widetilde G_{sk}^{M,O*}
    =
    \frac1{\sqrt m}
    \sum_{\ell=1}^m
    a_{s\ell}e_\ell
    (O_{kM}\circ P_{kM})
    \left(q^{-1/2}\widetilde S_{\ell k}\right).
\]
Define
\[
    \widetilde T_{sk}^{M,O*}
    :=
    \big\|
    \widetilde G_{sk}^{M,O*}
    \big\|_2
    =
    \big\|
    P_{kM}\widetilde C_k^{O*}(s)
    \big\|_{\Hb_k}.
\]
Finally, define
\[
    \widetilde T_n^{O*}
    :=
    \max_{k\in[K]}
    \max_{s\in[n]}
    \widetilde T_{sk}^{O*},
    \qquad
    \widetilde T_x^{O*}
    :=
    \max_{k\in[K]}
    \max_{s\in S_x}
    \widetilde T_{sk}^{O*},
\]
and
\[
    \widetilde T_x^{M,O*}
    :=
    \max_{k\in[K]}
    \max_{s\in S_x}
    \widetilde T_{sk}^{M,O*}.
\]

\begin{lemma}[Residual--oracle comparison on a coordinate subset]
\label{lem:residual_oracle_adaptive_subset}
Suppose Assumption~\ref{cond:stationary} holds. 
Fix \(\varnothing\ne\mathcal J\subseteq[K]\) and suppose that the tuning parameters $q,r,g,\tau$ for the residual bootstrap from Section~\ref{subsec:bootstrap} satisfy
\begin{equation*}
   \tau \in (0,1/2) \qquad 1\le r\le q,
    \qquad
    g+q+r\le\frac{\tau n}{2}
\end{equation*}
and also that
\begin{equation*}
    \omega_k\in[\tau n,(1-\tau)n]
    \quad \forall\,
    k\in\mathcal J\cap\mathcal S,
\end{equation*}
where this condition is understood to be vacuous if
\(\mathcal J\cap\mathcal S=\varnothing\).
Then, for every $\ell \in [n]$, there exists an event 
\(\mathcal E_n\in\mathcal F_n\) satisfying
\begin{equation*}
    \mathbb P(\mathcal E_n^c)
    \le
    \frac1n
    +
    2\left\lceil\frac n{\ell}\right\rceil
    \beta_{\ell},
\end{equation*}
such that, on \(\mathcal E_n\),
\begin{align}
&\mathbb P\left(
    \max_{k\in\mathcal J}\max_{s\in[n]}
    \|C_k^*(s)-C_k^{O*}(s)\|_{\Hb_k}
    >
    \frac{C}{\sqrt \tau} R_{n,\mathcal J}^{\mathrm{ad}}(q,g,\ell,K)
    \,\middle|\,
    \mathcal F_n
\right) 
\le\frac2n,
\label{eq:adaptive-subset-conditional-comparison}
\end{align}
where, recalling \(\Delta_{\mathcal J,\mathrm{wk}}(g,\ell)\) from
\eqref{eq:def-delta-J-wk}, 
\begin{align}
R_{n,\mathcal J}^{\mathrm{ad}}(q,g,\ell,K)
:={}&
\sqrt{\log(nK)}
\Bigg[
\log(nK)
\sqrt{\frac{q+g}{n}}
\left\{
1+
\frac{\ell\log(en/\ell)}{\sqrt q}
\right\}
+
\sqrt q
\Delta_{\mathcal J,\mathrm{wk}}(g,\ell)
\Bigg]
\label{eq:def-adaptive-subset-residual-rate}
\end{align}
and where $C>0$ is a constant that only depends on $D_\psi,C_\beta,c_\beta$.
\end{lemma}

\begin{proof}[Proof of Lemma~\ref{lem:residual_oracle_adaptive_subset}]
Set 
\begin{align}
\label{eq:helper-notation1}
\ell_0:=\ell,
\qquad
h:=q+r,
\qquad
L_n
:=\ell_0\log\left(\frac{en}{\ell_0}\right)
.
\end{align}
For each \(k\in[K]\), define the (random) deleted index set
\[
 G_k(g)
:=
\left\{
i\in[n]:
 N_{L,k}(g)<i\le
 N_{R,k}(g)
\right\},
\]
where \( N_{L,k}(g)\) and \( N_{R,k}(g)\) are defined in~\eqref{def:clipped-LR-general}.
Thus, by~\eqref{def:trimmed-residuals-clipped},
$\widehat\varepsilon_{i,k}(g)=0$ for all 
$i\in G_k(g)$.
For \((i,k)\in[n]\times[K]\), define
\[
d_{i,k}
:=
\widehat\varepsilon_{i,k}(g)-\varepsilon_k^{(i)}.
\]
Then, for \((s,k)\in[n]\times[K]\),
\[
C_k^*(s)-C_k^{O*}(s)
=
\frac1{\sqrt{mq}}
\sum_{\ell=1}^m e_\ell
\sum_{i\in I_\ell}d_{i,k}
\left\{\mathbf1(i\le s)-\frac sn\right\},
\]
where \(C_k^*(s)\) is the bootstrap CUSUM process for coordinate \(k\)
defined in~\eqref{def:bootstrap-cusum}. For each
\((s,k)\in[n]\times\mathcal J\), define
\[
V_{sk}
:=
\bigg[
\frac1{mq}
\sum_{\ell=1}^m
\Big\|
\sum_{i\in I_\ell}d_{i,k}
\Big\{\mathbf1(i\le s)-\frac sn\Big\}
\Big\|_{\Hb_k}^2
\bigg]^{1/2}.
\]
Then set
\[
D_{n,\mathcal J}
:=
\max_{k\in\mathcal J}\max_{s\in[n]}V_{sk}.
\]
Every \(V_{sk}\), and therefore \(D_{n,\mathcal J}\), is
\(\mathcal F_n\)-measurable because it depends only on the data. Let
\(C_{\mathrm G}>0\) be the absolute constant from
Lemma~\ref{lem:gaussian_conditional_tail_bound}. Conditionally on
\(\mathcal F_n\), each \(C_k^*(s)-C_k^{O*}(s)\) is a centered Gaussian
random element of \(\Hb_k\), being a linear function of the independent
standard Gaussian multipliers \(e_1,\ldots,e_m\).
Lemma~\ref{lem:gaussian_conditional_tail_bound} gives, for every
\(y>0\) and every \((s,k)\in[n]\times\mathcal J\), almost surely,
\[
\mathbb P\big(
\|C_k^*(s)-C_k^{O*}(s)\|_{\Hb_k}
>
C_{\mathrm G}\sqrt y\,V_{sk}
\,\big|\,
\mathcal F_n
\big)
\le
2e^{-y}.
\]
Define $C_{\mathrm{mult}}
:=
\sqrt2\,C_{\mathrm G}$.
Taking \(y=2\log(nK)\), using \(V_{sk}\le D_{n,\mathcal J}\), and
applying the conditional union bound over the at most
\(n|\mathcal J|\le nK\) pairs \((s,k)\), we obtain
\begin{align*}
&\mathbb P\left(
\max_{k\in\mathcal J}\max_{s\in[n]}
\|C_k^*(s)-C_k^{O*}(s)\|_{\Hb_k}
>
C_{\mathrm{mult}}D_{n,\mathcal J}\sqrt{\log(nK)}
\,\middle|\,
\mathcal F_n
\right)
\\
&\quad\le
\sum_{k\in\mathcal J}\sum_{s=1}^n
\mathbb P\left(
\|C_k^*(s)-C_k^{O*}(s)\|_{\Hb_k}
>
C_{\mathrm G}\sqrt{2\log(nK)}\,D_{n,\mathcal J}
\,\middle|\,
\mathcal F_n
\right)
\\
&\quad\le
\sum_{k\in\mathcal J}\sum_{s=1}^n
\mathbb P\left(
\|C_k^*(s)-C_k^{O*}(s)\|_{\Hb_k}
>
C_{\mathrm G}\sqrt{2\log(nK)}\,V_{sk}
\,\middle|\,
\mathcal F_n
\right)
\\
&\quad\le
2n|\mathcal J|e^{-2\log(nK)}
\le
2nK(nK)^{-2}
=
\frac{2}{nK}
\le
\frac2n.
\end{align*}
Consequently,
\begin{align}
&\mathbb P\left(
\max_{k\in\mathcal J}\max_{s\in[n]}
\|C_k^*(s)-C_k^{O*}(s)\|_{\Hb_k}
>
C_{\mathrm{mult}}D_{n,\mathcal J}\sqrt{\log(nK)}
\,\middle|\,
\mathcal F_n
\right)
\le\frac2n.
\label{eq:adaptive-subset-multiplier-reduction}
\end{align}
It remains to bound \(D_{n,\mathcal J}\). 
For each \((s,k)\in[n]\times\mathcal J\), define
\begin{align*}
V_{sk}^{\mathrm{del}}
&:=
\bigg[
\frac1{mq}
\sum_{\ell=1}^m
\Big\|
\sum_{i\in I_\ell\cap G_k(g)}
d_{i,k}
\left\{\mathbf1(i\le s)-\frac sn\right\}
\Big\|_{\Hb_k}^2
\bigg]^{1/2},
\\
V_{sk}^{\mathrm{ret}}
&:=
\bigg[
\frac1{mq}
\sum_{\ell=1}^m
\Big\|
\sum_{i\in I_\ell\setminus G_k(g)}
d_{i,k}
\left\{\mathbf1(i\le s)-\frac sn\right\}
\Big\|_{\Hb_k}^2
\bigg]^{1/2}.
\end{align*}
Set
\begin{align}
\label{eq:def-Dn}
D_{n,\mathcal J}^{\mathrm{del}}
:=
\max_{k\in\mathcal J}\max_{s\in[n]}
V_{sk}^{\mathrm{del}},
\qquad
D_{n,\mathcal J}^{\mathrm{ret}}
:=
\max_{k\in\mathcal J}\max_{s\in[n]}
V_{sk}^{\mathrm{ret}}.
\end{align}
The decomposition
$
d_{i,k}
=
d_{i,k}\mathbf1\{i\in G_k(g)\}
+
d_{i,k}\mathbf1\{i\notin G_k(g)\}
$
and the triangle inequality on the Hilbert space
\(\bigoplus_{\ell=1}^m\Hb_k\), equipped with its direct-sum norm, give $V_{sk}
\le
V_{sk}^{\mathrm{del}}
+
V_{sk}^{\mathrm{ret}}$ and 
consequently
\begin{equation}
\label{eq:adaptive-subset-D-reduction}
D_{n,\mathcal J}
\le
D_{n,\mathcal J}^{\mathrm{del}}
+
D_{n,\mathcal J}^{\mathrm{ret}}.
\end{equation}

We first introduce a uniform interval-sum bound that will be used for
both terms. Apply Lemma~\ref{lem:uniform_interval_sums_new} with
\[
\ell=\ell_0,
\qquad
\eta=\frac1{12n},
\]
and define
$\mathcal E_n
:=
\mathcal E_{\ell_0,\,1/(12n)}^{\mathrm{int}}$.
Then \(\mathcal E_n\in\mathcal F_n\) and
\begin{align*}
\mathbb P(\mathcal E_n^c)
&\le
\frac1{12n}
+
2\left\lceil
\frac n
{\ell_0}
\right\rceil
\beta_{
\ell_0
}\le
\frac1n
+
2\left\lceil
\frac n
{\ell_0}
\right\rceil
\beta_{
\ell_0}.
\end{align*}
Moreover, on \(\mathcal E_n\), simultaneously for every
\((k,I)\in[K]\times\mathcal I_n\),
\begin{align*}
\bigg\|
\sum_{i\in I}\varepsilon_k^{(i)}
\bigg\|_{\Hb_k}
&\le
C_{\mathrm{int}}\log(144n^3K)
\left\{
V\sqrt{|I|}
+
D_\psi L_n
\right\},
\end{align*}
where $L_n$ is defined in \eqref{eq:helper-notation1}, \(\mathcal I_n\) denotes the collection of nonempty integer
intervals in \([n]\), \(C_{\mathrm{int}}>0\) is the absolute constant
from Lemma~\ref{lem:uniform_interval_sums_new}, and \(V\) is the
dependence constant from that lemma. As stated there, \(V\) depends
only on \(C_\beta,c_\beta,D_\psi\).
The tuning conditions in \eqref{eq:adaptive-subset-block-conditions} imply \(n>8\), because
$
2
\le
q+r
\le
g+q+r
\le
\tau n/{2}
<
n/4.
$
In particular,  \(144\le n^3\), such that 
$
\log(144n^3K)
\le
\log((nK)^6)
=
6\log(nK)$.
Define
$
C_{\mathrm I}
:=
7C_{\mathrm{int}}(V\vee D_\psi).
$
Then, on \(\mathcal E_n\),
\begin{equation}
\label{eq:adaptive-subset-interval-bound}
\bigg\|
\sum_{i\in I}\varepsilon_k^{(i)}
\bigg\|_{\Hb_k}
\le
C_{\mathrm I}\log(nK)
\left\{
\sqrt{|I|}+L_n
\right\},
\qquad
(k,I)\in[K]\times\mathcal I_n.
\end{equation}
The constant \(C_{\mathrm I}\) depends only on
\(C_\beta,c_\beta,D_\psi\). We now 
show that, on \(\mathcal E_n\),
\begin{align}
D_{n,\mathcal J}^{\mathrm{del}}
&\le
C_{\mathrm{del}}\log(nK)
\sqrt{\frac{g+q}{n}}
\left(
1+\frac{L_n}{\sqrt q}
\right),
\label{eq:adaptive-subset-deleted-contribution}
\end{align}
where \(C_{\mathrm{del}}\) depends only on
\(C_\beta,c_\beta,D_\psi\). For
\((\ell,k)\in[m]\times[K]\), set
\[
d_{\ell,k}^{G}
:=
\big|I_\ell\cap G_k(g)\big|.
\]
Since \(d_{i,k}=-\varepsilon_k^{(i)}\) for
\(i\in G_k(g)\), it follows that
\begin{align*}
\bigg\|
\sum_{i\in I_\ell\cap G_k(g)}
d_{i,k}
\left\{\mathbf1(i\le s)-\frac sn\right\}
\bigg\|_{\Hb_k}
&=
\bigg\|
\sum_{i\in I_\ell\cap G_k(g)\cap[1,s]}
\varepsilon_k^{(i)}
-
\frac sn
\sum_{i\in I_\ell\cap G_k(g)}
\varepsilon_k^{(i)}
\bigg\|_{\Hb_k}
\\
&\le
\bigg\|
\sum_{i\in I_\ell\cap G_k(g)\cap[1,s]}
\varepsilon_k^{(i)}
\bigg\|_{\Hb_k}
+
\frac sn
\bigg\|
\sum_{i\in I_\ell\cap G_k(g)}
\varepsilon_k^{(i)}
\bigg\|_{\Hb_k}.
\end{align*}
Both \(I_\ell\cap G_k(g)\) and
\(I_\ell\cap G_k(g)\cap[1,s]\) are integer intervals.
Therefore, on \(\mathcal E_n\),
\eqref{eq:adaptive-subset-interval-bound} gives
\begin{align*}
\bigg\|
\sum_{i\in I_\ell\cap G_k(g)}
d_{i,k}
\left\{\mathbf1(i\le s)-\frac sn\right\}
\bigg\|_{\Hb_k}
& \le
C_{\mathrm I}\log(nK)
\left\{
\sqrt{
\big|I_\ell\cap G_k(g)\cap[1,s]\big|\cap[1,s]
}
+L_n
\right\}
\\
&\qquad\quad+
\frac sn C_{\mathrm I}\log(nK)
\left\{
\sqrt{
\big|I_\ell\cap G_k(g)\big|
}
+L_n
\right\}
\\
&\le
2C_{\mathrm I}\log(nK)
\left\{
\sqrt{\big|I_\ell\cap G_k(g)\big|}+L_n
\right\}.
\end{align*}
Recalling the definition of \(D_{n,\mathcal J}^{\mathrm{del}}\) from \eqref{eq:def-Dn}, we
obtain
\begin{align*}
\left(D_{n,\mathcal J}^{\mathrm{del}}\right)^2
&\le
\frac{4C_{\mathrm I}^2\log^2(nK)}{mq}
\max_{k\in\mathcal J}
\sum_{\ell\in[m]:
I_\ell\cap G_k(g)\ne\varnothing}
\left\{
\sqrt{\big|I_\ell\cap G_k(g)\big|}+L_n
\right\}^2
\\
&\le
\frac{8C_{\mathrm I}^2\log^2(nK)}{mq}
\max_{k\in\mathcal J}
\left[
\sum_{\ell=1}^m \big|I_\ell\cap G_k(g)\big|
+
L_n^2
\#\left\{
\ell\in[m]:
I_\ell\cap G_k(g)\ne\varnothing
\right\}
\right].
\end{align*}
We next bound the two quantities on the right-hand side of this display. First, by the definitions of the block-aligned endpoints in~\eqref{def:clipped-LR-general}, recalling $h=q+r$, 
\[
 N_{L,k}(g)
=
h\left\lfloor
\frac{\widetilde\omega_k-g}{h}
\right\rfloor,
\qquad
 N_{R,k}(g)
=
h\left\lceil
\frac{\widetilde\omega_k+g}{h}
\right\rceil.
\]
Thus,
\begin{align}
\label{eq:bound-on-tilde-g}
\sum_{\ell=1}^m \big|I_\ell\cap G_k(g)\big|
\le
| G_k(g)|
\le
 N_{R,k}(g)- N_{L,k}(g)
&\le \nonumber
h\left[
\frac{\widetilde\omega_k+g}{h}+1
-
\left\{
\frac{\widetilde\omega_k-g}{h}-1
\right\}
\right]
\\
&=
2g+2h
\le
2g+4q
\le
4(g+q),
\end{align}
where we used \(r\le q\). Second, let
\[
N_k
:=
\#\left\{
\ell\in[m]:
I_\ell\cap G_k(g)\ne\varnothing
\right\}.
\]
If \(N_k\le2\), then
$N_k
\le
2+| G_k(g)|/{q}.
$
If \(N_k>2\), discard the first and last big blocks intersecting
\( G_k(g)\). Because \( G_k(g)\) is an integer
interval, every remaining intersected big block is contained entirely
in \( G_k(g)\). There are \(N_k-2\) such blocks, each of
cardinality \(q\), and hence
$
| G_k(g)|
\ge
(N_k-2)q.
$
Consequently, in either case,
\[
N_k
\le
2+\frac{| G_k(g)|}{q}
\le
2+4\frac{g+q}{q}
\le
6\frac{g+q}{q},
\]
where the second inequality follows from \eqref{eq:bound-on-tilde-g}. Furthermore, the assumptions from \eqref{eq:adaptive-subset-block-conditions} give
$q + r \le g+q+r \le \tau n/{2} < n/4$.
Using \(\lfloor u\rfloor\ge u/2\) for \(u\ge2\) and $q \ge r$, we obtain
$
m
=
\left\lfloor{n}/{q+r}\right\rfloor
\ge
{n}/\{2(q+r)\}
\ge
{n}/{4q},
$
and hence $mq\ge n/4$.
Substituting these bounds gives
\begin{align*}
\big(D_{n,\mathcal J}^{\mathrm{del}}\big)^2
&\le
\frac{8C_{\mathrm I}^2\log^2(nK)}{mq}
\left\{
4(g+q)
+
6\frac{g+q}{q}L_n^2
\right\}
\\
&\le
192C_{\mathrm I}^2\log^2(nK)
\frac{g+q}{n}
\left(
1+\frac{L_n^2}{q}
\right).
\end{align*}
In view of \(\sqrt{1+u^2}\le1+u\) and defining $C_{\mathrm{del}}:=\sqrt{192}\,C_{\mathrm I}$, this proves~\eqref{eq:adaptive-subset-deleted-contribution}. We next bound \(D_{n,\mathcal J}^{\mathrm{ret}}\). By the definition of
the clipped change-point estimator in~\eqref{def:trimmed_omega_k}, we have 
$\tau n \le \widetilde\omega_k \le (1-\tau)n$ for all $k \in [K]$.
Thus, 
\begin{align*}
 N_{L,k}(g)
&=
h\left\lfloor
\frac{\widetilde\omega_k-g}{h}
\right\rfloor
\ge
\widetilde\omega_k-g-h
\ge
\tau n-g-h
\ge
\frac{\tau n}{2},
\\
n- N_{R,k}(g)
&=
n-
h\left\lceil
\frac{\widetilde\omega_k+g}{h}
\right\rceil
\ge
n-\widetilde\omega_k-g-h
\ge
\tau n-g-h
\ge
\frac{\tau n}{2},
\end{align*}
where the last inequalities use \(g+h\le\tau n/2\). Applying~\eqref{eq:adaptive-subset-interval-bound} to the two retained
segments over which $i$ may range, gives, uniformly over \(k\in[K]\),
\begin{align}
&\bigg\|
\frac1{ N_{L,k}(g)}
\sum_{i=1}^{ N_{L,k}(g)}
\varepsilon_k^{(i)}
\bigg\|_{\Hb_k}
+
\bigg\|
\frac1{n- N_{R,k}(g)}
\sum_{i= N_{R,k}(g)+1}^{n}
\varepsilon_k^{(i)}
\bigg\|_{\Hb_k}
\nonumber\\
&\hspace{6cm}\le
2C_{\mathrm I}\log(nK)
\bigg\{
\sqrt{\frac2{\tau n}}
+
\frac{2L_n}{\tau n}
\bigg\}.
\label{eq:adaptive-subset-retained-noise-means}
\end{align}
By the definition of $\mathcal W_{\mathcal J}(g,\ell_0)$ in \eqref{def:W_j(g)} and Corollary~\ref{cor:coordinatewise_beta_localization_new}, on
\(\mathcal E_n\),
\begin{equation*}
|\widehat\omega_k-\omega_k|
<
r_{\beta,k}^{(n)}(
\ell_0
)
\le g
\qquad
\forall\, k\in
(\mathcal J\cap\mathcal S)
\setminus
\mathcal W_{\mathcal J}\left(
g,
\ell_0
\right).
\end{equation*}
For every such \(k\), the interval
\([\tau n,(1-\tau)n]\) contains \(\omega_k\) by the interior assumption in \eqref{eq:intertior-condition}. We claim that
$
|\widetilde\omega_k-\omega_k|
\le
|\widehat\omega_k-\omega_k|.
$
Indeed, if
\(\widehat\omega_k\in[\tau n,(1-\tau)n]\), then
\(\widetilde\omega_k=\widehat\omega_k\). If
\(\widehat\omega_k<\tau n\), then
$
|\widetilde\omega_k-\omega_k|
=
\omega_k-\tau n
\le
\omega_k-\widehat\omega_k.
$
Finally, if \(\widehat\omega_k>(1-\tau)n\), then
$
|\widetilde\omega_k-\omega_k|
=
(1-\tau)n-\omega_k
\le
\widehat\omega_k-\omega_k.
$
Thus
\[
|\widetilde\omega_k-\omega_k|
\le
|\widehat\omega_k-\omega_k|
<g
\qquad \forall\, k\in
(\mathcal J\cap\mathcal S)
\setminus
\mathcal W_{\mathcal J}\left(
g,
\ell_0
\right).
\]
Consequently, for every such $k$,
\begin{align*}
 N_{L,k}(g)
&=
h\left\lfloor
\frac{\widetilde\omega_k-g}{h}
\right\rfloor
\le
\widetilde\omega_k-g
<
\omega_k,
\\
 N_{R,k}(g)
&=
h\left\lceil
\frac{\widetilde\omega_k+g}{h}
\right\rceil
\ge
\widetilde\omega_k+g
>
\omega_k.
\end{align*}
In other words, the true change point lies in the deleted region:
\[
\omega_k\in G_k(g) \qquad \forall\, k\in
(\mathcal J\cap\mathcal S)
\setminus
\mathcal W_{\mathcal J}\left(
g,
\ell_0
\right).
\]
Next, define
\begin{align*}
\pi_{L,k}
&:=
\frac1{ N_{L,k}(g)}
\sum_{j=1}^{ N_{L,k}(g)}
\mathbf1\{j>\omega_k\},
&  
\pi_{R,k}
& :=
\frac1{n- N_{R,k}(g)}
\sum_{j= N_{R,k}(g)+1}^{n}
\mathbf1\{j>\omega_k\}.
\\
\overline\varepsilon_{L,k}
&:=
\frac1{ N_{L,k}(g)}
\sum_{j=1}^{ N_{L,k}(g)}
\varepsilon_k^{(j)},
&  
\overline\varepsilon_{R,k}
&:=
\frac1{n- N_{R,k}(g)}
\sum_{j= N_{R,k}(g)+1}^{n}
\varepsilon_k^{(j)}.
\end{align*}
For \(k\in\mathcal S\) and \(i\notin G_k(g)\), the
definition of the residuals in \eqref{def:trimmed-residuals-clipped}  gives
\[
d_{i,k}
=
\begin{cases}
\bigl\{\mathbf1(i>\omega_k)-\pi_{L,k}\bigr\}\delta_k
-\overline\varepsilon_{L,k},
&
i\le N_{L,k}(g),
\\[0.4em]
\bigl\{\mathbf1(i>\omega_k)-\pi_{R,k}\bigr\}\delta_k
-\overline\varepsilon_{R,k},
&
i> N_{R,k}(g).
\end{cases}
\]
If
$
k\in
(\mathcal J\cap\mathcal S)
\setminus
\mathcal W_{\mathcal J}(g,\ell_0),
$
then
$
\pi_{L,k}=0
$ and $
\pi_{R,k}=1$, 
and the deterministic terms in the preceding display vanish. For
\(k\in\mathcal W_{\mathcal J}(g,\ell_0)\), since
\(\pi_{L,k},\pi_{R,k}\in[0,1]\),
\[
\left|
\mathbf1(i>\omega_k)-\pi_{L,k}
\right|\le1,
\qquad
\left|
\mathbf1(i>\omega_k)-\pi_{R,k}
\right|\le1,
\]
so the corresponding deterministic terms have norm at most
\(\Delta_k:=\|\delta_k\|_{\Hb_k}\). For \(k\notin\mathcal S\), the deterministic terms
vanish identically, and for these indices $\delta_k=0$. It thus follows from
\eqref{eq:adaptive-subset-retained-noise-means} that
\begin{align}
&\max_{k\in\mathcal J}
\max_{i\notin G_k(g)}
\|d_{i,k}\|_{\Hb_k}
\le
2C_{\mathrm I}\log(nK)
\bigg\{
\sqrt{\frac2{\tau n}}
+
\frac{2L_n}{\tau n}
\bigg\}
+
\Delta_{\mathcal J,\mathrm{wk}}(g,\ell_0),
\label{eq:adaptive-subset-outside-gap-bound}
\end{align}
where $\Delta_{\mathcal J,\mathrm{wk}}(g,\ell) = \max_{k \in W_{\mathcal J}(g,\ell_0)} \Delta_k$ as defined in \eqref{eq:def-delta-J-wk}.
For every \((\ell,s,k)\in[m]\times[n]\times\mathcal J\), the
Hilbert-space triangle inequality and
\(\left|\mathbf1(i\le s)-s/n\right|\le1\) give
\begin{align*}
\bigg\|
\sum_{i\in I_\ell\setminus G_k(g)}
d_{i,k}
\left\{\mathbf1(i\le s)-\frac sn\right\}
\bigg\|_{\Hb_k}
\le
\sum_{i\in I_\ell\setminus G_k(g)}
\|d_{i,k}\|_{\Hb_k}
\le
q
\max_{i\notin G_k(g)}
\|d_{i,k}\|_{\Hb_k}.
\end{align*}
Hence, by definition of $D_{n,\mathcal J}^{\mathrm{ret}}$ in \eqref{eq:def-Dn},
\begin{align*}
D_{n,\mathcal J}^{\mathrm{ret}}
&\le
\left[
\frac1{mq}
\sum_{\ell=1}^m q^2
\right]^{1/2}
\max_{k\in\mathcal J}
\max_{i\notin G_k(g)}
\|d_{i,k}\|_{\Hb_k}
=
\sqrt q
\max_{k\in\mathcal J}
\max_{i\notin G_k(g)}
\|d_{i,k}\|_{\Hb_k}.
\end{align*}
Therefore, from \eqref{eq:adaptive-subset-outside-gap-bound},
\begin{align*}
D_{n,\mathcal J}^{\mathrm{ret}}
\le
2C_{\mathrm I}\log(nK)
\bigg\{
\sqrt{\frac{2q}{\tau n}}
+
\frac{2\sqrt q L_n}{\tau n}
\bigg\}
+
\sqrt q\,
\Delta_{\mathcal J,\mathrm{wk}}(g,\ell_0).
\end{align*}
The tuning condition in \eqref{eq:adaptive-subset-block-conditions} gives \(q/n\le\tau/2\). Therefore,
$q / \sqrt{n(q+g)} \le \sqrt{q/n} \le \sqrt{\tau/2}$, which implies
\begin{align*}
\frac{2\sqrt q\,L_n}{\tau n}
&=
\frac2\tau
\frac{q}{\sqrt{n(q+g)}}
\sqrt{\frac{q+g}{n}}
\frac{L_n}{\sqrt q}
\le
\sqrt{\frac2\tau}
\sqrt{\frac{q+g}{n}}
\frac{L_n}{\sqrt q}.
\end{align*}
Moreover,
\[
\sqrt{\frac{2q}{\tau n}}
\le
\sqrt{\frac2\tau}
\sqrt{\frac{q+g}{n}}.
\]
It follows that
\begin{align}
D_{n,\mathcal J}^{\mathrm{ret}}
\le{}&
2C_{\mathrm I}\sqrt{\frac2\tau}
\log(nK)
\sqrt{\frac{q+g}{n}}
\left(
1+\frac{L_n}{\sqrt q}
\right)
+
\sqrt q\,
\Delta_{\mathcal J,\mathrm{wk}}(g,\ell_0).
\label{eq:adaptive-subset-retained-final}
\end{align}

Combining~\eqref{eq:adaptive-subset-D-reduction},
\eqref{eq:adaptive-subset-deleted-contribution}, and
\eqref{eq:adaptive-subset-retained-final}, and recalling $C_{\mathrm{del}} = \sqrt{192} C_{\mathrm I} \le \sqrt{192} C_{\mathrm I}/\sqrt \tau$, we obtain that
\begin{align*}
D_{n,\mathcal J}
&\le \nonumber
\frac{C_{\mathrm D}}{\sqrt\tau}
\bigg[
\log(nK)
\sqrt{\frac{q+g}{n}}
\left\{
1+
\frac{\ell_0\log(en/\ell_0)}{\sqrt q}
\right\}
+
\sqrt q\,
\Delta_{\mathcal J,\mathrm{wk}}(g,\ell_0)
\bigg]
\\
&=
\frac{C_{\mathrm D}}{\sqrt\tau} \,
\frac{R_{n,\mathcal J}^{\mathrm{ad}}(q,g,\ell_0,K)}{\sqrt{\log(nK)}} 
\end{align*}
where $C_{\mathrm D} :=
\{(\sqrt{192}+2\sqrt2)C_{\mathrm I}
\}\vee1$ and where $R_{n,\mathcal J}^{\mathrm{ad}}(q,g,\ell_0,K)$ is defined \eqref{eq:def-adaptive-subset-residual-rate}.
Combining this display with
\eqref{eq:adaptive-subset-multiplier-reduction} proves
\eqref{eq:adaptive-subset-conditional-comparison} with
$
C
:=
C_{\mathrm{mult}}C_{\mathrm D}.
$
The constant \(C\) depends only on
\(D_\psi,C_\beta,c_\beta\), and is independent of \(\tau\).
This completes the proof.
\end{proof}

\begin{lemma}[Oracle block-alignment comparison]
\label{lem:oracle_block_alignment}
Suppose Assumption~\ref{cond:stationary} holds. Let
\(q,r\in\mathbb N\) satisfy
\[
1\le r\le q,
\qquad
q\le\frac n4.
\]
Then there exist a constant
$
C_{\mathrm{align}}=C_{\mathrm{align}}(D_\psi)>0
$
and an event \(\Omega_{\mathrm{align}}\in\mathcal F_n\) satisfying
$
\mathbb P(\Omega_{\mathrm{align}})\ge1-1/n
$
such that, on \(\Omega_{\mathrm{align}}\cap\Omega_{\mathrm{big}}\) with $\Omega_{\mathrm{big}}$ from \eqref{def:Omega_big},
\[
\mathbb P\bigg(
\max_{k\in[K]}\max_{s\in[n]}
\big\|\widetilde C_k^{O*}(s)-C_k^{O*}(s)\big|_{\Hb_k}
>
C_{\mathrm{align}}
\frac{q\log^{3/2}(nK)}{\sqrt n}
\,\bigg|\,
\mathcal F_n\vee\widetilde{\mathcal F}_n
\bigg)
\le\frac2n,
\]
where $C_k^{O*}$ and $\widetilde C_k^{O*}$ are defined in 
\eqref{def:oracle-ck} and 
\eqref{def:berbee-coupled-block-oracle}, respectively,
\end{lemma}

\begin{proof}[Proof of Lemma~\ref{lem:oracle_block_alignment}]
Recall from~\eqref{eq:def-Bs} that
$
B_s
=
\lceil s/(q+r)\rceil.
$
On \(\Omega_{\mathrm{big}}\) from \eqref{def:Omega_big}, we have
\(\widetilde S_{\ell k}=S_{\ell k}\) for every
\((\ell,k)\in[m]\times[K]\). Hence, by definition in \eqref{def:oracle-ck} and \eqref{def:berbee-coupled-block-oracle}, 
\begin{align*}
\widetilde C_k^{O*}(s)-C_k^{O*}(s)
&=
\frac1{\sqrt{mq}}
\sum_{\ell=1}^m e_\ell
\sum_{i\in I_\ell}\varepsilon_k^{(i)}
\left\{
\mathbf1(\ell\le B_s)-\mathbf1(i\le s)
\right\}.
\end{align*}
Every summand in this display vanishes except
possibly the one with \(\ell=B_s\). If \(B_s>m\), then
\(s>m(q+r)\)
so every index in the retained big blocks is at most \(s\), and the entire
sum vanishes. Therefore, with
\[
D_{sk}
:=
\begin{cases}
\displaystyle
\sum_{i\in I_{B_s}}\varepsilon_k^{(i)}
\{1-\mathbf1(i\le s)\},
&B_s\in[m],\\[1.2ex]
0,
&B_s\notin[m],
\end{cases}
\]
we have, on \(\Omega_{\mathrm{big}}\),
\[
\widetilde C_k^{O*}(s)-C_k^{O*}(s)
=
\frac{e_{B_s}}{\sqrt{mq}}D_{sk}.
\]
Set
\[
A_n:=\max_{k\in[K]}\max_{s\in[n]}\|D_{sk}\|_{\Hb_k}.
\]
Each \(D_{sk}\) is either zero or a sum of at most \(q\) marginal noise elements. Therefore, by the Hilbert-space triangle
inequality, monotonicity of the Orlicz norm,
Assumption~\ref{cond:stationary}, and~\ref{orlicz:triangle},
\[
\left\|\|D_{sk}\|_{\Hb_k}\right\|_{\psi_1}
\le qD_\psi.
\]
Then~\ref{orlicz:tail_bound}, applied with \(p=1\), and the union bound
over the \(nK\) pairs \((s,k)\) yield an absolute constant \(C_0>0\) (we may take $C
_0=3$ since $n\geq 4$)
such that the event
\[
\Omega_{\mathrm{align}}
:=
\left\{
A_n\le C_0D_\psi q\log(nK)
\right\}
\in\mathcal F_n
\]
satisfies
$
\mathbb P(\Omega_{\mathrm{align}})\ge1-1/n.
$
Because \(r\le q\) and \(q\le n/4\), we have
$n/(q+r)\ge2$.
Therefore,
$
m
=
\lfloor n /(q+r) \rfloor
\ge
n / \{2(q+r) \} \ge n/(4q)$
and hence
$
mq\ge n/4.
$
Thus, on \(\Omega_{\mathrm{align}}\cap\Omega_{\mathrm{big}}\),
\[
\max_{k\in[K]}\max_{s\in[n]}
\|\widetilde C_k^{O*}(s)-C_k^{O*}(s)\|_{\Hb_k}
\le
2C_0D_\psi\frac{q\log(nK)}{\sqrt n}
\max_{\ell \in [m]}|e_\ell|.
\]
The multipliers are independent of \(\mathcal F_n\vee\widetilde{\mathcal F}_n\).
Moreover,
$
\left\||e_\ell|\right\|_{\psi_2\mid\mathcal F_n\vee\widetilde{\mathcal F}_n}
=
\left\||e_\ell|\right\|_{\psi_2}
=
\sqrt{8/3}.
$
Therefore, by~\ref{orlicz:tail_bound} and the union bound and since $m \le n$,
\begin{align*}
\mathbb P\left(
\max_{\ell \in [m]}|e_\ell|
>
\frac4{\sqrt3}\sqrt{\log n}
\,\middle|\,
\mathcal F_n\vee\widetilde{\mathcal F}_n
\right)\le
2m\exp\left\{
-\frac{(16/3)\log n}{8/3}
\right\}
=
2m e^{-2\log n}
\le
\frac2n.
\end{align*}
Combining this bound with the preceding deterministic inequality, we obtain,
on \(\Omega_{\mathrm{align}}\cap\Omega_{\mathrm{big}}\),
\begin{align*}
\mathbb P\left(
\max_{k\in[K]}\max_{s\in[n]}
\|\widetilde C_k^{O*}(s)-C_k^{O*}(s)\|_{\Hb_k}
>
\frac{8C_0D_\psi}{\sqrt3}
\frac{q\log(nK)\sqrt{\log n}}{\sqrt n}
\,\middle|\,
\mathcal F_n\vee\widetilde{\mathcal F}_n
\right)
\le\frac2n.
\end{align*}
Since \(\log n\le\log(nK)\), it follows that
$
\log(nK)\sqrt{\log n}
\le
\log^{3/2}(nK).
$
Therefore the asserted conditional bound holds with $C_{\mathrm{align}}
:= 8C_0 D_\psi/\sqrt 3$.
This completes the proof.
\end{proof}

\begin{lemma}[Conditional size bound for the coupled block-aligned
oracle process]
\label{lem:coupled_block_oracle_size}
Suppose Assumption~\ref{cond:stationary} holds. Let \(q,r\in\mathbb N\)
satisfy
\[
1\le r\le q \le\frac n4.
\]
Then there exist a constant
$
C_{\mathrm{size}}
=
C_{\mathrm{size}}(D_\psi,S_{\beta,0})>0
$
and an event
$
\Omega_{\mathrm{size}}^{O*}
\in\widetilde{\mathcal F}_n
$
satisfying
$
\mathbb P\left(\Omega_{\mathrm{size}}^{O*}\right)
\ge1-1/n
$
such that, on \(\Omega_{\mathrm{size}}^{O*}\),
\begin{align}
&\mathbb P\left(
\max_{k\in[K]}\max_{s\in[n]}
\|\widetilde C_k^{O*}(s)\|_{\Hb_k}
>
C_{\mathrm{size}}
\left(
1+\frac{q\log n}{\sqrt n}
\right)
\log^{3/2}(nK)
\,\middle|\,
\widetilde{\mathcal F}_n
\right)
\le\frac2n.
\label{eq:coupled-block-oracle-size}
\end{align}
\end{lemma}

\begin{proof}[Proof of Lemma~\ref{lem:coupled_block_oracle_size}]
For every \((s,k)\in[n]\times[K]\), define
\[
\widetilde V_{sk}
:=
\bigg\{
\frac1{mq}
\sum_{\ell=1}^m
a_{s\ell}^2
\|\widetilde S_{\ell k}\|_{\Hb_k}^2
\bigg\}^{1/2}.
\]
By \ref{orlicz:power} and
Lemma~\ref{lem:uified_bounds}(6), applied with \(Q=\operatorname{id}_{\Hb_k}\),
\begin{align*}
\|\widetilde V_{sk}\|_{\psi_1}
&\le
C_6
\bigg[
\bigg\{
\frac1{mq}
\sum_{\ell=1}^m
a_{s\ell}^2
\mathbb E\|\widetilde S_{\ell k}\|_{\Hb_k}^2
\bigg\}^{1/2}
+
D_\psi\sqrt{\frac qm}\log(m+1)
\bigg].
\end{align*}
Since
\(\widetilde S_{\ell k}\stackrel{d}=S_{\ell k}\),
Lemma~\ref{lem:uified_bounds}(1) gives
$
\mathbb E\|\widetilde S_{\ell k}\|_{\Hb_k}^2
\le
\widetilde C_{\beta,q}D_\psi^2q.
$
Moreover, with $S_{\beta,0}$ from \eqref{eq:def-S_beta},
\[
\widetilde C_{\beta,q}
\le
\{A_{4,1/2,1}+64\}(S_{\beta,0}\vee1).
\]
It follows that, for an absolute constant
\(C_{\mathrm{unif}}>0\),
\begin{align*}
\|\widetilde V_{sk}\|_{\psi_1}
&\le
C_{\mathrm{unif}}
D_\psi(S_{\beta,0}\vee1)^{1/2}
\left[
\frac{\|a_s\|_2}{\sqrt m}
+
\sqrt{\frac qm}\log(m+1)
\right].
\end{align*}
Lemma~\ref{lem:prop a_s_ell} gives
$
\|a_s\|_2^2\le m+2\le2m.
$
Furthermore, \(r\le q\) and \(q\le n/4\) imply
$
m\ge n/(4q)
$
while \(m+1\le n+1\le n^2\). Therefore,
$\sqrt{q/m} \log(m+1) \le (4q\log n )/\sqrt n$.
Consequently, 
\begin{equation}
\max_{k\in[K]}\max_{s\in[n]}
\|\widetilde V_{sk}\|_{\psi_1}
\le
C_V
\left(
1+\frac{q\log n}{\sqrt n}
\right),
\label{eq:coupled-block-size-proxy-orlicz}
\end{equation}
where
$
C_V
=
(\sqrt2+4)
C_{\mathrm{unif}}
D_\psi(S_{\beta,0}\vee1)^{1/2}.
$
Define
\[
\Omega_{\mathrm{size}}^{O*}
:=
\left\{
\max_{k\in[K]}\max_{s\in[n]}
\widetilde V_{sk}
\le
3C_V
\left(
1+\frac{q\log n}{\sqrt n}
\right)
\log(nK)
\right\}.
\]
This event belongs to \(\widetilde{\mathcal F}_n\). By
\eqref{eq:coupled-block-size-proxy-orlicz},
\ref{orlicz:tail_bound}, and the union bound,
\begin{align*}
\mathbb P\left(
(\Omega_{\mathrm{size}}^{O*})^c
\right)
&\le
2nK\exp\{-3\log(nK)\}
=
\frac{2}{(nK)^2}
\le
\frac1n,
\end{align*}
where the last inequality uses \(n\ge4\).
Conditionally on \(\widetilde{\mathcal F}_n\),
$\widetilde C_k^{O*}(s)
=
(mq)^{-1/2}
\sum_{\ell=1}^m
a_{s\ell}e_\ell\widetilde S_{\ell k}$
is a centered Gaussian random element of \(\Hb_k\).
Lemma~\ref{lem:gaussian_conditional_tail_bound} therefore gives, for an absolute constant $C_{\mathrm{G}}$ and for
every \(y>0\),
\[
\mathbb P\left(
\|\widetilde C_k^{O*}(s)\|_{\Hb_k}
>
C_{\mathrm G}\sqrt y\,\widetilde V_{sk}
\,\middle|\,
\widetilde{\mathcal F}_n
\right)
\le2e^{-y}.
\]
On \(\Omega_{\mathrm{size}}^{O*}\), take
\(y=2\log(nK)\). Then
\[
C_{\mathrm G}\sqrt y\,\widetilde V_{sk}
\le
3\sqrt2\,C_{\mathrm G}C_V
\left(
1+\frac{q\log n}{\sqrt n}
\right)
\log^{3/2}(nK).
\]
Define
$
C_{\mathrm{size}}
:=
3\sqrt2\,C_{\mathrm G}C_V.
$
A conditional union bound over the \(nK\) pairs \((s,k)\) yields
\begin{align*}
&\mathbb P\left(
\max_{k\in[K]}\max_{s\in[n]}
\|\widetilde C_k^{O*}(s)\|_{\Hb_k}
>
C_{\mathrm{size}}
\left(
1+\frac{q\log n}{\sqrt n}
\right)
\log^{3/2}(nK)
\,\middle|\,
\widetilde{\mathcal F}_n
\right)
\le
2nK e^{-2\log(nK)}
\le \frac2n.
\end{align*}
This proves~\eqref{eq:coupled-block-oracle-size}.
\end{proof}

\begin{lemma}[Boundary restriction for the squared coupled block-aligned
oracle process]
\label{lem:oracle_coupled_boundary_restriction_squared}
Suppose
Assumption~\ref{cond:stationary} holds. Fix \(x\in(0,3/16]\), and let \(S_x\) be as defined
in~\eqref{eq:definition-S_x}. Let \(q,r\in\mathbb N\) satisfy
\[
1\le r\le q \le \frac n4.
\]
Then there exist an absolute constant \(C_{\partial}>0\) and an event
$
\Omega_{\partial}^{O*}\in\widetilde{\mathcal F}_n
$
that may depend on $x$
satisfying
$
\mathbb P(\Omega_{\partial}^{O*})\ge1-1/n
$
such that, on \(\Omega_{\partial}^{O*}\), for every \(t>0\) satisfying
\[
t
\ge
C_{\partial}D_\psi^2(S_{\beta,0}\vee1)
\log^3(nK)
\Big(
x\vee\frac{q^2\log^2n}{n}
\Big),
\]
we have
\[
\mathbb P\left(
\max_{k\in[K]}
\max_{s\in[n]\setminus S_x}
\|\widetilde C_k^{O*}(s)\|_{\Hb_k}^2
>
t
\,\middle|\,
\widetilde{\mathcal F}_n
\right)
\le\frac2n.
\]
\end{lemma}

\begin{proof}[Proof of Lemma~\ref{lem:oracle_coupled_boundary_restriction_squared}]
In view of the representation
\[
\widetilde C_k^{O*}(s)
=
\frac1{\sqrt{mq}}
\sum_{\ell=1}^m
a_{s\ell}e_\ell\widetilde S_{\ell k},
\]
we may apply Lemma~\ref{lem:gaussian_conditional_tail_bound} to deduce that, for every \(y>0\) and every \((s,k)\in[n]\times[K]\), almost surely,
\[
\mathbb P\left(
\|\widetilde C_k^{O*}(s)\|_{\Hb_k}
>C_G\sqrt y\,V_{sk}
\,\middle|\,
\widetilde{\mathcal F}_n
\right)
\le2e^{-y},
\]
where  \(C_G>0\) denotes the absolute constant from
Lemma~\ref{lem:gaussian_conditional_tail_bound} and where 
\[
V_{sk}^2
:=
\frac1{mq}
\sum_{\ell=1}^m
a_{s\ell}^2
\|\widetilde S_{\ell k}\|_{\Hb_k}^2.
\]
Taking \(y=2\log(nK)\) and applying the conditional union bound, we
obtain, for every nonempty \(S\subset[n]\),
\begin{align}
\label{eq:conditional-boundary-maximum-short}
&\mathbb P\left(
\max_{k\in[K]}\max_{s\in S}
\|\widetilde C_k^{O*}(s)\|_{\Hb_k}
>
C_G\sqrt{2\log(nK)}
\max_{k\in[K]}\max_{s\in S}V_{sk}
\,\middle|\,
\widetilde{\mathcal F}_n
\right)
\nonumber\\
&\quad\le
\sum_{k=1}^K\sum_{s\in S}
\mathbb P\left(
\|\widetilde C_k^{O*}(s)\|_{\Hb_k}
>
C_G\sqrt{2\log(nK)}\,V_{sk}
\,\middle|\,
\widetilde{\mathcal F}_n
\right)
\nonumber\\
&\quad\le
2K|S|e^{-2\log(nK)}
=
\frac{2|S|}{n^2K}
\le
\frac{2}{nK}
\le
\frac2n.
\end{align}
Here the first inequality uses
\(V_{sk}\le\max_{k'\in[K],\,s'\in S}V_{s'k'}\). By \ref{orlicz:power} and Lemma~\ref{lem:uified_bounds}(3), applied with
\(Q=\operatorname{id}_{\Hb_k}\), and using
\(C_{\beta,q}\le S_{\beta,0}\) with $S_{\beta,0}$ from \eqref{eq:def-S_beta},
\begin{align}
\|V_{sk}\|_{\psi_1}
=
\|V_{sk}^2\|_{\psi_{1/2}}^{1/2}
&\le
C_3D_\psi(S_{\beta,0}\vee1)^{1/2}
\left[
\frac{\|a_s\|_2}{\sqrt m}
+
\sqrt{\frac qm}\log(m+1)
\right],
\label{eq:boundary-variance-proxy-orlicz-short}
\end{align}
where \(C_3>0\) is the absolute constant from
Lemma~\ref{lem:uified_bounds}(3). Lemma~\ref{lem:prop a_s_ell} gives
\[
\|a_s\|_2^2
\le
mV(s/n)+2,
\]
where $V(u)=u(1-u)$.
Thus, if \(s\in[n]\setminus S_x\), then \(V(s/n)\le x\)
by \eqref{eq:infsup-V-S_x} and hence
\[
\frac{\|a_s\|_2}{\sqrt m}
\le
\sqrt x+\sqrt{\frac2m}.
\]
Moreover, \(q\le n/4\) and \(r\le q\) imply $n/(q+r) \ge n/(2q) \ge 2$.
Using \(\lfloor z\rfloor\ge z/2\) for every \(z\ge2\), we obtain $m \ge n/(4q)$.
Therefore \(m^{-1/2}\le2\sqrt{q/n}\), and
\[
\sqrt{\frac2m}
\le
2\sqrt{\frac{2q}{n}}
\le
2\sqrt2\,\frac{q\log n}{\sqrt n},
\]
where we used \(\sqrt q\le q\) and \(\log n\ge1\). Moreover, since
\(m+1\le n+1\le n^2\),
\[
\sqrt{\frac qm}\log(m+1)
\le
\frac{2q}{\sqrt n}\log(n+1)
\le
4\frac{q\log n}{\sqrt n}.
\]
Consequently,
\[
\sqrt{\frac2m}
+
\sqrt{\frac qm}\log(m+1)
\le
(4+2\sqrt2)\frac{q\log n}{\sqrt n}
\le
7\frac{q\log n}{\sqrt n}.
\]
Set
\[
B_x
:=
\sqrt x\vee\frac{q\log n}{\sqrt n}.
\]
The preceding bounds imply
\[
\frac{\|a_s\|_2}{\sqrt m}
+
\sqrt{\frac qm}\log(m+1)
\le
\sqrt x+7\frac{q\log n}{\sqrt n}
\le8B_x
\]
for every \(s\in[n]\setminus S_x\). Hence
\eqref{eq:boundary-variance-proxy-orlicz-short} yields
\begin{equation}
\label{eq:boundary-variance-proxy-final-short}
\max_{k\in[K]}
\max_{s\in[n]\setminus S_x}
\|V_{sk}\|_{\psi_1}
\le
8C_4D_\psi(S_{\beta,0}\vee1)^{1/2}B_x.
\end{equation}
Set
$
D_1:=24C_4
$
and define the \(\widetilde{\mathcal F}_n\)-measurable event
\[
\Omega_{\partial}^{O*}
:=
\Big\{
\max_{k\in[K]}\max_{s\in[n]\setminus S_x}V_{sk}
\le
D_1D_\psi(S_{\beta,0}\vee1)^{1/2}B_x\log(nK)
\Big\}.
\]
By~\eqref{eq:boundary-variance-proxy-final-short},
\ref{orlicz:tail_bound}, and the union bound,
\begin{align*}
\mathbb P\bigl((\Omega_{\partial}^{O*})^c\bigr)
&\le
2nK\exp\Big\{
-\frac{D_1}{8C_4}\log(nK)
\Big\}
\\
&=
2nK\exp\{-3\log(nK)\}
=
\frac{2}{(nK)^2}
\le
\frac1n,
\end{align*}
where the last inequality uses \(n\ge4\) and \(K\ge1\). Thus
\(\mathbb P(\Omega_{\partial}^{O*})\ge1-1/n\), and, on this event,
\begin{equation*}
\max_{k\in[K]}
\max_{s\in[n]\setminus S_x}V_{sk}
\le
D_1D_\psi(S_{\beta,0}\vee1)^{1/2}
B_x\log(nK).
\end{equation*}
Together with the conditional maximal bound from \eqref{eq:conditional-boundary-maximum-short}, applied with
\(S=[n]\setminus S_x\), we obtain, on 
\(\Omega_{\partial}^{O*}\),
\begin{align*}
&\mathbb P\left(
\max_{k\in[K]}\max_{s\in[n]\setminus S_x}
\|\widetilde C_k^{O*}(s)\|_{\Hb_k}
>
\sqrt2\,C_GD_1D_\psi(S_{\beta,0}\vee1)^{1/2}
B_x\log^{3/2}(nK)
\,\middle|\,
\widetilde{\mathcal F}_n
\right)
\le\frac2n.
\end{align*}
Squaring the threshold and observing that
$B_x^2 = x \vee (q^2\log^2n/n)$ proves the result upon taking $C_{\partial}
:=
2C_G^2D_1^2$.
\end{proof}

The following result is a bootstrap version of Lemma~\ref{lem:coupled_big_block_projection_error}.

\begin{lemma}[Conditional projection error for the coupled block-aligned
oracle statistic]
\label{lem:oracle_coupled_projection_error_2}
Suppose Assumptions~\ref{cond:stationary} and~\ref{cond:scores} hold.
Let \(q,r\in\mathbb N\) satisfy
\[
1\le r\le q \le \frac n4.
\]
Then there exists a constant
$
C_{\mathrm{proj}}
=
C_{\mathrm{proj}}(D_\psi,C_U,\gamma_U,C_\beta,c_\beta)>0
$
such that, for every
\(M\in\mathbb N_{\ge1}\) and every nonempty
\(S\subset[n]\),
there exists an event
$
\Omega_{\mathrm{proj}}^{O*}\in\widetilde{\mathcal F}_n
$, potentially depending on $M$ and $S$,
satisfying
$
\mathbb P(\Omega_{\mathrm{proj}}^{O*})\ge 1-1/n
$
such that, on \(\Omega_{\mathrm{proj}}^{O*}\),
\begin{align*}
\mathbb P\bigg(&
\max_{k\in[K]}\max_{s\in S}
\|\widetilde C_k^{O*}(s)\|_{\Hb_k}^2
>
\max_{k\in[K]}\max_{s\in S}
\|P_{kM}\widetilde C_k^{O*}(s)\|_{\Hb_k}^2
+\Delta_{\mathrm{proj}}^O
\,\Bigm|\,
\widetilde{\mathcal F}_n
\bigg)
\le\frac2n,
\end{align*}
where
\[
\Delta_{\mathrm{proj}}^O
:=
C_{\mathrm{proj}}\log^3(nK)
\left[
M^{1-2\gamma_U}
+\frac1q
+\frac{q^2\log^2n}{n}
\right].
\]
\end{lemma}

\begin{proof}[Proof of Lemma~\ref{lem:oracle_coupled_projection_error_2}]
The proof is similar to the proof of Lemma~\ref{lem:coupled_big_block_projection_error}. Write
$
P_{kM}^\perp:=\operatorname{id}_{\Hb_k}-P_{kM}
$
and note that
\begin{align}
&\max_{k\in[K]}\max_{s\in S}
\|\widetilde C_k^{O*}(s)\|_{\Hb_k}^2
\le 
\max_{k\in[K]}\max_{s\in S}
\|P_{kM}\widetilde C_k^{O*}(s)\|_{\Hb_k}^2
+
\max_{k\in[K]}\max_{s\in S}
\|P_{kM}^\perp\widetilde C_k^{O*}(s)\|_{\Hb_k}^2
\label{eq:conditional-projection-reduction-short}
\end{align}
by \eqref{eq:orthogonal}.
For \((s,k)\in[n]\times[K]\), define
\[
V_{sk,M}^\perp
:=
\bigg\{
\frac1{mq}
\sum_{\ell=1}^m
a_{s\ell}^2
\|P_{kM}^\perp\widetilde S_{\ell k}\|_{\Hb_k}^2
\bigg\}^{1/2}.
\]
Let \(C_G>0\) be the absolute constant from
Lemma~\ref{lem:gaussian_conditional_tail_bound}. Conditionally on
\(\widetilde{\mathcal F}_n\), that lemma and
\ref{orlicz:tail_bound} give, for every \(y>0\),
\[
\mathbb P\left(
\|P_{kM}^\perp\widetilde C_k^{O*}(s)\|_{\Hb_k}
>
C_G\sqrt y\,V_{sk,M}^\perp
\,\middle|\,
\widetilde{\mathcal F}_n
\right)
\le2e^{-y}.
\]
Taking \(y=2\log(nK)\) and applying the conditional union bound over
the at most \(K|S|\le nK\) pairs gives
\begin{align}
\mathbb P\bigg(&
\max_{k\in[K]}\max_{s\in S}
\|P_{kM}^\perp\widetilde C_k^{O*}(s)\|_{\Hb_k}
>
C_G\sqrt{2\log(nK)}
\max_{k\in[K]}\max_{s\in S}V_{sk,M}^\perp
\,\Bigm|\,
\widetilde{\mathcal F}_n
\bigg)
\le\frac2n.
\label{eq:conditional-projection-gaussian-short}
\end{align}
By \ref{orlicz:power} and Lemma~\ref{lem:uified_bounds}(6) it holds for an absolute constant $C_6$ that
\begin{align}
\|V_{sk,M}^\perp\|_{\psi_1}
\le C_6
\bigg[
\bigg\{
\frac1{mq}
\sum_{\ell=1}^m
a_{s\ell}^2
\mathbb E\|P_{kM}^\perp\widetilde S_{\ell k}\|_{\Hb_k}^2
\bigg\}^{1/2}
+
D_\psi\sqrt{\frac qm}\log(m+1)
\bigg],
\label{eq:conditional-proxy-refined-HJ-short}
\end{align}
Since \(\widetilde S_{\ell k}\stackrel d=S_{\ell k}\),
\begin{align*}
\frac1q
\mathbb E\|P_{kM}^\perp\widetilde S_{\ell k}\|_{\Hb_k}^2
&=
\left\|
P_{kM}^\perp
\Cov(q^{-1/2}S_{\ell k})
P_{kM}^\perp
\right\|_{\Tr}
\\
&\le
\left\|
P_{kM}^\perp
\left\{
\Cov(q^{-1/2}S_{\ell k})-\mathcal K_k
\right\}
P_{kM}^\perp
\right\|_{\Tr}
+
\left\|
P_{kM}^\perp\mathcal K_kP_{kM}^\perp
\right\|_{\Tr}.
\end{align*}
Here we used Lemma~\ref{lem:properties-covariance}(2) and
\(P_{kM}^\perp=(P_{kM}^\perp)^*\).
Set $C_1:=64D_\psi^2S_{\beta,1}$.
Lemma~\ref{lem:convergence_to_long_term_trace} gives
\[
\left\|
\Cov(q^{-1/2}S_{\ell k})-\mathcal K_k
\right\|_{\Tr}
\le\frac{C_1}{q}.
\]
By the ideal property of the trace norm, \citep[Chapter~2]{simon2005trace} it holds that
\(\|ABC\|_{\Tr}\le
\|A\|_{\mathrm{op}}\|B\|_{\Tr}\|C\|_{\mathrm{op}}\), whenever $A, C$ are bounded and $B$ is trace-class. Hence
\begin{align*}
&\left\|
P_{kM}^\perp
\left\{
\Cov(q^{-1/2}S_{\ell k})-\mathcal K_k
\right\}
P_{kM}^\perp
\right\|_{\Tr}
\le
\|P_{kM}^\perp\|_{\mathrm{op}}^2
\left\|
\Cov(q^{-1/2}S_{\ell k})-\mathcal K_k
\right\|_{\Tr}
\le\frac{C_1}{q},
\end{align*}
because \(\|P_{kM}^\perp\|_{\mathrm{op}}\le1\).
Moreover, Assumption~\ref{cond:scores} and an integral comparison give
\[
\left\|
P_{kM}^\perp\mathcal K_kP_{kM}^\perp
\right\|_{\Tr}
=
\sum_{j=M+1}^{r_k}\lambda_{kj}
\le
\frac{C_U^2}{2\gamma_U-1}M^{1-2\gamma_U}.
\]
Define $C_2 = \max\{C_1, C_U^2/(2\gamma_U-1)\}$.
Consequently,
\begin{equation}
\label{eq:projected-block-second-moment-short}
\frac1q
\mathbb E\|P_{kM}^\perp\widetilde S_{\ell k}\|_{\Hb_k}^2
\le
C_2\left(M^{1-2\gamma_U}+\frac1q\right).
\end{equation}
Lemma~\ref{lem:prop a_s_ell} implies
\(\|a_s\|_2^2\le m+2\le2m\), because \(m\ge2\). Therefore,
\eqref{eq:projected-block-second-moment-short} gives
\begin{align*}
&\bigg\{
\frac1{mq}
\sum_{\ell=1}^m
a_{s\ell}^2
\mathbb E\|P_{kM}^\perp\widetilde S_{\ell k}\|_{\Hb_k}^2
\bigg\}^{1/2}
\le
\sqrt{2C_2}
\Big(M^{1-2\gamma_U}+\frac1q\Big)^{1/2}.
\end{align*}
Furthermore, \(m\ge n/(4q)\) and
\(m+1\le n+1\le n^2\) imply
\[
D_\psi\sqrt{\frac qm}\log(m+1)
\le
4D_\psi\frac{q\log n}{\sqrt n}.
\]
Set
\[
R_{n,M}
:=
M^{1-2\gamma_U}
+\frac1q
+\frac{q^2\log^2n}{n}
\]
and $C_V = C_6 \{ \sqrt {2C_2} + 4 D_\psi\}$.
Then~\eqref{eq:conditional-proxy-refined-HJ-short} yields
\begin{equation}
\label{eq:conditional-proxy-final-short}
\|V_{sk,M}^\perp\|_{\psi_1}
\le
C_VR_{n,M}^{1/2}.
\end{equation}
Set $D_V:=3C_V$ and define
\[
\Omega_{\mathrm{proj}}^{O*}
:=
\bigg\{
\max_{k\in[K]}\max_{s\in S}V_{sk,M}^\perp
\le
D_V\log(nK)R_{n,M}^{1/2}
\bigg\}.
\]
By~\eqref{eq:conditional-proxy-final-short},
\ref{orlicz:tail_bound}, and the union bound,
\begin{align*}
\mathbb P\big((\Omega_{\mathrm{proj}}^{O*})^c\big)
&\le
2K|S|e^{-3\log(nK)}
\le
2nK(nK)^{-3}
=
\frac{2}{(nK)^2}
\le
\frac1n.
\end{align*}
Combining~\eqref{eq:conditional-projection-gaussian-short} with the
definition of \(\Omega_{\mathrm{proj}}^{O*}\), we obtain, on that event,
\begin{align*}
\mathbb P\bigg(
\max_{k\in[K]}\max_{s\in S}
\|P_{kM}^\perp\widetilde C_k^{O*}(s)\|_{\Hb_k}^2
>
2C_G^2D_V^2\log^3(nK)R_{n,M}
\,\Bigm|\,
\widetilde{\mathcal F}_n
\bigg)
\le\frac2n.
\end{align*}
The conclusion follows from
\eqref{eq:conditional-projection-reduction-short} after setting
$
C_{\mathrm{proj}}
:=
2C_G^2D_V^2$.
\end{proof}

\begin{lemma}[Uniform covariance concentration for the projected coupled
oracle process]
\label{lem:oracle_coupled_covariance_concentration}
Suppose Assumptions~\ref{cond:stationary} and~\ref{cond:scores} hold.
Let \(q,r\in\mathbb N\) satisfy
\[
1\le r\le q \le \frac n4.
\]
For \(k\in[K]\) and $M \in \N$, let \(\check M_k:=M\wedge r_k\), and, for
\(j \in [\check M_k], \ell \in [m]\), define
\[
\widetilde Z_{\ell,kj}
:=
q^{-1/2}
\left\langle
\widetilde S_{\ell k},Z_{kj}
\right\rangle_{\Hb_k}.
\]
For \(s,t\in[n]\), \(k,k'\in[K]\),
\(j\in[\check M_k]\), and \(j'\in[\check M_{k'}]\), define
\[
\widehat\Gamma_{skj,tk'j'}
:=
\frac1m
\sum_{\ell=1}^m
a_{s\ell}a_{t\ell}
\widetilde Z_{\ell,kj}\widetilde Z_{\ell,k'j'},
\qquad
\Gamma_{skj,tk'j'}
:=
\mathbb E [\widehat\Gamma_{skj,tk'j'}].
\]
Then there exist a constant
$
C_{\mathrm{cov}}
=
C_{\mathrm{cov}}(D_\psi,C_\beta,c_\beta)>0
$
and an event
\(\Omega_{\mathrm{cov}}^{O*}\in\widetilde{\mathcal F}_n\) satisfying
$\mathbb P(\Omega_{\mathrm{cov}}^{O*}) \ge 1 - 1/n
$
such that, on \(\Omega_{\mathrm{cov}}^{O*}\),
\begin{align*}
&\max_{s,t\in[n]}
\max_{k,k'\in[K]}
\max_{\substack{
j\in[\check M_k]\\
j'\in[\check M_{k'}]
}}
\left|
\widehat\Gamma_{skj,tk'j'}-
\Gamma_{skj,tk'j'}
\right|
\le
C_{\mathrm{cov}}
\left[
\frac1{\sqrt m}
+\frac{q\log^2n}{m}
\right]
\log^2(nKM).
\end{align*}
\end{lemma}

\begin{proof}[Proof of Lemma~\ref{lem:oracle_coupled_covariance_concentration}]
Fix an admissible tuple \((s,t,k,k',j,j')\), and set
\[
b_\ell:=a_{s\ell}a_{t\ell},
\qquad
\xi_\ell
:=
b_\ell
\left\{
\widetilde Z_{\ell,kj}\widetilde Z_{\ell,k'j'}
-
\mathbb E[
\widetilde Z_{\ell,kj}\widetilde Z_{\ell,k'j'}]
\right\}.
\]
Since the coupled block sums $\tilde S_{\ell k}$ are independent and identically distributed,
the variables \(\xi_1,\ldots,\xi_m\) are independent and centered. Also recalling the definition of~\eqref{eq:def-a_sl} it holds that 
\[
|b_\ell|\le1,
\qquad
\sum_{\ell=1}^m b_\ell^2\le m,
\qquad
\widehat\Gamma_{skj,tk'j'}-\Gamma_{skj,tk'j'}
=
\frac1m\sum_{\ell=1}^m\xi_\ell.
\]
We first record two uniform bounds for the block scores. For
\[
\phi_{kj}^{(i)}
:=
\langle\varepsilon_k^{(i)},Z_{kj}\rangle_{\Hb_k},
\]
Assumption~\ref{cond:stationary} gives
\(
\|\phi_{kj}^{(i)}\|_{\psi_1}\le D_\psi
\).
Consequently, from \ref{orlicz:tail_bound},
\[
Q'(\phi_{kj}^{(i)},u)
\le D_\psi\log(2/u),
\qquad u\in(0,1),
\]
where
\(
Q'(W,u):=\inf\{z\ge0:\mathbb P(|W|>z)\le u\}
\).
Moreover,
\[
C_R
:=
\int_0^1
\inf\{h\in\mathbb N:\beta_h\le2u\}
\log^4(2/u)\,du
<\infty,
\]
and \(C_R\) depends only on \(C_\beta,c_\beta\). Since the
\(\alpha\)-mixing coefficients satisfy \(\alpha_h\le\beta_h/2\),
Theorem~2.5 of~\cite{rio2017asymptotic} gives for an absolute constant $C$
\[
\mathbb E\Big|
\sum_{i\in I_\ell}\phi_{kj}^{(i)}
\Big|^4
\le
Cq^2D_\psi^4C_R;
\]
see also \eqref{eq:rio1} for an extended version of this argument.
Using \(\widetilde S_\ell\stackrel d=S_\ell\), we therefore obtain a constant $C_1$ that depends only on $D_{\psi},C_\beta,c_\beta$ such that 
\begin{equation}
\label{eq:coupled-score-fourth-short}
\mathbb E|\widetilde Z_{\ell,kj}|^4\le C_1.
\end{equation}
On the other hand, the Orlicz triangle inequality gives
\begin{equation}
\label{eq:coupled-score-orlicz-short}
\|\widetilde Z_{\ell,kj}\|_{\psi_1}
\le
q^{-1/2}
\sum_{i\in I_\ell}\|\phi_{kj}^{(i)}\|_{\psi_1}
\le D_\psi\sqrt q.
\end{equation}
By \eqref{eq:coupled-score-fourth-short} and Cauchy--Schwarz,
\[
\mathbb E\xi_\ell^2
\le
b_\ell^2
\mathbb E\left[
\widetilde Z_{\ell,kj}^2
\widetilde Z_{\ell,k'j'}^2
\right]
\le C_1b_\ell^2.
\]
Hence independence and centeredness imply that, with $C_2=\sqrt C_1$ depending only on $C_\beta, c_\beta, D_{\psi}$,
\begin{equation}
\label{eq:covariance-sum-expectation-short}
\mathbb E\Big|
\sum_{\ell=1}^m\xi_\ell
\Big|
\le
\Big(
\sum_{\ell=1}^m\mathbb E\xi_\ell^2
\Big)^{1/2}
\le C_2\sqrt m.
\end{equation}
Furthermore, \eqref{eq:coupled-score-orlicz-short}, \ref{orlicz:triangle_up_to_constant} and
\ref{orlicz:product} give a constant $C_3$ that depends only on $D_\psi$ such that 
\[
\|\xi_\ell\|_{\psi_{1/2}}
\le C_3q|b_\ell|
\le C_3q.
\]
Thus, by~\ref{orlicz:max} there is a constant $C_4$ that depends only on $D_\psi$ such that
\begin{equation}
\label{eq:covariance-maximal-summand-short}
\Big\|
\max_{\ell \in [m]}|\xi_\ell|
\Big\|_{\psi_{1/2}}
\le
C_4q\log^2(m+1).
\end{equation}
Then by~\ref{orlicz:banach}, with \(\mathcal B=\mathbb R\) and \(p=1/2\), together with \eqref{eq:covariance-sum-expectation-short} and
\eqref{eq:covariance-maximal-summand-short}, it holds that for a constant $C_5$ that depends only on $C_\beta, c_\beta, D_{\psi}$ that 
\[
\Big\|
\sum_{\ell=1}^m\xi_\ell
\Big\|_{\psi_{1/2}}
\le
C_5\left\{
\sqrt m+q\log^2(m+1)
\right\}.
\]
Consequently it holds that 
\begin{equation*}
\left\|
\widehat\Gamma_{skj,tk'j'}-\Gamma_{skj,tk'j'}
\right\|_{\psi_{1/2}}
\le
C_5\left[
\frac1{\sqrt m}
+\frac{q\log^2(m+1)}m
\right].
\end{equation*}
There are at most \((nKM)^2\) admissible covariance entries. Applying~\ref{orlicz:tail_bound}  and taking a union bound gives an
event \(\Omega_{\mathrm{cov}}^{O*}\in\widetilde{\mathcal F}_n\) with
$
\mathbb P(\Omega_{\mathrm{cov}}^{O*}) \ge 1-1/n
$
on which
\begin{align*}
&\max_{s,t\in[n]}
\max_{k,k'\in[K]}
\max_{\substack{
j\in[\check M_k]\\
j'\in[\check M_{k'}]
}}
\left|
\widehat\Gamma_{skj,tk'j'}-\Gamma_{skj,tk'j'}
\right|
\le
C_6
\left[
\frac1{\sqrt m}
+\frac{q\log^2(m+1)}m
\right]
\log^2(nKM)
\end{align*}
for a constant $C_6$ that depends only on $C_\beta, c_\beta, D_{\psi}$. Since \(m+1\le n+1\), the asserted bound follows after setting 
\(C_{\mathrm{cov}}=C_6\).
\end{proof}

\begin{lemma}[Conditional Gaussian comparison for the projected coupled
oracle process]
\label{lem:projected_coupled_oracle_gaussian_comparison}
Fix \(\rho>0\) and \(x\in(0,3/16]\). Suppose
Assumptions~\ref{cond:stationary} and~\ref{cond:scores} hold. Let
\(q,r,M\in\mathbb N\) satisfy
\[
1\le r\le q \le\frac n4,
\]
and
\begin{equation}
\label{eq:M-condition-conditional-Gaussian-comparison}
M^{2\gamma_L}
\le
\frac{
C_L^2x(m\wedge q)
}{
2C_4\bigl(D_\psi^2S_{\beta,1}\vee C_{\mathcal K}\bigr)
},
\end{equation}
where \(C_4\) is the universal constant from
Lemma~\ref{lem:bound-on-operator-norm-covariances}. Recall
\(G_{sk}^M\) from~\eqref{eq:definition-GMsk}.
Then there exist a constant
$
C_{\mathrm{GB}}
=
C_{\mathrm{GB}}
(\rho,D_\psi,C_L,C_U,\gamma_L,\gamma_U,C_\beta,c_\beta)>0
$
and an event
\(\Omega_{\mathrm{GB}}^{O*}\in\widetilde{\mathcal F}_n\) satisfying
$
\mathbb P(\Omega_{\mathrm{GB}}^{O*})\ge 1-1/n
$
such that, on \(\Omega_{\mathrm{GB}}^{O*}\),
\begin{align*}
&\sup_{t\ge\rho}
\bigg|
\mathbb P\Big(
\max_{s\in S_x}\max_{k\in[K]}
\big\|G_{sk}^M\big\|_2^2
\le t
\Big)
-
\mathbb P\Big(
\max_{s\in S_x}\max_{k\in[K]}
\big\|(O_{kM}\circ P_{kM})\widetilde C_k^{O*}(s)\big\|_2^2
\le t
\,\Big|\,
\widetilde{\mathcal F}_n
\Big)
\bigg|
\\
&\hspace{4.7cm} \le
C_{\mathrm{GB}}x^{-1}
M^{2/3}\left[
\frac1{\sqrt m}
+\frac{q\log^2n}{m}
\right]^{1/3}
\log^{\,11/6+6\kexp}(nKM),
\end{align*}
where $\kexp:=\gamma_L/({2\gamma_U-1})$.
\end{lemma}

\begin{proof}[Proof of Lemma~\ref{lem:projected_coupled_oracle_gaussian_comparison}]
Let \(\mathcal I_x:=S_x\times[K]\), and define 
\[
X
:=
(G_{sk}^M)_{(s,k)\in\mathcal I_x},
\qquad
Y
:=
\bigl((O_{kM}\circ P_{kM})
\widetilde C_k^{O*}(s)\bigr)_{(s,k)\in\mathcal I_x},
\]
with $G_{sk}^M$ from \eqref{eq:definition-GMsk} and $O_{kM}$ from \eqref{eq:def-OkM}. Note that $X$ is centered Gaussian, whereas \(Y\) is centered Gaussian conditionally on
\(\widetilde{\mathcal F}_n\). By the
definitions in Lemma~\ref{lem:oracle_coupled_covariance_concentration} and the representations in \eqref{eq:covariance-GMsk} and \eqref{eq:covariance-GMsk-alternative-expresssion},
their covariance entries satisfy
\[
\Cov(X)_{skj,tk'j'}
=
\Gamma_{skj,tk'j'},
\qquad
\Cov(Y\mid\widetilde{\mathcal F}_n)_{skj,tk'j'}
=
\widehat\Gamma_{skj,tk'j'}.
\]
Let $
\Omega_{\mathrm{GB}}^{O*}
:=
\Omega_{\mathrm{cov}}^{O*}$
be the event from that lemma. Then
\(
\mathbb P((\Omega_{\mathrm{GB}}^{O*})^c)\le1/n
\), and, on \(\Omega_{\mathrm{GB}}^{O*}\), there is a constant $C$ depending only on $C_\beta, c_\beta, D_\psi$ such that 
\begin{equation}
\label{eq:conditional-Gaussian-covariance-distance-short}
\left\|
\Cov(Y\mid\widetilde{\mathcal F}_n)-\Cov(X)
\right\|_{\infty}
\le
C\delta_{\mathrm{GB}} := 
C \left[
\frac1{\sqrt m}
+\frac{q\log^2n}{m}
\right]
\log^2(nKM).
\end{equation}

Next, Condition~\eqref{eq:M-condition-conditional-Gaussian-comparison} and
Corollary~\ref{lem:good_eigen_values_2} imply that, for every
\((s,k)\in\mathcal I_x\),
\begin{equation}
\label{eq:conditional-Gaussian-decay-short}
\Cov(G_{sk}^M)
\in
\operatorname{Decay}
\left(
\frac{xC_L^2}{2},
2C_U^2,
2\gamma_L,
2\gamma_U
\right).
\end{equation}
In particular, by \eqref{eq:trace-as-sum-eigenvalues} and definition of $\operatorname{Decay}(\cdot)$ in Definition~\ref{def:covariance-decay},
\[
\sup_{(s,k)\in\mathcal I_x}
\Tr\big(  \Cov(X_{sk}) \big)
\le
2C_U^2\sum_{j=1}^\infty j^{-2\gamma_U}
\le C_1,
\]
where $C_1$ depends only on $\gamma_U$ and $C_U$. Note that $\Tr(A) \le M \| A \|_\infty$ for a positive semidefinite $(M \times M)$-matrix $A$. By \eqref{eq:conditional-Gaussian-covariance-distance-short}, on
\(\Omega_{\mathrm{GB}}^{O*}\) and for a constant $C_2$ that may have the same parameter dependencies as $C$ and $C_1$ and that may be chosen larger than $C_1 \vee 1$
\[
\sup_{(s,k)\in\mathcal I_x}
\Tr\big( \Cov(Y_{sk}\mid\widetilde{\mathcal F}_n)\big)
\le
C_2 (1 + M\delta_{\mathrm{GB}}).
\]
Consequently, 
\begin{equation}
\label{eq:conditional-Gaussian-U-short}
U
:= 
\max_{(s,k)\in\mathcal I_x}
\left\{
\sqrt{\Tr\big( \Cov(X_{sk})\big)}
\vee
\sqrt{\Tr\big( \Cov(Y_{sk} \mid \widetilde{\mathcal F}_n\big)}
\right\}
\le C_2\{1+\sqrt{M\delta_{\mathrm{GB}}}\}.
\end{equation}

We apply Lemma~\ref{lem:gaussbootsphere-new} conditionally on
\(\widetilde{\mathcal F}_n\), with block index set \(\mathcal I_x\),
largest block dimension at most \(M\), and auxiliary integer $n\ge 4$. The heterogeneous block dimensions cause no change: one may pad each block
with zero coordinates up to dimension \(M\). Since the present result
concerns squared norms, the lower threshold in that lemma is
\(\sqrt\rho\).  Using
\eqref{eq:conditional-Gaussian-covariance-distance-short} and
\eqref{eq:conditional-Gaussian-decay-short}, we obtain for a constant $C_3$ that may depend on $D_\psi, C_
\beta, c_\beta, C_U, C_L, \rho, \gamma_U, \gamma_L$
\begin{align}
&\sup_{t\ge\rho}
\left|
\mathbb P\left(
\max_{(s,k)\in\mathcal I_x}\|X_{sk}\|_2^2\le t
\right)
-
\mathbb P\left(
\max_{(s,k)\in\mathcal I_x}\|Y_{sk}\|_2^2\le t
\,\middle|\,
\widetilde{\mathcal F}_n
\right)
\right|
\nonumber\\
&\qquad\le
C_3x^{-1}
\left[
M^{2/3}\delta_{\mathrm{GB}}^{1/3}
+n^{-2/3}
\right]
\log^{\,7/6+6\kexp}
\left(
en(|\mathcal I_x|+1)(U\vee1)
\right).
\label{eq:conditional-Gaussian-comparison-before-simplification}
\end{align}
Here all powers of \(C_L,C_U\) and \(\rho\) have been absorbed into
\(C_3\); the factor \(x^{-1}\) follows from the lower eigenvalue constant
\(C_L\sqrt{x/2}\) in
\eqref{eq:conditional-Gaussian-decay-short}. It remains to simplify the right-hand side. It holds that 
\[
n^{-2/3}
\le
M^{2/3}\delta_{\mathrm{GB}}^{1/3},
\]
since $M \ge 1$ and
\[
n^{-2/3}
\le
n^{-1/6}
\le
m^{-1/6}
\le
\left[
\left\{
\frac1{\sqrt m}
+
\frac{q\log^2 n}{m}
\right\}
\log^2(nKM)
\right]^{1/3}
= \delta_{\mathrm{GB}}^{1/3}.
\]
Moreover, \(|\mathcal I_x|\le nK\), and
\eqref{eq:conditional-Gaussian-U-short}, together with
\(q\le n\) and \(m\ge2\), gives
\[
\log\left(
en(|\mathcal I_x|+1)(U\vee1)
\right)
\le
C_4\log(nKM),
\]
for a constant $C_4$ that may have the same parameter dependencies as the constant $C_3$. Substituting these two bounds into
\eqref{eq:conditional-Gaussian-comparison-before-simplification} proves
the claim.
\end{proof}

\subsection{Localization}

Recall that \[
    \Delta_k:=\|\delta_k\|_{\mathcal H_k},
    \qquad
    \theta_k:=\omega_k/n,
    \qquad
    \gamma_k:=\theta_k\wedge(1-\theta_k).
\]

\begin{lemma}[Localization under a general noise process]
\label{lem:general_localization_lemma}
Fix some nonempty collection of changed coordinates
\(\chgset\subseteq\mathcal S\), with \(\mathcal S\) from
\eqref{eq:change-point-coordinates}. For \(r\in[n-1]\), define the local
noise envelope
\begin{align}
\label{eq:definition-Hnr}
\mathcal H_n(r) := \mathcal H_n(r,\chgset)
:=
\max_{k\in\chgset}
\sup_{\substack{
0\leq s\leq n\\
|s-\omega_k|\geq r
}}
\frac{
\sqrt n\,
\|C_{n,k}^{O}(s)-C_{n,k}^{O}(\omega_k)\|_{\Hb_k}
}{
|s-\omega_k|\,\gamma_k\Delta_k
},
\end{align}
where the supremum over an empty set is understood to be zero.
Then, for every \(r\in[n]\), we have the following inclusion of events:
\[
\{ \mathcal H_n(r)<1/5 \}
\subseteq
\Big\{
\max_{k\in\chgset}
|\widehat\omega_k-\omega_k|
<
r
\Big\}.
\]
\end{lemma}

\begin{lemma}[Localization under polynomially $\beta$-mixing noise]
\label{lem:polynomial_beta_localization}
Suppose that Assumption~\ref{cond:stationary} holds, and let $\chgset\subseteq\mathcal S$ with $\mathcal S$ from \eqref{eq:change-point-coordinates} be some nonempty collection of
changed coordinates.
Then there exists a universal constant $C_0>0$ such that, 
for every $\eta\in(0,1/12)$ and every integer $\ell \in [n]$,  we have
\[
    \mathbb P\Big(
        \max_{k\in\chgset}
        |\widehat\omega_k-\omega_k|
        <
        r_\beta(\ell,\eta)
    \Big)
    \ge
    1-\eta
    -
    2
    \left\lceil\frac{n}{\ell}\right\rceil
    \beta_\ell,
\]
where
\begin{align}
\label{eq:def-rbeta}
    r_\beta(\ell,\eta)
    :=
    \Big\lceil
        C_0
        \Big\{
            \frac{
                V^2x_{n,\chgset,\eta}^{\,2}
            }{
                \nu_\chgset^2
            }
            +
            \frac{
                D_\psi\ell\Lambda_{n,\ell}
                x_{n,\chgset,\eta}
            }{
                \nu_\chgset
            }
        \Big\}
    \Big\rceil,
\end{align}
with $V$ some constant that may depend only on $D_\psi$, $C_\beta$, and $c_\beta$ and with
\begin{align}
\label{eq:xnk}
\nu_\chgset
     :=
     \min_{k\in\chgset}\gamma_k\Delta_k,
     \qquad
    x_{n,\chgset,\eta}
    :=
    \log\left(
        \frac{12  n|\chgset|}{\eta}
    \right),
    \qquad
    \Lambda_{n,\ell}
    :=
    \log\left(\frac{en}{\ell}\right).
\end{align}
\end{lemma}

\begin{lemma}[Uniform interval-sum bound under polynomial mixing]
\label{lem:uniform_interval_sums_new}
Suppose Assumption~\ref{cond:stationary} holds, and let
\(\mathcal I_n\) denote the collection of all nonempty integer
intervals contained in \([n]\). For every \(\ell\in[n]\) and
\(\eta\in(0,1/12)\), define
\[
\Lambda_{n,\ell}
:=
\log\left(\frac{en}{\ell}\right),
\qquad
x_{n,\eta}
:=
\log\left(\frac{12n^2K}{\eta}\right);
\]
see also \eqref{eq:xnk}.
Then there exist a universal constant \(C>0\), a constant
\(V>0\) depending only on \(D_\psi,C_\beta,c_\beta\), and an event
$
\mathcal E_{\ell,\eta}^{\mathrm{int}}
\in\mathcal F_n
$
such that
\[
\Prob\left(
\{\mathcal E_{\ell,\eta}^{\mathrm{int}}\}^c
\right)
\le
\eta
+
2\left\lceil\frac n\ell\right\rceil\beta_\ell
\]
and, on \(\mathcal E_{\ell,\eta}^{\mathrm{int}}\), for every
\((k,I)\in[K]\times\mathcal I_n\),
\[
\Big\|
\sum_{i\in I}\varepsilon_k^{(i)}
\Big\|_{\Hb_k}
\le
Cx_{n,\eta}
\Big\{
V\sqrt{|I|}
+
D_\psi\ell\Lambda_{n,\ell}
\Big\}.
\]
\end{lemma}

\begin{corollary}[Simultaneous coordinatewise localization]
\label{cor:coordinatewise_beta_localization_new}
Suppose Assumption~\ref{cond:stationary} holds, and let
\(\mathcal S\subseteq[K]\) be the collection of changed coordinates
from~\eqref{eq:change-point-coordinates}. Assume that $\mathcal S \ne \varnothing$.
There exists a constant
\[
    C_{\mathrm{loc}}
    =
    C_{\mathrm{loc}}(D_\psi,C_\beta,c_\beta)>0
\]
such that, for every $\ell \in \N$, on the event
\(\mathcal E_{\ell,\,1/(12n)}^{\mathrm{int}}\) from
Lemma~\ref{lem:uniform_interval_sums_new}, we have
\[
    \forall\, k\in\mathcal S: \quad |\widehat\omega_k-\omega_k|
    <
    r_{\beta,k}^{(n)}(\ell),
\]
where
\begin{equation}
\label{eq:def-coordinatewise-rbeta-new}
    r_{\beta,k}^{(n)}(\ell)
    :=
    \left\lceil
        C_{\mathrm{loc}}
        \frac{\log(nK)}{\nu_k}
        \left\{
            \frac{\log(nK)}{\nu_k}
            +
            \ell\log\left(\frac{en}{\ell}\right)
        \right\}
    \right\rceil 
\end{equation}
with $\nu_k:=\gamma_k\Delta_k>0$ from \eqref{eq:def-nu_k}.
Consequently, by Lemma~\ref{lem:uniform_interval_sums_new},
\[
    \mathbb P\left(
        \exists k\in\mathcal S:
        |\widehat\omega_k-\omega_k|
        \ge
        r_{\beta,k}^{(n)}(\ell)
    \right)
    \le
    \frac1n
    +
    2\left\lceil\frac n\ell\right\rceil\beta_\ell.
\]
\end{corollary}

\begin{proof}[Proof of Lemma~\ref{lem:general_localization_lemma}]
A tedious but straightforward calculation shows that
\begin{align}
\label{eq:cnk-dnk-cnko}
C_{n,k}(s)
=
D_{n,k}(s)+C_{n,k}^{O}(s),
\end{align}
where \(D_{n,k}(s)\) is the deterministic CUSUM process defined by
\[
D_{n,k}(s)
=
d_{n,k} \, a_{n,k}(s),
\]
with
\[
d_{n,k}
=
-\sqrt n\,\theta_k(1-\theta_k)\delta_k,
\qquad
a_{n,k}(s)
:=
\begin{cases}
\dfrac{s}{\omega_k},
& 0\leq s\leq\omega_k,\\[1.2ex]
\dfrac{n-s}{n-\omega_k},
& \omega_k<s\leq n.
\end{cases}
\]
Note that \(0\leq a_{n,k}(s)\leq1\), that \(a_{n,k}(\omega_k)=1\),
and that
\begin{align}
\label{eq:sup-dnk}
\|D_{n,k}(\omega_k)\|_{\mathcal H_k}
=
\sqrt n\,\theta_k(1-\theta_k)\Delta_k.
\end{align}

Fix \(k\in\chgset\), and abbreviate
\[
\omega:=\omega_k,
\qquad
\theta:=\theta_k,
\qquad
\gamma:=\gamma_k,
\qquad
\Delta:=\Delta_k,
\qquad
d:=d_{n,k},
\]
as well as
\[
a(s):=a_{n,k}(s),
\qquad
C(s):=C_{n,k}(s),
\qquad
C^O(s):=C_{n,k}^{O}(s),
\qquad
D(s):=D_{n,k}(s).
\]

We will show that, on the event \(\{\mathcal H_n(r)<1/5\}\), every
\(s\in\{1,\dots,n-1\}\) with \(|s-\omega|\ge r\) satisfies
\[
\|C(s)\|_{\mathcal H_k}<\|C(\omega)\|_{\mathcal H_k}.
\]
As a consequence, every maximizer \(\widehat \omega_k\) must satisfy
\(|\widehat\omega_k-\omega_k|<r\). Since \(k\in\chgset\) was arbitrary,
this proves the result.

Hence, fix \(s\in\{1,\ldots,n-1\}\) such that \(|s-\omega|\geq r\).
Suppose, toward a contradiction, that
\begin{align}
\label{eq:proof-by-contradiction}
\|C(s)\|_{\mathcal H_k}
\geq
\|C(\omega)\|_{\mathcal H_k}.
\end{align}
By the definition \(D(s)=a(s)d\), with \(D(\omega)=d\), and by the
decomposition \(C(s)=D(s)+C^O(s)\) from \eqref{eq:cnk-dnk-cnko},
\begin{align*}
C(s)
&=
a(s)d+C^O(s)
\\
&=
a(s)\{C(\omega)-C^O(\omega)\}+C^O(s)
\\
&=
a(s)C(\omega)
+
\{C^O(s)-C^O(\omega)\}
+
\{1-a(s)\}C^O(\omega).
\end{align*}
Hence, under \eqref{eq:proof-by-contradiction}, the triangle inequality gives
\begin{align}
\label{eq:c-co-ineq}
(1-a(s))\|C(\omega)\|_{\mathcal H_k}
\le
\|C^O(s)-C^O(\omega)\|_{\mathcal H_k}
+
(1-a(s))\|C^O(\omega)\|_{\mathcal H_k}.
\end{align}
Using \eqref{eq:cnk-dnk-cnko} and the triangle inequality gives
\[
\|D(\omega)\|_{\mathcal H_k}
\le
\|C(\omega)\|_{\mathcal H_k}
+
\|C^O(\omega)\|_{\mathcal H_k}.
\]
Combining this display with \eqref{eq:c-co-ineq} gives
\[
L(s)
:=
(1-a(s))\|D(\omega)\|_{\mathcal H_k}
\le
\|C^O(s)-C^O(\omega)\|_{\mathcal H_k}
+
2(1-a(s))\|C^O(\omega)\|_{\mathcal H_k}.
\]
A direct calculation using \eqref{eq:sup-dnk} gives
\[
L(s)
=
\begin{cases}
\dfrac{(\omega-s)(1-\theta)\Delta}{\sqrt n},
& s\leq\omega,\\[1.2ex]
\dfrac{(s-\omega)\theta\Delta}{\sqrt n},
& s>\omega,
\end{cases}
\]
and hence
\[
L(s)
\ge
\frac{|s-\omega|\gamma\Delta}{\sqrt n}.
\]
Therefore, by the definition of \(\mathcal H_n(r)\),
\[
\|C^O(s)-C^O(\omega)\|_{\mathcal H_k}
\le
L(s)\mathcal H_n(r).
\]
Moreover,
\begin{align*}
2(1-a(s))\|C^O(\omega)\|_{\mathcal H_k}
&=
2L(s)
\frac{\|C^O(\omega)\|_{\mathcal H_k}}
{\|D(\omega)\|_{\mathcal H_k}}
\\
&=
2L(s)
\frac{\|C^O(\omega)\|_{\mathcal H_k}}
{\sqrt n\,\theta(1-\theta)\Delta}
\\
&\le
4L(s)
\frac{\|C^O(\omega)\|_{\mathcal H_k}}
{\sqrt n\,\gamma\Delta},
\end{align*}
where we used \(\theta(1-\theta)=\gamma(1-\gamma)\) and \(\gamma\leq1/2\).

If \(\omega\ge n/2\), then
\[
\frac{\|C^O(\omega)\|_{\mathcal H_k}}
{\sqrt n\,\gamma\Delta}
=
\frac{\omega}{n}
\frac{
\sqrt n\|C^O(0)-C^O(\omega)\|_{\mathcal H_k}
}{
|0-\omega|\gamma\Delta
}
\le
\mathcal H_n(r).
\]
Indeed, this is immediate when \(r\le n/2\). If \(r>n/2\), then the
constraints \(|s-\omega|\ge r>n/2\) and \(\omega\ge n/2\) imply
\(\omega\ge s\), and hence
\[
|0-\omega|
=
\omega
=
\omega-s+s
=
|\omega-s|+s
\ge r.
\]
Likewise, if \(\omega\le n/2\), then
\[
\frac{\|C^O(\omega)\|_{\mathcal H_k}}
{\sqrt n\,\gamma\Delta}
=
\frac{n-\omega}{n}
\frac{
\sqrt n\|C^O(n)-C^O(\omega)\|_{\mathcal H_k}
}{
|n-\omega|\gamma\Delta
}
\le
\mathcal H_n(r).
\]
Overall,
\[
2(1-a(s))\|C^O(\omega)\|_{\mathcal H_k}
\le
4 L(s)\mathcal H_n(r).
\]
Combining the previous inequalities yields
\[
L(s)
\le
5L(s)\mathcal H_n(r).
\]
Since \(s\ne\omega\), we have \(L(s)>0\). Thus, on the event
\(\{\mathcal H_n(r)<1/5\}\), it is impossible that
\[
\|C(s)\|_{\mathcal H_k}
\ge
\|C(\omega)\|_{\mathcal H_k}.
\]
This proves the assertion.
\end{proof}

\begin{proof}[Proof of Lemma~\ref{lem:polynomial_beta_localization}] 
Throughout, we write $r_\beta=r_\beta(\ell,\eta)$. The assertion is trivial if $r_\beta\ge n$, so assume that $r_\beta < n$.
Partition $\{1,\ldots,n\}$ into consecutive blocks
\[
    B_1,\ldots,B_m,
    \qquad
    m:=\left\lceil\frac{n}{\ell}\right\rceil,
\]
where every block except possibly the last has length $\ell$; note that the blocking structure is different than in other proofs in this manuscript.
For each block $j$, write $E_j = (\eps^{(i)})_{i \in B_j}$.
By Berbee's coupling lemma applied to both $E_1, E_3, E_5, \dots$ and to $E_2, E_4, E_6, \dots$, we can construct a sequence $(\eps^{(i), \star})_{i \in [n]}$ such that each block variable $E_j^* = (\eps^{(i), \star})_{i \in B_j}$ has the same distribution as $E_j$, such that $E_1^\star, E_3^\star, E_5^\star, \dots$ are mutually independent, such that $E_2^\star, E_4^\star, E_6^\star, \dots$ are mutually independent and such that 
\begin{align*}
\Prob\big( \exists i \in [n]: \eps^{(i)} \ne \eps^{(i),\star}\big) = \Prob\big( \exists j \in [m]: E_j \ne E_j^\star\big) \le 2 m\beta_\ell \le 2 \left\lceil\frac{n}{\ell}\right\rceil \beta_\ell.
\end{align*}
As a consequence, it is sufficient to show that 
\[
    \mathbb P\Big(
        \max_{k\in\chgset}
        |\widehat\omega_k^\star-\omega_k|
        <
        r_\beta
    \Big)
    \ge
    1-\eta,
\]
where $\widehat\omega_k^\star$ is defined analogously as $\widehat\omega_k$, but based on $(\eps^{(i),\star})_i$ instead of $(\eps^{(i)})_i$. 

An application of Lemma~\ref{lem:general_localization_lemma} (with  $(\eps^{(i),\star})_i$ instead of $(\eps^{(i)})_i$) implies that
\[
    \mathbb P\Big(
        \max_{k\in\chgset}
        |\widehat\omega_k^\star-\omega_k|
        \ge
        r_\beta
    \Big)
    \le \Prob\Big(\mathcal H_n^\star (r_\beta)  \ge  \frac15 \Big)
\]
with  $\mathcal H_n^\star (r)$ the star-versions of $\mathcal H_n(r)$ from  \eqref{eq:definition-Hnr}. We need to prove that the right-hand side is bounded by $\eta$. For that purpose consider the event 
\begin{align}
\label{eq:bound-interval-sums}
    \mathcal E_{\ell, \eta}^\star(C) = \bigcap_{k \in \chgset} \bigcap_{I \in \mathcal I_k} \bigg\{\Big\|
        \sum_{i\in I}\eps_k^{(i), \star}
    \Big\|_{\Hb_k}
    \le
    Cx_{n,\chgset,\eta}
    \left(
        V\sqrt{|I|}
        +
        D_\psi\ell\Lambda_{n,\ell}
    \right) \bigg\},
\end{align}
where $\mathcal I_k$ comprises all intervals $I \subseteq[n]$ of the form $I=[n], I=[\omega_k]$ or $I=\{(s \wedge \omega_k)+1, \dots, s\vee \omega_k\}$ for some $s \in [n-1]$ with $s \ne \omega_k$. The proof is finished once we show that we can choose a universal constant $C$ such that 
\begin{align}
\label{eq:to-be-shown}
\text{(i)} \quad \Prob\big(\mathcal E_{\ell, \eta}^\star(C)\big) \ge 1-\eta \qquad \text{and} \qquad \text{(ii)} \quad\mathcal E_{\ell, \eta}^\star(C) \subseteq \Big\{\mathcal H_n^\star (r_\beta)  <  \frac15\Big\}.
\end{align}

We start by proving (i).
For that purpose, fix $k \in \chgset$ and an interval $I \in \mathcal I_k$, and write $d:=|I|$. 
Let
\[
    J_{\mathrm o}(I)
    :=
    \{j:\ j\ \text{is odd and }B_j\subseteq I\}, \qquad
    J_{\mathrm e}(I)
    :=
    \{j:\ j\ \text{is even and }B_j\subseteq I\}.
\]
We can then write
\begin{align} \label{eq:main-decomp-star}
\sum_{i\in I}\eps_k^{(i)^, \star} 
= 
\sum_{j\in J_{\mathrm o}(I)} Y_{j,k}^\star + 
\sum_{j\in J_{\mathrm e}(I)} Y_{j,k}^\star +
\sum_{i \in N(I)} \eps^{(i), \star}_k,
\end{align}
where $N(I)$ contains all indices $i \in I$ that are not part of any set $B_j$ with $j \in J_{\mathrm o}(I) \cup J_{\mathrm e}(I)$, and where
\[
    Y_{j,k}^\star
    :=
    \sum_{i\in B_j}\eps_k^{(i),\star}, \qquad j \in [m].
\]

We will bound each of the three sums on the right-hand side of \eqref{eq:main-decomp-star} individually, and start by treating the sum over $J_{\mathrm o}(I)$. 
By the triangle inequality for the $\psi_1$-norm, see \ref{orlicz:triangle},
strict stationarity, and Assumption~\ref{cond:stationary},
\begin{align*}
    \left\|
        \|Y_{j,k}^\star\|_{\Hb_k}
    \right\|_{\psi_1}
    =
    \left\|
        \|Y_{j,k}\|_{\Hb_k}
    \right\|_{\psi_1}
    \le
    D_\psi |B_j|
    \le
    D_\psi\ell,
\end{align*}
where $Y_{j,k}:=\sum_{i\in B_j}\eps_k^{(i)}$ denotes the non-star version of $Y_{j,k}^\star$.
Hence, since $|J_{\mathrm o}(I)|\le m$, 
\ref{orlicz:max} implies
\begin{align}
\label{eq:max-ystar}
    \Big\|
        \max_{j\in J_{\mathrm o}(I)}
        \|Y_{j,k}^\star\|_{\Hb_k}
    \Big\|_{\psi_1}
    \le
    C_1 D_\psi\ell\log(em)
    \le
    C_1 D_\psi\ell\Lambda_{n,\ell}
\end{align}
for some universal constant $C_1>0$.
Moreover by~\eqref{eq:bound_on_square_sum_beta_mixing} there is a constant $V$ depending only on $D_\psi$, $c_\beta$ and $C_\beta$ such that
\begin{align*}
    \mathbb E
    \|Y_{j,k}^\star\|_{\Hb_k}^{2}
    =
    \mathbb E
    \|Y_{j,k}\|_{\Hb_k}^{2}
    \le
    V^2|B_j|.
\end{align*}
Since the variables $(Y_{j,k}^\star)_{j\in J_{\mathrm o}(I)}$ are independent
and centered, and since the complete blocks contained in $I$ have total
length at most $d$, we obtain that
\begin{align}
\label{eq:sum-ystar}
    \sum_{j\in J_{\mathrm o}(I)}
    \mathbb E\|Y_{j,k}^\star\|_{\Hb_k}^{2}
    \le
    V^2d.
\end{align}

Next, note the following consequence of~\ref{orlicz:banach}: if $U_1,\ldots,U_N$ are independent, centered, Hilbert-valued random
variables, then there exists a universal constant $C_2>0$ such that
\begin{align*}
\bigg\|\Big\|\sum_{j=1}^{N}U_j\Big\|\bigg\|_{\psi_1} 
\leq 
C_2 \Big\{\Big( \sum_{j=1}^{N} \E\| U_j \|^2\Big)^{1/2}+\Big\|\max_{j=1}^{N}\|U_{j}\|\Big\|_{\psi_1}\Big\}.
\end{align*}
Together with \eqref{eq:max-ystar} and \eqref{eq:sum-ystar}, we obtain that, for some universal constant $C_3>0$, 
\[
\bigg\|\Big\|
            \sum_{j\in J_{\mathrm o}(I)}
            Y_{j,k}^\star
        \Big\|_{\Hb_k}\bigg\|_{\psi_1} \le C_3 \left\{
            V\sqrt d
            +
            D_\psi\ell\Lambda_{n,\ell}
        \right\},
\]
The ordinary $\psi_1$ tail bound from \ref{orlicz:tail_bound} then yields,
for every $x\ge0$,
\begin{align}
\label{eq:tail-ystar}
    \mathbb P\bigg(
        \Big\|
            \sum_{j\in J_{\mathrm o}(I)}
            Y_{j,k}^\star
        \Big\|_{\Hb_k}
        >
        C_3x
        \left\{
            V\sqrt d
            +
            D_\psi\ell\Lambda_{n,\ell}
        \right\}
    \bigg)
    \le
    2e^{-x}.
\end{align}
The same estimate holds for the complete even blocks, i.e., with $J_{\mathrm o}(I)$ replaced by $J_{\mathrm e}(I)$.

Finally, regarding the last sum on the right-hand side of \eqref{eq:main-decomp-star}, note that $N(I)$ contains at most $2 \ell$ elements. Hence, by the triangle inequality for the $\psi_1$-norm, see \ref{orlicz:triangle},
strict stationarity, and Assumption~\ref{cond:stationary},
\[
\bigg\| \Big\| \sum_{i \in N(I)} \eps_k^{(i),\star} \Big\|_{\Hb_k} \bigg\|_{\psi_1} \le 2 D_\psi \ell.
\]
Therefore, by \ref{orlicz:tail_bound}, for every $x \ge 0$,
\begin{align}
\label{eq:tail-ystar2}
    \mathbb P\bigg(
        \Big\|
             \sum_{i \in N(I)} \eps_k^{(i),\star}  \Big\|_{\Hb_k}
        >
        2x D_\psi \ell \Big)
    \le
    2e^{-x}.
\end{align}
Recalling the main decomposition from \eqref{eq:main-decomp-star}, and combining \eqref{eq:tail-ystar} and \eqref{eq:tail-ystar2}, we have found a universal constant $C_4>0$ such that, for any $x \ge 0$,
\begin{align}
\label{eq:tail-epsstar}
    \mathbb P\bigg(
        \Big\|
             \sum_{i \in I} \eps_k^{(i),\star}  \Big\|_{\Hb_k}
        >
        C_4x
        \left\{
            V\sqrt d
            +
            D_\psi\ell\Lambda_{n,\ell}
        \right\} \bigg)
    \le
    6e^{-x}.
\end{align}
Recalling the definition of $\mathcal E_{n,k}^\star$ from \eqref{eq:bound-interval-sums} and noting that $|\mathcal I_k|=n+1 \le 2n$, the union bound then implies that 
\[
\Prob \Big( (\mathcal E_{n,k}^\star)^c\Big) \le 12 n |\chgset|e^{- x_{n,\chgset, \eta}} = \eta.
\]
We have hence shown (i) in \eqref{eq:to-be-shown} with $C=C_{4}$.

It remains to prove (ii) in \eqref{eq:to-be-shown}, so suppose we are on the event $\mathcal E_{\ell, \eta}^*(C)$ from \eqref{eq:bound-interval-sums}. We need to show that $\mathcal H_n(r_\beta) < 1/5$, where
\begin{align*}
\mathcal H_n(r) 
=
\max_{k\in\chgset}
\sup_{\substack{
1\leq s\leq n-1\\
|s-\omega_k|\geq r
}}
\frac{
\sqrt n\,
\|C_{n,k}^{O}(s)-C_{n,k}^{O}(\omega_k)\|_{\mathcal H_k}
}{
|s-\omega_k|\,\gamma_k\Delta_k
};
\end{align*}
here, and in the remaining parts of this proof, we will omit the stars everywhere.
Recall from \eqref{eq:definition-CnkO} that
\[
    \sqrt{n} C_{n,k}^O(s)
    =
        \sum_{i=1}^{s}\eps_k^{(i)}
        -
        \frac{s}n
        \sum_{i=1}^{n}\eps_k^{(i)}.
\]
Let $d:=|s-\omega_k|$. If \(s>\omega_k\), then \(d=s-\omega_k\) and
\begin{equation}\label{eq:COnk-difference-decomposition1}
\sqrt{n}\{C_{n,k}^O(s)-C_{n,k}^O(\omega_k)\}
=
\sum_{i=\omega_k+1}^{s}\varepsilon_k^{(i)} - \frac{d}n \sum_{i=1}^{n}\eps_k^{(i)}.
\end{equation}
If $s<\omega_k$, then \(d=\omega_k-s\) and
\begin{equation}\label{eq:COnk-difference-decomposition2}
\sqrt{n}\{C_{n,k}^O(s)-C_{n,k}^O(\omega_k)\}
=
-
\sum_{i=s+1}^{\omega_k}\varepsilon_k^{(i)} + \frac{d}n \sum_{i=1}^{n}\eps_k^{(i)}.
\end{equation}
Therefore, after taking the Hilbert norm and applying the triangle inequality,
\begin{align*}
\sqrt n
\Big\|
C_{n,k}^{O}(s)-C_{n,k}^{O}(\omega_k)
\Big\|_{\mathcal H_k}
&\le 
\Big\|
\sum_{i=(s\wedge\omega_k)+1}^{s\vee\omega_k}
\varepsilon_k^{(i)}
\Big\|_{\Hb_k}
+ \frac{d}{n}
\Big\|
\sum_{i=1}^{n}\varepsilon_k^{(i)}
\Big\|_{\mathcal H_k}.
\end{align*}
On the event $\mathcal E_{\ell, \eta}^*(C)$ from \eqref{eq:bound-interval-sums}, we hence obtain the bound
\begin{align*}
\sqrt n
\Big\|
C_{n,k}^{O}(s)-C_{n,k}^{O}(\omega_k)
\Big\|_{\mathcal H_k}
&\le 
Cx_{n,\chgset,\eta}
    \left\{
        V\sqrt d
        +
        D_\psi\ell\Lambda_{n,\ell}
    \right\}
    +
    C\frac{d}{n}
    x_{n,\chgset,\eta}
    \left\{
        V\sqrt n
        +
        D_\psi\ell\Lambda_{n,\ell}
    \right\}.
\end{align*}
Since $r < n$ and $r \le d$, this yields the bound
\[
    \mathcal H_n(r_\beta)
    \le
    \frac{Cx_{n,\chgset,\eta}}{\nu_\chgset}
    \left\{
        \frac{
            V
        }{\sqrt r_\beta}
        +
        \frac{
            D_\psi\ell\Lambda_{n,\ell}
        }{r_\beta}
    \right\}.
\]
By the definition of $r_\beta$ in \eqref{eq:def-rbeta} and by taking the universal constant $C_0$ sufficiently large,
\[
    \frac{Cx_{n,\chgset,\eta}V}{\nu_\chgset\sqrt{r_\beta}}
    \le
    \frac{1}{11},
    \qquad
    \frac{Cx_{n,\chgset,\eta}D_\psi \ell \Lambda_{n,\ell}}{\nu_\chgset r_\beta}
    \le
    \frac{1}{11}.
\]
Thus $\mathcal H_n(r_\beta)<1/5$, as was to be shown.
\end{proof}

\begin{proof}[Proof of Lemma~\ref{lem:uniform_interval_sums_new}]
The proof is very similar to the proof of Lemma~\ref{lem:polynomial_beta_localization}.
Partition \(\{1,\ldots,n\}\) into consecutive blocks
\[
B_1,\ldots,B_m,
\qquad
m:=\left\lceil\frac{n}{\ell}\right\rceil,
\]
where every block except possibly the last has length \(\ell\); note
that the blocking structure is different than in other proofs in this
manuscript. For each block \(j\), write
$
E_j=(\varepsilon^{(i)})_{i\in B_j}.
$
By Berbee's coupling lemma applied to both
\(E_1,E_3,E_5,\ldots\) and \(E_2,E_4,E_6,\ldots\), we can construct a
sequence \((\varepsilon^{(i),\star})_{i\in[n]}\) such that each block
variable $E_j^\star=(\varepsilon^{(i),\star})_{i\in B_j}$
has the same distribution as \(E_j\), such that
\(E_1^\star,E_3^\star,E_5^\star,\ldots\) are mutually independent,
such that \(E_2^\star,E_4^\star,E_6^\star,\ldots\) are mutually
independent, and such that
\begin{align}
\label{eq:coupling-failure-interval-new_2}
\Prob\Big(
\exists i\in[n]:
\varepsilon^{(i)}\ne\varepsilon^{(i),\star}
\Big)
&=
\Prob\Big(
\exists j\in[m]:
E_j\ne E_j^\star
\Big)
\le
2m\beta_\ell
\le
2\left\lceil\frac n\ell\right\rceil\beta_\ell.
\end{align}

Let \(\mathcal I_n\) be the collection of all integer intervals
contained in \([n]\), as in the statement of the lemma, and define
the starred interval event
\[
\mathcal E_{\ell,\eta}^{\mathrm{int},\star}(C)
:=
\bigcap_{k\in[K]}
\bigcap_{I\in\mathcal I_n}
\bigg\{
\Big\|
\sum_{i\in I}\varepsilon_k^{(i),\star}
\Big\|_{\Hb_k}
\le
Cx_{n,\eta}
\Big(
V\sqrt{|I|}
+
D_\psi\ell\Lambda_{n,\ell}
\Big)
\bigg\}.
\]
We first prove that \(C\) can be chosen universally so that
\[
\Prob\left(
\mathcal E_{\ell,\eta}^{\mathrm{int},\star}(C)
\right)
\ge1-\eta.
\]

For that purpose, fix
\(k\in[K]\) and \(I\in\mathcal I_n\),
and write \(d:=|I|\). Let
\[
J_{\mathrm o}(I)
:=
\{j:\ j\text{ is odd and }B_j\subseteq I\},
\qquad
J_{\mathrm e}(I)
:=
\{j:\ j\text{ is even and }B_j\subseteq I\}.
\]
As in \eqref{eq:main-decomp-star}, we can then write
\begin{align*}
\sum_{i\in I}\varepsilon_k^{(i),\star}
=
\sum_{j\in J_{\mathrm o}(I)}Y_{j,k}^\star
+
\sum_{j\in J_{\mathrm e}(I)}Y_{j,k}^\star
+
\sum_{i\in N(I)}\varepsilon_k^{(i),\star},
\end{align*}
where \(N(I)\) contains all indices \(i\in I\) that are not part of
any set \(B_j\) with
\(j\in J_{\mathrm o}(I)\cup J_{\mathrm e}(I)\), and where
\[
Y_{j,k}^\star
:=
\sum_{i\in B_j}\varepsilon_k^{(i),\star},
\qquad j\in[m].
\]
If either \(J_{\mathrm o}(I)\) or \(J_{\mathrm e}(I)\) is empty, the
corresponding expressions are understood to be zero.
Using the same arguments that led to \eqref{eq:tail-epsstar}, we find a universal constant $C_4 >0$ such that 
\begin{align}
\label{eq:tail-epsstar-interval-new_2}
    \mathbb P\bigg(
        \Big\|
             \sum_{i \in I} \eps_k^{(i),\star}  \Big\|_{\Hb_k}
        >
        C_4x
        \left\{
            V\sqrt d
            +
            D_\psi\ell\Lambda_{n,\ell}
        \right\} \bigg)
    \le
    6e^{-x}.
\end{align}
There are at most \(n^2\) integer intervals contained in \([n]\).
Taking \(x=x_{n,\eta}\) in
\eqref{eq:tail-epsstar-interval-new_2} and applying the union bound over
all \(k\in[K]\) and \(I\in\mathcal I_n\) gives
\[
\Prob\left(
\{\mathcal E_{\ell,\eta}^{\mathrm{int},\star}(C_4)\}^c
\right)
\le
6Kn^2e^{-x_{n,\eta}}
\le
\eta.
\]

Define the event in the original variables by
\[
\mathcal E_{\ell,\eta}^{\mathrm{int}}
:=
\bigcap_{k\in[K]}
\bigcap_{I\in\mathcal I_n}
\bigg\{
\Big\|
\sum_{i\in I}\varepsilon_k^{(i)}
\Big\|_{\Hb_k}
\le
C_5x_{n,\eta}
\Big(
V\sqrt{|I|}
+
D_\psi\ell\Lambda_{n,\ell}
\Big)
\bigg\}.
\]
This event belongs to \(\mathcal F_n\). On the event that every
coupled block agrees with its original counterpart,
\(\mathcal E_{\ell,\eta}^{\mathrm{int},\star}(C_5)\) is equal to
\(\mathcal E_{\ell,\eta}^{\mathrm{int}}\). Therefore,
\eqref{eq:coupling-failure-interval-new_2} gives
\[
\Prob\left(
\{\mathcal E_{\ell,\eta}^{\mathrm{int}}\}^c
\right)
\le
\eta
+
2\left\lceil\frac n\ell\right\rceil\beta_\ell.
\]
This is the claimed result.
\end{proof}

\begin{proof}
Let \(C_{\mathrm{int}}>0\) and \(V>0\) be the constants from
Lemma~\ref{lem:uniform_interval_sums_new}. Apply that lemma with
\(\eta=1/(12n)\) and $\ell \in [n]$ fixed. Since \(n\ge4\) and \(K\ge1\),
$
    x_{n,\eta} = \log(144n^3K)
    \le
    7\log(nK).
$
It follows that, on
\(\mathcal E_{\ell,\,1/(12n)}^{\mathrm{int}}\),
\begin{equation}
\label{eq:coordinatewise-interval-bound-new}
    \max_{k\in[K]}\max_{I\in\mathcal I_n}
    \Big\|
        \sum_{i\in I}\varepsilon_k^{(i)}
    \Big\|_{\Hb_k}
    \le
    C_{\mathrm I}\log(nK)
    \left\{
        \sqrt{|I|}
        +
        \ell\log\left(\frac{en}{\ell}\right)
    \right\},
\end{equation}
where \(\mathcal I_n\) is the collection of nonempty integer intervals
in \([n]\) and one may take
$
    C_{\mathrm I}
    :=
    7C_{\mathrm{int}}(V\vee D_\psi),
$
which only depends on $D_\psi,C_\beta,c_\beta$.

Fix
\(k\in\mathcal S\), and write 
$
    r_k:=r_{\beta,k}^{(n)}(\ell)
$
for the expression defined in \eqref{eq:coordinatewise-interval-bound-new}.
If \(r_k\ge n\), then
\(|\widehat\omega_k-\omega_k|<r_k\) follows from the range of the
change-point estimator. It therefore remains to consider \(r_k<n\).

For \(s\in\{0,\ldots,n\}\) satisfying
\(d:=|s-\omega_k|\ge r_k\), the definition of the oracle CUSUM process
in \eqref{eq:definition-CnkO} and the triangle inequality give
\begin{align*}
&\sqrt n
\left\|
    C_{n,k}^O(s)-C_{n,k}^O(\omega_k)
\right\|_{\Hb_k}
\le
\Big\|
    \sum_{i=(s\wedge\omega_k)+1}^{s\vee\omega_k}
    \varepsilon_k^{(i)}
\Big\|_{\Hb_k}
+
\frac dn
\Big\|
    \sum_{i=1}^n\varepsilon_k^{(i)}
\Big\|_{\Hb_k}.
\end{align*}
Applying~\eqref{eq:coordinatewise-interval-bound-new} to these two
integer intervals yields
\begin{align*}
\sqrt n
\left\|
    C_{n,k}^O(s)-C_{n,k}^O(\omega_k)
\right\|_{\Hb_k}
&\le
    C_{\mathrm I}\log(nK)
    \left\{
        \sqrt d+\ell\log\left(\frac{en}{\ell}\right)
    \right\}
\\
&\qquad\quad+
    C_{\mathrm I}\frac dn\log(nK)
    \left\{
        \sqrt n+\ell\log\left(\frac{en}{\ell}\right)
    \right\}.
\end{align*}
Dividing by \(d\nu_k\), taking the supremum over
\(d\ge r_k\), and using \(r_k\le d\le n\), we obtain
\begin{equation}
\label{eq:coordinatewise-envelope-bound-new}
    \mathcal H_n(r_k,\{k\})
    \le
    \frac{2C_{\mathrm I}\log(nK)}{\nu_k}
    \left\{
        \frac1{\sqrt{r_k}}
        +
        \frac{
            \ell\log\!\left(\frac{en}{\ell}\right)
        }{r_k}
    \right\},
\end{equation}
where \(\mathcal H_n\) is defined in
\eqref{eq:definition-Hnr}.
Choose the constant in~\eqref{eq:def-coordinatewise-rbeta-new} so that
$
    C_{\mathrm{loc}}
    \ge
    (22C_{\mathrm I})^2
    \vee
    22C_{\mathrm I}.
$
Hence, by the definition of \(r_k\)
\begin{align*}
    r_k
    &\ge
    C_{\mathrm{loc}}
    \frac{\log^2(nK)}{\nu_k^2}
    \ge (22C_{\mathrm{I}})^2\frac{\log^2(nK)}{\nu_k^2},
\end{align*}
and 
\begin{align*}
    r_k
    &\ge
    C_{\mathrm{loc}}
    \frac{
        \ell\log\!\left(\frac{en}{\ell}\right)\log(nK)
    }{\nu_k}
    \ge 
    22C_{\mathrm{I}}
    \frac{
        \ell\log\!\left(\frac{en}{\ell}\right)\log(nK)
    }{\nu_k}.
\end{align*}
Consequently,
\[
    \frac{2C_{\mathrm I}\log(nK)}
    {\nu_k\sqrt{r_k}}
    \le
    \frac1{11},
\qquad
    \frac{
        2C_{\mathrm I}\ell\log\!\left(\frac{en}{\ell}\right)
        \log(nK)
    }{
        \nu_k r_k
    }
    \le
    \frac1{11}.
\]
Together with~\eqref{eq:coordinatewise-envelope-bound-new}, this gives
\[
    \mathcal H_n(r_k,\{k\})
    \le
    \frac2{11}
    <
    \frac15.
\]
Lemma~\ref{lem:general_localization_lemma}, applied with
\(\chgset=\{k\}\) and \(r=r_k\), therefore gives
$
    |\widehat\omega_k-\omega_k|<r_k.
$
Because \(k\in\mathcal S\) was arbitrary and
\(\mathcal E_{\ell,\,1/(12n)}^{\mathrm{int}}\) controls all coordinates
and integer intervals simultaneously, the conclusion holds for every
\(k\in\mathcal S\) without an additional union bound.
\end{proof}

\section{Miscellaneous Definitions and Results}
\label{sec:miscellaneous}

\begin{definition}[Polynomial Eigenvalue Decay]
\label{def:covariance-decay}
Let \(\Hb_1,\dots,\Hb_K\) be separable Hilbert spaces, let
$
\mathcal H=\bigoplus_{k \in [K]} \Hb_k,
$
and let \(\mathcal C:\mathcal H\to\mathcal H\) be a positive
semidefinite, self-adjoint, trace-class covariance operator. For each
\(k\in[K]\), let \(\pi_k:\mathcal H\to\Hb_k\) denote the canonical
coordinate projection and define
\[
\mathcal C_k:=\pi_k\mathcal C\pi_k^* : \Hb_k \to \Hb_k
\]
Let
\[
r_k:=\operatorname{rank}(\mathcal C_k) := \dim(\mathrm{ran} \mathcal C_k) \in\mathbb N_{\ge 0}\cup\{\infty\},
\]
and suppose that \(r_k\geq 1\). Denote the positive eigenvalues of
\(\mathcal C_k\), counted with multiplicity and arranged in
nonincreasing order, by
\[
\lambda_1^{(k)}\geq\lambda_2^{(k)}
\geq\cdots\geq\lambda_{r_k}^{(k)}>0,
\]
where, when \(r_k=\infty\), the sequence is indexed by all
\(j\in\mathbb N\). For constants satisfying
\[
0<B\leq 1\leq A<\infty,
\qquad
\frac12<a\leq b<\infty,
\]
we write
\[
\mathcal C\in\operatorname{Decay}(B^2,A^2,2b,2a)
\]
if, for every \(k\in[K]\) and every \(j \in [r_k] \cap \N\),
\[
B^2j^{-2b}
\leq
\lambda_j^{(k)}
\leq
A^2j^{-2a}.
\]
\end{definition}

Note that Assumption~\ref{cond:scores} can be rewritten as $\mathcal K \in \operatorname{Decay}(C_L^2,C_U^2,2\gamma_L,2\gamma_U)$, provided we additionally require $C_L \le 1$.

\begin{lemma}\label{lem:gaussbootsphere-new}
Fix \(\rho>0\), and let \(M\geq 1\) and \(K\geq 1\). Let \(X=(X_1, \dots, X_K)\) and \(Y=(Y_1, \dots, Y_K)\) be centered Gaussian random vectors taking values in \((\mathbb R^M)^K\).
Assume that, in the finite-dimensional sense of Definition~\ref{def:covariance-decay},
\[
\Cov(X)
\in
\operatorname{Decay}
\left(
C_L^2,
C_U^2,
2\gamma_L,
2\gamma_U
\right),
\]
where $0<C_L\leq 1\leq C_U<\infty$ and $1/2<\gamma_U\leq\gamma_L<\infty$
Then there exists a constant
$
D=D(\gamma_U,\gamma_L)>0
$
depending only on \(\gamma_U\) and \(\gamma_L\) such that, for every auxiliary integer
\(m\in\mathbb N_{\ge4}\),
\begin{align*}
&\sup_{t\geq \rho}
\bigg|
\mathbb P\Big(
\max_{k\in[K]}\|X_k\|_2\leq t
\Big)
-
\mathbb P\Big(
\max_{k\in[K]}\|Y_k\|_2\leq t
\Big)
\bigg|
\\
&\quad\leq
D
\left(
\frac{C_U^{2\kexp+2}}
{C_L^3\rho^{\kexp+1/2}}
\right)^{2/3}
\left[
M^{2/3}
\|\Sigma_1-\Sigma_2\|_{\infty}^{1/3}
+
m^{-2/3}
\right]
\log^{7/6+2\kexp}\big(
em(K+1)(U\vee1)
\big),
\end{align*}
where $\Sigma_1:=\Cov(X), \Sigma_2:=\Cov(Y)$ and $\kexp = \gamma_L / (2 \gamma_U-1)$ and where
\[
U:=
\max_{k\in[K]}
\left\{
\sqrt{\Tr(\Cov(X_k))}
\vee
\sqrt{\Tr(\Cov(Y_k))}
\right\}.
\]
\end{lemma}

\begin{proof}
    The result is a straightforward adaptation of Proposition 2 (page 59) from \cite{many_means}.
\end{proof}

\begin{lemma}\label{new:high_dimensional_CLT_mixed_norms}
Let \(M_1,\ldots,M_K\in\mathbb N\), and let
$\mathcal H_M:=\bigoplus_{k \in [K]} \mathbb R^{M_k}$
be equipped with the direct-sum inner product. Suppose
\(\{X_i\}_{i=1}^n\) is a collection of independent centered random vectors
taking values in \(\mathcal H_M\). Let \(\{Y_i\}_{i=1}^n\) be a
collection of independent centered Gaussian random vectors in \(\mathcal H_M\), independent
of \(\{X_i\}_{i=1}^n\), such that
$\Cov(Y_i)=\Cov(X_i)$ for each $i \in [n]$.
Define
\[
S_n^X:=\frac1{\sqrt n}\sum_{i=1}^n X_i,
\qquad
S_n^Y:=\frac1{\sqrt n}\sum_{i=1}^n Y_i,
\qquad
M:=\max_{k\in[K]}M_k.
\]
Let \(J_n\) and \(H_n\) satisfy \(1\le J_n\le H_n<\infty\), and assume that
\begin{equation}\label{eq:Jn}
\max_{k\in[K]}
\sup_{v\in\mathbb R^{M_k}:\|v\|_2=1}
\frac1n\sum_{i=1}^n
\mathbb E\left[
\left|
\left\langle \pi_kX_i,v\right\rangle
\right|^3
\right]
\le J_n,
\end{equation}
and
\begin{equation}\label{eq:Hn3}
\max_{k\in[K]}\max_{i\in[n]}
\big\|
\left\|\pi_kX_i\right\|_2
\big\|_{\psi_1}
\le H_n.
\end{equation}
Let \(\rho>0\). Suppose that \(\Cov(S_n^X)\) satisfies
\[
\Cov(S_n^X)
\in
\operatorname{Decay}(C_L^2,C_U^2,2\gamma_L,2\gamma_U)
\]
in the finite-dimensional sense of Definition~\ref{def:covariance-decay}, where $0<C_L \le 1 \le C_U < \infty$ and $1/2 < \gamma_U \le \gamma_L < \infty.$ 
Then there exists a constant
$D=D(\rho,C_U,\gamma_U,\gamma_L)>0$
such that
\begin{align*}
d_{\mathrm{K}}^{(\rho)}
\Big(\max_{k\in[K]}\left\|\pi_kS_n^X\right\|_2^2, \max_{k\in[K]} \left\|\pi_kS_n^Y\right\|_2^2 \Big)
&\le
\frac{D}{C_L^3}
J_n
\left(
\frac{M^4H_n^2}{n}
\right)^{1/6}
\log^{3\kexp+2}(K+1)
\log^{3\kexp+3/2}(n),
\end{align*}
where $d_{\mathrm K}^{(\rho)}$ is the restricted K-distance from \eqref{eq:restricted-k-distance}.
\end{lemma}

\begin{proof}
    The result is a straightforward adaptation of Theorem 4 (page 52) from \cite{many_means}.
\end{proof}

\begin{lemma}[Improved anti-concentration for mixed norms of Gaussian variables]
\label{lem:acnew}
For fixed \(M_1,\ldots,M_K\in\mathbb N\) with $K \in \N$, let
\[
g\in \mathcal H_M=\bigoplus_{k \in [K]} \mathbb R^{M_k}
\]
be a centered Gaussian random vector with covariance operator \(\mathcal K\) satisfying
\[
\mathcal K
\in
\operatorname{Decay}
\left(
C_L^2,
C_U^2,
2\gamma_L,
2\gamma_U
\right)
\]
in the sense of Definition~\ref{def:covariance-decay}, where $0<C_L \le 1 \le C_U < \infty$ and $1/2 < \gamma_U \le \gamma_L < \infty.$
Define
\[
S:=\max_{k\in[K]}\|\pi_k g\|_{2},
\]
with $\| \cdot \|_2$ the Euclidean norm on $\R^{M_k}$.
Then, for any $\rho>0$, there exist constants
\[
D=D(\rho,C_U,\gamma_U,\gamma_L)>0,
\qquad
D'=D'(\rho,C_U,\gamma_U,\gamma_L)>0,
\]
such that, for any $\delta>0$ and $m\in\mathbb N_{\ge3}$,
\begin{equation*}
\sup_{t\ge\rho}
\Prob\left(S\in[t-\delta,t+\delta]\right)
\le
\frac{D}{C_L^3}
\left(\delta+\frac1m\right)
\log^{1/2+3\kexp}(K+1)
\log^{3\kexp}(m)
\end{equation*}
and
\begin{equation*}
\sup_{t\ge\rho}
\Prob\left(S^2\in[t-\delta,t+\delta]\right)
\le
\frac{D'}{C_L^3}
\left(\delta+\frac1m\right)
\log^{1/2+3\kexp}(K+1)
\log^{3\kexp}(m),
\end{equation*}
where $\kexp = \gamma_L / (2 \gamma_U-1)$.
\end{lemma}

\begin{proof}
    The result is a straightforward adaptation of Lemma 10 and Corollary 2 (page 37) from \cite{many_means}.
\end{proof}

\begin{lemma}[Simultaneous power under Gaussian comparison]
\label{lem:simultaneous_power_generic}
Fix \(\gamma\in(0,1)\), \(\rho>0\), and \(P\in\mathcal P\), and let
\(\varnothing\ne\chgset\subseteq\mathcal S_P\) be deterministic. Recall \(B_{\mathcal J,P}\) and
\(c_{\mathcal J,P}^{\mathrm{BB}}\) from
\eqref{eq:def-B-J-P} and~\eqref{eq:def-BB-quantile-J-P},
respectively. Define 
\begin{equation}\label{noise:kappa}
T_{n,\chgset}^{O}
:=
\max_{k\in\chgset}\max_{s\in[n]}
\|C_{n,k}^{O}(s)\|_{\Hb_k}.
\end{equation}
Suppose that there exist \(\varpi\in[0,1]\), \(0<\delta<\gamma \wedge (1-\gamma)\), and
an event \(\Omega\in\mathcal F_n\) satisfying
\begin{align}
\mathbb P_P\big(\Omega^c\big) &\le \varpi,
\label{eq:generic-power-calibration-event} 
\\
\mathbb P\big( B_{[K],P}^2\le\rho \big)
&<
1-\gamma-\delta,
\label{eq:generic-power-bootstrap-threshold}
\\
\sup_{t\ge\rho}
\left|
\mathbb P_P\left(
    (T_{n,\chgset}^{O})^2\le t
\right)
-
\mathbb P\left(
    B_{\chgset,P}^2\le t
\right)
\right|
&\le\delta
\label{eq:generic-power-oracle-comparison},
\\
\text{On } \Omega: \qquad \sup_{t\ge\rho}
\left|
\mathbb P_P\left(
    (T_n^*)^2\le t
    \,\middle|\,
    \mathcal F_n
\right)
-
\mathbb P\left(
    B_{[K],P}^2\le t
\right)
\right|
&\le\delta.
\label{eq:generic-power-bootstrap-comparison}
\end{align}
Then, for every \(z\ge\sqrt\rho\), if
\begin{equation}
\sqrt n\,\theta_k(1-\theta_k)\Delta_k
>
c_{[K],P}^{\mathrm{BB}}(1-\gamma+\delta)+z,
\qquad k\in\chgset,
\label{eq:generic-power-signal-at-z}
\end{equation}
then
\begin{equation}
\mathbb P_P\big(
    \chgset\subseteq\widehat{\mathcal S}_\gamma
\big)
\ge
\mathbb P\big(
    B_{\chgset,P}\le z
\big)
-\varpi-\delta.
\label{eq:generic-power-brownian-lower-bound}
\end{equation}
In particular, for every \(\eta\in(0,1)\) such that
\begin{equation}
c_{\chgset,P}^{\mathrm{BB}}(1-\eta)\ge\sqrt\rho,
\label{eq:generic-power-oracle-threshold}
\end{equation}
if
\begin{equation}
\sqrt n\,\theta_k(1-\theta_k)\Delta_k
>
c_{[K],P}^{\mathrm{BB}}(1-\gamma+\delta)
+
c_{\chgset,P}^{\mathrm{BB}}(1-\eta),
\qquad k\in\chgset,
\label{eq:generic-power-two-quantile-condition}
\end{equation}
then
\begin{equation}
\mathbb P_P\big(
    \chgset\subseteq\widehat{\mathcal S}_\gamma
\big)
\ge
1-\eta-\varpi-\delta.
\label{eq:generic-simultaneous-power-bound}
\end{equation}
\end{lemma}

\begin{proof}[Proof of Lemma~\ref{lem:simultaneous_power_generic}]
Recall from~\eqref{def:marginal_test} that, for every \(k\in[K]\), and
in particular for every \(k\in\chgset\),
\[
T_{n,k}
=
\max_{s\in[n]}
\|C_{n,k}(s)\|_{\Hb_k}.
\]
For brevity, write
\[
\mathfrak m_{k}
:=
\sqrt n\,\theta_k(1-\theta_k)\Delta_k,
\qquad k\in\chgset.
\]
By~\eqref{eq:cnk-dnk-cnko},
\[
C_{n,k}(s)
=
D_{n,k}(s)+C_{n,k}^{O}(s),
\qquad s\in[n].
\]
Evaluating the preceding decomposition at the true change point
\(s=\omega_k\) and applying the reverse triangle inequality gives
\begin{align}
T_{n,k}
\ge
\|C_{n,k}(\omega_k)\|_{\Hb_k}
\ge
\|D_{n,k}(\omega_k)\|_{\Hb_k}
-
\|C_{n,k}^{O}(\omega_k)\|_{\Hb_k}
\ge
\mathfrak m_{k}
-
T_{n,\chgset}^{O},
\qquad k\in\chgset,
\label{eq:power-proof-reverse-triangle}
\end{align}
where we have used that
\[
\|D_{n,k}(\omega_k)\|_{\Hb_k}
=
\max_{s\in[n]}\|D_{n,k}(s)\|_{\Hb_k}
=
\mathfrak m_{k}
\]
as a consequence of~\eqref{eq:sup-dnk}.
Recall the conditional bootstrap \((1-\gamma)\)-quantile
\(\widehat Q_{1-\gamma}\) defined immediately below
\eqref{eq:definition-Tstar}, and define its unsquared counterpart by
$
\widehat c_{1-\gamma}
:=
\widehat Q_{1-\gamma}^{1/2}.
$
Recall the definition of $\widehat S_\gamma$ from \eqref{eq:def-hat-S_gamma}.
Because both \(T_{n,k}\) and \(\widehat c_{1-\gamma}\) are
nonnegative, it follows from \eqref{eq:power-proof-reverse-triangle} that
\begin{align}
\mathbb P_P\big(
    \chgset\subseteq\widehat{\mathcal S}_\gamma
\big)
&=
\mathbb P_P\big(
\forall k\in\chgset:
T_{n,k}^2>\widehat Q_{1-\gamma}
\big)
\nonumber\\
&=
\mathbb P_P\big(
\forall k\in\chgset:
T_{n,k}>\widehat c_{1-\gamma}
\big)
\nonumber\\
&\ge
\mathbb P_P\big(
\forall k\in\chgset:
\mathfrak m_{k}>
\widehat c_{1-\gamma}+T_{n,\chgset}^{O}
\big).
\label{eq:power-proof-detection-probability}
\end{align}
For \(t\in\mathbb R\), define
\[
F_n^*(t)
:=
\mathbb P_P\big(
(T_n^*)^2\le t
\,\big|\,
\mathcal F_n
\big),
\qquad
F_{[K],P}^{\mathrm{BB}}(t)
:=
\mathbb P\big(
B_{[K],P}^2\le t
\big),
\]
such that we can rewrite \eqref{eq:generic-power-bootstrap-comparison} as
\[
\text{On } \Omega: \qquad \sup_{t\ge\rho}
\left|
F_n^*(t)-F_{[K],P}^{\mathrm{BB}}(t)
\right|
\le\delta.
\]
Further, we then have, for $a \in (0,1)$,
\begin{align}
\begin{aligned} \label{eq:squared-quantiles}
    \widehat Q_{a} &:= \inf\left\{
t\in\mathbb R:
F_n^*(t)\ge a
\right\}
=
(\widehat c_a )^2
\\
Q_{[K],P}^{\mathrm{BB}}(a)
&:=
\inf\left\{
t\in\mathbb R:
F_{[K],P}^{\mathrm{BB}}(t)\ge a
\right\}
=
\big\{ c_{[K],P}^{\mathrm{BB}}(a) \big\}^2.
\end{aligned}
\end{align}
Since $0<\delta < \gamma \wedge (1-\gamma)$, we have
$
1-\gamma-\delta\in(0,1),$ and $
1-\gamma+\delta\in(0,1)$.
On the event \(\Omega\), apply
Lemma~\ref{lem:koltoquant} with
\[
F=F_n^*,
\qquad
G=F_{[K],P}^{\mathrm{BB}},
\qquad
t_*=\rho,
\qquad
\alpha=1-\gamma.
\]
Condition~\eqref{eq:generic-power-bootstrap-threshold} first gives
$
\widehat Q_{1-\gamma}>\rho,
$
and we then obtain that, on $\Omega$,
\begin{equation}
Q_{[K],P}^{\mathrm{BB}}(1-\gamma-\delta)
\le
\widehat Q_{1-\gamma}
\le
Q_{[K],P}^{\mathrm{BB}}(1-\gamma+\delta).
\label{eq:power-proof-squared-quantile-comparison}
\end{equation}
Consequently, by \eqref{eq:squared-quantiles}, 
\begin{equation}
\widehat c_{1-\gamma}
\le
c_{[K],P}^{\mathrm{BB}}(1-\gamma+\delta).
\label{eq:power-proof-unsquared-quantile-comparison}
\end{equation}
Fix \(z\ge\sqrt\rho\) satisfying
\eqref{eq:generic-power-signal-at-z}. On \(\Omega\),
\eqref{eq:power-proof-unsquared-quantile-comparison} and
\eqref{eq:generic-power-signal-at-z} imply
\[
\Omega\cap\{T_{n,\chgset}^{O}\le z\}
\subseteq
\left\{
\forall k\in\chgset:
\mathfrak m_{k}>
\widehat c_{1-\gamma}+T_{n,\chgset}^{O}
\right\}.
\]
Together with~\eqref{eq:power-proof-detection-probability} and \eqref{eq:generic-power-calibration-event}, this yields
\begin{equation}
\mathbb P_P\big(
\chgset\subseteq\widehat{\mathcal S}_\gamma
\big)
\ge 
\mathbb P_P\big(\Omega \cap \{T_{n,\chgset}^O \le z\} 
\big)
\ge
\mathbb P_P\left(T_{n,\chgset}^{O}\le z\right)-\varpi.
\label{eq:power-proof-oracle-reduction-at-z}
\end{equation}
Because \(z^2\ge\rho\), the oracle comparison
\eqref{eq:generic-power-oracle-comparison}, applied at \(t=z^2\),
gives
\begin{align}
\mathbb P_P\left(T_{n,\chgset}^{O}\le z\right)
&=
\mathbb P_P\left((T_{n,\chgset}^{O})^2\le z^2\right)
\nonumber\\
&\ge
\mathbb P\left(B_{\chgset,P}^2\le z^2\right)-\delta
\nonumber\\
&=
\mathbb P\left(B_{\chgset,P}\le z\right)-\delta.
\label{eq:power-proof-oracle-comparison-at-z}
\end{align}
Combining~\eqref{eq:power-proof-oracle-reduction-at-z} and
\eqref{eq:power-proof-oracle-comparison-at-z} proves
\eqref{eq:generic-power-brownian-lower-bound}. Finally, take $
z=c_{\chgset,P}^{\mathrm{BB}}(1-\eta).
$
Condition~\eqref{eq:generic-power-oracle-threshold} ensures that
\(z\ge\sqrt\rho\), while
\eqref{eq:generic-power-two-quantile-condition} is precisely
\eqref{eq:generic-power-signal-at-z}. By the definition of the
generalized quantile and the right-continuity of the distribution
function of \(B_{\chgset,P}\),
\[
\mathbb P\left(
B_{\chgset,P}
\le
c_{\chgset,P}^{\mathrm{BB}}(1-\eta)
\right)
\ge
1-\eta.
\]
Substitution into~\eqref{eq:generic-power-brownian-lower-bound}
proves~\eqref{eq:generic-simultaneous-power-bound}.
\end{proof}

\begin{lemma}\label{lem:prop a_s_ell}
Let $\langle \cdot,\cdot\rangle$ denote the standard Euclidean inner product on $\R^m$. For each $s \in [n]$, recall the vector $a_s = (a_{s1},\ldots,a_{sm})^\top \in \R^m$ with  $a_{s\ell}
    :=
    \mathbf 1\{\ell\le B_s\}-s/n$ and $B_s = \lceil s/ (q+r) \rceil$; see \eqref{eq:def-a_sl}, and recall the Brownian bridge covariance kernel $K:I^2\to \R, K(u,v) = u \wedge v - uv$ defined in \eqref{def:brownian_covariance_kernel_variance_2}. Then
\begin{equation}\label{eq:brownian_bridge_covariance_bound}
\sup_{(s,t) \in [n] \times [n]}
\Big|
\langle a_s, a_t \rangle
-
mK\Big(\frac sn,\frac tn\Big)
\Big|
< 2.
\end{equation}
\end{lemma}

Note that~\eqref{eq:brownian_bridge_covariance_bound} implies that for all $s\in [n]$, if $m\geq 1$, then
\begin{equation}\label{bound_norm_a_s}
\|a_s\|_2^2 < 2+mK\Big(\frac sn, \frac sn\Big) \leq 2+\frac{m}{4}\leq \frac 94 m,
\end{equation}
and hence $\|a_s\|_2 \leq (3/2) \sqrt{m}\leq 2\sqrt{m}$.

\begin{proof}[Proof of Lemma~\ref{lem:prop a_s_ell}]
Define
\[
M_s:=B_s\wedge m=\min\{\lceil s/(q+r)\rceil,m\},
\]
such that 
\[
\sum_{\ell=1}^m \mathbf 1\{\ell\le B_s\}=M_s.
\]
Assume first that $s\le t$, and note that $B_s \le B_t$. Then
\begin{align*}
\langle a_s,a_t\rangle
&=
\sum_{\ell=1}^m
\Big(\mathbf 1\{\ell\le B_s\}-\frac sn\Big)
\Big(\mathbf 1\{\ell\le B_t\}-\frac tn\Big)=
M_s- \frac tn M_s- \frac sn M_t+ m \frac sn \frac tn.
\end{align*}
On the other hand,
\[
mK\Big(\frac sn,\frac tn\Big)=m \frac sn \Big( 1- \frac tn\Big).
\]
Therefore,
\begin{align*}
\langle a_s,a_t\rangle-mK\Big(\frac sn,\frac tn\Big)
&=
\Big( 1- \frac tn\Big) \Big( M_s - m \frac sn \Big) - \frac sn \Big (M_t - m \frac tn\Big)
= 
\Big( 1- \frac tn\Big) z_s - \frac sn z_t,
\end{align*}
where
\[
z_s:=M_s-\frac{ms}{n}.
\]
Suppose we have shown that $0\le z_s<2$ for all $s \in [n]$. Then 
\begin{align*}
\Big| \langle a_s,a_t\rangle-mK\Big(\frac sn,\frac tn\Big) \Big| 
& \le 
\Big( 1- \frac tn\Big) z_s + \frac sn z_t
< 2 \Big( 1- \frac tn + \frac sn\Big) \le 2.
\end{align*}
This is the assertion for $s \le t$, and the case $t\le s$ follows by symmetry.

It remains to show  that $0\le z_s<2$ for all $s \in [n]$. First, since
$m=\lfloor n/(q+r)\rfloor \le n/(q+r)$, we have
\[
\frac{ms}{n}\le \frac{s}{q+r}\le B_s.
\]
Thus, $M_s = \min(B_s, m) \ge \min(ms/n, m) = ms/n$, and hence $z_s \ge 0$.
Moreover, writing $n = m(q+r) + \rho$ for some $0 \le \rho < h$, and observing that $M_s \le B_s \le s/(q+r) + 1$,  we obtain
\begin{align*}
z_s
\le
\frac{s}{q+r}+1-\frac{ms}{n}
&=
1+s\left(\frac1{q+r}-\frac mn\right)
=
1+\frac{s\rho}{(q+r)n}
\le
1+\frac{\rho}{q+r}
<2.
\end{align*}
Thus $z_s \le 2$. This completes the proof.
\end{proof}

\begin{lemma}[Rate of convergence to the long-run covariance operator in trace norm]
\label{lem:convergence_to_long_term_trace}
Let \(\Hb\) be a separable Hilbert space, and let
\((X_i)_{i\in\mathbb Z}\) be a centered strictly stationary sequence of
\(\Hb\)-valued random elements with $\beta$-mixing coefficients $\beta_h$ as defined in Definition~\ref{def:beta_mixing_for_sequence}. Suppose also that
\[
U
:=
\bigl\|\|X_0\|_{\Hb}\bigr\|_{\psi_1}
<\infty,
\qquad
S_{\beta,1}
:=
\sum_{h=1}^{\infty}
h\beta_h\log^2\Big(\frac4{\beta_h}\Big)
<\infty;
\]
see also~\eqref{eq:def-S_beta}. For \(h\in \N_{\ge 0}\), define
\[
\Gamma_h
:=
\Cov(X_0, X_h) =
\E[X_0\otimes X_h].
\]
Then the series
\[
\mathcal K_{\infty}
:=
\Gamma_0+
\sum_{h=1}^{\infty}
\bigl(\Gamma_h+\Gamma_h^*\bigr)
\]
converges absolutely in trace norm and defines a positive semidefinite,
self-adjoint trace-class operator. Moreover, for every \(N\in\mathbb N\),
\[
\Big\|
\Cov\Big(
N^{-1/2}\sum_{i=1}^{N}X_i
\Big)
-
\mathcal K_{\infty}
\Big\|_{\Tr}
\le
64U^2S_{\beta,1}\frac1N
\]
and
\begin{align}
\mathbb E\Big[ 
\Big\|
N^{-1/2}\sum_{i=1}^{N}X_i
\Big\|_{\Hb}^{2}
\Big]
\le
64U^2(1+S_{\beta,0}).
\label{eq:bound_on_square_sum_beta_mixing}
\end{align}
\end{lemma}

\begin{proof}
Lemma~\ref{lem:properties-covariance}(5), shows that for all $h\geq 0$ the operator $\Gamma_h$ is trace-class, i.e, that \(\|\Gamma_h\|_{\Tr}<\infty\).
We next derive a sharper bound on\(\|\Gamma_h\|_{\Tr}\) for \(h\ge0\) than the one provided in Lemma~\ref{lem:properties-covariance}(5). Let
\(A\in\mathcal B(\Hb)\) satisfy $
\|A\|_{\op}\le1$.
Then \(AX_0\) is centered and $
\|AX_0\|_{\Hb}
\le
\|X_0\|_{\Hb}$
so that
\[
\bigl\|\|AX_0\|_{\Hb}\bigr\|_{\psi_1}
\le
\bigl\|\|X_0\|_{\Hb}\bigr\|_{\psi_1}
\le U,
\]
using monotonicity of the $\psi_1$ norm. 
Strict stationarity also gives
\[
\bigl\|\|X_h\|_{\Hb}\bigr\|_{\psi_1}
=
\bigl\|\|X_0\|_{\Hb}\bigr\|_{\psi_1}
\le U.
\]
Finally, by monotonicity of the \(\beta\)-mixing coefficient, $\beta\bigl(\sigma(AX_0),\sigma(X_h)\bigr) \le \beta_h$.
We may therefore apply Lemma~\ref{lem:expectation_beta_psi_1_control} to \(AX_0\) and \(X_h\), and obtain
\begin{equation}
\label{eq:transformed-lag-inner-product-bound}
\big|
\E\langle AX_0,X_h\rangle_{\Hb}
\big|
\le
32U^2\beta_h
\log^2\left(\frac4{\beta_h}\right),
\end{equation}
where the right-hand side is understood to equal \(0\) if \(\beta_h=0\) as indicated in the statement of Lemma~\ref{lem:expectation_beta_psi_1_control}. By Lemma~\ref{lem:trace_norm_variational_rank_one}(iii) it holds that
\[
\Tr\bigl(A(x\otimes y)\bigr)
=
\langle Ax,y\rangle_{\Hb}.
\]
Now note that since \(T\to\Tr(AT)\) is a bounded linear functional on $\mathcal{B}_{\Tr}$ it holds 
that
\[
\Tr(A\Gamma_h)
=
\Tr\bigl(A\E[X_0\otimes X_h]\bigr)
=
\E\Tr\bigl(A(X_0\otimes X_h)\bigr)
=
\E\langle AX_0,X_h\rangle_{\Hb}.
\]
We used that the expectation in the second expression is defined as a Bochner integral, and that the Bochner integral commutes with bounded linear operators \cite{DiestelUhl1977}, Theorem 6, page 47. Consequently, by Lemma~\ref{lem:trace_norm_variational_rank_one}(ii) and
\eqref{eq:transformed-lag-inner-product-bound},
\begin{align}
\|\Gamma_h\|_{\Tr}
=
\sup_{\|A\|_{\op}\le1}
\left|\Tr(A\Gamma_h)\right|
=
\sup_{\|A\|_{\op}\le1}
\left|
\E\langle AX_0,X_h\rangle_{\Hb}
\right|
\le
32U^2\beta_h
\log^2\left(\frac4{\beta_h}\right).
\label{eq:lag-covariance-trace-bound}
\end{align}
In particular the hypothesis $S_{\beta,1}<\infty$ implies that
\[
\sum_{h=1}^{\infty}\|\Gamma_h\|_{\Tr}<\infty.
\]
Since
\[
\|\Gamma_h+\Gamma_h^*\|_{\Tr}
\le
\|\Gamma_h\|_{\Tr}+\|\Gamma_h^*\|_{\Tr}
=
2\|\Gamma_h\|_{\Tr},
\]
the series
\[
\Gamma_0+
\sum_{h=1}^{\infty}
\bigl(\Gamma_h+\Gamma_h^*\bigr)
\]
converges absolutely in trace norm and this absolute convergence implies that, since $\mathcal{B}_{\Tr}$ is a Banach space,  it converges to a limit \(\mathcal K_{\infty}\) in $\mathcal{B}_{\Tr}$. Moreover, \(\Gamma_0\) is self-adjoint, and each operator
\(\Gamma_h+\Gamma_h^*\) is self-adjoint. Hence \(\mathcal K_{\infty}\) is
self-adjoint. It remains to prove the asserted approximation bound and positivity. For \(N\in\mathbb N\), we have 
\begin{align*}
\mathcal K_N
:=
\Cov\Big(
N^{-1/2}\sum_{i=1}^{N}X_i
\Big)
&=
\E\Big[
\Big(
N^{-1/2}\sum_{i=1}^{N}X_i
\Big)
\otimes
\Big(
N^{-1/2}\sum_{j=1}^{N}X_j
\Big)
\Big]\\
&=
\frac1N
\sum_{i=1}^{N}\sum_{j=1}^{N}
\E[X_i\otimes X_j],
\end{align*}
where we used that $X_j$ is centered.
Strict stationarity and the identity $(x \otimes y)^*=(y \otimes x)$ from Lemma~\ref{lem:norm_kronecker_product} implies that $\E[X_i\otimes X_j]=\Gamma_{j-i}$ for $i\le j$ and $\E[X_i\otimes X_j]=\Gamma_{i-j}^*$ for $i>j$.
For every \(h\in\{1,\ldots,N-1\}\), there are \(N-h\) pairs
\((i,j)\) satisfying \(j-i=h\). It follows that
\begin{equation}
\label{eq:finite-sample-long-run-covariance-expansion}
\mathcal K_N
=
\Gamma_0+
\sum_{h=1}^{N-1}
\left(1-\frac hN\right)
\bigl(\Gamma_h+\Gamma_h^*\bigr).
\end{equation}
Subtracting the defining series for \(\mathcal K_{\infty}\) from
\eqref{eq:finite-sample-long-run-covariance-expansion}, we obtain
\[
\mathcal K_N-\mathcal K_\infty
=
-\frac1N
\sum_{h=1}^{N-1}
h\bigl(\Gamma_h+\Gamma_h^*\bigr)
-
\sum_{h=N}^{\infty}
\bigl(\Gamma_h+\Gamma_h^*\bigr).
\]
Therefore,
\begin{align*}
\|\mathcal K_N-\mathcal K_\infty\|_{\Tr}
&\le
\frac1N
\sum_{h=1}^{N-1}
h\|\Gamma_h+\Gamma_h^*\|_{\Tr}
+
\sum_{h=N}^{\infty}
\|\Gamma_h+\Gamma_h^*\|_{\Tr}\\
&\le
\frac2N
\sum_{h=1}^{N-1}
h\|\Gamma_h\|_{\Tr}
+
2\sum_{h=N}^{\infty}
\|\Gamma_h\|_{\Tr}\\
&\le
\frac2N
\sum_{h=1}^{N-1}
h\|\Gamma_h\|_{\Tr}
+
\frac2N
\sum_{h=N}^{\infty}
h\|\Gamma_h\|_{\Tr}\\
&=
\frac2N
\sum_{h=1}^{\infty}
h\|\Gamma_h\|_{\Tr}\\
&\le
64U^2S_{\beta,1}\frac1N,
\end{align*}
where the final inequality follows from
\eqref{eq:lag-covariance-trace-bound}. Finally, every \(\mathcal K_N\) is positive semidefinite. The preceding
bound shows that
$
\lim_{N \to \infty} \|\mathcal K_N-\mathcal K_\infty\|_{\Tr}\ = 0,
$ which in turn yields 
$
\lim_{N \to \infty} \|\mathcal K_N-\mathcal K_\infty\|_{\op} =  0.
$
Hence, for every
\(x\in\Hb\),
\[
\langle\mathcal K_\infty x,x\rangle_{\Hb}
=
\lim_{N\to\infty}
\langle\mathcal K_Nx,x\rangle_{\Hb}
\ge0.
\]
Thus \(\mathcal K_\infty\) is positive semidefinite.

It remains to show \eqref{eq:bound_on_square_sum_beta_mixing}.
By centeredness and strict stationarity,
\begin{align*}
\mathbb E\Big[
\Big\|
\sum_{i=1}^{N}X_i
\Big\|_{\Hb}^{2}
\Big]
&=
N\mathbb E\|X_0\|_{\Hb}^{2}
+
2\sum_{h=1}^{N-1}(N-h)
\mathbb E\langle X_0,X_h\rangle_{\Hb}
\nonumber\\
&\le
N\mathbb E\|X_0\|_{\Hb}^{2}
+
2\sum_{h=1}^{N-1}(N-h)
\left|
\mathbb E\langle X_0,X_h\rangle_{\Hb}
\right|.
\end{align*}
By~\ref{orlicz:expectation}, with \(p=2\), and the definition of \(U\),
$
\mathbb E\|X_0\|_{\Hb}^{2}
\le
4U^2.
$
Moreover, applying
\eqref{eq:transformed-lag-inner-product-bound} with
\(A=\operatorname{id}_{\Hb}\) gives, for every \(h\in \N\),
\[
\left|
\mathbb E\langle X_0,X_h\rangle_{\Hb}
\right|
\le
32U^2\beta_h
\log^2\left(\frac4{\beta_h}\right).
\]
Consequently, using \(N-h\le N\) and the definition of
\(S_{\beta,0}\),
\begin{align*}
\mathbb E\Big[
\Big\|
\sum_{i=1}^{N}X_i
\Big\|_{\Hb}^{2}
\Big]
&\le
4NU^2
+
64NU^2
\sum_{h=1}^{N-1}
\beta_h\log^2\left(\frac4{\beta_h}\right)
\le
64U^2(1+S_{\beta,0})N.
\end{align*}
which proves \eqref{eq:bound_on_square_sum_beta_mixing}.
\end{proof}

\begin{lemma}[Control of a Hilbert-space inner product]
\label{lem:expectation_beta_psi_1_control}
Let \(\Hb\) be a separable Hilbert space, and let \(X\) and \(Y\) be centered
\(\Hb\)-valued random elements defined on a probability space
\((\Omega,\mathcal F,\Prob)\). Suppose that, for some \(U<\infty\) and $b \in [0,1]$
\[
\bigl\|\|X\|_{\Hb}\bigr\|_{\psi_1}\le U,
\qquad
\bigl\|\|Y\|_{\Hb}\bigr\|_{\psi_1}\le U,
\qquad 
\beta\bigl(\sigma(X),\sigma(Y)\bigr)\le b,
\]
where \(\beta(\cdot,\cdot)\) is the coefficient from
Definition~\ref{beta_mixing_individ}. Then
\[
\big|
\E\langle X,Y\rangle_{\Hb}
\big|
\le
32U^2b
\log^2\left(\frac4b\right),
\]
where the right-hand side is understood to equal \(0\) when \(b=0\).
\end{lemma}

\begin{proof}
If \(U=0\), then
\[
\E\|X\|_{\Hb}
\le
A_{4,1,1}\bigl\|\|X\|_{\Hb}\bigr\|_{\psi_1}
\le
A_{4,1,1}U
=
0
\]
by~\ref{orlicz:expectation}. Hence \(X=0\)
almost surely, and the assertion is immediate.

Henceforth, suppose that \(U>0\). For a nonnegative real-valued random
variable \(Z\), define its upper-tail quantile function by
\[
Q_Z(u)
:=
\inf\bigl\{
t\ge0:\Prob(Z>t)\le u
\bigr\},
\qquad u\in(0,1).
\]
Also define
\[
\overline{\alpha}
:=
\alpha\bigl(\sigma(X),\sigma(Y)\bigr)
=
\sup_{\substack{A\in\sigma(X)\\B\in\sigma(Y)}}
\left|
\Prob(A\cap B)-\Prob(A)\Prob(B)
\right|.
\]
By equation~(1.11) of \cite{Bra05}, for any two \(\sigma\)-fields
\(\mathcal A\) and \(\mathcal B\),
$
2\alpha(\mathcal A,\mathcal B)
\le
\beta(\mathcal A,\mathcal B).
$
Since
$\beta\bigl(\sigma(X),\sigma(Y)\bigr)\le b$ by assumption, 
we obtain
\begin{equation}
\label{eq:alpha-beta-upper-bound-expectation}
\overline{\alpha}
\le
\frac12\beta\bigl(\sigma(X),\sigma(Y)\bigr)
\le
\frac b2.
\end{equation}

Observing that $X$ and $Y$ are centered, Lemma~2.1 of \cite{lu2022almost}, which is attributed there to
Lemma~2 of \cite{merlevede1997sharp}, yields
\begin{equation}
\label{eq:inner-product-quantile-bound}
\big|
\E\langle X,Y\rangle_{\Hb}
\big|
\le
18
\int_0^{\overline{\alpha}}
Q_{\|X\|_{\Hb}}(u)
Q_{\|Y\|_{\Hb}}(u)
\,\diff u.
\end{equation}
By Lemma~\ref{orlicz:tail_bound}, for every \(t\ge0\),
\[
\Prob\bigl(\|X\|_{\Hb}>t\bigr)
\le
2e^{-t/U},
\qquad
\Prob\bigl(\|Y\|_{\Hb}>t\bigr)
\le
2e^{-t/U}.
\]
Therefore, for every \(u\in(0,1)\),
\begin{equation}
\label{eq:psi1-quantile-bound-expectation}
Q_{\|X\|_{\Hb}}(u)
\le
U\log\left(\frac2u\right),
\qquad
Q_{\|Y\|_{\Hb}}(u)
\le
U\log\left(\frac2u\right).
\end{equation}

If \(b=0\), then
\eqref{eq:alpha-beta-upper-bound-expectation} implies that
\(\overline{\alpha}=0\). Hence
\eqref{eq:inner-product-quantile-bound} gives
$
\E\langle X,Y\rangle_{\Hb}=0.
$
Thus the desired conclusion holds when \(b=0\). We may therefore assume
for the remainder of the proof that \(b>0\).

Combining
\eqref{eq:alpha-beta-upper-bound-expectation},
\eqref{eq:inner-product-quantile-bound}, and
\eqref{eq:psi1-quantile-bound-expectation}, we obtain
\[
\left|
\E\langle X,Y\rangle_{\Hb}
\right|
\le
18U^2
\int_0^{b/2}
\log^2\left(\frac2u\right)\,du.
\]
It is straightforward to check that \(u\mapsto\log^2(2/u)\) has antiderivative $u\big[ \log^2(2/u)+2\log(2/u)+2 \big]$, and so by the fundamental theorem of calculus,
\begin{equation}
\label{eq:precise-beta-upper-bound-inner-product}
\left|
\E\langle X,Y\rangle_{\Hb}
\right|
\le 18 U^2 
\int_0^{b/2}
\log^2\left(\frac2u\right)\,\diff u
=
9U^2b
\left[
\log^2\left(\frac4b\right)
+
2\log\left(\frac4b\right)
+
2
\right].
\end{equation}
Finally, let $L=\log(4/b)$.
Since \(0<b\le1\), we have $L \ge \log 4 > 11/8$.
Consequently,
\[
1+\frac2L+\frac2{L^2}
\le
1+\frac{16}{11}+\frac{128}{121}
=
\frac{425}{121},
\]
and therefore
\[
9\bigl(L^2+2L+2\bigr)
\le
\frac{3825}{121}L^2
<
32L^2.
\]
Applying this inequality to
\eqref{eq:precise-beta-upper-bound-inner-product} gives
\[
\left|
\E\langle X,Y\rangle_{\Hb}
\right|
\le
32U^2b
\log^2\left(\frac4b\right)
\]
as asserted.
\end{proof}

\begin{lemma}[Conditional Gaussian multiplier tail bound]
\label{lem:gaussian_conditional_tail_bound}
Let \(Y_1,\ldots,Y_m\) be random elements in a separable Hilbert space \(\Hb\), and define
$
\sigma_Y:=\sigma(Y_1,\ldots,Y_m).
$
Let \(e_1,\ldots,e_m\) be independent standard normal random variables independent of \(\sigma_Y\). There exists an absolute constant \(C>0\) such
that, almost surely,
\[
\bigg\|
\Big\|
\sum_{\ell=1}^m e_\ell Y_\ell
\Big\|_{\Hb}
\bigg\|_{\psi_2\mid \sigma_Y}
\le
C
\bigg(
\sum_{\ell=1}^m
\|Y_\ell\|_{
\Hb}^2
\bigg)^{1/2}.
\]
Consequently, by \ref{orlicz:tail_bound}, after increasing \(C\) if necessary, for every \(t>0\),
almost surely,
\[
\mathbb P\bigg(
\Big\|
\sum_{\ell=1}^m e_\ell Y_\ell
\Big\|_{\Hb}
>
C\sqrt t
\bigg(
\sum_{\ell=1}^m
\|Y_\ell\|_{
\Hb}^2
\bigg)^{1/2}
\,\bigg|\,\sigma_Y
\bigg)
\le2e^{-t}.
\]
\end{lemma}

\begin{proof}[Proof of Lemma~\ref{lem:gaussian_conditional_tail_bound}]
Conditionally on \(\sigma_Y\), the random elements \(e_1Y_1,\ldots,e_mY_m\) are independent and centered. Therefore, by~\ref{orlicz:banach}, applied conditionally with \(p=2\),
\begin{align*}
\bigg\|
\Big\|
\sum_{\ell=1}^m e_\ell Y_\ell
\Big\|_{\Hb}
\bigg\|_{\psi_2\mid \sigma_Y}
\le
C\Bigg\{
\mathbb E\bigg[
\Big\|
\sum_{\ell=1}^m e_\ell Y_\ell
\Big\|_{\Hb}
\,\bigg|\,\sigma_Y
\bigg]
+
\bigg(
\sum_{\ell=1}^m
\big\|
\|e_\ell Y_\ell\|_{\Hb}
\big\|_{\psi_2\mid\sigma_Y}^2
\bigg)^{1/2}
\Bigg\}.
\end{align*}
By the conditional Cauchy--Schwarz inequality,
\begin{align*}
\mathbb E\bigg[
\Big\|
\sum_{\ell=1}^m e_\ell Y_\ell
\Big\|_{\Hb}
\,\bigg|\,\sigma_Y
\bigg]
\le
\bigg\{
\mathbb E\bigg[
\Big\|
\sum_{\ell=1}^m e_\ell Y_\ell
\Big\|_{\Hb}^2
\,\bigg|\,\sigma_Y
\bigg] \bigg\}^{1/2}
=
\bigg(
\sum_{\ell=1}^m\|Y_\ell\|_{\Hb}^2
\bigg)^{1/2}.
\end{align*}
where we have used that 
\(\mathbb E[e_\ell e_{\ell'}]=\mathbf1\{\ell=\ell'\}\).
Furthermore,
\begin{align*}
\sum_{\ell=1}^m
\big\|
\|e_\ell Y_\ell\|_{\Hb}
\big\|_{\psi_2\mid\sigma_Y}^2
=
\sum_{\ell=1}^m
\|e_\ell\|_{\psi_2}^2\|Y_\ell\|_{\Hb}^2
=
\frac83
\sum_{\ell=1}^m\|Y_\ell\|_{\Hb}^2,
\end{align*}
where the first equality follows from
\ref{orlicz:taking_out_what_is_known} and the independence of
\(e_\ell\) and \(\sigma_Y\) and the second one from $\|e_\ell\|_{\psi_2}=\sqrt{8/3}$.
Combining the preceding bounds proves the conditional Orlicz-norm
inequality. 
\end{proof}

\begin{lemma}\label{lem:psi_2_controlled_by_L_2}
Let $\mathcal H$ be a separable Hilbert space, and let $g$ be a centered Gaussian random element in $\mathcal H$ with covariance operator $\Cov(g)$. Then
\[
    \big\|\|g\|_{\mathcal H}\big\|_{\psi_2}^2
    \le
    \frac{8}{3}
    \|\Cov(g)\|_{\operatorname{Tr}}
    =
    \frac{8}{3}
    \mathbb E\|g\|_{\mathcal H}^2.
\]
\end{lemma}

\begin{proof}[Proof of Lemma~\ref{lem:psi_2_controlled_by_L_2}]
Let $(b_j)_{j\in J}$ be an orthonormal basis of $\mathcal H$, where $J=\{1,\dots,d\}$ if $\mathcal H$ is finite-dimensional and $J=\mathbb N$ otherwise. Put $g_j := \langle g,b_j\rangle_{\mathcal H}$ for $j\in J$, and note that each $g_j$ is a centered Gaussian random variable. By Parseval's identity, we have $\|g\|_{\mathcal H}^2 = \sum_{j\in J} g_j^2$ almost surely.
Using the identity $\|X\|_{\psi_2}^2=\|X^2\|_{\psi_1}$ and the triangle
inequality for the $\psi_1$-norm, we obtain
\[
    \big\|\|g\|_{\mathcal H}\big\|_{\psi_2}^2
    =
    \big\|\|g\|_{\mathcal H}^2\big\|_{\psi_1}                         
    =
    \Big\|\sum_{j\in J} g_j^2 \Big\|_{\psi_1}                        
    \le
    \sum_{j\in J} \|g_j^2\|_{\psi_1}                                   
    =
    \sum_{j\in J} \|g_j\|_{\psi_2}^2
    \le 
    \frac{8}{3}\sum_{j\in J}\operatorname{Var}(g_j),
\]
where the last inequality follows from the fact $\|Z\|_{\psi_2}^2 = 8\sigma^2/3$ for a centered Gaussian random variable $Z\sim N(0,\sigma^2)$.  Finally, by Lemma~\ref{lem:properties-covariance},
\[
    \sum_{j\in J}\operatorname{Var}(g_j)
    =
    \sum_{j\in J}
    \mathbb E\langle g,b_j\rangle_{\mathcal H}^2
    =
    \mathbb E\|g\|_{\mathcal H}^2
    =
    \|\Cov(g)\|_{\operatorname{Tr}},
\]
which yields the asserted bound.
\end{proof}

The following is Lemma 36 in \cite{many_means}.

\begin{lemma}[A Kolmogorov bound implies a bound on quantiles]
\label{lem:koltoquant}
Let \(F\) and \(G\) be cumulative distribution functions, and let
\(q_F\) and \(q_G\) be their respective quantile functions.
For any cumulative distribution function \(H\), define its 
quantile function for $u \in \R$ by
\[
q_H(u)
:=
\inf\{t\in\mathbb R:H(t)\ge u\},
\]
where the infimum is taken in the extended real line and
\(\inf\varnothing=\infty\). Let \(\alpha\in(0,1)\), \(t_*\in\mathbb R\), and \(\delta>0\), and
suppose
\begin{equation*}
\sup_{t\ge t_*}|F(t)-G(t)|\le\delta.
\end{equation*}
If \(q_F(\alpha)>t_*\), then
\[
q_G(\alpha-\delta)
\le
q_F(\alpha)
\le
q_G(\alpha+\delta).
\]
Further,
\[
G(t_*)+\delta<\alpha
\]
implies \(q_F(\alpha)>t_*\).
\end{lemma}

\begin{lemma}\label{lem:uified_bounds}
Let $\Hb$ be a separable Hilbert space, and suppose that $X_1, \dots, X_n$ is a centered strictly stationary sequence in $\Hb$ with $\beta$-mixing coefficients $\beta_h$. Let $Q:\Hb \to \Hb$ be a bounded linear operator and
$U:=\big\|\|Q\circ X_1\|_{\Hb}\big\|_{\psi_1}<\infty$.
For $\ell \in [m]$ let $S_\ell=\sum_{i\in I_\ell}X_i$ and let $\tilde{S}_1,\dots,\tilde{S}_m$ be any independent sequence satisfying $\tilde{S}_\ell=_d S_\ell$ for all $\ell \in [m]$; here, $I_\ell$ is from \eqref{def:big_little_blocks}. Let  $C_{\beta, q}
:= \sum_{h=1}^{q-1}\beta_h\log^2(4/\beta_h)$ and let
\[
\tilde{C}_{\beta, q}=\big(A_{4,1/2,1}
+
64C_{\beta, q}\big),
\]
with $A_{4,1/2,1}$ the absolute constant from \ref{orlicz:expectation}.
Then, for any  $a=[a_1, \dots, a_m]^{\top}\in [-1,1]^m$, it holds
\begin{enumerate}[(1)]
\item $\E\Big[\|Q\circ  S_{1}\|^2_{\Hb}\Big]\leq \tilde{C}_{\beta, q}U^2q$,
\item $\E\Big[\Big\|\sum_{\ell=1}^ma_\ell\cdot Q\circ\tilde{S}_\ell\Big\|_\Hb\Big]\leq \big(\tilde{C}_{\beta, q})^{1/2}U\|a\|_2\sqrt{q}$,
\item $\Big\|\Big\|\sum_{\ell=1}^ma_\ell\cdot Q\circ\tilde{S}_\ell\Big\|_\Hb\Big\|_{\psi_1}\leq C_3U\big(C_{\beta, q}\vee 1)^{1/2}\Big[\|a\|_2\sqrt{q}+q\log(m+1)\Big]$,
\item $\Big\|\Big\|\sum_{\ell=1}^ma_\ell\cdot Q\circ\tilde{S}_\ell\Big\|_\Hb\Big\|_{\psi_1}\leq C_4\left[\E\Big[\Big\|\sum_{\ell=1}^m a_\ell\cdot Q\circ\tilde{S}_\ell\Big\|_{\Hb}\Big]+q U\log(m+1)\right]$,
\item $\Big\|\sum_{\ell=1}^m a^2_\ell\cdot\|Q\circ\tilde{S}_\ell\|^2_\Hb\Big\|_{\psi_{1/2}}^{1/2}
\leq C_5U\big(C_{\beta, q}\vee 1)^{1/2}\Big[\|a\|_2\sqrt{q}+q\log(m+1)\Big]$,
\item 
$\Big\|\sum_{\ell=1}^m a^2_\ell\cdot\|Q\circ\tilde{S}_\ell\|^2_\Hb\Big\|_{\psi_{1/2}}^{1/2}
\le C_6
\left[
\left\{
\sum_{\ell=1}^m
a_\ell^2
\mathbb E
\|Q\circ\widetilde S_\ell\|_{\Hb}^2
\right\}^{1/2}
+
qU\log(m+1)
\right]$,
\end{enumerate}
where the $C_j$ are absolute constants. 
\end{lemma}

\begin{proof}[Proof of Lemma~\ref{lem:uified_bounds}]
We prove the statements one by one. First,
\begin{align*}
\E\big[\big\|Q\circ  S_{1}\big\|^2_{\Hb}\big]
&=
\E\Big[\sum_{i,i'=1}^{q}\langle Q \circ X_i,Q \circ X_{i^{'}}\rangle_\Hb\Big]
\\&=
q\E\Big[\langle Q\circ X_1,Q \circ X_1\rangle_\Hb\Big]
+
2\sum_{h=1}^{q-1}\big(q-h)\E\Big[\langle Q \circ X_1,Q \circ X_{(1+h)}\rangle_\Hb\Big]
\\& \le 
q\E\Big[\| Q\circ X_1\|^2_\Hb\Big]
+
2q\sum_{h=1}^{q-1}\Big|\E\Big[\langle Q \circ X_1,Q \circ X_{(1+h)}\rangle_\Hb\Big]\Big|
\\
&
\le
qA_{4,1/2,1}\Big\|\| Q\circ X_1\|^2_\Hb\Big\|_{\psi_{1/2}}
+
q64U^2\sum_{h=1}^{q-1}\beta_h\log^2(4/{\beta_h})
\\&
=
qA_{4,1/2,1}\Big\|\| Q\circ X_1\|_\Hb\Big\|^2_{\psi_1}
+
q64U^2\sum_{h=1}^{q-1}\beta_h\log^2(4/{\beta_h})\\&
\le
qU^2\Big(A_{4,1/2,1}
+
64C_{\beta, q}\Big).
\end{align*}
In the first line we used the definition of an inner-product norm and bi-linearity of the inner-product. In the second line we used linearity of expectation, stationarity of the sequence $X_1, \dots, X_q$ and symmetry of the inner-product. In the fourth line we used~\ref{orlicz:expectation} and Lemma~\ref{lem:expectation_beta_psi_1_control}. In the fifth line we used~\ref{orlicz:power}. This completes the proof of (1).

Second,
\begin{align*}
&\phantom{{}={}}
\E\Big[\Big\|\sum_{\ell=1}^m a_\ell\cdot Q\circ\tilde{S}_\ell\Big\|^2_\Hb\Big] 
\\&=
\E\Big[\Big\langle \sum_{\ell=1}^m a_\ell\cdot Q\circ\tilde{S}_\ell, \sum_{\ell=1}^m a_\ell\cdot Q\circ\tilde{S}_\ell\Big\rangle_\Hb\Big] 
\\&=
\sum_{\ell=1}^{m}\E\Big[\langle  a_\ell\cdot Q\circ\tilde{S}_\ell, a_{\ell}\cdot Q\circ\tilde{S}_{\ell}\rangle_\Hb\Big]+2\sum_{\ell\neq  \ell^{'}}\E\Big[\langle  a_\ell\cdot Q\circ\tilde{S}_\ell, a_{\ell^{'}}\cdot Q\circ\tilde{S}_{\ell^{'}}\rangle_\Hb\Big] 
\\&=
\sum_{\ell=1}^{m}a^2_\ell\E\Big[\| Q\circ\tilde{S}_\ell\|_\Hb\Big]
= 
\| a \|_2^2 \E\Big[\| Q\circ\tilde{S}_1\|_\Hb\Big],
\end{align*}
where we used independence of $\tilde{S}_\ell$ and $\tilde{S}_\ell^{'}$ for $\ell \neq \ell^{'}$ and stationarity of the sequence $\tilde{S}_1,\dots,\tilde{S}_m$. Plugging in the bound proved in the first part completes the proof of (2) after an application of the Cauchy-Schwarz inequality. 

Third, by~\ref{orlicz:banach} applied to the norm $\| \cdot \|_\Hb$ and \ref{orlicz:max} it holds that
\begin{align*}
&\phantom{{}={}}
\Big\|\Big\|\sum_{\ell=1}^ma_\ell\cdot Q\circ\tilde{S}_\ell\Big\|_\Hb\Big\|_{\psi_1}
\\&\leq
A_{11,1}\Big(\E\Big[\Big\|\sum_{\ell=1}^m a_\ell\cdot Q\circ\tilde{S}_\ell\Big\|_{\Hb}\Big]+\Big\|\max_{\ell=1}^{m}\big\| a_\ell\cdot Q\circ\tilde{S}_\ell\big\|_\Hb\Big\|_{\psi_1}\Big)\\
&\leq 
A_{11,1}\Big(\E\Big[\Big\|\sum_{\ell=1}^m a_\ell\cdot Q\circ\tilde{S}_\ell\Big\|_{\Hb}\Big] +A_{8,1}\|a\|_\infty\log(m+1)\Big\|\Big\|\sum_{i\in I_1}Q\circ  X_i\Big\|_\Hb\Big\|_{\psi_1}\Big)\\
&\leq 
A_{11,1}\Big(\E\Big[\Big\|\sum_{\ell=1}^m a_\ell\cdot Q\circ\tilde{S}_\ell\Big\|_{\Hb}\Big] +A_{8,1}\log(m+1)qU\Big).
\end{align*}
To pass from the third to fourth line we used $\|a\|_\infty \leq 1$ and~\ref{orlicz:triangle} to conclude that
\begin{equation*}
\Big\|\big\|\sum_{i\in I_1}Q\circ  X_i\big\|_\Hb\Big\|_{\psi_1}
\leq q U.
\end{equation*}
The penultimate display yields (4) by setting $C_4=A_{11,1}\times (A_{8,1}\vee 1)$.  Then (3) follows by inserting the bound of (2) into (4) and using
\[
\widetilde C_{\beta,q}^{1/2}
\le
\bigl(A_{4,1/2,1}+64\bigr)^{1/2}
(C_{\beta,q}\vee1)^{1/2}.
\]
Then choose $C_3=C_4\bigl\{1\vee(A_{4,1/2,1}+64)^{1/2}\bigr\}$.

Fourth, we prove (5) and (6), starting with (5).  To lighten notation in the following derivation we omit the operator $Q$ when performing calculations. Since $Q$ is linear, all calculations go through when each $\tilde{S}_\ell$ below is replaced by $ Q \circ \tilde{S}_\ell$. By the triangle inequality from \ref{orlicz:triangle_up_to_constant} it holds that
\begin{align*}
\Big\|\sum_{\ell=1}^m a^2_\ell\|\tilde{S}_\ell\|^2_\Hb\Big\|_{\psi_{1/2}}
&\leq 
A_{2,1/2}\Big\{\Big\|\sum_{\ell=1}^ma^2_\ell\Big( \|\tilde{S}_\ell\|^2_\Hb-\E\big[\|\tilde{S}_\ell\|^2_\Hb\big]\Big)\Big\|_{\psi_{1/2}}+\sum_{\ell=1}^m\Big\|a^2_\ell\E\big[\|\tilde{S}_\ell\|^2_\Hb\big]\Big\|_{\psi_{1/2}}\Big\}.
\end{align*}
With respect to the second summand it holds that
\begin{align*}
\sum_{\ell=1}^m\Big\|a^2_\ell\E\big[\|\tilde{S}_\ell\|^2_\Hb\big]\Big\|_{\psi_{1/2}}
&=
\|1 \|_{\psi_{1/2}} \E\big[\|S_1\|_\Hb^2\big]  \sum_{\ell=1}^m a^2_\ell
\le
[\log(2)]^{-2}\|a\|^2_2qU^2\tilde{C}_{\beta, q},
\end{align*}
where we used that $\|1\|_{\psi_{1/2}}=[\log(2)]^{-2}$ and plugged in the bound from the first part.
With respect to the first summand it holds that
\begin{align*}
&\phantom{{}={}}
\Big\|\sum_{\ell=1}^m a^2_\ell \Big(\|\tilde{S}_\ell\|^2_\Hb-\E\big[\|\tilde{S}_\ell\|^2_\Hb\big]\Big)\Big\|_{\psi_{1/2}}
\\
&\leq A_{11, 1/2}\Bigg\{\E\Big[\Big|\sum_{\ell=1}^m\ a^2_\ell \Big( \|\tilde{S}_\ell\|^2_\Hb-\E\big[\|\tilde{S}_\ell\|^2_\Hb\big]\Big)\Big|\Big]
+\Big\|\max_{\ell\in[m]}\left|a^2_\ell\Big( \|\tilde{S}_\ell\|^2_\Hb-\E\big[\|\tilde{S}_\ell\|^2_\Hb\big]\Big)\right|\Big\|_{\psi_{1/2}}\Bigg\}
\\
&\leq 
A_{11, 1/2}\Big\{2\sum_{\ell=1}^m a^2_\ell \E\Big[\|\tilde{S}_\ell\|^2_\Hb\Big]
+
A_{8,1/2}[\log(m+1)]^2 \Big\| \|\tilde{S}_1\|^2_\Hb-\E\big[\|\tilde{S}_1\|^2_\Hb\big]\Big\|_{\psi_{1/2}}\Big\},
\end{align*}
where we used \ref{orlicz:banach} in the second line and \ref{orlicz:max} and $\|a\|_\infty \le 1$ in the third line. Using the first bound of this lemma, the first sum on the right-hand side is bounded by $2 \|a \|^2_2 \tilde C_{\beta, q} U^2q$. Finally, 
\begin{align*}
\Big\| \|\tilde{S}_1\|^2_\Hb- \E\big[\|\tilde{S}_1\|^2_\Hb\big]\Big\|_{\psi_{1/2}}
&\leq 
A_{2, 1/2} \Big(\big\| \|\tilde{S}_1\|^2_\Hb\big\|_{\psi_{1/2}}+\big\|1\big\|_{\psi_{1/2}}\E\big[\|\tilde{S}_1\|^2_\Hb\big]\Big)
\\&\leq 
A_{2, 1/2}\Big(q^2U^2+[\log(2)]^{-2}\tilde{C}_{\beta, q} U^2q\Big),
\end{align*}
again using the first bound of this lemma and the fact that 
\[
\big\| \|\tilde{S}_1\|^2_\Hb\big\|_{\psi_{1/2}} = \big\| \|\tilde{S}_1\|_\Hb\big\|_{\psi_{1}}^2 \le q^2U^2
\]
by \ref{orlicz:power} and the triangle inequality. Since \(q\ge1\) and
\(\widetilde C_{\beta,q}\le(A_{4,1/2,1}+64)(C_{\beta,q}\vee1)\),
assembling the preceding estimates and using
\(\sqrt{u+v}\le\sqrt u+\sqrt v\) proves (5), after increasing the
absolute constant \(C_5\).
To prove (6) we repeat much of the argument used to prove (5).
After centering and applying~\ref{orlicz:triangle_up_to_constant}, we have
\begin{align*}
\Big\|
\sum_{\ell=1}^m a_\ell^2\|Q\widetilde S_\ell\|_{\Hb}^2
\Big\|_{\psi_{1/2}}
&\le
A_{2,1/2}
\left[
S_1
+
S_2
\right],
\end{align*}
where 
\[
S_1=\left\|
\sum_{\ell=1}^m(a_\ell^2\|Q\widetilde S_\ell\|_{\Hb}^2-\mathbb E[a_\ell^2\|Q\widetilde S_\ell\|_{\Hb}^2])
\right\|_{\psi_{1/2}},\qquad S_2=\left\|
\sum_{\ell=1}^m \mathbb E[a_\ell^2\|Q\widetilde S_\ell\|_{\Hb}^2]
\right\|_{\psi_{1/2}}.
\]
Note that 
\begin{equation*}
S_2=\left\|
\sum_{\ell=1}^m\mathbb E[a_\ell^2\|Q\widetilde S_\ell\|_{\Hb}^2]
\right\|_{\psi_{1/2}}=\left\|1
\right\|_{\psi_{1/2}}\left|\sum_{\ell=1}^m\mathbb E[a_\ell^2\|Q\widetilde S_\ell\|_{\Hb}^2]\right|
=
\{\log(2)\}^{-2}
\sum_{\ell=1}^m\mathbb E[a_\ell^2\|Q\widetilde S_\ell\|_{\Hb}^2],
\end{equation*}
where we used that $\|1\|_{\psi_{\frac 12}}=\{\log(2)\}^{-2}$ after pulling out a positive constant. 
Applying~\ref{orlicz:banach} to  $S_1$  gives
\begin{align*}
&\left\|
\sum_{\ell=1}^m
\left(
a_\ell^2\|Q\widetilde S_\ell\|_{\Hb}^2
-
\mathbb E\!\left[
a_\ell^2\|Q\widetilde S_\ell\|_{\Hb}^2
\right]
\right)
\right\|_{\psi_{1/2}}\\
&\le
A_{11,1/2}
\Bigg[
\mathbb E\left|
\sum_{\ell=1}^m
\left(
a_\ell^2\|Q\widetilde S_\ell\|_{\Hb}^2
-
\mathbb E\!\left[
a_\ell^2\|Q\widetilde S_\ell\|_{\Hb}^2
\right]
\right)
\right|
+
\left\|
\max_{\ell\in[m]}
\left|
a_\ell^2\|Q\widetilde S_\ell\|_{\Hb}^2
-
\mathbb E\!\left[
a_\ell^2\|Q\widetilde S_\ell\|_{\Hb}^2
\right]
\right|
\right\|_{\psi_{1/2}}
\Bigg].
\end{align*}
Then by the triangle inequality for the $L^1$  norm it holds that
\begin{align*}
&\mathbb E\left|
\sum_{\ell=1}^m
\left(
a_\ell^2\|Q\widetilde S_\ell\|_{\Hb}^2
-
\mathbb E\!\left[
a_\ell^2\|Q\widetilde S_\ell\|_{\Hb}^2
\right]
\right)
\right|\le
\sum_{\ell=1}^m
\mathbb E\left|
a_\ell^2\|Q\widetilde S_\ell\|_{\Hb}^2
-
\mathbb E\!\left[
a_\ell^2\|Q\widetilde S_\ell\|_{\Hb}^2
\right]
\right|
\le
2\sum_{\ell=1}^m
a_\ell^2
\mathbb E\|Q\widetilde S_\ell\|_{\Hb}^2.
\end{align*}
Also \ref{orlicz:max} gives
\begin{align*}
&\phantom{{}={}} \left\|
\max_{\ell\in[m]}
\left|
a_\ell^2\|Q\widetilde S_\ell\|_{\Hb}^2
-
\mathbb E\!\left[
a_\ell^2\|Q\widetilde S_\ell\|_{\Hb}^2
\right]
\right|
\right\|_{\psi_{1/2}}
\\& \le
A_{8,1/2}\log^2(m+1)
\max_{\ell\in[m]}
\Big\|a_\ell^2\|Q\widetilde S_\ell\|_{\Hb}^2
-
\mathbb E\!\left[
a_\ell^2\|Q\widetilde S_\ell\|_{\Hb}^2
\right]\Big\|_{\psi_{1/2}}.
\end{align*}
Now for each $\ell\in[m]$ it holds by ~\ref{orlicz:triangle_up_to_constant} that 
\begin{align*}
\|a_\ell^2\|Q\widetilde S_\ell\|_{\Hb}^2
-
\mathbb E\!\left[
a_\ell^2\|Q\widetilde S_\ell\|_{\Hb}^2
\right]\|_{\psi_{1/2}}&\leq A_{2,1/2}\{\|a_\ell^2\|Q\widetilde S_\ell\|_{\Hb}^2\|_{\psi_{1/2}}
+
\|\mathbb E\!\left[
a_\ell^2\|Q\widetilde S_\ell\|_{\Hb}^2
\right]\|_{\psi_{1/2}}\}\\
&\leq A_{2,1/2}\{\|\|Q\widetilde S_\ell\|_{\Hb}^2\|_{\psi_{1/2}}
+
\|\mathbb E\!\left[
\|Q\widetilde S_\ell\|_{\Hb}^2
\right]\|_{\psi_{1/2}}\}\\
&\leq A_{2,1/2}\{\|\|Q\widetilde S_\ell\|_{\Hb}\|^2_{\psi_{1}}
+
\mathbb E\!\left[
\|Q\widetilde S_\ell\|_{\Hb}^2
\right]\|1\|_{\psi_{1/2}}\}\\
&\leq A_{2,1/2}\{\|\|Q\widetilde S_\ell\|_{\Hb}\|^2_{\psi_{1}}
+
A_{4, 1/2,1}\|\|Q\widetilde S_\ell\|_{\Hb}^2\|_{\psi_{1/2}}
\|1\|_{\psi_{1/2}}\}\\
&=A_{2,1/2}\{\|\|Q\widetilde S_\ell\|_{\Hb}\|^2_{\psi_{1}}
+
A_{4, 1/2,1}\|\|Q\widetilde S_\ell\|_{\Hb}\|^2_{\psi_{1}}
\|1\|_{\psi_{1/2}}\}\\
&\leq D_1\|\|Q\widetilde S_\ell\|_{\Hb}\|^2_{\psi_{1}},
\end{align*}
where
$D_1=A_{2,1/2}\{1+A_{4,1/2,1}\|1\|_{\psi_{1/2}}\}$.
By the triangle inequality in \(\Hb\), stationarity, and
\ref{orlicz:triangle},
\[
\left\|
\|Q\widetilde S_\ell\|_{\Hb}
\right\|_{\psi_1}
=
\left\|
\|QS_\ell\|_{\Hb}
\right\|_{\psi_1}
\le
\sum_{i\in I_\ell}
\left\|
\|QX_i\|_{\Hb}
\right\|_{\psi_1}
=qU.
\]
Consequently, for every $\ell \in [m]$ it holds that 
\[
\|a_\ell^2\|Q\widetilde S_\ell\|_{\Hb}^2
-
\mathbb E\!\left[
a_\ell^2\|Q\widetilde S_\ell\|_{\Hb}^2
\right]\|_{\psi_{1/2}}\le D_1a_\ell^2
\left\|
\|Q\widetilde S_\ell\|_{\Hb}
\right\|_{\psi_1}^2
\le
D_1 q^2U^2,
\]
where we used that $a^2_\ell \leq 1$.
Aggregating all of the bounds we have proved shows that there are absolute constants $D_2$ and $D_3$ such that
\[
S_2\leq D_2\sum_{\ell=1}^ma_\ell^2\mathbb E[\|Q\widetilde S_\ell\|_{\Hb}^2], \qquad
S_1 
\leq 
D_3\bigg\{\log^2(m+1)q^2U^2 + \sum_{\ell=1}^ma_\ell^2\mathbb E[\|Q\widetilde S_\ell\|_{\Hb}^2] \bigg\}.
\]
so that there is an absolute constant $C_6$ such that, using ~\ref{orlicz:power} and sub-additivity of the square root function,
\[
\left\|
\sum_{\ell=1}^m a_\ell^2\|Q\widetilde S_\ell\|_{\Hb}^2
\right\|^{1/2}_{\psi_{1/2}}
\leq C_6
\bigg\{
\Big(\sum_{\ell=1}^ma_\ell^2\mathbb E[\|Q\widetilde S_\ell\|_{\Hb}^2]\Big)^{1/2}
+\log(m+1)qU
\bigg\}.
\]
The proof is finished.
\end{proof}

The following Lemma is a straightforward adaptation of a martingale maximal inequality proved in \cite{pinelis1994optimum}, Theorem 3.3. 

\begin{lemma}[Pinelis maximal Bernstein inequality]
\label{lem:pinelis-maximal-bernstein}
Let \(X_1,\ldots,X_M\) be independent, centered random elements of a
separable Hilbert space \(\mathcal H\). Suppose that there exists a constant
\(B>0\) such that, for every integer \(p\ge2\),
\[
    \max_{i \in [M]}
    \mathbb E\|X_i\|_{\mathcal H}^p
    \le
    \frac{p!}{2}B^{p} .
\]
Then it holds for every \(x\ge 0\) that 
\[
\mathbb P\Big(
    \max_{m \in [M]}
    \Big\|
        \sum_{i=1}^{m}X_i
    \Big\|_{\mathcal H}
    \geq
     B\{\sqrt{Mx}+x\}
\Big)
\le
2e^{-x/4}.
\]
\end{lemma}

\begin{proof}[Proof of Lemma~\ref{lem:pinelis-maximal-bernstein}]
Regard
\[
S_m=\sum_{i=1}^m X_i,
\qquad m=0,\ldots,M,
\]
as a martingale with respect to
$\mathcal F_m:=\sigma(X_1,\ldots,X_m)$,
where following our convention throughout this manuscript, an empty sum is defined to equal $0 \in \Hb$.
Its martingale differences are \(d_i=X_i\). A Hilbert space is a \((2,D)\)-smooth separable Banach space for $D=1$, using the Definition given in Equation (2.1) of \cite{pinelis1994optimum} and recalling the parallelogram equality for norms induced by an inner-product.  Moreover,
by independence,
\[
\E_{i-1}\|d_i\|_{\mathcal H}^p
=
\E\|X_i\|_{\mathcal H}^p.
\]
Hence, using the notation of Theorem 3.3 of \cite{pinelis1994optimum} the moment assumption in the formulation of the lemma we are proving gives 
\[
\left\|
\sum_{i=1}^M
\E_{i-1}\|d_i\|_{\mathcal H}^p
\right\|_{\infty} \le \sum_{i=1}^M\left\|
\E_{i-1}\|d_i\|_{\mathcal H}^p
\right\|_{\infty}
\le \frac{p!}{2}\,
(MB^2)B^{p-2}
=\frac{p!}{2}\,
V^2_M B^{p-2},
\]
where we define $V^2_M:=(MB^2)>0$. In this context, the norm $\|\cdot\|_{\infty}$ refers to the essential supremum of a random variable, following \cite{pinelis1994optimum}. 
Applying Theorem 3.3 of \cite{pinelis1994optimum} to the martingale
\((S_m)_{m=0}^M\) gives for every $t\geq 0$
\[
\mathbb P\Big(
\max_{m \in [M]}\|S_m\|_{\mathcal H}\ge t
\Big)
\le
2\exp\Big\{
-\frac{t^2}
{V_M^2+V_M\sqrt{V_M^2+2Bt}}
\Big\}.
\]
Now, for given $x \ge 0$, set 
\[
t=V_M\sqrt{x}+Bx.
\]
Note that by sub-additivity of the square root function, it holds that
\[
\sqrt{V_M^2+2Bt}
\le
V_M+\sqrt{2Bt},
\]
and therefore
\[
V_M^2+V_M\sqrt{V_M^2+2Bt}
\le
2V_M^2+V_M\sqrt{2Bt}.
\]
If \(Bt\le V_M^2\), then
$
V_M\sqrt{2Bt}\le \sqrt2\,V_M^2,
$
and hence for $c=1/(2+\sqrt{2})\approx 0.2928932\geq 1/4$ it holds that
\[
\frac{t^2}
{V_M^2+V_M\sqrt{V_M^2+2Bt}}
\ge
c\frac{t^2}{V_M^2}.
\]
If \(Bt>V_M^2\), then
$2V_M^2\le2Bt$
and
$
V_M\sqrt{2Bt}
<
\sqrt2\,Bt$.
Consequently, again for $c=1/(2+\sqrt{2})$ it holds that 
\[
\frac{t^2}
{V_M^2+V_M\sqrt{V_M^2+2Bt}}
\ge
c\frac{t}{B}.
\]
These two bounds together imply that
\[
\Prob\Big(
\max_{m \in [M]}\|S_m\|_{\mathcal H}\ge t
\Big)
\le
2\exp\Big\{
-\frac 14\min\Big(
\frac{t^2}{V_M^2},
\frac{t}{B}
\Big)
\Big\}.
\]
Then note that
\[
\frac 14\frac{t^2}{V_M^2}=\frac 14\frac{\{V_M\sqrt{x}+Bx\}^2}{V_M^2}\geq \frac 14 x.
\]
Likewise,
\[
\frac 14  \times \frac{t}{B}\ge \frac 14 x.
\]
Thus, 
\[
\frac 14\min\Big(
\frac{t^2}{V_M^2},
\frac{t}{B}\Big)\geq \frac 14 x.
\]
and so,
\[
\Prob\Big(
\max_{m \in [M]}\|S_m\|_{\mathcal H}\ge t
\Big)
\le
2\exp\Big\{
-\frac 14x
\Big\},
\]
which is what we sought to prove.
\end{proof}

\begin{lemma}\label{lem:berbee}
Let $\eps^{(1)}, \dots, \eps^{(n)}$ be a strictly stationary sequence taking values in the Borel space $\mathcal{S}$. For each $h \in \N$, let 
\[
\beta_h=\max_{1 \leq \ell \leq n-h}\beta(\sigma(\eps^{(1)}, \dots, \eps^{(\ell)}), \sigma(\eps^{(\ell+h)}, \dots, \eps^{(n)}),
\]
for $\beta(\mathcal{A}, \mathcal{B})$ the $\beta$-mixing coefficient of the $\sigma$-fields $\mathcal{A}$ and $\mathcal{B}$. For each $\ell \in [m]$, recall the set $I_{\ell}$, and for $\ell \in [m]$ let 
$E_\ell =(\eps^{(i)})_{i \in I_\ell}\in \mathcal{S}^q$.
Then there exists a sequence of independent random elements $\tilde{E}_\ell \in \mathcal{S}^q$ satisfying  $\tilde{E}_\ell =_dE_\ell$ for all $\ell \in [m]$ and also 
$\Prob\big((E_\ell)_{\ell \in [m]} \neq (\tilde{E}_\ell)_{\ell \in [m]}\big) \leq m\beta_r$.
\end{lemma}

\begin{proof}[Proof of Lemma~\ref{lem:berbee}]
    This is a direct consequence of Berbee's coupling lemma \cite[Lemma 4.1]{Dedecker2002}.
\end{proof}

\begin{definition}[\(\beta\)-mixing coefficient between two sub-\(\sigma\)-fields]
\label{beta_mixing_individ}
Let \((\Omega,\mathcal F,\Prob)\) be a probability space, and let
\(\mathcal F_1,\mathcal F_2\subseteq\mathcal F\) be sub-\(\sigma\)-fields. Define
\[
\beta(\mathcal F_1,\mathcal F_2)
:=
\frac12
\sup
\bigg\{
\sum_{i=1}^{I}\sum_{j=1}^{J}
\big|
\Prob(A_i\cap B_j)-\Prob(A_i)\Prob(B_j)
\big|
\bigg\},
\]
where the supremum is taken over all \(I,J\in\mathbb N\), all finite
\(\mathcal F_1\)-measurable partitions $
\Omega=\bigsqcup_{i=1}^{I}A_i$
and all finite \(\mathcal F_2\)-measurable partitions $
\Omega=\bigsqcup_{j=1}^{J}B_j.
$ Here the notation $\bigsqcup$ emphasizes that the unions are disjoint. 
\end{definition}

\begin{definition}[\(\beta\)-mixing coefficients of a strictly stationary sequence]
\label{def:beta_mixing_for_sequence}
Let \((X_i)_{i\in\mathbb Z}\) be a strictly stationary sequence of random elements. For each \(k\in\mathbb N\), define
\[
\beta_k
:=
\beta\!\left(
\sigma(X_i:i\le 0),
\sigma(X_i:i\ge k)
\right).
\]
\end{definition}

\section{Orlicz Norms}
\label{sec:orlicz}

For $x\in\R$ and $p>0$, let $\psi_p(x) = \exp(|x|^p)-1$. For a real-valued random variable $X$ let $\|X\|_{\psi_{p}} = \inf\{B>0: \E[\psi_p(|X/B|)] \leq 1\}$. A function $\psi$ is an Orlicz function if $\psi$ is convex, non-decreasing, and satisfies $\psi(0)=0$. If $p>1$, then $\psi_{p}$ is an Orlicz function. For $Z \sim N(0,1)$, we have $C_{Z} = \|Z\|_{\psi_2}=\sqrt{8/3}$ \citep[see Section 2.2.1]{VW96}, and we also have $\|1\|_{\psi_p}=(\log 2)^{-1/p}$. 

Let $\mathcal{X}$ and $\mathcal{Y}$ be Banach spaces, and let $X \in \mathcal{X}$ and $Y \in \mathcal{Y}$ be independent random elements. Let $g: \mathcal{X} \times \mathcal{Y} \rightarrow \R$ be a measurable function. For each fixed $y \in \mathcal{Y}$ let $h(y) = \|g(X,y)\|_{\psi_1}$. We then define the conditional Orlicz norm as
\begin{equation}\label{conditional_orlicz_norm}
\|g(X,Y)\|_{\psi_{1} \mid \sigma(Y)}:=h(Y).
\end{equation}

\begin{lemma}\label{prop_Orlicz}
Let $X_{1},X_{2}, \dots$ and $Y$ be real-valued random variables unless otherwise stated, and let $\psi$ be an Orlicz function. 
\begin{enumerate}
\renewcommand{\theenumi}{[P1]}
\renewcommand{\labelenumi}{\theenumi}
\item \label{orlicz:triangle}
 If $p \geq 1$, we have $\|X+Y\|_{\psi_p} \leq \|X\|_{\psi_p}+\|Y\|_{\psi_p}$.
\renewcommand{\theenumi}{[P2]}
\renewcommand{\labelenumi}{\theenumi}
\item \label{orlicz:triangle_up_to_constant}
If $p\in(0,1)$, there is a universal constant $A_{2,p}$ such that, for any $n\in\N$ and any random variables $X_1, \dots, X_n$, 
\begin{align} \label{eq:orlicz-triangle}
\Big\| \sum_{i=1}^n X_i \Big\|_{\psi_p} \leq A_{2,p}\sum_{i=1}^n \|X_i\|_{\psi_p}.
\end{align}
\renewcommand{\theenumi}{[P3]}
\renewcommand{\labelenumi}{\theenumi}
\item \label{orlicz:tail_bound} For any $x>0$, we have $\Prob(|X| \geq x) \leq 2\exp(-x^p/\|X\|_{\psi_p}^p)$. 
\renewcommand{\theenumi}{[P4]}
\renewcommand{\labelenumi}{\theenumi}
\item \label{orlicz:expectation} For any $p,q>0$ there exists a universal constant $A_{4,q,p}$ such that $(\E[|X|^{p}])^{1/p} \leq A_{4,q,p}\|X\|_{\psi_{q}}$. More precisely, we have $A_{4,1,p} = (2 \Gamma(p+1))^{1/p}$.
\renewcommand{\theenumi}{[P5]}
\renewcommand{\labelenumi}{\theenumi}
\item \label{orlicz:power}For any positive pair $(q,p)$ $\|X\|_{\psi_p}=\||X|^{p/q}\|^{q/p}_{\psi_q}$  so long as both sides of the equality are finite. 
\renewcommand{\theenumi}{[P6]}
\renewcommand{\labelenumi}{\theenumi}
\item \label{orlicz:product} For all $p$, $\|XY\|_{\psi_p} \leq \|X\|_{\psi_{2p}}\|Y\|_{\psi_{2p}}$. 
\renewcommand{\theenumi}{[P7]}
\renewcommand{\labelenumi}{\theenumi}
\item \label{orlicz:integral} If $X$ is a real valued stochastic process with Riemann integrable sample paths defined on a compact subset of the real line,  $I$, then for all $p\geq 1$, $\|\int_{I}X(t)dt\|_{\psi_p} \leq \int_{I}\|X(t)\|_{\psi_p}dt$. If $0<p<1$ there is a universal constant $A_{7,p}$ such that $\|\int_{I}X(t)dt\|_{\psi_p} \leq A_{7,p}\int_{I}\|X(t)\|_{\psi_p}dt$.
    \renewcommand{\theenumi}{[P8]}
\renewcommand{\labelenumi}{\theenumi}
\item \label{orlicz:max} 
For all $p>0$, there exists a universal constant $A_{8, p}$ such that
$\|\max_{j=1}^{n}X_j\|_{\psi_p} \leq A_{8, p} \log^{1/p}(n+1)\max_{j=1}^{n}\|X_j\|_{\psi_p}$
for any $n\in\N$ and any random variables $X_1, \dots, X_n$.
\renewcommand{\theenumi}{[P9]}
\renewcommand{\labelenumi}{\theenumi}
\item \label{orlicz:ind_sum}
Let $X_1, \dots, X_n$ be independent random elements of a separable Hilbert space $\Hb$. For all $p\in(0,2]$, there exists a universal constant $A_{9,p}$ such that
$\|\|n^{-1/2}\sum_{i=1}^{n}X_i\|_\Hb\|_{\psi_p} \leq A_{9,p}\max_{i=1}^{n}\|\|X_i\|_\Hb\|_{\psi_p}$.
\renewcommand{\theenumi}{[P10]}
\renewcommand{\labelenumi}{\theenumi}
\item \label{orlicz:limit}
If $p \geq 1$, $\|\liminf_{n\rightarrow \infty}X_n\|_{\psi_p} \leq \liminf_{n \rightarrow \infty}\|X_n\|_{\psi_p}$. If $0<p<1$, there is a universal constant $A_{10,p}$ such that $\|\liminf_{n\rightarrow \infty}X_n\|_{\psi_p} \leq A_{10,p}\liminf_{n \rightarrow \infty}\|X_n\|_{\psi_p}$.
\renewcommand{\theenumi}{[P11]}
\renewcommand{\labelenumi}{\theenumi}
\item \label{orlicz:banach}
If $X_1, \dots X_n$ are independent random elements of the Banach space $(\mathcal{B},\|\cdot\|)$ and $0<p\leq 1$ then there exists a constant $A_{11,p}$ that depends on $p$ only such that
\[
\bigg\|\Big\|\sum_{i=1}^{n}X_i\Big\|\bigg\|_{\psi_p} 
\leq 
A_{11,p}\Big\{\E\Big[\Big\|\sum_{i=1}^{n}X_i\Big\|\Big]+\Big\|\max_{i=1}^{n}\|X_{i}\|\Big\|_{\psi_p}\Big\}.
\]
Moreover, for $p=2$, we have
\[
\bigg\|\Big\|\sum_{i=1}^{n}X_i\Big\|\bigg\|_{\psi_2} 
\leq 
A_{11,2}\Big\{\E\Big[\Big\|\sum_{i=1}^{n}X_i\Big\|\Big]+\Big(\sum_{i=1}^{n}\Big\|\|X_{i}\|\Big\|^2_{\psi_2}\Big)^{1/2}\Big\}.
\]

\renewcommand{\theenumi}{[P12]}
\renewcommand{\labelenumi}{\theenumi}
\item \label{orlicz:taking_out_what_is_known}
If $X$ and $Y$ are independent random variables and $\sigma(X)$ is the $\sigma$-field generated by $X$ then 
\[
\big\|XY\|_{\psi_p\mid \sigma(X)} =X\big\|Y\|_{\psi_p}
\]
\end{enumerate}
\end{lemma}

Proofs of statements \ref{orlicz:triangle} -- \ref{orlicz:limit} can be found in \cite[Section 8]{Dec24}. A proof of \ref{orlicz:banach} can be found in \cite{talagrand1989isoperimetry}. \ref{orlicz:taking_out_what_is_known} is elementary.

\section{Background on Hilbert Spaces}\label{sec:Hilbert_space_defs}

The purpose of this section is to collect results about Hilbert spaces that we will need in the rest of the article. Almost all of the material comes from \cite{TE15}, but it is scattered throughout the book. 
A Hilbert space $\Hb$ is a complete inner-product space, and as such comes equipped with an inner-product, $\langle \cdot, \cdot \rangle_{\mathbb{H}}$. In this manuscript, we will always take $\Hb$ to be separable. Being complete, $\mathbb{H}$ is also a complete normed vector space (with norm $\|\cdot\|_{\mathbb{H}} = \sqrt{ \langle \cdot, \cdot \rangle_{\mathbb{H}}}$). Therefore we have:

\begin{lemma}\label{def:abs_conv_implies_conv}
If $\{x_j\}_{j=1}^{\infty} \subset \Hb$ satisfies $\sum_{j=1}^{\infty}\|x_j\|_{\Hb}<\infty$, then there is some $L \in \Hb$ such that 
\[
\lim_{N \to \infty} \Big\|\sum_{j=1}^{N}x_j-L \Big\|_{\Hb}=0.
\]
\end{lemma}

A collection $\{e_i\}_{i \in I} \subset \mathbb{H}$, where $I$ is an arbitrary index set, is called an orthonormal collection if $\langle e_i, e_j \rangle = \delta_{ij}$, where $\delta_{ij} = 1$ if $i =j$ and $0$ otherwise.  Given a set $S \subset \Hb$, we define $\textrm{cl}(S)$, the closure of $S$, to be the intersection of all closed sets containing $S$. In general, $\textrm{cl}(S)$ is a closed set. Given a collection of vectors $\{f_i\}_{i \in I}$, where $I$ is an arbitrary index set, we define $\lin\{f_i\}_{i \in I}$ to be the set of all finite linear combinations of elements of $\{f_i\}_{i \in I}$.  If $\textrm{cl}(\lin\{f_i\}_{i \in I})=\Hb$, and  $\{f_i\}_{i \in I}$ is an orthonormal collection,  $\{f_i\}_{i \in I}$ is called a complete orthonormal system for $\Hb$, sometimes abbreviated to CONS. The following result is Parseval's equality and is Theorem 2.4.13 in \cite{TE15}

\begin{lemma}\label{lem:parseval}
If $x \in \Hb$ and $\{e_j\}_{j=1}^{\infty}$ is a CONS, then $\|x\|^2_{\Hb} = \sum_{j=1}^{\infty}\langle x, e_j \rangle^2$.
\end{lemma}

\subsection{Linear Operators and Operator Norms}

If $T: \Hb_1 \to \Hb_2$ is a linear map between two Hilbert spaces, we define
\begin{equation*}
\textrm{Im}(T) = \{y \in \Hb_2 \mid \exists x \in H_1: T(x)=y\},
\end{equation*}
and 
\begin{equation*}
\textrm{Ker}(T) = \{x \in \Hb_1 \mid T(x)=0\}.
\end{equation*}

\subsubsection{Bounded Operators}
Given a linear map $T:\Hb_1 \to \Hb_2$ between two Hilbert spaces, define 
\[
\|T\|_{\op}=\sup_{x \in \Hb_1, \|x\|_{ \mathbb{H}}=1}\|T(x)\|_{\Hb_2}.
\]
$T$ is said to be bounded if $\| T \|_{\op}< \infty$. Let
\[
\mathcal{B}(\Hb_1,\Hb_2) = \big\{T:\mathbb{H}_1\to \mathbb{H}_2 \mid \|T\|_{\op}<\infty\big\}
\]
denote the set of bounded operators, which is a Banach space by Theorem 3.1.3 of \cite{TE15}. If $\Hb=\Hb_1=\Hb_2$ so that $T$ is a linear operator on $\Hb$, we write $\mathcal{B}(\Hb)$.  The following is Theorem 3.3.1 of \cite{TE15}:

\begin{lemma}
If $T \in \mathcal{B}(\mathbb{H}_1,\mathbb{H}_2)$, then there is a unique $T^*\in \mathcal{B}(\mathbb{H}_2,\mathbb{H}_1)$ such that for all $(x_1, x_2) \in (\Hb_1, \Hb_2)$, $\langle T(x_1), x_2 \rangle_{\Hb_2} =\langle x_1, T^{*}x_2 \rangle_{\Hb_1}$.
\end{lemma}

If $T \in \mathcal{B}(\Hb)$ satisfies $T=T^*$, $T$ is said to be self-adjoint. A self-adjoint operator is said to be non-negative or positive semi-definite if for all $x \in \Hb$, $\langle Tx, x \rangle_\Hb \geq 0.$
 The following is Theorem 3.4.3 of \cite{TE15}:
 
\begin{lemma}
If $T$ is bounded and non-negative, then there exists a bounded non-negative operator $S$ satisfying $S^2:=S\circ S =T$, and we write $S=T^{1/2}$. Further, if $T$ is a bounded operator on $\Hb$, $T^{*}T$ is self-adjoint and non-negative.
\end{lemma}

\subsubsection{Compact Operators}
The following is Definition 4.1.1 from \cite{TE15}.
\begin{definition}[Compact Linear Transformation]
A linear transformation $T$ from a normed space $\mathbb{X}_1$ into a normed space $\mathbb{X}_2$ is called a compact transformation if for any bounded sequence $\{x_n\} \subset \mathbb{X}_1$, $\{Tx_n\}$ contains a convergent subsequence in $\mathbb{X}_2$.
\end{definition}

The following is Definition 3.4.6 in \cite{TE15}.

\begin{definition}\label{def:kronecker_product}
Let $x_1, x_2 \in \mathbb{H}$. Then,
\begin{equation}
(x_1 \otimes x_2):\mathbb{H}\rightarrow \mathbb{H}
\end{equation}
is the linear operator defined for each $y \in \mathbb{H}$ by
\begin{equation}\label{def:tensor_product}
(x_1 \otimes x_2)(y)=\langle x_1, y \rangle_{\mathbb{H}} x_2. 
\end{equation}
\end{definition}
A simple calculation shows $\otimes$ is bi-linear and that $\|x_1 \otimes x_2\|_{\op}=\|x_1\|_\Hb \|x_2\|_\Hb$, whence $(x_1 \otimes x_2) \in \mathcal B(\Hb)$; see Theorem 3.4.7 in \cite{TE15} and Lemma~\ref{lem:norm_kronecker_product} below.

The following is Theorem 4.2.4 of \cite{TE15}:
 
\begin{theorem}
Suppose $T$ is a compact, self-adjoint operator. The set of non-zero eigenvalues of $T$ is either finite or consists of a sequence that tends to 0. Each non-zero eigenvalue has finite multiplicity, and eigenvectors corresponding to different eigenvalues are orthogonal. Let $\lambda_1, \lambda_2, \dots$ be the eigenvalues of $T$ ordered so that $|\lambda_1| \geq |\lambda_2|\geq \dots$, and let $e_1, e_2, \dots$ be the corresponding orthonormal eigenvectors obtained from the Gram-Schmidt procedure, necessary in case that there are eigenvalues of non-unit multiplicity. Then $\{e_j\}$ is a CONS for $\textrm{cl} \hspace{.1cm}\textrm{Im}(T)$, and 
\[
T=\sum_{j \geq 1}\lambda_j(e_j \otimes e_j),
\]
i.e. for every $x \in \mathbb{H}$,

\[
T(x)=\sum_{j \geq 1}\lambda_j \langle x, e_j \rangle_{\mathbb{H}} e_j.
\]
\end{theorem}

The following is Lemma 4.2.7 in \cite{TE15}.

\begin{lemma}\label{lem:min_max_principle_Hilbert_space}
Let $T$ be a non-negative definite self-adjoint compact operator on a Hilbert space $\mathbb{H}$ with eigenvalues $\lambda_1 \geq \lambda_2 \geq \dots \geq 0$. Then,

\begin{equation}\label{eq_1:lem:min_max_principle_Hilbert_space}
\lambda_k=\max_{v_1, \dots v_k \in \mathbb{H}}\min_{v \in \lin\{v_1, \dots, v_k\}}\frac{\langle Tv, v \rangle }{\|v\|^2}
\end{equation}
and 
\begin{equation}\label{eq_2:lem:min_max_principle_Hilbert_space}
\lambda_k=\min_{v_1, \dots v_{k-1} \in \mathbb{H}}\max_{v \in \lin\{v_1, \dots, v_{k-1}\}^{\perp}}\frac{\langle Tv, v \rangle }{\|v\|^2},
\end{equation}
where the maximum and minimum in Equations \eqref{eq_1:lem:min_max_principle_Hilbert_space} and \eqref{eq_2:lem:min_max_principle_Hilbert_space} are each attained when $v_1, \dots, v_k$ are equal to $e_1, \dots, e_k$, the eigenvectors corresponding to $\lambda_1, \dots, \lambda_k$.
\end{lemma}

The following is Theorem 4.2.8 of \cite{TE15}.

\begin{lemma}\label{lem:weyl_eigen_value_bound_Hilbert_space}
Let $T$ and ${\tilde T}$ be non-negative definite, self-adjoint, compact operators with eigenvalue sequences $\{\lambda_j\}$ and $\{\tilde{\lambda}_j\}$ respectively. Then, 
\[
\sup_{k \geq 1}|\lambda_j -\tilde{\lambda}_j| \leq \|{T} - {\tilde T}\|_{\op}. 
\]
\end{lemma}

    \subsubsection{Hilbert-Schmidt Operators}
Let $\{e_i\}$ be a CONS for $\Hb_1$, and let $T \in \mathcal{B}(\Hb_1, \Hb_2)$. If $T$ satisfies 
\begin{equation}
\|T\|_{\HS}:=\Big(\sum_{i=1}^{\infty}\|T(e_i)\|^2_{\Hb_2}\Big)^{1/2}<\infty,
\end{equation}
then $T$ is called a Hilbert-Schmidt operator. The collection of $T \in \mathcal{B}(H_1, H_2)$ such that $\|T\|_{\HS}$ is finite is called the collection of Hilbert-Schmidt operators, $\mathcal{B}_{\HS}$. The space of Hilbert-Schmidt operators is a Hilbert space for a particular choice of inner-product. Further more, every Hilbert-Schmidt operator is a compact operator, and if $T$ is a self-adjoint positive semi-definite Hilbert-schmidt operator, we have
\begin{equation} \label{eq:hilber-schmidt-eigenvalues}
\|T\|_{\HS} = \Big(\sum_{j=1}^{\infty}\lambda^2_j \Big)^{1/2}.
\end{equation}

\subsubsection{Trace-class operators}

\begin{definition}\label{def:trace_norm}
If $T \in \mathcal{B}(\Hb)$, we define the trace-norm of $T$ as follows. Given any CONS $\{e_i\}_{i=1}^{\infty}$ for $\Hb$ define
\[
\|T\|_{\Tr}:=\sum_{j=1}^{\infty}\langle (T^*T)^{1/2}e_j,e_j \rangle_{\Hb}.
\]
Then define
\[
\mathcal{B}_{\Tr}=\mathcal{B}_{\Tr}(\Hb):=\{T \in \mathcal{B}(\Hb) \, : \, \|T\|_{\Tr}<\infty\},
\]
\end{definition}
It is well know than that the normed vector space $\mathcal{B}_{\Tr}$ is a Banach space. If $T\in \mathcal{B}_{\Tr}$ the trace of $T$ is defined, following \citet[page 114]{TE15}, as follows. Given any CONS $\{e_i\}_{i=1}^{\infty}$ for $\Hb$ let
\[
\Tr(T):=\sum_{j=1}^\infty \langle Te_j, e_j\rangle_\Hb.
\]
If $T$ is self-adjoint and non-negative, then
\begin{equation}
\label{eq:trace-as-sum-eigenvalues}
\Tr(T)=\|T\|_\Tr=\sum_{j\ge 1}\lambda_j,
\end{equation}
where $\{\lambda_j\}$ is the collection of non-zero eigenvalues of $T$ \citep[Equation (4.35)]{TE15}. In the case of general $T\in \mathcal B_\Tr$, it holds that
\begin{equation*}
\|T\|_\Tr=\sum_{j\ge 1}\sigma_j
\end{equation*}
where $\{\sigma_j\}$ is the collection of non-zero eigenvalues of $T^*T$, also called the singular values of $T$.
Every trace class operator $T \in \mathcal B(\Hb)$ is also a Hilbert-Schmidt operator, and it satisfies 
\begin{align} \label{eq:operator-norms-inequality}
\big\|T\big\|_{\op} \leq \big\|T\big\|_{\HS} \leq \big\|T\big\|_{\Tr}.
\end{align}

\begin{lemma}\label{lem:norm_kronecker_product}
For $x, y \in \Hb$, we have $(x \otimes y)^*=(y \otimes x)$ and
\[
\|x \otimes y\|_{\op}=\|x \otimes y\|_{\HS}=\|x \otimes y\|_{\Tr} = \|x\|_\Hb\|y\|_\Hb.
\]
\end{lemma}

\begin{proof}
In view of \eqref{eq:operator-norms-inequality} and $\|x \otimes y\|_{\op} = \|x\|_\Hb\|y\|_\Hb$ by Theorem 3.4.7 in \cite{TE15}, it is sufficient to show that $\|x \otimes y\|_{\Tr} \le  \|x\|_\Hb\|y\|_\Hb$. Without loss of generality, by bi-linearity of $\otimes$, assume that $\|x\|_\Hb=\|y\|_\Hb=1$.

First note that for any $a, b \in \Hb$ we have
\[
\langle (x \otimes y)(a), b \rangle =  \langle x,a\rangle \langle  y, b \rangle 
\]
Interchanging $y$ with $x$ and $a$ with $b$ (and symmetry of the real inner product) gives 
\[
\langle a, (y \otimes x)b \rangle=\langle  y, b \rangle \langle x,a\rangle 
\]
Thus $(x \otimes y)^*= y\otimes x$. Next, for any  $w \in \Hb$, we have
\begin{align*}
\big((x \otimes y)^*(x \otimes y)\big)(w)  
=
(y \otimes x) \big( \langle x,w\rangle y\big)
=
\langle x,w\rangle \langle y,y \rangle x 
= 
\langle x,w\rangle  x = (x \otimes x) w,
\end{align*}
which shows that $(x \otimes y)^*(x \otimes y) = (x\otimes x)$.
Let $e_1=x$ and let $e_2, e_3 \dots $ be a complete orthonormal system for $\textrm{span}\{x\}^{\perp}$, so that $e_1, e_2, \dots$ is a complete orthonormal system for $\Hb$. Then, since  $(x \otimes x) = (x\otimes x)^{2}$,
\[
\big\|x\otimes y\big\|_{\Tr} = \sum_{j=1}^{\infty} \langle (x\otimes x)e_j, e_j \rangle_\Hb = \sum_{j=1}^{\infty} \langle (e_1\otimes e_1)e_j, e_j \rangle_\Hb =\langle (e_1\otimes e_1)e_1, e_1 \rangle_\Hb=1.
\]
\end{proof}

\begin{lemma}[Variational characterization of the trace norm and rank-one trace identity]
\label{lem:trace_norm_variational_rank_one}
Let \(\Hb\) be a separable Hilbert space.
Then the following assertions hold.
\begin{enumerate}[(i)]
\item If \(T\in\mathcal B_{\Tr}(\Hb)\) and \(A\in\mathcal B(\Hb)\), then
\(AT\) and \(TA\) are trace class and
\begin{equation}
\label{eq:trace_ideal_inequality}
\bigl|\Tr(AT)\bigr|
\le
\|A\|_{\op}\|T\|_{\Tr}.
\end{equation}

\item For every \(T\in\mathcal B_{\Tr}(\Hb)\),
\begin{equation*}
\|T\|_{\Tr}
=
\sup_{\substack{A\in\mathcal B(\Hb)\\ \|A\|_{\op}\le1}}
\bigl|\Tr(AT)\bigr|.
\end{equation*}

\item Let \(x,y\in\Hb\). Then recalling Definition~\ref{def:kronecker_product} it holds that  \(x\otimes y\) is trace class and, for every
\(A\in\mathcal B(\Hb)\),
\begin{equation*}
\Tr\bigl(A(x\otimes y)\bigr)
=
\langle Ax,y\rangle_{\Hb}.
\end{equation*}
\end{enumerate}
\end{lemma}

\begin{proof}
The estimate
\eqref{eq:trace_ideal_inequality} is standard; see
\citet[Page 267]{Con90}. For \(T\in\mathcal B_{\Tr}(\Hb)\) it holds that,
\begin{equation}
\label{eq:trace_norm_variational_compact}
\|T\|_{\Tr}
=
\sup_{\substack{C \textrm{ is compact}\\ \|C\|_{\op}\le1}}
\bigl|\Tr(CT)\bigr|,
\end{equation}
see \citet[Page 268]{Con90}.
Since every compact operator is bounded,
\eqref{eq:trace_norm_variational_compact} implies
\[
\|T\|_{\Tr}
\le
\sup_{\substack{A\in\mathcal B(\Hb)\\ \|A\|_{\op}\le1}}
\bigl|\Tr(AT)\bigr|.
\]
Conversely, by~(i), every \(A\in\mathcal B(\Hb)\) satisfying
\(\|A\|_{\op}\le1\) obeys
\[
\bigl|\Tr(AT)\bigr|
\le
\|T\|_{\Tr}.
\]
Taking the supremum over all such \(A\) yields the reverse inequality and
proves~(ii). Finally, (iii) is given in \citet[page 268]{Con90}.
\end{proof}

\subsection{Mean and Covariance of a random elements in Hilbert spaces}

Let $f$ and $g$ be random elements of $\Hb$, and suppose $\E[\|f\|^2_{\Hb}] \vee \E[\|g\|^2_{\Hb}] <\infty$. Recalling the tensor product $\otimes$ from Definition~\ref{def:kronecker_product}, the cross-covariance operator of $f$ and $g$ is defined as
\begin{equation*} 
\Cov(f,g)=\int_{\Omega} (f - \E[f]) \otimes(g-\E[g]) \,\diff\mathbb{P} \in \mathcal{B}_{\HS}(\Hb);
\end{equation*}
where the integral is taken in the Bochner sense for $\mathcal{B}_{\HS}(\Hb)$-valued functions. The covariance operator of $f$ is defined as $\Cov(f) := \Cov(f,f)$. 
The following is a slight extension of Theorem 7.2.5 in \cite{TE15}

\begin{lemma} \label{lem:properties-covariance}
If $\Exp[\| f \|_\Hb^2]<\infty$ and $\E[f]=0$ we have:
\begin{enumerate}[(1)]
\item $\langle \Cov(f)u, v \rangle_{\Hb} =\E[\langle f,u \rangle_{\Hb}\langle f,v \rangle_{\Hb}
]$ for all $v, u \in \Hb$,
\item $\Cov(f)$ is a non-negative trace-class operator with $\| \Cov(f)\|_{\Tr} = \Exp[\| f \|_\Hb^2]$. 
\item $\Prob\big(f \in \mathrm{cl}(\mathrm{Im}(\Cov(f)))\big)=1$.
\item[(4)] For $T\in\mathcal B(\Hb, \Hb')$, we have $\Cov (T f) = T \Cov(f) T^*$.
\end{enumerate}
If $g$ is another random in $\Hb$ with $\Exp[\| g \|_\Hb^2]<\infty$ and $\E[g]=0$, defined on the same probability space as $f$, then
\begin{enumerate}[(1)]
\item[(5)] $\Cov(f,g)$ is a trace-class operator with 
\[
\|\Cov(f,g)\|_{\Tr} \leq \Exp[\| f \|_\Hb\| g \|_\Hb] \le ( \Exp[\| f \|_\Hb^2]\Exp[\| g \|_\Hb^2])^{1/2}.
\]
\end{enumerate}
\end{lemma}

\putbib[biblio]
\end{bibunit}

\begin{thebibliography}{33}

\bibitem[\protect\citeauthoryear{Aitchison}{1982}]{aitchison1982statistical}
\begin{barticle}[author]
\bauthor{\bsnm{Aitchison},~\bfnm{J.}\binits{J.}}
(\byear{1982}).
\btitle{The statistical analysis of compositional data}.
\bjournal{J. Roy. Statist. Soc. Ser. B}
\bvolume{44}
\bpages{139--177}.
\bnote{With discussion}.
\bmrnumber{676206}
\end{barticle}
\endbibitem

\bibitem[\protect\citeauthoryear{Aston and Kirch}{2012}]{MR2922864}
\begin{barticle}[author]
\bauthor{\bsnm{Aston},~\bfnm{John A.~D.}\binits{J.~A.~D.}} \AND
  \bauthor{\bsnm{Kirch},~\bfnm{Claudia}\binits{C.}}
(\byear{2012}).
\btitle{Detecting and estimating changes in dependent functional data}.
\bjournal{J. Multivariate Anal.}
\bvolume{109}
\bpages{204--220}.
\bdoi{10.1016/j.jmva.2012.03.006}
\bmrnumber{2922864}
\end{barticle}
\endbibitem

\bibitem[\protect\citeauthoryear{Aue, Rice and S\"onmez}{2018}]{Aue2018}
\begin{barticle}[author]
\bauthor{\bsnm{Aue},~\bfnm{Alexander}\binits{A.}},
  \bauthor{\bsnm{Rice},~\bfnm{Gregory}\binits{G.}} \AND
  \bauthor{\bsnm{S\"onmez},~\bfnm{Ozan}\binits{O.}}
(\byear{2018}).
\btitle{Detecting and dating structural breaks in functional data without
  dimension reduction}.
\bjournal{J. R. Stat. Soc. Ser. B. Stat. Methodol.}
\bvolume{80}
\bpages{509--529}.
\bdoi{10.1111/rssb.12257}
\bmrnumber{3798876}
\end{barticle}
\endbibitem

\bibitem[\protect\citeauthoryear{Bai, Hu and Wu}{2026}]{Bai2026}
\begin{barticle}[author]
\bauthor{\bsnm{Bai},~\bfnm{Lujia}\binits{L.}},
  \bauthor{\bsnm{Hu},~\bfnm{Qirui}\binits{Q.}} \AND
  \bauthor{\bsnm{Wu},~\bfnm{Weichi}\binits{W.}}
(\byear{2026}).
\btitle{Inference for structural changes in nonstationary functional time
  series with partial measurement error}.
\bjournal{Journal of the Royal Statistical Society Series B: Statistical
  Methodology}
\bpages{qkag072}.
\bdoi{10.1093/jrsssb/qkag072}
\end{barticle}
\endbibitem

\bibitem[\protect\citeauthoryear{Banerjee, Laha and Lakra}{2020}]{MR4176154}
\begin{barticle}[author]
\bauthor{\bsnm{Banerjee},~\bfnm{Buddhananda}\binits{B.}},
  \bauthor{\bsnm{Laha},~\bfnm{Arnab~K.}\binits{A.~K.}} \AND
  \bauthor{\bsnm{Lakra},~\bfnm{Arjun}\binits{A.}}
(\byear{2020}).
\btitle{Data-driven dimension reduction in functional principal component
  analysis identifying the change-point in functional data}.
\bjournal{Stat. Anal. Data Min.}
\bvolume{13}
\bpages{529--536}.
\bdoi{10.1002/sam.11471}
\bmrnumber{4176154}
\end{barticle}
\endbibitem

\bibitem[\protect\citeauthoryear{Bastian}{2025}]{Bastian2025}
\begin{barticle}[author]
\bauthor{\bsnm{Bastian},~\bfnm{Patrick}\binits{P.}}
(\byear{2025}).
\btitle{Choosing the right norm for change point detection in functional data}.
\bjournal{Electron. J. Stat.}
\bvolume{19}
\bpages{4637--4672}.
\bdoi{10.1214/25-ejs2451}
\bmrnumber{4965756}
\end{barticle}
\endbibitem

\bibitem[\protect\citeauthoryear{Bastian, Basu and
  Dette}{2024}]{bastian2024multiple}
\begin{barticle}[author]
\bauthor{\bsnm{Bastian},~\bfnm{Patrick}\binits{P.}},
  \bauthor{\bsnm{Basu},~\bfnm{Rupsa}\binits{R.}} \AND
  \bauthor{\bsnm{Dette},~\bfnm{Holger}\binits{H.}}
(\byear{2024}).
\btitle{Multiple change point detection in functional data with applications to
  biomechanical fatigue data}.
\bjournal{Ann. Appl. Stat.}
\bvolume{18}
\bpages{3109--3129}.
\bdoi{10.1214/24-aoas1926}
\bmrnumber{4815956}
\end{barticle}
\endbibitem

\bibitem[\protect\citeauthoryear{Berkes et~al.}{2009}]{berkes2009detecting}
\begin{barticle}[author]
\bauthor{\bsnm{Berkes},~\bfnm{Istv\'an}\binits{I.}},
  \bauthor{\bsnm{Gabrys},~\bfnm{Robertas}\binits{R.}},
  \bauthor{\bsnm{Horv\'ath},~\bfnm{Lajos}\binits{L.}} \AND
  \bauthor{\bsnm{Kokoszka},~\bfnm{Piotr}\binits{P.}}
(\byear{2009}).
\btitle{Detecting changes in the mean of functional observations}.
\bjournal{J. R. Stat. Soc. Ser. B Stat. Methodol.}
\bvolume{71}
\bpages{927--946}.
\bdoi{10.1111/j.1467-9868.2009.00713.x}
\bmrnumber{2750251}
\end{barticle}
\endbibitem

\bibitem[\protect\citeauthoryear{Cai and Hall}{2006}]{cai2006prediction}
\begin{barticle}[author]
\bauthor{\bsnm{Cai},~\bfnm{T.~Tony}\binits{T.~T.}} \AND
  \bauthor{\bsnm{Hall},~\bfnm{Peter}\binits{P.}}
(\byear{2006}).
\btitle{Prediction in functional linear regression}.
\bjournal{Ann. Statist.}
\bvolume{34}
\bpages{2159--2179}.
\bdoi{10.1214/009053606000000830}
\bmrnumber{2291496}
\end{barticle}
\endbibitem

\bibitem[\protect\citeauthoryear{Chakraborty and
  Sakhanenko}{2023}]{Chakraborty2023}
\begin{barticle}[author]
\bauthor{\bsnm{Chakraborty},~\bfnm{Nilanjan}\binits{N.}} \AND
  \bauthor{\bsnm{Sakhanenko},~\bfnm{Lyudmila}\binits{L.}}
(\byear{2023}).
\btitle{Novel multiplier bootstrap tests for high-dimensional data with
  applications to {MANOVA}}.
\bjournal{Comput. Statist. Data Anal.}
\bvolume{178}
\bpages{Paper No. 107619, 18}.
\bdoi{10.1016/j.csda.2022.107619}
\bmrnumber{4488688}
\end{barticle}
\endbibitem

\bibitem[\protect\citeauthoryear{Dette and Kutta}{2021}]{DetteKutta2021}
\begin{barticle}[author]
\bauthor{\bsnm{Dette},~\bfnm{Holger}\binits{H.}} \AND
  \bauthor{\bsnm{Kutta},~\bfnm{Tim}\binits{T.}}
(\byear{2021}).
\btitle{Detecting structural breaks in eigensystems of functional time series}.
\bjournal{Electron. J. Stat.}
\bvolume{15}
\bpages{944--983}.
\bdoi{10.1214/20-ejs1796}
\bmrnumber{4255289}
\end{barticle}
\endbibitem

\bibitem[\protect\citeauthoryear{Fremdt et~al.}{2014}]{MR3147328}
\begin{barticle}[author]
\bauthor{\bsnm{Fremdt},~\bfnm{Stefan}\binits{S.}},
  \bauthor{\bsnm{Horv\'ath},~\bfnm{Lajos}\binits{L.}},
  \bauthor{\bsnm{Kokoszka},~\bfnm{Piotr}\binits{P.}} \AND
  \bauthor{\bsnm{Steinebach},~\bfnm{Josef~G.}\binits{J.~G.}}
(\byear{2014}).
\btitle{Functional data analysis with increasing number of projections}.
\bjournal{J. Multivariate Anal.}
\bvolume{124}
\bpages{313--332}.
\bdoi{10.1016/j.jmva.2013.11.009}
\bmrnumber{3147328}
\end{barticle}
\endbibitem

\bibitem[\protect\citeauthoryear{Hall and Van~Keilegom}{2007}]{HK07}
\begin{barticle}[author]
\bauthor{\bsnm{Hall},~\bfnm{Peter}\binits{P.}} \AND
  \bauthor{\bsnm{Van~Keilegom},~\bfnm{Ingrid}\binits{I.}}
(\byear{2007}).
\btitle{Two-sample tests in functional data analysis starting from discrete
  data}.
\bjournal{Statist. Sinica}
\bvolume{17}
\bpages{1511--1531}.
\bmrnumber{2413533}
\end{barticle}
\endbibitem

\bibitem[\protect\citeauthoryear{Holm}{1979}]{holm1979simple}
\begin{barticle}[author]
\bauthor{\bsnm{Holm},~\bfnm{Sture}\binits{S.}}
(\byear{1979}).
\btitle{A simple sequentially rejective multiple test procedure}.
\bjournal{Scand. J. Statist.}
\bvolume{6}
\bpages{65--70}.
\bmrnumber{538597}
\end{barticle}
\endbibitem

\bibitem[\protect\citeauthoryear{Horv\'ath, Kokoszka and
  Rice}{2014}]{Horvath2014}
\begin{barticle}[author]
\bauthor{\bsnm{Horv\'ath},~\bfnm{Lajos}\binits{L.}},
  \bauthor{\bsnm{Kokoszka},~\bfnm{Piotr}\binits{P.}} \AND
  \bauthor{\bsnm{Rice},~\bfnm{Gregory}\binits{G.}}
(\byear{2014}).
\btitle{Testing stationarity of functional time series}.
\bjournal{J. Econometrics}
\bvolume{179}
\bpages{66--82}.
\bdoi{10.1016/j.jeconom.2013.11.002}
\bmrnumber{3153649}
\end{barticle}
\endbibitem

\bibitem[\protect\citeauthoryear{Jirak}{2015}]{jirak2015uniform}
\begin{barticle}[author]
\bauthor{\bsnm{Jirak},~\bfnm{Moritz}\binits{M.}}
(\byear{2015}).
\btitle{Uniform change point tests in high dimension}.
\bjournal{Ann. Statist.}
\bvolume{43}
\bpages{2451--2483}.
\bdoi{10.1214/15-AOS1347}
\bmrnumber{3405600}
\end{barticle}
\endbibitem

\bibitem[\protect\citeauthoryear{Jirak}{2016}]{jirak2016optimal}
\begin{barticle}[author]
\bauthor{\bsnm{Jirak},~\bfnm{Moritz}\binits{M.}}
(\byear{2016}).
\btitle{Optimal eigen expansions and uniform bounds}.
\bjournal{Probab. Theory Related Fields}
\bvolume{166}
\bpages{753--799}.
\bdoi{10.1007/s00440-015-0671-3}
\bmrnumber{3568039}
\end{barticle}
\endbibitem

\bibitem[\protect\citeauthoryear{Jirak}{2018}]{jirak2018rate}
\begin{barticle}[author]
\bauthor{\bsnm{Jirak},~\bfnm{Moritz}\binits{M.}}
(\byear{2018}).
\btitle{Rate of convergence for {H}ilbert space valued processes}.
\bjournal{Bernoulli}
\bvolume{24}
\bpages{202--230}.
\bdoi{10.3150/16-BEJ870}
\bmrnumber{3706754}
\end{barticle}
\endbibitem

\bibitem[\protect\citeauthoryear{Kumar et~al.}{2024}]{kumar2024estimation}
\begin{barticle}[author]
\bauthor{\bsnm{Kumar},~\bfnm{Shivam}\binits{S.}},
  \bauthor{\bsnm{Xu},~\bfnm{Haotian}\binits{H.}},
  \bauthor{\bsnm{Cho},~\bfnm{Haeran}\binits{H.}} \AND
  \bauthor{\bsnm{Wang},~\bfnm{Daren}\binits{D.}}
(\byear{2024}).
\btitle{Estimation and inference for change points in functional regression
  time series}.
\bjournal{arXiv preprint arXiv:2405.05459}.
\end{barticle}
\endbibitem

\bibitem[\protect\citeauthoryear{Kutta, Dette and
  Wang}{2025}]{kutta2025multiscale}
\begin{barticle}[author]
\bauthor{\bsnm{Kutta},~\bfnm{Tim}\binits{T.}},
  \bauthor{\bsnm{Dette},~\bfnm{Holger}\binits{H.}} \AND
  \bauthor{\bsnm{Wang},~\bfnm{Shixuan}\binits{S.}}
(\byear{2025}).
\btitle{Multiscale Change Point Detection for Functional Time Series}.
\bjournal{arXiv preprint arXiv:2511.06870}.
\end{barticle}
\endbibitem

\bibitem[\protect\citeauthoryear{Li, Li and Shang}{2024}]{li2024detection}
\begin{barticle}[author]
\bauthor{\bsnm{Li},~\bfnm{Degui}\binits{D.}},
  \bauthor{\bsnm{Li},~\bfnm{Runze}\binits{R.}} \AND
  \bauthor{\bsnm{Shang},~\bfnm{Han~Lin}\binits{H.~L.}}
(\byear{2024}).
\btitle{Detection and estimation of structural breaks in high-dimensional
  functional time series}.
\bjournal{Ann. Statist.}
\bvolume{52}
\bpages{1716--1740}.
\bdoi{10.1214/24-aos2414}
\bmrnumber{4804825}
\end{barticle}
\endbibitem

\bibitem[\protect\citeauthoryear{Lopes}{2025}]{lopes2025improved}
\begin{barticle}[author]
\bauthor{\bsnm{Lopes},~\bfnm{Miles~E.}\binits{M.~E.}}
(\byear{2025}).
\btitle{Improved rates of bootstrap approximation for the operator norm: a
  coordinate-free approach}.
\bjournal{Ann. Inst. Henri Poincar\'e{} Probab. Stat.}
\bvolume{61}
\bpages{2262--2290}.
\bdoi{10.1214/24-aihp1477}
\bmrnumber{4947554}
\end{barticle}
\endbibitem

\bibitem[\protect\citeauthoryear{Lu et~al.}{2022}]{lu2022almost}
\begin{barticle}[author]
\bauthor{\bsnm{Lu},~\bfnm{Jianya}\binits{J.}},
  \bauthor{\bsnm{Wu},~\bfnm{Wei~Biao}\binits{W.~B.}},
  \bauthor{\bsnm{Xiao},~\bfnm{Zhijie}\binits{Z.}} \AND
  \bauthor{\bsnm{Xu},~\bfnm{Lihu}\binits{L.}}
(\byear{2022}).
\btitle{Almost sure invariance principle of $\beta-$ mixing time series in
  Hilbert space}.
\bjournal{arXiv preprint arXiv:2209.12535}.
\end{barticle}
\endbibitem

\bibitem[\protect\citeauthoryear{Padilla et~al.}{2022}]{3600270.3602960}
\begin{binproceedings}[author]
\bauthor{\bsnm{Padilla},~\bfnm{Carlos Misael~Madrid}\binits{C.~M.~M.}},
  \bauthor{\bsnm{Zhao},~\bfnm{Zifeng}\binits{Z.}},
  \bauthor{\bsnm{Wang},~\bfnm{Daren}\binits{D.}} \AND
  \bauthor{\bsnm{Yu},~\bfnm{Yi}\binits{Y.}}
(\byear{2022}).
\btitle{Change-point detection for sparse and dense functional data in general
  dimensions}.
In \bbooktitle{Proceedings of the 36th International Conference on Neural
  Information Processing Systems}.
\bseries{NIPS '22}.
\bpublisher{Curran Associates Inc.}, \baddress{Red Hook, NY, USA}.
\end{binproceedings}
\endbibitem

\bibitem[\protect\citeauthoryear{Pinelis}{1994}]{pinelis1994optimum}
\begin{barticle}[author]
\bauthor{\bsnm{Pinelis},~\bfnm{Iosif}\binits{I.}}
(\byear{1994}).
\btitle{Optimum bounds for the distributions of martingales in {B}anach
  spaces}.
\bjournal{Ann. Probab.}
\bvolume{22}
\bpages{1679--1706}.
\bmrnumber{1331198}
\end{barticle}
\endbibitem

\bibitem[\protect\citeauthoryear{Rice, Wirjanto and
  Zhao}{2020}]{RiceWirjantoZhao2020}
\begin{barticle}[author]
\bauthor{\bsnm{Rice},~\bfnm{Gregory}\binits{G.}},
  \bauthor{\bsnm{Wirjanto},~\bfnm{Tony}\binits{T.}} \AND
  \bauthor{\bsnm{Zhao},~\bfnm{Yuqian}\binits{Y.}}
(\byear{2020}).
\btitle{Tests for conditional heteroscedasticity of functional data}.
\bjournal{J. Time Series Anal.}
\bvolume{41}
\bpages{733--758}.
\bdoi{10.1111/jtsa.12532}
\bmrnumber{4176358}
\end{barticle}
\endbibitem

\bibitem[\protect\citeauthoryear{Romano and Wolf}{2005}]{RomanoWolf2005}
\begin{barticle}[author]
\bauthor{\bsnm{Romano},~\bfnm{Joseph~P.}\binits{J.~P.}} \AND
  \bauthor{\bsnm{Wolf},~\bfnm{Michael}\binits{M.}}
(\byear{2005}).
\btitle{Exact and approximate stepdown methods for multiple hypothesis
  testing}.
\bjournal{J. Amer. Statist. Assoc.}
\bvolume{100}
\bpages{94--108}.
\bdoi{10.1198/016214504000000539}
\bmrnumber{2156821}
\end{barticle}
\endbibitem

\bibitem[\protect\citeauthoryear{Shang and Ji}{2023}]{ShangJi2023}
\begin{barticle}[author]
\bauthor{\bsnm{Shang},~\bfnm{Han~Lin}\binits{H.~L.}} \AND
  \bauthor{\bsnm{Ji},~\bfnm{Kaiying}\binits{K.}}
(\byear{2023}).
\btitle{Forecasting intraday financial time series with sieve bootstrapping and
  dynamic updating}.
\bjournal{J. Forecast.}
\bvolume{42}
\bpages{1973--1988}.
\bdoi{10.1002/for.3000}
\bmrnumber{4667846}
\end{barticle}
\endbibitem

\bibitem[\protect\citeauthoryear{Sharipov, Tewes and
  Wendler}{2016}]{sharipov2016sequential}
\begin{barticle}[author]
\bauthor{\bsnm{Sharipov},~\bfnm{Olimjon}\binits{O.}},
  \bauthor{\bsnm{Tewes},~\bfnm{Johannes}\binits{J.}} \AND
  \bauthor{\bsnm{Wendler},~\bfnm{Martin}\binits{M.}}
(\byear{2016}).
\btitle{Sequential block bootstrap in a {H}ilbert space with application to
  change point analysis}.
\bjournal{Canad. J. Statist.}
\bvolume{44}
\bpages{300--322}.
\bdoi{10.1002/cjs.11293}
\bmrnumber{3536199}
\end{barticle}
\endbibitem

\bibitem[\protect\citeauthoryear{Thevenon et~al.}{2013}]{thevenon2013human}
\begin{barticle}[author]
\bauthor{\bsnm{Thevenon},~\bfnm{Florian}\binits{F.}},
  \bauthor{\bsnm{Wirth},~\bfnm{Stefanie~B}\binits{S.~B.}},
  \bauthor{\bsnm{Fujak},~\bfnm{Marian}\binits{M.}},
  \bauthor{\bsnm{Pot{\'e}},~\bfnm{John}\binits{J.}} \AND
  \bauthor{\bsnm{Girardclos},~\bfnm{St{\'e}phanie}\binits{S.}}
(\byear{2013}).
\btitle{Human impact on the transport of terrigenous and anthropogenic elements
  to peri-alpine lakes (Switzerland) over the last decades}.
\bjournal{Aquatic Sciences}
\bvolume{75}
\bpages{413--424}.
\end{barticle}
\endbibitem

\bibitem[\protect\citeauthoryear{VanderDoes, Rice and
  Wendler}{2025}]{fChange2025}
\begin{bmanual}[author]
\bauthor{\bsnm{VanderDoes},~\bfnm{Jeremy}\binits{J.}},
  \bauthor{\bsnm{Rice},~\bfnm{Gregory}\binits{G.}} \AND
  \bauthor{\bsnm{Wendler},~\bfnm{Martin}\binits{M.}}
(\byear{2025}).
\btitle{fChange: Functional Change Point Detection and Analysis}
\bnote{R package version 2.1.0}.
\bdoi{10.32614/CRAN.package.fChange}
\end{bmanual}
\endbibitem

\bibitem[\protect\citeauthoryear{Verzelen et~al.}{2023}]{verzelen2023optimal}
\begin{barticle}[author]
\bauthor{\bsnm{Verzelen},~\bfnm{Nicolas}\binits{N.}},
  \bauthor{\bsnm{Fromont},~\bfnm{Magalie}\binits{M.}},
  \bauthor{\bsnm{Lerasle},~\bfnm{Matthieu}\binits{M.}} \AND
  \bauthor{\bsnm{Reynaud-Bouret},~\bfnm{Patricia}\binits{P.}}
(\byear{2023}).
\btitle{Optimal change-point detection and localization}.
\bjournal{Ann. Statist.}
\bvolume{51}
\bpages{1586--1610}.
\bdoi{10.1214/23-aos2297}
\bmrnumber{4658569}
\end{barticle}
\endbibitem

\bibitem[\protect\citeauthoryear{Zhou and Dette}{2023}]{ZD23}
\begin{barticle}[author]
\bauthor{\bsnm{Zhou},~\bfnm{Zhou}\binits{Z.}} \AND
  \bauthor{\bsnm{Dette},~\bfnm{Holger}\binits{H.}}
(\byear{2023}).
\btitle{Statistical inference for high-dimensional panel functional time
  series}.
\bjournal{J. R. Stat. Soc. Ser. B. Stat. Methodol.}
\bvolume{85}
\bpages{523--549}.
\bdoi{10.1093/jrsssb/qkad015}
\bmrnumber{4718536}
\end{barticle}
\endbibitem

\end{thebibliography}


\begin{thebibliography}{16}

\bibitem[\protect\citeauthoryear{Bradley}{2005}]{Bra05}
\begin{barticle}[author]
\bauthor{\bsnm{Bradley},~\bfnm{Richard~C.}\binits{R.~C.}}
(\byear{2005}).
\btitle{Basic properties of strong mixing conditions. {A} survey and some open
  questions}.
\bjournal{Probab. Surv.}
\bvolume{2}
\bpages{107--144}.
\bnote{Update of, and a supplement to, the 1986 original}.
\bdoi{10.1214/154957805100000104}
\bmrnumber{2178042}
\end{barticle}
\endbibitem

\bibitem[\protect\citeauthoryear{Conway}{1990}]{Con90}
\begin{bbook}[author]
\bauthor{\bsnm{Conway},~\bfnm{John~B.}\binits{J.~B.}}
(\byear{1990}).
\btitle{A course in functional analysis},
\bedition{second} ed.
\bseries{Graduate Texts in Mathematics}
\bvolume{96}.
\bpublisher{Springer-Verlag, New York}.
\bmrnumber{1070713}
\end{bbook}
\endbibitem

\bibitem[\protect\citeauthoryear{Decker}{2024}]{Dec24}
\begin{bphdthesis}[author]
\bauthor{\bsnm{Decker},~\bfnm{David~Colin}\binits{D.~C.}}
(\byear{2024}).
\btitle{Hypothesis Tests for High-Dimensional Functional Data},
\btype{PhD thesis},
\bpublisher{University of Toronto},
\baddress{Toronto, Canada}
\bnote{Available at \url{http://hdl.handle.net/1807/140607}}.
\end{bphdthesis}
\endbibitem

\bibitem[\protect\citeauthoryear{Decker, Kong and Volgushev}{2025}]{many_means}
\begin{bmisc}[author]
\bauthor{\bsnm{Decker},~\bfnm{Colin}\binits{C.}},
  \bauthor{\bsnm{Kong},~\bfnm{Dehan}\binits{D.}} \AND
  \bauthor{\bsnm{Volgushev},~\bfnm{Stanislav}\binits{S.}}
(\byear{2025}).
\btitle{Simultaneous hypothesis testing for comparing many functional means}.
\end{bmisc}
\endbibitem

\bibitem[\protect\citeauthoryear{Dedecker and Louhichi}{2002}]{Dedecker2002}
\begin{bincollection}[author]
\bauthor{\bsnm{Dedecker},~\bfnm{J\'er\^ome}\binits{J.}} \AND
  \bauthor{\bsnm{Louhichi},~\bfnm{Sana}\binits{S.}}
(\byear{2002}).
\btitle{Maximal inequalities and empirical central limit theorems}.
In \bbooktitle{Empirical process techniques for dependent data}
\bpages{137--159}.
\bpublisher{Birkh\"auser Boston, Boston, MA}.
\bmrnumber{1958779}
\end{bincollection}
\endbibitem

\bibitem[\protect\citeauthoryear{Diestel and Uhl}{1977}]{DiestelUhl1977}
\begin{bbook}[author]
\bauthor{\bsnm{Diestel},~\bfnm{J.}\binits{J.}} \AND
  \bauthor{\bsnm{Uhl},~\bfnm{J.~J.}\binits{J.~J.} \bsuffix{Jr.}}
(\byear{1977}).
\btitle{Vector measures}.
\bseries{Mathematical Surveys}
\bvolume{No. 15}.
\bpublisher{American Mathematical Society, Providence, RI}
\bnote{With a foreword by B. J. Pettis}.
\bmrnumber{453964}
\end{bbook}
\endbibitem

\bibitem[\protect\citeauthoryear{Hsing and Eubank}{2015}]{TE15}
\begin{bbook}[author]
\bauthor{\bsnm{Hsing},~\bfnm{Tailen}\binits{T.}} \AND
  \bauthor{\bsnm{Eubank},~\bfnm{Randall}\binits{R.}}
(\byear{2015}).
\btitle{Theoretical foundations of functional data analysis, with an
  introduction to linear operators}.
\bseries{Wiley Series in Probability and Statistics}.
\bpublisher{John Wiley \& Sons, Ltd., Chichester}.
\bdoi{10.1002/9781118762547}
\bmrnumber{3379106}
\end{bbook}
\endbibitem

\bibitem[\protect\citeauthoryear{Li, Li and Shang}{2024}]{li2024detection}
\begin{barticle}[author]
\bauthor{\bsnm{Li},~\bfnm{Degui}\binits{D.}},
  \bauthor{\bsnm{Li},~\bfnm{Runze}\binits{R.}} \AND
  \bauthor{\bsnm{Shang},~\bfnm{Han~Lin}\binits{H.~L.}}
(\byear{2024}).
\btitle{Detection and estimation of structural breaks in high-dimensional
  functional time series}.
\bjournal{Ann. Statist.}
\bvolume{52}
\bpages{1716--1740}.
\bdoi{10.1214/24-aos2414}
\bmrnumber{4804825}
\end{barticle}
\endbibitem

\bibitem[\protect\citeauthoryear{Lu et~al.}{2022}]{lu2022almost}
\begin{barticle}[author]
\bauthor{\bsnm{Lu},~\bfnm{Jianya}\binits{J.}},
  \bauthor{\bsnm{Wu},~\bfnm{Wei~Biao}\binits{W.~B.}},
  \bauthor{\bsnm{Xiao},~\bfnm{Zhijie}\binits{Z.}} \AND
  \bauthor{\bsnm{Xu},~\bfnm{Lihu}\binits{L.}}
(\byear{2022}).
\btitle{Almost sure invariance principle of $\beta-$ mixing time series in
  Hilbert space}.
\bjournal{arXiv preprint arXiv:2209.12535}.
\end{barticle}
\endbibitem

\bibitem[\protect\citeauthoryear{Merlev\`ede, Peligrad and
  Utev}{1997}]{merlevede1997sharp}
\begin{barticle}[author]
\bauthor{\bsnm{Merlev\`ede},~\bfnm{Florence}\binits{F.}},
  \bauthor{\bsnm{Peligrad},~\bfnm{Magda}\binits{M.}} \AND
  \bauthor{\bsnm{Utev},~\bfnm{Sergey}\binits{S.}}
(\byear{1997}).
\btitle{Sharp conditions for the {CLT} of linear processes in a {H}ilbert
  space}.
\bjournal{J. Theoret. Probab.}
\bvolume{10}
\bpages{681--693}.
\bdoi{10.1023/A:1022653728014}
\bmrnumber{1468399}
\end{barticle}
\endbibitem

\bibitem[\protect\citeauthoryear{Pinelis}{1994}]{pinelis1994optimum}
\begin{barticle}[author]
\bauthor{\bsnm{Pinelis},~\bfnm{Iosif}\binits{I.}}
(\byear{1994}).
\btitle{Optimum bounds for the distributions of martingales in {B}anach
  spaces}.
\bjournal{Ann. Probab.}
\bvolume{22}
\bpages{1679--1706}.
\bmrnumber{1331198}
\end{barticle}
\endbibitem

\bibitem[\protect\citeauthoryear{Rio}{2017}]{rio2017asymptotic}
\begin{bbook}[author]
\bauthor{\bsnm{Rio},~\bfnm{Emmanuel}\binits{E.}}
(\byear{2017}).
\btitle{Asymptotic theory of weakly dependent random processes},
\bedition{french} ed.
\bseries{Probability Theory and Stochastic Modelling}
\bvolume{80}.
\bpublisher{Springer, Berlin}.
\bdoi{10.1007/978-3-662-54323-8}
\bmrnumber{3642873}
\end{bbook}
\endbibitem

\bibitem[\protect\citeauthoryear{Simon}{2005}]{simon2005trace}
\begin{bbook}[author]
\bauthor{\bsnm{Simon},~\bfnm{Barry}\binits{B.}}
(\byear{2005}).
\btitle{Trace ideals and their applications},
\bedition{second} ed.
\bseries{Mathematical Surveys and Monographs}
\bvolume{120}.
\bpublisher{American Mathematical Society, Providence, RI}.
\bdoi{10.1090/surv/120}
\bmrnumber{2154153}
\end{bbook}
\endbibitem

\bibitem[\protect\citeauthoryear{Talagrand}{1989}]{talagrand1989isoperimetry}
\begin{barticle}[author]
\bauthor{\bsnm{Talagrand},~\bfnm{Michel}\binits{M.}}
(\byear{1989}).
\btitle{Isoperimetry and integrability of the sum of independent {B}anach-space
  valued random variables}.
\bjournal{Ann. Probab.}
\bvolume{17}
\bpages{1546--1570}.
\bmrnumber{1048946}
\end{barticle}
\endbibitem

\bibitem[\protect\citeauthoryear{van~de Geer}{2000}]{vandeGeer2000}
\begin{bbook}[author]
\bauthor{\bparticle{van~de} \bsnm{Geer},~\bfnm{Sara~A.}\binits{S.~A.}}
(\byear{2000}).
\btitle{Applications of empirical process theory}.
\bseries{Cambridge Series in Statistical and Probabilistic Mathematics}
\bvolume{6}.
\bpublisher{Cambridge University Press, Cambridge}.
\bmrnumber{1739079}
\end{bbook}
\endbibitem

\bibitem[\protect\citeauthoryear{van~der Vaart and Wellner}{1996}]{VW96}
\begin{bbook}[author]
\bauthor{\bparticle{van~der} \bsnm{Vaart},~\bfnm{Aad~W.}\binits{A.~W.}} \AND
  \bauthor{\bsnm{Wellner},~\bfnm{Jon~A.}\binits{J.~A.}}
(\byear{1996}).
\btitle{Weak convergence and empirical processes}.
\bseries{Springer Series in Statistics}.
\bpublisher{Springer-Verlag, New York}
\bnote{With applications to statistics}.
\bdoi{10.1007/978-1-4757-2545-2}
\bmrnumber{1385671}
\end{bbook}
\endbibitem

\end{thebibliography}
\end{document}